\documentclass[a4paper,twoside,11pt,reqno]{amsart}
\usepackage{mathrsfs,mathtools,amssymb,latexsym,eucal,extarrows,caption,dsfont} 
\usepackage[headings]{fullpage}
\usepackage[dvips]{graphicx} 
\usepackage[usenames]{color}
\usepackage[T1]{fontenc} 
\usepackage[utf8]{inputenc}
\usepackage[shortlabels]{enumitem}
\usepackage{tikz-cd} 

\usepackage[pdfpagelabels, pdftex]{hyperref}
\definecolor{darkblue}{rgb}{0,0,.85} 
\definecolor{darkred}{rgb}{0.84,0,0}
\hypersetup{
  pdftitle={},
  pdfauthor={},
  pdfsubject={},
  pdfkeywords={},
  colorlinks=true,    
  linkcolor=darkred,     
  citecolor=darkblue,     
  filecolor=darkblue,      
  urlcolor=darkblue,       
  breaklinks=true,
  bookmarksopen=true,
  bookmarksnumbered=true,
  pdfpagemode=UseOutlines,
  plainpages=false}
\usepackage[capitalise]{cleveref} 

\usepackage[kerning=true,spacing=true]{microtype}
    \microtypecontext{spacing=nonfrench}

\usetikzlibrary{arrows.meta, positioning, decorations.markings,decorations.pathmorphing} 

\usepackage{scalerel,stackengine}
\stackMath
\newcommand\widerhat[1]{%
\savestack{\tmpbox}{\stretchto{%
  \scaleto{%
    \scalerel*[\widthof{\ensuremath{#1}}]{\kern-.6pt\bigwedge\kern-.6pt}%
    {\rule[-\textheight/2]{1ex}{\textheight}}
  }{\textheight}%
}{0.5ex}}%
\stackon[1pt]{#1}{\tmpbox}%
}

\DeclareRobustCommand{\SkipTocEntry}[5]{}

    \DeclareMathOperator{\rH}{H}

    \renewcommand{\mathbb}{\mathds}
    \newcommand{\bA}{{\mathbb A}}
    \newcommand{\bC}{{\mathbb C}}
    \newcommand{\bF}{{\mathbb F}}
    \newcommand{\bN}{{\mathbb N}}
    \newcommand{\bP}{{\mathbb P}}
    \newcommand{\bQ}{{\mathbb Q}}
    
    \newcommand{\bZ}{{\mathbb Z}}
    \newcommand{\bD}{{\mathbb D}}
    
    \newcommand{\cC}{{\mathcal C}}
    
    \newcommand{\cG}{{\mathcal G}}
    \newcommand{\cM}{{\mathcal M}}
    \newcommand{\cO}{{\mathcal O}}
    
    \newcommand{\cX}{{\mathcal X}}
    
    \newcommand{\cU}{{\mathcal U}}

    \newcommand{\fp}{{\mathfrak p}}
    \newcommand{\fq}{{\mathfrak q}}

    \newcommand{\fW}{{\mathfrak{W}}}
    \newcommand{\fX}{{\mathfrak{X}}}
    \newcommand{\fY}{{\mathfrak{Y}}}
    \newcommand{\fZ}{{\mathfrak{Z}}}

    \DeclareRobustCommand\longtwoheadrightarrow {\relbar\joinrel\twoheadrightarrow}
    \newcommand{\surj}{\longtwoheadrightarrow} 
   
    \newcommand{\inj}{\hookrightarrow}
    \newcommand{\isomto}{\xlongrightarrow{\,\smash{\raisebox{-0.65ex} 
    {\ensuremath{\displaystyle\sim}}}\,}}
    
    \newcommand{\la}{\longrightarrow} 
    \tikzset{ 
        open/.style = {decoration = {markings, mark = at position 0.5 with { \node[transform shape, scale = .7] {$\circ$}; } }, postaction = {decorate} }
    }

    \DeclareMathOperator{\id}{id}
    \DeclareMathOperator{\pr}{pr}
    \DeclareMathOperator{\Hom}{Hom}
    \DeclareMathOperator{\Ext}{Ext}
    \DeclareMathOperator{\Aut}{Aut}

    \DeclareMathOperator{\Gal}{Gal}
    \DeclareMathOperator{\Spec}{Spec}
    \DeclareMathOperator{\Spf}{Spf}
    \DeclareMathOperator{\Spa}{Spa}
    \DeclareMathOperator{\idem}{idem}
    \DeclareMathOperator{\dlog}{dlog}
            
    \newcommand{\et}{\text{\rm \'et}} 
    \newcommand{\op}{{\rm op}}
    \newcommand{\sep}{{\rm sep}} 
    \newcommand{\sh}{{\rm sh}}   
    \newcommand{\gp}{{\rm gp}} 
    
    \newcommand{\rh}{{\rm h}}  

    \newcommand{\rt}{{\rm t}}  

    \newcommand{\cat}[1]{\mathbf{#1}}    

    \newcommand{\Gsets}[1]{{#1\text{-}\sets}}   

    \newcommand{\FEt}{\cat{F\acute{E}t}}  
    \newcommand{\sets}{\cat{sets}} 
    \newcommand{\Sets}{\cat{Sets}}
    \newcommand{\Sch}{\cat{Sch}}

    \newcommand{\hZ}{{\widehat{\bZ}}}

    \newcommand{\ov}[1]{\overline{#1}}
    \newcommand{\vtheta}{\vartheta}
    \newcommand{\ph}{\varphi}
    \newcommand{\ep}{\varepsilon}

    \newcommand{\stacks}[2][Tag]{\cite[\href{https://stacks.math.columbia.edu/tag/#2}{#1~#2}]{StacksProject}}

    \newcommand{\typeVd}{(\mathrm{V}_{\rm div})}
    \newcommand{\pitame}{\pi_1^{\mathrm{t}}}
    \newcommand{\FEtt}{{\cat{F\acute{E}t}}^{\!\mathrm{t}}}
    \DeclareMathOperator{\Desc}{DD}
    \DeclareMathOperator{\spe}{sp}
    \newcommand{\rig}{{\rm rig}}

    \makeatletter
    \def\subsubsection{\@startsection{subsubsection}{3}%
      \z@{.5\linespacing\@plus.7\linespacing}{-.5em}%
      {\normalfont\bfseries}}
    \makeatother
    
    \crefformat{subsection}{\S#2#1#3}
    
    \crefname{prop}{Proposition}{Propositions}
    \crefname{lem}{Lemma}{Lemmas}
    \crefname{cor}{Corollary}{Corollaries}
    \crefname{defi}{Definition}{Definitions}
    \crefname{rmk}{Remark}{Remarks}
    \crefname{rmks}{Remarks}{Remarks}
    \crefname{nota}{Notation}{Notations}
    \crefname{setup}{Setup}{Setups}
    \crefname{ex}{Example}{Examples}
    \crefname{fact}{Fact}{Facts}

    \AddToHook{env/prop/begin}{\crefalias{thm}{prop}}
    \AddToHook{env/lem/begin}{\crefalias{thm}{lemma}}
    \AddToHook{env/cor/begin}{\crefalias{thm}{cor}}
    \AddToHook{env/defi/begin}{\crefalias{thm}{defi}}
    \AddToHook{env/rmk/begin}{\crefalias{thm}{rmk}}
    \AddToHook{env/rmks/begin}{\crefalias{thm}{rmks}}
    \AddToHook{env/nota/begin}{\crefalias{thm}{nota}}
    \AddToHook{env/setup/begin}{\crefalias{thm}{setup}}
    \AddToHook{env/ex/begin}{\crefalias{thm}{ex}} 
    \AddToHook{env/fact/begin}{\crefalias{thm}{fact}}
    
\makeatletter

\renewcommand{\tocsection}[3]{%
  \indentlabel{\@ifnotempty{#2}{\bfseries\ignorespaces#1 #2\quad}}\bfseries#3}

\renewcommand{\tocsubsection}[3]{%
  \indentlabel{\@ifnotempty{#2}{\ignorespaces#1 #2\quad}}#3}

\newcommand\@dotsep{4.5}
\def\@tocline#1#2#3#4#5#6#7{\relax
  \ifnum #1>\c@tocdepth 
  \else
    \par \addpenalty\@secpenalty\addvspace{#2}%
    \begingroup \hyphenpenalty\@M
    \@ifempty{#4}{%
      \@tempdima\csname r@tocindent\number#1\endcsname\relax
    }{%
      \@tempdima#4\relax
    }%
    \parindent\z@ \leftskip#3\relax \advance\leftskip\@tempdima\relax
    \rightskip\@pnumwidth plus1em \parfillskip-\@pnumwidth
    #5\leavevmode\hskip-\@tempdima{#6}\nobreak
    \leaders\hbox{$\m@th\mkern \@dotsep mu\hbox{.}\mkern \@dotsep mu$}\hfill
    \nobreak
    \hbox to\@pnumwidth{\@tocpagenum{\ifnum#1=1\bfseries\fi#7}}\par
    \nobreak
    \endgroup
  \fi}
\AtBeginDocument{%
\expandafter\renewcommand\csname r@tocindent0\endcsname{0pt}
}
\def\l@subsection{\@tocline{2}{0pt}{2.5pc}{5pc}{}}
\makeatother
    
    \newtheorem{thm}{Theorem}[subsection]
    \newtheorem{prop}[thm]{Proposition}
    \newtheorem{lem}[thm]{Lemma}
    \newtheorem{cor}[thm]{Corollary}
    \newtheorem{thmABC}{Theorem} 
    
    \theoremstyle{definition}
    \newtheorem{defi}[thm]{Definition}
    \newtheorem{defiABC}[thmABC]{Definition} 
    
    \newtheorem{rmk}[thm]{Remark}
    \newtheorem{rmks}[thm]{Remarks}
    
    \newtheorem{setup}[thm]{Setup}
    \newtheorem{ex}[thm]{Example}
    \newtheorem{fact}[thm]{Fact} 
    
    \numberwithin{equation}{subsection} 
    
\usepackage[style=alphabetic, 
    backend=biber, 
    natbib=true,
    url=false,
    doi=false, 
    isbn=false,
    eprint=false, maxbibnames=99,
    maxalphanames=6, minalphanames=3]{biblatex}
\DeclareFieldFormat[article, inbook, incollection, inproceedings, misc, thesis, unpublished]{title}{\textit{#1}}
\DeclareFieldFormat{journaltitle}{{\rmfamily #1}}
\DeclareFieldFormat{booktitle}{{\rmfamily #1}}
\DeclareFieldFormat{pages}{#1}
\renewrobustcmd*{\bibinitdelim}{\addnbthinspace{}}
\renewrobustcmd*{\bibnamedelima}{\addnbthinspace{}}
\renewrobustcmd*{\bibnamedelimd}{\addnbthinspace{}}

\DeclareBibliographyDriver{article}{%
  \usebibmacro{author/editor}%
  \newunit
  \usebibmacro{title}%
  \newunit
  \usebibmacro{journal}%
  \setunit*{\addspace}%
  \printtext{\textbf{\printfield{volume}}}%
  \setunit*{\addspace}%
  \printtext[parens]{\printfield{year}}%
  \setunit{\addcomma\space}%
  \iffieldundef{number}{}{\printtext{no.~\printfield{number}}}%
  \setunit{\addcomma\addspace}%
  \printfield{pages}%
  \iffieldundef{pubstate}{}{\addcomma\addspace\printfield{pubstate}%
  	 \iffieldundef{doi}{}{\addcomma\addspace\printfield{doi}}%
	 }%
  \usebibmacro{finentry}}
    
\begin{document}

\hrule width\hsize
\vskip 1cm

\title{Tame fundamental groups of rigid spaces} 

\author{Piotr Achinger}
\address{Piotr Achinger
\newline \indent
Instytut Matematyczny PAN, Śniadeckich 8, 00-656 Warsaw, Poland 
\newline \indent
Kyiv School of Economics, Shpaka 3, 02-000 Kyiv, Ukraine}
\email{pachinger@impan.pl}

\author{Katharina H\"ubner}
\address{Katharina H\"ubner
\newline \indent
Institut f\"ur Mathematik, Goethe--Universit\"at Frankfurt 
\newline \indent
Robert-Mayer-Stra\ss e~6--8, 60325~Frankfurt am Main, Germany}
\email{huebner@math.uni-frankfurt.de}

\author{Marcin Lara}
\address{Marcin Lara
\newline \indent
Instytut Matematyczny PAN, Śniadeckich 8, 00-656 Warsaw, Poland}
\email{marcin.lara@impan.pl}

\author{Jakob Stix}
\address{Jakob Stix
\newline \indent
Institut f\"ur Mathematik, Goethe--Universit\"at Frankfurt
\newline \indent
Robert-Mayer-Str.~6--8, 60325~Frankfurt am Main, Germany}
\email{stix@math.uni-frankfurt.de}

\newcommand{\grants}{The first author (PA) was supported by the project KAPIBARA funded by the European Research Council (ERC) under the European Union's Horizon 2020 research and innovation programme (grant agreement No 802787). The second to fourth authors (KH, ML, and JS) acknowledge support by Deutsche Forschungsgemeinschaft (DFG) through the Collaborative Research Centre TRR 326 \emph{Geometry and Arithmetic of Uniformized Structures} project number 444845124. The third author (ML) was later supported by the National Science Centre, Poland, grant number 2023/51/D/ST1/02294. For the purpose of Open Access, the author has applied a CC-BY public copyright licence to any Author Accepted Manuscript (AAM) version arising from this submission.}

\date{\today} 

\maketitle

\begin{quotation} 
    \noindent {\small {\bf Abstract} --- 
    We introduce the tame \'etale fundamental group $\pi_1^{\rm t}(X/K)$ of a rigid space $X$ over a non-archimedean field $K$. We show that if $X$ is qcqs and $K$ has topologically finitely generated tame Galois group (e.g.\ algebraically closed or a local field), then $\pitame(X/K)$ is topologically finitely generated. If $X$ is moreover the rigid generic fibre of a strictly semistable formal scheme such that the smooth locus of its special fibre admits a projective snc compactification, then $\pi_1^{\rm t}(X/K)$ is topologically finitely presented. The proofs rely on techniques of logarithmic geometry (extended beyond its usual scope of finitely generated monoids), in particular on an analogous finiteness statement for the tame log \'etale fundamental group, and on the  ``vertical compactification'' of a map of adic spaces.}
\end{quotation}

\hypersetup{linkcolor=black}
{\small \tableofcontents}
\hypersetup{linkcolor=darkred}

\section{Introduction}
\label{s:intro}

Rigid-analytic spaces over non-archimedean fields of mixed or positive characteristic have intricate local and global topology. For example, the \'etale fundamental group of the affinoid unit disc is not finitely generated\footnote{We call a profinite group finitely generated if it is topologically finitely generated, i.e.\ if it admits a finitely generated dense subgroup.}. For a smooth and proper rigid-analytic variety $X$ over $\mathbb{C}_p$, it is known thanks to Scholze \cite{Scholze} that the \'etale cohomology groups $\rH^*(X,\bZ/n\bZ)$ are finite, and it is expected\footnote{We learned of this question from Bogdan Zavyalov.} that its \'etale fundamental group is finitely generated, maybe even finitely presented. This is known to hold if $X$ is the analytification of an algebraic variety, but the general case of this question seems to require genuinely new ideas. 

The `culprit' responsible for the lack of finite generation in the affinoid case, and for the difficulty in the smooth and proper case, is the phenomenon of wild ramification. 
Indeed, a finite \'etale cover of a proper rigid-analytic variety may be wildly ramified along the special fibre of a formal model.
In this paper, we avoid this issue and develop a theory of tame fundamental groups $\pitame(X/K)$ of rigid spaces over a non-archimedean field $K$. We show that this invariant is reasonably well-behaved, most notably that it is finitely generated if $X$ is qcqs and $K$ is algebraically closed or a local field.

This will go in parallel to the story for schemes, which we will now recall briefly as a point of reference. Let $X$ be a connected scheme of finite type over an algebraically closed field $k$. A finite \'etale morphism $f\colon Y\to X$ is tame (or ``tamely ramified at infinity'') if for every $y\in Y$, the extension $k(y)/k(f(y))$ is tamely ramified with respect to every valuation of $k(y)$ which is trivial on $k$ (see \cref{sec:tame pi1} for a review of this and other notions of tameness). Tame covers form a Galois category corresponding to the tame \'etale fundamental group $\pitame(X/k)$, a quotient of the \'etale fundamental group $\pi_1(X)$. Finite generation of $\pitame(X/k)$ has been established in various parts by Grothendieck (if $X$ is a smooth curve \cite[XIII, Corollaire 2.12]{SGA1} or a proper variety \cite[X, Th\'eor\`eme 2.9]{SGA1}), Esnault--Kindler \cite{EsnaultKindler2016:LefschetzTheoremsTamely} (if $X$ admits a projective snc compactification), and in \cref{thm:naked pi1tame fg} of this paper in the general case. Recently, finite presentation of $\pitame(X/k)$ has been proved in \cite{EsnaultShustermanSrinivas2022:FinitePresentationTame} (if $X$ admits a projective snc compactification) and  \cite{LaraSrinivasStix2024:FundamentalGroupsProper} (if $X$ is proper), but remains an open question in general.

\addtocontents{toc}{\SkipTocEntry}
\subsection*{Main results}

Let $K$ be a non-archimedean field, i.e.\ a field complete with respect to a rank one valuation. By a rigid space over $K$ we shall mean an adic space locally of finite type over $\Spa(K)$. The difficulty in giving the right definition of tameness in the context of rigid spaces stems from subtleties surrounding the idea of ``boundary at infinity''. In fact, there are two natural notions of tameness, absolute and relative tameness, and it is the relative one that one should use in order to have hopes for finiteness, as we explain below.

Absolute tameness makes use of the fact that every point of an adic space is equipped with a valuation subring $k(x)^+\subseteq k(x)$ of its residue field. 
By definition (see \cite[Definition~3.3]{Hubner2021:AdicTameSite}), a finite \'etale morphism $f\colon Y\to X$ of adic spaces is tame if for every $y\in Y$, the extension of valued fields $(k(y), k(y)^+)/(k(f(y)), k(f(y))^+)$ is tamely ramified. Using this absolute notion of tame covers, we can define the \emph{absolute} tame fundamental group $\pitame(X)$ of a connected adic space $X$. 

However, the tame fundamental group $\pitame(X)$ of a connected rigid space $X$ over $K$ is not topologically finitely generated in general. For example, if $X$ is the affinoid unit disc over $K = \bC_p$, the surjective specialisation map of étale fundamental groups
\[ 
    \pi_1(X) \la \pi_1(\bA^1), \qquad \bA^1 = \Spec(\overline{\mathbb{F}}_p[T])
\]
factors through $\pitame(X)$, and as $\pi_1(\bA^1)$ is not finitely generated, neither is $\pitame(X)$. This is because, by definition of $X = \Spa(K\langle T\rangle)$, we are only testing tameness at valuations on $K\langle T\rangle$ for which $|T(x)|\leq 1$, or which have a centre on $\bA^1$. At the same time, the wild ramification expressed by the non-finite generation of $\pi_1(\bA^1)$ happens ``at infinity'' of $\bA^1$ and is witnessed by a valuation of $\overline{\mathbb{F}}_p[T]$ with $|T(x)|>1$. We refer to \cref{ex:unit-disc-Cp} for more details. In order to address this disparity, we need to add certain boundary points to $X$ corresponding to the \emph{infinity in the special fibre}. This example motivates our \emph{relative} notion of tameness:

\begin{defiABC}[{see \cref{def:tame-rig}}]
    A finite \'etale morphism of rigid spaces $f\colon Y\to X$ is {\bf tame relative to $K$} if for every $y\in Y$, the field extension $k(y)/k(f(y))$ is tamely ramified with respect to every valuation on $k(y)$ whose valuation ring $V$ satisfies $K^\circ\subseteq V\subseteq k(y)^+$.
\end{defiABC}

If $X$ is quasi-compact and separated, one can rephrase this condition using Huber's universal compactification \cite[\S 5.1]{HuberBook}: $Y\to X$ is tame relative to $K$ if the induced finite \'etale morphism $\ov{Y}\to\ov{X}$ of universal compactifications (relative to $K$)  is tame in the absolute sense defined previously (see \cref{prop:tame-S-compactification}). If $X$ is connected and non-empty, then covers $Y\to X$ which are tame relative to $K$ form a Galois category, leading to the \textit{tame relative to $K$} fundamental group $\pitame(X/K)$. 

This turns out to be a well-behaved invariant. For example, $\pitame(X/K)=1$ if $X$ is the unit disc and $K=\overline{K}$ (see \S\ref{ss:tamepi1-rig-fgfp} for further examples).  The following two theorems are our main results regarding $\pitame(X/K)$.

\begin{thmABC}[{see \cref{thm:fg-pi1-rig}}] \label{thmABC:analytic pi1tame fg}
    Suppose that the tame Galois group of $K$ is finitely generated (e.g.\ $K$ is algebraically closed or it is discretely valued and its residue field is either algebraically closed or finite). Let $X$ be a non-empty connected, quasi-compact and quasi-separated rigid-analytic space over $K$. 
    Then $\pitame(X/K)$ is finitely generated.
\end{thmABC}

\begin{thmABC}[{see \cref{thm:fp-sometimes}}] \label{thmABC:analytic pi1tame fp}
    Suppose that $K$ is algebraically closed and let $X$ be a non-empty connected rigid-analytic space over $K$. Suppose moreover that $X$ admits a strictly semistable formal model $\fX$ such that the smooth locus of the special fibre $\fX_k$ admits a projective snc compactification. Then $\pitame(X/K)$ is finitely presented. 
\end{thmABC}

Finite generation of the tame fundamental group is not only a desirable outcome by itself.
It is also crucial in establishing further properties of the tame fundamental group.
In upcoming work, we aim to prove a Künneth formula and invariance under algebraically closed base field extensions for the tame fundamental group.
A major input in our approach to the Künneth formula is the present result on finite generation.

\addtocontents{toc}{\SkipTocEntry}
\subsection*{Proof strategy}

Before we discuss the contents of the paper in detail, let us briefly explain the overall strategy. The proof of \cref{thmABC:analytic pi1tame fg} starts with a reduction, using desingularisation results due to Temkin and Gabber, to the case $K$ algebraically closed and $X$ admitting a semistable formal model $\fX$. The next step is to compare the tame fundamental group of $X/K$ with a suitably defined tame Kummer \'etale fundamental group $\pitame(\fX^{\log}_k/k)$ of the special fibre $\fX_k$ endowed with its natural log structure. This is difficult to do in part because the log structures in question do not satisfy the customary finiteness conditions, and also because it is unclear why the two notions of ``ramification at infinity'' for $X/K$ and for $\fX^{\log}_k/k$ should be compatible. Having done this step, we are left with showing that $\pitame(Y/k)$ is finitely generated for a certain class of log schemes $Y$ over $k$ with non-finitely generated log structures. Using a ``formal gluing'' type argument along the log stratification, this question is reduced to the case of locally constant log structures. In turn, this case is easily reduced to the case of trivial log structures: the finite generation of $\pitame(Y/k)$ for a scheme of finite type over $k$. Using alterations, finite generation follows from the case of smooth varieties admitting a projective snc compactification, known\footnote{The result form \cite{EsnaultKindler2016:LefschetzTheoremsTamely} exploits Lefschetz theorems for generic hyperplane sections, and, ironically, finally in the case of curves lifts back to characteristic $0$ and appeals to topology via the Riemann Existence Theorem.} thanks to \cite{EsnaultKindler2016:LefschetzTheoremsTamely}.

\addtocontents{toc}{\SkipTocEntry}
\subsection*{Tame fundamental groups of schemes}

The structure of our paper traces the above argument backwards. In \cref{sec:pit schemes} we treat tame fundamental groups of schemes. For a map of schemes $X\to S$ with $X$ connected and non-empty, one can define the tame fundamental group $\pitame(X/S)$ (\cref{defi:tamepi1 schemes}) by imposing tameness conditions on finite \'etale covers $Y\to X$ (\cref{defi:tame cover of schemes X over S}). Following the constructions in \cite{Temkin2011:RelativeRiemannZariskiSpacesa} and \cite{Hubner2021:AdicTameSite}, this is the tame fundamental group of the discretely ringed adic space $\Spa(X,S)$.
This definition is non-standard in the classical literature on tame covers but should be considered the most natural one.
In \cref{sec:adic tame discrete case} we compare it with existing notions of tameness and also report on improvements on the comparison theorems in \cite{KerzSchmidt2010:DifferentNotionsTameness} due to new developments in non-archimedean geometry, see \cite{HuebnerTemkin2026:WildLocus}.

Van Kampen/descent type arguments for Galois categories play a big role in our paper, and we review the relevant facts in \cref{sec:descent for Gal cats}, in particular, giving a descent criterion for finite generation and presentation of tame fundamental groups (\cref{cor:DD-fgfp}) that essentially goes back to \cite[IX, \S5]{SGA1}. 
These descent techniques find their first application in \cref{thm:naked pi1tame fg} (resp.~\cref{thm:naked pi1tame fp}) to show that the tame fundamental group of a non-empty connected scheme of finite type over an algebraically closed field is finitely generated (resp.~finitely presented under some extra conditions).
Finite generation and presentation of the tame fundamental group of a sufficiently nice variety was previously known due to \cite{EsnaultShustermanSrinivas2022:FinitePresentationTame} but the above described descent techniques allow us to deduce finite generation in full generality and finite presentation under some extra conditions that are far milder than smoothness.

In later sections the descent results from \cref{sec:descent for Gal cats} will also be crucial in treating the tame fundamental groups of log schemes (\cref{s:fg-pi1-log}) and adic spaces (\cref{s:fg-pi1-rig}).

\addtocontents{toc}{\SkipTocEntry}
\subsection*{Tame fundamental groups of log schemes}

In Section~\ref{s:fg-pi1-log} we turn to log schemes.
For later application in \cref{s:fg-pi1-rig} it is essential that we work with monoids that are not necessarily finitely generated.
More precisely, we will encounter log schemes of finite type over $\Spec(M \to k)$, where~$k$ is the residue field of an algebraically closed non-archimedean field $(K,K^\circ)$ and $M = K^\circ \cap K^\times$.
These log schemes are not fs, and accordingly this section relies on the extension of logarithmic geometry beyond fs log schemes in our companion paper \cite{AHLS-Log} (see the review in \S\ref{ss:log-review}), which in particular defines the (Kummer) \'etale fundamental group $\pi_1(X)$ of a saturated log scheme $X$. In this paper, we only need to consider log schemes of type $\typeVd$, i.e.\ those which locally admit a chart by a monoid which is finitely generated over a divisible valuative submonoid. Under this simplifying assumption, every Kummer \'etale map is locally of finite presentation as a map of schemes. In \cref{def:FKEt-pi1-tame} we introduce the tame fundamental group $\pitame(X/S)$ of a log scheme $X$ 
(with underlying scheme $\underline{X}$)
relative to a map of schemes $\underline{X}\to S$ by imposing tameness conditions on the finite (or integral) maps of schemes underlying Kummer \'etale covers. In case the log structure on $X$ is locally constant, we show in \cref{cor:torus-fibration-tame} that we have an exact sequence 
\[
        \begin{tikzcd}
            \Hom\big(\ov \cM^\gp_{X,\bar x} ,\hZ'(1)\big) \ar[r] & \pitame(X/S,\bar x) \ar[r] & \pitame(\underline{X}/S,\bar x) \ar[r] & 1.
        \end{tikzcd}
    \]
Thanks to our $\typeVd$ assumption, the group on the left is a finitely generated $\widehat{\bZ}'$-module (the divisible part of the monoid does not contribute), and we conclude that $\pitame(X/S)$ is finitely generated or finitely presented if and only if $\pitame(\underline{X}/S)$ has this property. If $X$ is a log scheme of type $\typeVd$ over an algebraically closed field $k$, then $X$ admits a dense open $U$ on which the log structure is locally constant. In \S\ref{ss:strat}, we employ this observation and a delicate formal gluing argument to show the following result. 

\begin{thmABC}[{see \cref{thm:pi1tame log is fg}}] \label{thmABC:log pi1tame fg}
    Let $X$ be a non-empty connected log scheme of type $\typeVd$ whose underlying scheme is of finite type over an algebraically closed field $k$.  Then $\pitame(X/k,\bar x)$ is finitely generated.
\end{thmABC}

\noindent 
We also show how $\pitame(X/k)$ can be computed in practice in some situations (smooth over a log point) by performing ``surgery'' that disassembles a log scheme into strata closures with their compactifying log structures (\S\ref{ss:surgery}). This gives sufficient conditions for $\pitame(X/k)$ to be finitely presented (\cref{cor:surgery-fp}), a crucial ingredient in the proof of \cref{thmABC:analytic pi1tame fp}. 

\addtocontents{toc}{\SkipTocEntry}
\subsection*{Vertical compactifications of adic spaces}

Section~\ref{s:compactifications} deals with vertical compactifications of adic spaces and can be read independently of the previous material. The goal is to build, for a map of adic spaces $X\to S$, a space $\Spa(X,S)$ of ``valuations on $X$ with a centre on $S$'' which exists under some mild conditions. If $X\to S$ is a map of analytic adic spaces (for example, $X$ is a rigid space over $K$ and $S=\Spa(K)$), then $\Spa(X,S)$ is the universal compactification constructed by Huber \cite[\S 5]{HuberBook}. If $X_0\to S_0$ is a map of schemes, then $\Spa(X_0^{\rm ad}, S_0^{\rm ad})$ is the discretely ringed adic space $\Spa(X_0,S_0)$ mentioned above. Our key case of interest, the space $\Spa(\fX^{\rm ad}, \Spa(K^\circ))$ for an adic formal scheme $\fX$ over $K^\circ$, interpolates between the two constructions. The space $\Spa(X,S)$ has been constructed by Solodov \cite{Solodov-vertical-compactification} following Huber's approach. We prove additional facts about it, and characterise its points as ``root triples'' $(x^\circ, k(x)^+, x^\succ)$ where $x^\circ\in X$ is a point without vertical generalisations, $k(x)^+$ is a valuation subring of $k(x^\circ)$ contained in its valuation ring $k(x^\circ)^+=k(x^\circ)^\circ$, and where $x^\succ$ is a centre of $k(x)^+$ on $S$ (see \cref{defi:root-triple}). These triples are the test objects used in our definition of tameness (\cref{def:tame-rig}).
We want to emphasize that although the definition of tameness only involves the underlying set of $\Spa(X,S)$, the arguments used in \cref{ss:tameness-comparison} rely on the adic space structure of $\Spa(X,S)$, which is the main novelty of this section compared to \cite{Solodov-vertical-compactification}.

\addtocontents{toc}{\SkipTocEntry}
\subsection*{Tame fundamental groups of rigid spaces} 

In Section~\ref{s:fg-pi1-rig} we finally turn our attention to rigid spaces over a non-archimedean field $K$. We introduce the tame fundamental group $\pitame(X/K)$ by imposing a pointwise tameness condition on root triples (\cref{def:tamepi1-rigid}). We show a variant of the fundamental exact sequence (\cref{prop:weak-fundamental-exact-seq}), whose existence allows us to reduce the proof of \cref{thmABC:analytic pi1tame fg} to the case when $K$ is algebraically closed. 
In \S\ref{ss:alterations-v-descent} we use the fact that surjections between qcqs rigid spaces are of descent for finite \'etale covers combined with  uniformisation theorems due to Gabber \cite{TravauxGabberIX} and Temkin \cite{Temkin2017:AlteredLocalUniformization} to reduce the proof further to the case when $X$ admits a strictly semistable formal model $\fX$. At this point, we are enabled to use log geometry, since $\fX$ endowed with the standard log structure (\cref{prop:log-sm-vertical}) becomes $\fX^{\log}$ which is of type $\typeVd$ and log smooth over $\Spf(K^\circ)^{\log}$.

The two most difficult steps in this
  section are the following. First, we show using approximation techniques and a lemma of Zavyalov (\cref{lem:Zavyalov}) that a finite \'etale map $Y\to X$ which is tame (in the absolute sense) extends uniquely to a Kummer \'etale map $\fY^{\log} \to \fX^{\log}$ (\cref{thm:Abhyankar-formal-more-precise}). In terms of fundamental groups, this means that we have an isomorphism $\pitame(X) \simeq \pi_1(\fX^{\log}_k)$ where $\fX^{\log}_k$ is the log special fibre of $\fX^{\log}$ (see \cref{cor:generic fibre and eta normalisation} and \cref{prop:FEt-formal-vs-sp-fibre}). 
  Second, we show that this isomorphism induces an isomorphism (\cref{prop:tameness-comparison} and \cref{cor:Abhy-infty-equiv})
\[
    \pitame(X/K) \simeq \pitame(\fX^{\log}_k/k)
\]
of the tame quotients.
The argument for this is delicate, relying on the properties of the vertical compactification in \cref{s:compactifications} and a ``cdh uniformisation property'' (\cref{def:cdh-unif}) of the underlying schemes of certain log schemes. Given the above isomorphism and the aforementioned reductions, \cref{thmABC:analytic pi1tame fg} now follows from \cref{thmABC:log pi1tame fg}, and \cref{thmABC:analytic pi1tame fp} from \cref{cor:surgery-fp}.

\addtocontents{toc}{\SkipTocEntry}
\subsection*{Acknowledgements}

First and foremost, we thank IMPAN at Sopot for excellent working conditions during our stay in September 2023 when this project started. We also thank Heidelberg University for hosting us during an intense few weeks in February 2024. We thank Paul Alexander Helminck, Ben Heuer, Michael Temkin, Alex Youcis, and Bogdan Zavyalov for valuable discussions.

\addtocontents{toc}{\protect\SkipTocEntry}
\subsection*{Funding}

\grants

\addtocontents{toc}{\SkipTocEntry}
\section*{Notation and conventions}

\begin{itemize}[leftmargin=*]
    \item 
    If $X$ is a (formal) scheme, or an adic space, we denote by $\FEt_X$ the category of finite \'etale morphisms to $X$. 
    \item 
    Monoids are commutative, in additive notation, and with~$0$.
    \item 
    We denote log schemes by single symbols e.g.~$X$ (not $(X, \cM_X)$) and the underlying scheme is denoted by $\underline{X}$. For log schemes and maps between them, we use adjectives like regular, smooth, \'etale to mean the respective notions in log geometry. In particular, if $X$ is a log scheme of type $\typeVd$, we denote by $\FEt_X$ the category of finite Kummer \'etale maps to $X$, see \cref{def:Kummer-etale-map}. See \S\ref{ss:log-review} for a review of log geometry as used in this paper.

    \item 
    By \emph{locally} we mean working locally in the \'etale topology; in the logarithmic context this means in the strict \'etale topology, i.e.~the \'etale topology after forgetting the log structure. We nevertheless sometimes say \emph{\'etale locally} to stress the use of the \'etale topology.
    
    \item 
    If $X$ is a non-empty connected (log) (formal) scheme or an adic space, we denote by $\pi_1(X,\bar x)$ the fundamental group of the Galois category $\FEt_X$ with respect to the fibre functor induced by a given geometric (log) point $\bar x\to X$. In particular, we never use the notation $\pi_1^{\log}$ for the fundamental group of a log scheme.
    
    \item 
    For a Huber pair $(A,A^+)$ we denote by $A^\circ\subseteq A$ the ring of powerbounded elements and use $\Spa(A)$ as an abbreviation for $\Spa(A,A^\circ)$. 
    
    \item 
    By a \emph{non-archimedean field} we mean a field $K$ which is complete with respect to a rank one valuation. We denote its valuation ring by $K^\circ=K^+$. A \emph{rigid space over $K$} is an adic space locally of finite type over $\Spa(K)$. See \S\ref{ss:adic-spaces} for our conventions regarding adic spaces, rigid spaces, and formal schemes. 
    
    \item 
    For a point $x$ of an adic space $X$, we denote by $k(x)$ the residue field of $\cO_{X, x}$, by $k(x)^+\subseteq k(x)$ the corresponding valuation ring (so that $\cO_{X,x}^+ \subseteq \cO_{X,x}$ is the preimage of $k(x)^+$), and by $k(x)^\succ$ the residue field of $k(x)^+$ (or equivalently of $\cO_{X,x}^+$), which we call the \emph{specialisation field} of $x$. The last convention is not standard, but it seems that the ``residue field of the residue field'' does not have an established name or notation in the literature. 
        
    \item We call a profinite group finitely generated (resp.~presented) if it is topologically finitely generated (resp.~presented).
    
    \item 
    Once we work over a fixed algebraically closed field $k$ (or over a valuation ring $K^\circ$ with residue field $k$), we denote by $\hZ{}'(1) = \varprojlim \mu_n(k)$ the Tate module of $k^\times$. If $p$ is the characteristic exponent of $k$, the group $\hZ{}'(1)$ is isomorphic to $\hZ' = \prod_{\ell\neq p}\bZ_\ell$, the prime-to-$p$ completion of $\bZ$.

    \item 
    The fundamental group of a monoid $P$ is defined  as $\pi_1(P) = \Hom(P^{\rm gp}, \widehat{\bZ}'(1))$ in \cref{defi:pi1-of-monoid}.
    \item 
    For a field $K$ we denote its absolute Galois group by $\pi_1(K)$ where the omitted base point is some chosen algebraic closure $\ov{K}$ of $K$ so that $\pi_1(K) = \Gal(K^\sep/K)$ with the separable closure $K^\sep$ of $K$ in $\ov{K}$. If $(K,K^+)$ is a henselian valued field, then we denote the maximal tamely ramified quotient of the absolute Galois group by $\pitame(K)$.

    \item For a map of schemes $X\to S$, an \emph{snc compactification} is a factorisation $X\to \overline{X}\to S$ where $\overline{X}\to S$ is smooth and proper and where $X\to \overline{X}$ is an open immersion whose complement is the support of an effective Cartier divisor $D\subseteq \ov{X}$ which is snc (simple normal crossings) relative to $S$, i.e.\ Zariski locally on $\overline{X}$, the divisor $D$ is the preimage of the union of coordinate hyperplanes under an \'etale map to $\bA^n_S$.

    \item For an adic formal scheme $\fX$ over $K^\circ$, we denote by $\fX^{\rm log}$ the log formal scheme obtained by endowing $\fX$ with the \emph{standard log structure} $\cM_\fX^{\rm std}=\cO_\fX\times_{\cO_\fX\otimes K} (\cO_\fX\otimes K)^\times$, i.e.\ the one induced by the rigid generic fibre (see \S\ref{ss:log-formal-schemes}). 

\end{itemize}

\section{Tame fundamental groups of schemes}
\label{sec:pit schemes}

The goal of this section is to study the tame fundamental group of a connected scheme~$X$ relative to a base scheme $S$.
It is defined as the automorphism group of a fibre functor of the Galois category of tame covers of~$X$.
Accordingly, we start with a discussion of tame extensions of valued fields (\cref{sec:tame extensions valued fields}) and then of tame covers of schemes including a comparison of different notions of tameness found in the literature (\cref{sec:adic tame discrete case}).
The definition of tameness we will be using is the maximal one, which imposes tameness conditions on all valuations with a centre on $S$ of every residue field of~$X$. 
After a review of Galois categories in \cref{sec:Gal cats}, we are in the position to define the tame fundamental group in \cref{sec:tame pi1}.
The final results of this section are \cref{thm:naked pi1tame fg} and \cref{thm:naked pi1tame fp} showing that the tame fundamental group of a connected scheme of finite type over an algebraically closed field~$k$ is finitely generated and even finitely presented if it satisfies additional regularity assumptions.
We deduce these theorems from existing results \cite{EsnaultShustermanSrinivas2022:FinitePresentationTame} on finite generation and presentation of the tame fundamental group of a smooth variety with a projective snc compactification.
The reduction to the smooth case is done by a descent argument, for which we lay foundations in \cref{sec:descent for Gal cats}.

\subsection{Tame extensions of valued fields} 
\label{sec:tame extensions valued fields}

In this opening subsection we review tame extensions of valued fields, which form the basis of any discussion of tame phenomena in arithmetic geometry.
We give various characterisations of tameness and study the Galois group of the maximal tame extension.
Moreover, we explain how to describe tame extensions of compositions of valuations.

Throughout we will use the notation $(K,K^+)$ for a valued field with valuation ring~$K^+$ and write $K^\succ$ for the residue field of~$K^+$, which we also call the \textbf{specialisation field} of~$K$.
We denote the corresponding value group by~$\Gamma_K$, for which we use multiplicative notation, so the valuation of~$K$ is a multiplicative map
\[
 \lvert \cdot \rvert \colon K \longrightarrow \Gamma_K \cup \{0\}
\]
sending~$1$ to~$1$ and~$0$ to~$0$ and satisfying the strong triangle inequality.

For the valued field $(K,K^+)$ we fix an algebraic closure $(\overline{K},\overline{K}^+)$.
Then we have intermediate extensions of valued fields
\[
 (K,K^+) \subseteq (K^\rh,K^{\rh,+}) \subseteq (K^\sh,K^{\sh,+}) \subseteq (\overline{K},\overline{K}^+),
\]
where $(K^\rh,K^{\rh,+})$ is the {\bf henselisation}, and $(K^\sh,K^{\sh,+})$ is the {\bf strict henselisation} (or {\bf maximal unramified extension}).
Accordingly, $(K,K^+)$ is {\bf henselian} (resp.\ {\bf strictly henselian}) if and  only if $(K,K^+) = (K^\rh,K^{\rh,+})$ (resp.\ $(K,K^+) = (K^\sh,K^{\sh,+})$).

\begin{defi}
\label{defi:tame-extension-of-valued-fields}
    A finite extension of valued fields $(K,K^+) \to (L,L^+)$ is called \textbf{tame} if the residue characteristic does not divide the degree $[L^\sh:K^\sh]$.
    An algebraic extension of valued fields is \textbf{tame} if all of its finite subextensions are tame.
\end{defi}

From the definition it is clear that compositions and base changes of tame extensions are tame.
Consequently, there is a \textbf{maximal tamely ramified} extension $(K^t,K^{t,+})$ of $(K,K^+)$ (also called \textbf{tame henselisation}) complementing the above tower:
\[
 (K,K) \subseteq (K^\rh,K^{\rh,+}) \subseteq (K^\sh,K^{\sh,+}) \subseteq (K^t,K^{t,+}) \subseteq (\overline{K},\overline{K}^+).
\]
The valued field $(K,K^+)$ is called {\bf tamely henselian} if $(K,K^+) = (K^t,K^{t,+})$.
In the literature, the term {\bf tame closed} is often used for tamely henselian.

The structure of tame extensions is best understood via Galois theory.
Suppose that $(L,L^+)/(K,K^+)$ is a Galois extension of valued fields.
Its inertia group~$I$ is naturally identified with $\Gal(L^\sh/K^\sh)$.
We can now understand tame extensions in terms of the Kummer pairing
\[
 I \times (\Gamma_L/\Gamma_K) \longrightarrow \mu(L^\succ) \subseteq \bQ/\bZ'(1) 
\]
that maps $(\sigma,\gamma) \in I \times \Gamma_L$ to the residue class of $\sigma(a)/a$ in $(L^\succ)^\times$, where $a \in L$ is an element with $\lvert a \rvert = \gamma$.
It induces a surjective group homomorphism (the tame character)
\begin{equation} \label{Kummer-hom}
 I \longrightarrow \Hom(\Gamma_L/\Gamma_K,\bQ/\bZ'(1)), 
\end{equation}
whose kernel~$I_p$, called wild inertia, is a $p$-group, where $p \ge 0$ is the residue characteristic of~$K$ (see \cite[Proposition~6.2.12]{GabberRamero}).
Applying these considerations to the extension $K^t/K^\rh$, we obtain the following structural result for the Galois group of the maximal tame extension

\begin{lem} \label{lem:max tame extension}
    Suppose that $(K,K^+)$ is henselian of residue characteristic $p \ge 0$.
    \begin{enumerate}[ref=(\arabic*)]
        \item \label{lemitem:tame galois sequence}
        There is a short exact sequence
        \[
        1 \la \Hom(\Gamma_K,\hZ'(1)) \la \Gal(K^t/K) \la \Gal(K^\sh/K) \la 1
        \]
        where the map on the left comes from the Kummer theory pairings for $p \nmid n$:
        \[
        \Gal(K^t/K^\sh) \times \Gamma_{K^t}/\Gamma_K \la \bQ/\bZ'(1), 
        \qquad (\sigma,\gamma) \mapsto [\sigma(a)/a]
        \]
        with $a \in K^t$ of valuation $\gamma$, and $[\sigma(a)/a]$ is the residue class of $\sigma(a)/a$ in $(K^{t,\succ})^\times$.
  
        \item \label{lemitem:tame galois split}
        The sequence in \labelcref{lemitem:tame galois sequence} is split and 
        $\Gal(K^\sh/K)$ is canonically isomorphic to the absolute Galois group $\Gal_{K^\succ}$ of the residue field.
    \end{enumerate}
\end{lem}
\begin{proof}
    We first note that 
    \[
        \Hom(\Gamma_K,\hZ'(1)) = \varinjlim_{p \mid n} \Hom\big(\Gamma_K/n\Gamma_K,\mu_n(K^{t,\succ})\big) = \Hom(\Gamma_{K^t}/\Gamma_K,\bQ/\bZ'(1)).
    \]
     Since~$K$ is henselian, the extension $K^\sh/K$ is Galois with group isomorphic to $\Gal_{K^\succ}$.
    The kernel of the surjection $\Gal(K^t/K) \twoheadrightarrow \Gal(K^\sh/K) = \Gal_{K^\succ}$ is the inertia group of the extension $K^t/K$ by definition.
    Via the Kummer pairing it identifies with $\Hom(\Gamma_K,\hZ'(1))$ because the wild inertia group is trivial (see the discussion above).
    This accounts for the exact sequence in \labelcref{lemitem:tame galois sequence}.
    That the sequence splits follows from  \cite{MelnikovTavgen85} or  \cite[Theorem~2.2]{KuhlmannPankRoquette}.
\end{proof}  

\begin{cor}
\label{cor:pitame injective for fields}
    Let $K'/K$ be a finite extension of henselian valued fields. Then the induced map $\Gal(K'^t/K') \to \Gal(K^t/K)$ is an isomorphism onto an open subgroup.
\end{cor}

\begin{proof}
We compare both tame fundamental groups by means of the short exact sequence in \cref{lem:max tame extension}~\labelcref{lemitem:tame galois sequence}.
We first note that $K'K^\sh$ agrees with the strict henselisation of $K'$  since $K'K^\sh$ is henselian with separably closed specialisation field and unramified over~$K'$.
Hence, the induced map of the unramified (strongly \'etale) quotients yields an isomorphism onto an open subgroup.
\[
\Gal(K'^\sh/K') \isomto \Gal(K^\sh/K'\cap K^\sh) \subseteq \Gal(K^\sh/K).
\]

For the comparison of tame inertia groups, we note that the inclusion of value groups $\Gamma_K \inj \Gamma_{K'}$
is of finite index $e_{K'/K} \mid [K':K]$. It follows that $\Ext^1(\Gamma_{K'}/\Gamma_K, \hZ'(1))$ is finite and so the induced map 
\[
\Hom(\Gamma_K,\hZ'(1)) \la \Hom(\Gamma_{K'},\hZ'(1))
\]
is injective with open image of finite index. The claim now follows from  the diagram
\[
 \begin{tikzcd}
    1 \ar[r]    & \Hom(\Gamma_{K'},\hZ'(1)) \ar[r] \ar[d,hookrightarrow]   & \Gal(K'^t/K') \ar[r]    \ar[d]  & \Gal(K'^\sh/K') \ar[r]  \ar[d,hookrightarrow]   & 1 \\
    1 \ar[r]    & \Hom(\Gamma_K,\hZ'(1)) \ar[r]   & \Gal(K^t/K) \ar[r]    & \Gal(K^\sh/K) \ar[r]  & 1. 
 \end{tikzcd}
\]
and an application of the snake lemma.
\end{proof}

Also from Kummer theory we obtain Abhyankar's lemma for tame extensions of valued fields:

\begin{lem} \label{lem:tameness-conditions-galois}
    Let $(L,L^+)/(K,K^+)$ be a finite Galois extension of residue characteristic $p \ge 0$.
    The following conditions are equivalent:
    \begin{enumerate}[(a)]
        \item 
        \label{lemitem:tame}
        The extension $(K,K^+) \to (L,L^+)$ is tame.
        \item
        \label{lemitem:tame-galois}
        The wild inertia group~$I_p$ is trivial and the inertia group~$I$ is abelian and via the Kummer pairing isomorphic to $\Hom(\Gamma_L/\Gamma_K,\bQ/\bZ'(1))$.
        \item
        \label{lemitem:tame-abhyankar}
        There exist a finite unramified extension $(K',K'^+)/(K,K^+)$, elements $a_1, \ldots, a_r\in K'$, and an integer $n\geq 1$ not divisible by the residue characteristic such that we have an inclusion of extensions of $K$
        \[ 
            L \subseteq K'(\sqrt[n]{a_1}, \ldots, \sqrt[n]{a_r}).
        \]
    \end{enumerate}
\end{lem}

\begin{proof}
    \labelcref{lemitem:tame} implies \labelcref{lemitem:tame-galois} since $I_p \subseteq I$ is a $p$-group and~$I$ has order prime to~$p$ by assumption.
    
    If~\labelcref{lemitem:tame-galois} holds, the homomorphism (\ref{Kummer-hom}) is an isomorphism as explained above.
    Consequently, the extension $L^\sh/K^\sh$ is abelian and unravelling the definition of the Kummer pairing we see that an element $\gamma \in \Gamma_L/\Gamma_K$ of order~$n$ `corresponds' to the extension $K^\sh(x)/K^\sh$, where $x \in L^\sh$ is such that $\lvert x \rvert = \gamma$ and $x^n \in K^\sh$.
    In total we get that $L^\sh = K^\sh(a_1^{1/q_1},\ldots,a_k^{1/q_r})$  for positive integers~$q_i$ prime to~$p$ and $a_i \in K^\sh$ (see \cite[Corollary 6.2.14]{GabberRamero} for a detailed argument).
    Taking for~$n$ the least common multiple of the~$q_i$, we obtain inclusions
    \[
      L \subseteq L^\sh \subseteq K^\sh(\sqrt[n]{a_1}, \ldots, \sqrt[n]{a_r}).      
    \]
    Since~$L/K$ is finite, $L$ is already contained in $K'(\sqrt[n]{a_1}, \ldots, \sqrt[n]{a_r})$ for a finite subextension~$K'$ of $K^\sh/K$ with $a_i \in K'$.
    This shows the claim of \labelcref{lemitem:tame-abhyankar}.
   
    Finally, from~\labelcref{lemitem:tame-abhyankar} we conclude
    \[
     L^\sh \subseteq K'(\sqrt[n]{a_1}, \ldots, \sqrt[n]{a_r})^\sh = K^\sh(\sqrt[n]{a_1}, \ldots, \sqrt[n]{a_r})
    \]
    and the latter field is abelian of degree prime to~$p$ over~$K^\sh$.
    This shows~\labelcref{lemitem:tame}.
\end{proof}

For extensions that are not necessarily Galois we deduce the following characterisations.

\begin{lem} \label{lem:tameness-conditions}
    Let $(K,K^+) \to (L,L^+)$ be a finite extension of valued fields of residue characteristic~$p$. 
    The following conditions are equivalent: 
    \begin{enumerate}[(a)]
        \item 
        \label{lemitem:tame1}
        The extension $(K,K^+) \to (L,L^+)$ is tame.
        \item
        \label{lemitem:tame-abhyankar1}
        There exist a finite unramified extension $(K',K'^+)/(K,K^+)$, elements $a_1, \ldots, a_r\in K'$, and an integer $n\geq 1$ not divisible by the residue characteristic such that we have an inclusion of extensions of $K$
        \[ 
            L \subseteq K'(\sqrt[n]{a_1}, \ldots, \sqrt[n]{a_r}).
        \]
        \item 
        \label{lemitem:tame-numerical}
        The ramification index $e_{L/K} = \lvert \Gamma_L/\Gamma_K \rvert$ is not divisible by~$p$, the specialisation field extension $L^\succ/K^\succ$ is separable (of degree $f_{L/K}$), and the extension is defectless: the degree $[(L,L^+)^{\rh}:(K,K^+)^{\rh}]$ of the extension of the respectively henselised fields agrees with $e_{L/K} \cdot f_{L/K}$.
    \end{enumerate}
\end{lem}

\begin{proof}
    If \labelcref{lemitem:tame1} holds, the extension $L^\sh/K^\sh$ is abelian by \cref{lem:tameness-conditions-galois}~\labelcref{lemitem:tame-galois}.
    Applying the characterisation 
    \cref{lem:tameness-conditions-galois}~\labelcref{lemitem:tame-abhyankar} to $L^\sh/K^\sh$, we get $a_1,\ldots,a_r \in K^\sh$ and an 
    $n \ge 1$ such that
    \[
        L^\sh \subseteq K^\sh(\sqrt[n]{a_1}, \ldots, \sqrt[n]{a_r}).
    \]
    Choosing a suitable finite subextension $K'$ of $K^\sh/K$ as in the proof of \cref{lem:tameness-conditions-galois} we obtain~\labelcref{lemitem:tame-abhyankar1}. 

    From~\labelcref{lemitem:tame-abhyankar1} we deduce that~$e_{L/K}$ divides $n^r$, which is prime to~$p$.
    Regarding residue field extensions, we know that
    \[
        L^\succ \subseteq L^{\sh\succ} \subseteq K^\sh(\sqrt[n]{a_1}, \ldots, \sqrt[n]{a_r})^\succ = K^{\sh\succ} = K^{\succ\sep}.
    \]
    Moreover, $L^\rh$ is contained in $K'^\rh(\sqrt[n]{a_1}, \ldots, \sqrt[n]{a_r})$, which is a defectless extension of~$K^\rh$, so~$L^\rh$ is also defectless.

    If~\labelcref{lemitem:tame-numerical} holds, then $[L^\sh:K^\sh] = e_{L/K}$ is prime to~$p$ and hence $L/K$ is tame, i.e.\ \labelcref{lemitem:tame1} holds.
\end{proof}

\subsubsection*{Tameness of composite valuations}

In \cref{ss:tameness-comparison-pre} we will need to investigate tame covers of the universal compactification of a rigid space.
The relevant valuations are naturally compositions of two valuations.
As a preparation we study tame extensions of composite valuations.
Our aim is to express the tameness condition in terms of tameness of the composite parts.

\begin{defi}
    Two valued fields $(K_1,K_1^+)$ and $(K_2,K_2^+)$ are said to be \textbf{composable} if $K_1^\succ = K_2$.
    In this case their \textbf{composition} is the valued field $(K,K^+)$ where $K = K_1$ and $K^+$ is the preimage of~$K_2^+$ under the projection $K_1^+ \surj K_1^\succ = K_2$.
    We use the following notation for the composition:
    \[
        (K,K^+) = (K_1,K_1^+) \succ (K_2,K_2^+).
    \]
\end{defi}

The ring $K^+$ in the above definition is indeed a valuation ring of $K$ because for all $x \in K^\times$ we have $x \in K^+$ or $x^{-1} \in K^+$.
The composition $(K,K^+) = (K_1,K_1^+) \succ (K_2,K_2^\succ)$ is illustrated by the following cartesian diagram.

\[
     \begin{tikzcd}
         K = K_1 & K_1^+ \ar[d,twoheadrightarrow] \ar[l,hook']  & K^+ \ar[d,twoheadrightarrow] \ar[l,hook'] \\
         & K_1^\succ = K_2 & K_2^+ \ar[l,hook'] \ar[d,twoheadrightarrow] \\
         & & K_2^\succ = K^\succ.
     \end{tikzcd}
\]

\begin{prop} \label{henselisation-composition}
    Let $(K,K^+) = (K_1,K_1^+) \succ (K_2,K_2^+)$ be a composition of valued fields such that~$K^\succ = K_2^\succ$ is of characteristic $p \ge 0$.
    \begin{enumerate}[ref=(\arabic*)]
        \item \label{lemitem:hens-comp1}
        $(K,K^+)$ is henselian if and only if $(K_1,K_1^+)$ and $(K_2,K_2^+)$ are henselian.
        \item \label{lemitem:hens-comp2}
        $(K,K^+)$ is strictly henselian if and only if $(K_1,K_1^+)$ is henselian and $(K_2,K_2^+)$ is strictly henselian.
        
        \item \label{lemitem:hens-comp3}
        $(K,K^+)$ is tamely henselian 
        if and only if $(K_1,K_1^+)$ is henselian with value group divisible by all primes except possibly~$p$ and $(K_2,K_2^+)$ is tamely henselian.
    \end{enumerate}
\end{prop}

\begin{proof}
    Assertion~\labelcref{lemitem:hens-comp1} has been shown in \cite[Chapter~F, Propositions~9 and 10]{Ribenboim:Theorie-valuations}. 
    From this~\labelcref{lemitem:hens-comp2} follows immediately as a henselian local ring is strictly henselian if and only if its residue field is separably closed and the residue fields of $K^+$ and~$K_2^+$ are the same.
    
    For~\labelcref{lemitem:hens-comp3} remember that a valued field of residue characteristic $p > 0$ is tamely henselian if and only if it is strictly henselian and its value group is divisible by all primes except~$p$.
    The value group $\Gamma_{K_2}$ is a convex subgroup of~$\Gamma_K$ with quotient group~$\Gamma_{K_1}$.
    Therefore, $\Gamma_K$ is divisible by a prime $\ell \ne p$ if and only if $\Gamma_{K_1}$ and~$\Gamma_{K_2}$ are.
    This shows the equivalence.
\end{proof}

\begin{prop} \label{composition-tame}
    Let $(L,L^+)/(K,K^+)$ be an algebraic extension of valued fields of residue characteristic $p \ge 0$.
    Suppose that $(L,L^+) = (L_1,L_1^+) \succ (L_2,L_2^+)$ and set $(K_1,K_1^+) := (K,L_1^+ \cap K)$ and $(K_2,K_2^+) := (K_1^\succ,L_2^+ \cap K_1^\succ)$ so that $(K,K^+) = (K_1,K_1^+) \succ (K_2,K_2^\succ)$.
    Then
    \begin{enumerate}[ref=(\arabic*)]
        \item \label{lemitem:composition etale}
        $(L,L^+)/(K,K^+)$ is unramified if and only if $(L_1,L_1^+)/(K_1,K_1^+)$ and $(L_2,L_2^+)/(K_2,K_2^+)$ are unramified.
        
        In particular, the henselisation and strict henselisation of $(K,K^+)$ decompose as
        \begin{align*}
            (K^\rh,K^{\rh,+}) &= (K'_1,{K'_1}^+) \succ (K_2^\rh,K_2^{\rh,+}),   \\
            (K^\sh,K^{\sh,+}) &= (K''_1,{K''_1}^+) \succ (K_2^\sh,K_2^{\sh,+}),
        \end{align*}
        where $(K'_1,{K'_1}^+)$ and $(K''_1,{K''_1}^+)$ are the unique unramified extensions of $(K_1^\rh,K_1^{\rh,+})$ with residue fields $K_2^\rh$ and $K_2^\sh$, respectively.
        \item \label{lemitem:composition tame}
        $(L,L^+)/(K,K^+)$ is tame if and only if $(L_1,L_1^+)/(K_1,K_1^+)$ and $(L_2,L_2^+)/(K_2,K_2^+)$ are tame and the ramification index of $(L_1,L_1^+)/(K_1,K_1^+)$ is prime to~$p$ \footnote{It might happen that the residue characteristic of $(K_1,K_1^+)$ is~$0$ and the residue characteristic of $(K_2,K_2^+)$ is $p > 0$.
        In this case extracting $p$-th roots in $(K_1,K_1^+)$ is tame but does not lead to a tame extension of $(K,K^+)$.
        This is the reason for the extra condition on the ramification index.}.

        In particular, the tame henselisation of $(K,K^+)$ decomposes as
        \[
            (K^t,K^{t,+}) = (K'_1,{K'_1}^+) \succ (K_2^t,K_2^{t,+}),
        \]
         where $(K'_1,{K'_1}^+) \subseteq (K_1^t,K_1^{t,+})$ is the minimal extension of $(K_1^\rh,K_1^{\rh,+})$ with residue field~$K_2^t$ and value group divisible by all primes except possibly~$p$.
    \end{enumerate}
\end{prop}

\begin{proof}
    Let $\fp \subseteq K^+$ be the maximal ideal of~$K_1^+$, so that $K_1^+ = K^+_\fp$ and $K_2^+ = K^+/\fp$.
    
    In order to prove~\labelcref{lemitem:composition etale} suppose that $(L,L^+)/(K,K^+)$ is unramified.
    Then $\fp_L := \fp L^+$ is the maximal ideal of~$L_1^+$.
    Moreover, $L_1^+$ and $L_2^+$ are the base changes of~$L^+$ via $K^+ \hookrightarrow K_1^+$ and $K^+ \twoheadrightarrow K_2^+$, respectively.
    In particular, both extensions $L_1^+/K_1^+$ and $L_2^+/K_2^+$ are unramified.
    
    Suppose now that $(L_1,L_1^+)/(K_1,K_1^+)$ and $(L_2,L_2^+)/(K_2,K_2^+)$ are unramified.
    We can check that $(L,L^+)/(K,K^+)$ is unramified after base changing to $(K^\sh,K^{\sh,+})$.
    So we may assume that $(K,K^+)$ is strictly henselian.
    By \cref{henselisation-composition}~\labelcref{lemitem:hens-comp2}, also $(K_2,K_2^+)$ is strictly henselian and $(K_1,K_1^+)$ is henselian.
    Therefore, the unramified extension $(L_2,L_2^+)/(K_2,K_2^+)$ is trivial.
    Now $(L_1,L_1^+)/(K_1,K_1^+)$ is an unramified extension with trivial residue field extension of a henselian valued field, hence trivial as well.
    This shows that $(L,L^+) = (K,K^+)$, which is trivially unramified.

    For the second part of~\labelcref{lemitem:composition etale} we note that the unramified extension $(K^h,K^{h,+})/(K,K^+)$ induces unramified extensions of $(K_1,K_1^+)$ and $(K_2,K_2^+)$ that have to contain the respective henselisations by \cref{henselisation-composition}~\labelcref{lemitem:hens-comp1}.
    Invoking that the residue field extension over~$K_2^\succ$ has to be trivial we obtain the desired decomposition.
    The argument for the strict henselisation is analogous.

    In order to prove~\labelcref{lemitem:composition tame} we first assume that $(L,L^+)/(K,K^+)$ is tame.
    We already know from~\labelcref{lemitem:composition etale} that passing to the strict henselisation of $(K,K^+)$ induces unramified extensions of $(K_1,K_1^+)$ and $(K_2,K_2^+)$.
    We may therefore assume that $(K,K^+)$ is strictly henselian.
    Then $L/K$ is of degree prime to~$p$ and the same holds for $L_1/K_1$ and $L_2/K_2$, so both extensions $(L_1,L_1^+)/(K_1,K_1^+)$ and $(L_2,L_2^+)/(K_2,K_2^+)$ are tame and the ramification index of the former extension is prime to~$p$.
    
    For the converse we reduce to the situation where $(K,K^+)$ is tamely henselian.
    Then we argue in the same way as before using that a tame extension of henselian valued fields with trivial residue field extension and trivial value group extension is trivial.
\end{proof}

\begin{rmk}
    The extension $\bQ_p((T^{1/p}))/\bQ_p((T))$ certainly defines a tame extension of $(\bQ_p((T)),\bQ_p[[T]])$ since the residue field has characteristic~$0$.
    However, if we consider it over the composition
    \[
     (\bQ_p((T)),\bQ_p[[T]]) \succ (\bQ_p,\bZ_p) = (\bQ_p((T)),\bZ_p + T\bQ_p[[T]]),
    \]
    the same extension is not tame as it is ramified of degree~$p$ and the residue characteristic is now~$p$.
    One might be tempted to declare such an extension to be tame since it induces tame extensions over both constituents $(\bQ_p((t)),\bQ_p[[T]])$ and $(\bQ_p,\bZ_p)$.
    However, this is not a good idea for various reasons.

    Firstly, the module of log differentials for the extension over $(\bQ_p((T)),\bZ_p + T\bQ_p[[T]])$ is generated by $\dlog T^{1/p}$ subject to the relation $0 = \dlog T = p \dlog T^{1/p}$.
    So it is isomorphic to $\bZ/p\bZ$, which is not zero as expected for tame extensions.

    Secondly, the base change of the extension to $\bQ_p(((pT)^{1/p}))$ transforms it into the extension $\bQ_p((S))(p^{1/p})/\bQ_p((S))$, where $S = (pT)^{1/p}$.
    This extension is certainly wild with respect to the valuation ring $\bZ_p + T\bQ_p[[S]]$ but a base change of a tame extension should be tame.

    In what follows, we will work with rigid spaces in mixed or positive characteristic.
    In this setup the above described subtlety does not come up as for all relevant valued fields $(K,K^+)$ the residue characteristic of $K^\circ$ (the power bounded elements) is~$p$.
    In other words: the only possible localisation of~$K^+$ of residue characteristic~$0$ equals~$K$.
\end{rmk}

\subsection{Tame covers of schemes} 
\label{sec:adic tame discrete case}

In this section, we introduce tame covers of schemes over a base scheme~$S$.
Given its valuation-theoretic nature, the tameness of an \'etale (or just locally quasi-finite) morphism of schemes $Y\to X$ relative to $S$ is best defined using the associated map of discretely ringed adic spaces $\Spa(Y, S)\to \Spa(X, S)$, see \cite[\S3.1]{Temkin2011:RelativeRiemannZariskiSpacesa} and \cite{Hubner2021:AdicTameSite}.

In the literature, tame covers of schemes have been studied for decades, but there are varying notions of tameness used for their definition. 
After giving the definition of tame covers we use in this article (sometimes we refer to this as \emph{adic tameness} because it relies on tameness of the associated morphism of adic spaces), we compare it to the existing notions of tameness.
We will see that under some mild assumptions most of the tameness notions are equivalent.
This was already established in \cite{KerzSchmidt2010:DifferentNotionsTameness}.
However, they had to assume resolution of singularities for some of the implications and we will report on new developments that allow for a proof without using resolutions.

\begin{defi}
\label{defi:Spa for schemes}
Let $f\colon X\to S$ be a morphism of schemes. We define $\Spa(X,S)$ as the set of triples $(x,V,s)$, where $x \in X$ is a point, $V$ is a valuation subring of $k(x)$, and $s \colon \Spec(V) \to S$ is a map (that we call a \textbf{centre of~$V$ on~$S$}) fitting inside a commutative square
\[
 \begin{tikzcd}
     \Spec(k(x)) \ar[r] \ar[d] & X \ar[d] \\
     \Spec(V) \ar[r,"s"] & S.
 \end{tikzcd}
\]
\end{defi}

Often we denote a triple $(x,V,s)\in \Spa(X,S)$ simply by $x$, and the valuation ring $V\subseteq k(x)$ will be denoted by $k(x)^+$. By abuse of notation we sometimes also write~$s$ for the image of the closed point under the map $\Spec(V) \to S$ in the definition.
This is justified as for a given point $s \in S$ there is at most one map $\Spec(V) \to S$ as above mapping the closed point of~$\Spec(V)$ to~$s$.

    A commutative square 
    \[ 
        \begin{tikzcd}
            X'\ar[d] \ar[r,"g"] & X\ar[d] \\
            S'\ar[r,"h"] & S
        \end{tikzcd}
    \]
    induces a map $\Spa(X', S')\to \Spa(X, S)$ sending $(x', V',s')$ to $(x,V,s)$ where $x=g(x')$, $V = V'\cap k(x)$, and $s$ is induced by $h\circ s'$.

The construction of $\Spa(X,S)$ was introduced in \cite[\S~3.1]{Temkin2011:RelativeRiemannZariskiSpacesa}, where it is also shown that $\Spa(X,S)$ naturally carries the structure of an adic space.
In \cref{s:compactifications} we will see that $\Spa(X,S)$ is a special instance of a vertical compactification, namely of $\Spa(X,X)$ over $\Spa(S,S)$ (see \cref{example-Spa-schemes}).

\begin{defi}
\label{defi:tame cover of schemes X over S}
    Let $f\colon Y \to X$ be an \'etale morphism of schemes over $S$.   
    \begin{enumerate}
        \item 
        We say $f\colon Y \to X$ is \textbf{tame relative to $S$} if for every point $y = (k,k^+)$ in $\Spa(Y, S)$ with image $x = f(y)$ in $\Spa(X, S)$ the induced extension
        $
            \big(k(y),k(y)^+\big) / \big(k(x), k(x)^+\big)
        $
        is tamely ramified.
        \item 
        The category of all finite \'etale morphisms $Y\to X$ which are tame relative to $S$ is denoted by $\FEtt_{X/S}\subseteq \FEt_X$. 
    \end{enumerate}
\end{defi}

The base change functor $g^\ast \colon \FEt_X \to \FEt_{X'}$ induces a base change functor
\[
    g^\ast \colon \FEtt_{X/S} \la \FEtt_{X'/S'}
\]
on the respective full subcategories of tame covers relative to $S$ (resp. to $S'$).

\subsubsection*{Comparison with other notions of tameness}

In \cite{KerzSchmidt2010:DifferentNotionsTameness} Kerz and Schmidt compare various tameness conditions that we review quickly in order to compare with (adic) tameness as understood in this note. 
Let $X$ be a separated scheme of finite type over an integral, pure-dimensional, separated and excellent base scheme $S$.
Then we consider for a finite \'etale map $Y \to X$ the following properties. 

The oldest definition of a tame cover goes back to \cite[Exposé~X]{SGA1} and requires~$X$ to be regular with a regular compactification~$\ov X$, proper over the base~$S$, whose boundary is the support of a normal crossing divisor $D$.
In this setting we define \textbf{Grothendieck-Murre tameness} as follows:
A finite étale cover $Y \to X$ is \textbf{tame along~$D$} if the normalisation $\ov{Y} \to \ov{X}$ is tamely ramified at rank~$1$ discrete valuation rings corresponding to irreducible components of $D$.

\begin{prop}
    For~$X$ and~$S$ as above, a finite étale cover $Y \to X$ is tame relative to~$S$ if and only if it is tame along~$D$.
\end{prop}

\begin{proof}
    By the valuative criterion of properness we may assume that $S = \ov X$.
    A discrete valuation ring~$V_P$ corresponding to the irreducible component~$P$ of~$D$ defines the point $(\eta,V_P,\eta_P)$ of $\Spa(X,\ov X)$, where~$\eta$ is the generic point of~$X$ and~$\eta_P$ the generic point of~$P$.
    Tameness of $Y \to X$ along~$D$ translates to tameness of $\Spa(Y,\ov X) \to \Spa(X,\ov X)$ over these points $(\eta,V_P,\eta_P)$.
    In particular, tameness of $Y \to X$ relative to~$S$ implies tameness along~$D$.

    In order to show that Grothendieck-Murre tameness along~$D$ implies tameness relative to~$\ov X$ we reduce to the case where~$\ov X = \Spec(A)$ is strictly henselian.
    Then Abhyankar's lemma \cite[XIII Proposition~5.2]{SGA1} tells us that~$Y$ is dominated by a cover of the form
    \[
     \Spec(A[a_1^{1/n},\ldots,a_r^{1/n}]) \longrightarrow \Spec(A),
    \]
    where $a_1,\ldots,a_r \in A$ define the irreducible components of~$D$ and~$n$ is prime to the residue characteristic~$p$ of~$A$.
    Since the degree of this cover is prime to~$p$, it is tame (compare with the proof of \cite[Proposition~4.2]{KerzSchmidt2010:DifferentNotionsTameness}).
\end{proof}

We could rephrase the proof above as Abhyankar's lemma telling us that $\ov Y \to \ov X$ is Kummer étale with respect to the log structure defined by~$D$.
We will encounter a variant of this argument in a more sophisticated (nonnoetherian) setting in \cref{ss:log-abhyankar}.

Grothendieck-Murre tameness is a well-behaved concept only in the presence of a normal crossing compactification.
In order to define tame covers also under less favourable circumstances, various notions of tameness have been introduced.
For a separated scheme~$X$ of finite type over the base~$S$ (which is still assumed to be integral, pure-dimensional, separated, and excellent), we have the following list of tameness conditions for a finite étale cover $Y \to X$:
\begin{enumerate}
    \item \label{def:curve tameness}
    \textbf{Curve tameness} (following Wiesend \cite{Wiesend2008:TamelyRamifiedCovers}):    
    the pull-back $Y_C \to C$ along $C \to X$ to every connected scheme $C$ of finite type over $S$ with $\dim_S C = 1$ (see \cite[Definition on p.653]{KerzSchmidt2010:DifferentNotionsTameness}, this might deviate from the Krull dimension) is tame.
    Any such $C$ has a canonical set of discrete valuations of its function field as points at infinity with respect to which tameness can be defined.
\end{enumerate}
If in addition~$X$ is normal, the following tameness concepts are available:
\begin{enumerate}
    \setcounter{enumi}{1}
    \item  \label{def:divisor tameness}
    \textbf{Divisor tameness:} every discrete rank $1$ valuation of the function field $k(X)$ of $X$ arising from a boundary divisor in a normal compactification $X \subseteq \ov X$ over $S$ is tamely ramified in the function field of $Y$.
    \item  \label{def:chain tameness}
    \textbf{Chain tameness:} for every normal compactification $\ov X \to S$ of $X$,
    every discrete rank $d = \dim_S(X)$ valuation of $k(X)$ dominating a Parshin chain on $\ov X$ whose valuation ring dominates a local ring $\cO_{S,s}$ on $S$ is tamely ramified in $k(Y)$.
    \item  \label{def:valuation tameness}
    \textbf{Valuation tameness:} every valuation of $k(X)$ with centre on $S$ is tamely ramified in $k(Y)$.   
\end{enumerate}
All of the above definitions can be interpreted as placing tameness conditions on points of $\Spa(Y,S)$ of a special kind and are thus implied by adic tameness.
For instance, for curve tameness we only use points $(y,V,s) \in \Spa(Y,S)$ where~$y$ is the image of the generic point of a test curve~$C$ as in the definition above.
For curve tameness, divisor tameness and chain tameness these special points arise from the geometry of the scheme~$Y$.
The corresponding valuations are rather well behaved and geometric in nature.
In \cite{KerzSchmidt2010:DifferentNotionsTameness} the authors show that these notions of tameness are equivalent using some smart scheme theoretic arguments to obtain the following result.

\begin{prop}[\cite{KerzSchmidt2010:DifferentNotionsTameness}, Theorem~4.4]
    Let~$X$ be a regular, separated scheme of finite type over~$S$.
    Then curve tameness, divisor tameness, and chain tameness are equivalent for finite étale covers of~$X$.
\end{prop}

Valuation tameness includes tameness conditions for arbitrary $S$-valuations of the function field of~$X$.
It is a priori stronger than the above more geometric notions of tameness and closer to our (adic) definition of tameness.
Accordingly, it is not so hard to prove a comparison between the two as follows.

\begin{prop}[\cite{Hubner2021:AdicTameSite}, Proposition~9.4]
    Let $X/S$ be regular, separated, and of finite type.
    Then adic tameness of finite étale covers of~$X$ is equivalent to valuation tameness.
\end{prop}

In \cite{KerzSchmidt2010:DifferentNotionsTameness} Kerz and Schmidt bridge the gap between valuation tameness and the more geometric tameness notions by resorting to resolution of singularities.
Using alterations, Schmidt proves an unconditional comparison for regular varieties over a field of positive characteristic  in \cite[Theorem~4.4]{Schmidt-quasi-purity}.
The following is the most general comparison theorem, accomplished by using non-archimedean geometry in order to avoid the need for resolution of singularities.

\begin{prop}[\cite{HuebnerTemkin2026:WildLocus}, Theorem~8.11, Corollary~8.12]
    Suppose~$S$ is of finite type over a Dedekind scheme and~$X$ is separated and of finite type over~$S$.
    \begin{enumerate}
        \item If~$X$ is regular, valuation tameness for finite étale covers over~$X$ implies adic tameness.
        \item Curve tameness for finite étale covers over~$X$ implies adic tameness.
    \end{enumerate}
\end{prop}

The last notion of tameness we would like to discuss is \textbf{numerical tameness} with respect to a fixed normal compactification $\ov X \to S$ of $X$:
a finite étale cover $Y \to X$ is numerically tame if for all geometric points $\bar x \in \ov{X} \setminus X$ the restriction of $Y \to X$ to the scheme of nearby points (also called the link of $X$ near $\bar x$), $ U_{\bar x}^{\sh} = X \times_{\bar X} \Spec(\cO^\sh_{\bar X, \bar x}) $ decomposes into connected components whose Galois closures have degree prime to the characteristic of $\bar x$.

This notion of tameness differs from the preceding ones in two aspects: it requires a fixed normal compactification and it does not involve valuations.
In fact, numerical tameness depends on the choice of the compactification as Examples~2 and~3 in the appendix of \cite{KerzSchmidt2010:DifferentNotionsTameness} illustrate.
It is a priori stronger than the above discussed notions of tameness.

\begin{prop}[\cite{KerzSchmidt2010:DifferentNotionsTameness}, Theorem~5.3] \label{numerical-implies-tame}
    Let~$X$ be normal, separated, and of finite type over~$S$ with a normal compactification~$\ov X$.
    Then numerical tameness of finite étale covers of~$X$ along $\ov X \setminus X$ implies adic tameness (and hence curve, divisor, and chain tameness).
\end{prop}

\begin{proof}
    A variant of this has been shown in \cite[Theorem~5.3]{KerzSchmidt2010:DifferentNotionsTameness} for regular schemes~$X$ and divisor tameness as a conclusion.
    A similar argument can be used here:
    Let $Y \to X$ be a finite étale cover and denote by $\ov Y$ the normalisation of~$\ov X$ in~$Y$.
    We start with a point $(x,V,s)$ of $\Spa(X,S)$ and denote by $\bar{x}$ the unique centre of~$x$ on~$\ov X$ lying over~$s$.
    The centre map $\Spec(V^\sh )\to \ov X$ naturally factors through $\Spec(\cO_{\ov X,\bar{x}}^\sh)$.
    Now by definition the cover $Y \to X$ restricted to the link~$U_{\bar x}^{\sh}$ splits into connected components with Galois closures of degree prime to the residue characteristic~$p$ of~$\bar{x}$.
    Restricting further to $X \times_{\ov X} \Spec(V^\sh)$ this remains true, which implies that $Y \to X$ is tame at all points of $\Spa(Y,S)$ over $(x,V,s)$.
\end{proof}

A converse holds if we are allowed to change the compactification.

\begin{prop}[\cite{HuebnerTemkin2026:WildLocus}, Theorem~9.3]
    Suppose that~$X$ is normal, separated, and of finite type over~$S$.
    If $Y \to X$ is a finite étale cover that is tame, there exists a compactification~$\ov X$ of~$X$ over~$S$ such that $Y \to X$ is numerically tame along $\ov X \setminus X$.
\end{prop}

The following diagram summarises the notions of tameness we have discussed.

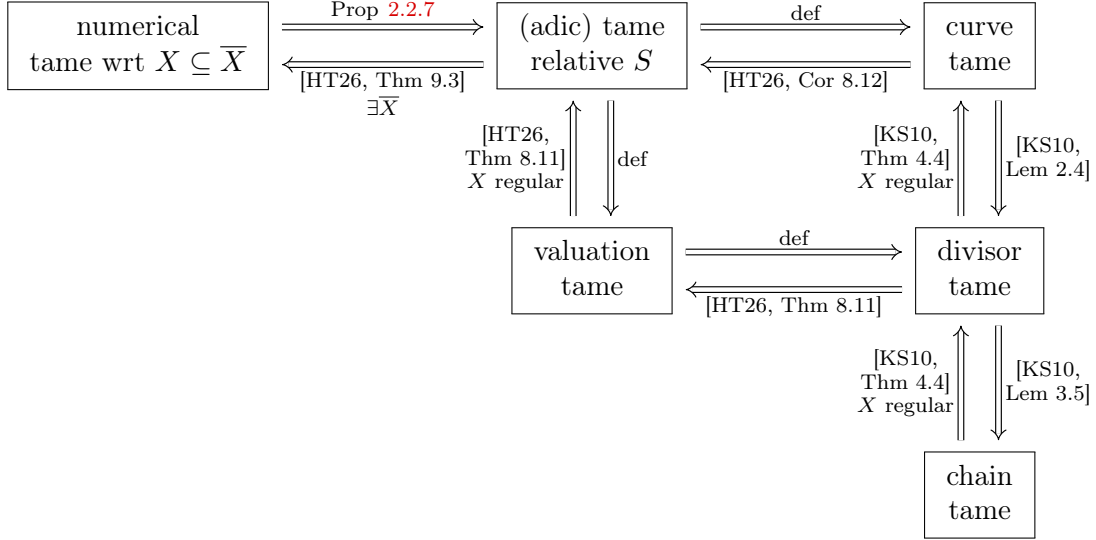
\[
\begin{tikzcd}[column sep=7em, row sep=4em]
\fbox{\begin{tabular}{c}
numerical \\
tame wrt $X \subseteq \ov{X}$
\end{tabular}} 
\ar[r, Rightarrow, shift left=1.5ex, "\text{Prop}~\ref{numerical-implies-tame}"] 
& 
\fbox{\begin{tabular}{c}
(adic) tame \\
relative $S$
\end{tabular}}
\ar[l, Rightarrow, shift left=1.5ex, "\substack{\text{[HT26, Thm~9.3]} \\ \exists \ov X}"]
\ar[d, Rightarrow, shift left=1.5ex, "\text{def}"] 
\ar[r, Rightarrow, shift left=1.5ex, "\text{def}"] 
& 
\fbox{\begin{tabular}{c}
curve \\
tame
\end{tabular}}
\ar[l, Rightarrow, shift left=1.5ex, "\text{[HT26, Cor~8.12]}"]
\ar[d, Rightarrow, shift left=1.5ex, "\substack{\text{[KS10,} \\ \text{Lem~2.4]}}"]
\\
& 
\fbox{\begin{tabular}{c}
valuation \\
tame
\end{tabular}}
\ar[u, Rightarrow, shift left=1.5ex, "\substack{\text{[HT26,} \\ \text{Thm~8.11]} \\ \text{$X$ regular}}"]
\ar[r, Rightarrow, shift left=1.5ex, "\text{def}"]
&
\fbox{\begin{tabular}{c}
divisor \\
tame
\end{tabular}}
\ar[l, Rightarrow, shift left=1.5ex, "\text{[HT26, Thm~8.11]}"]
\ar[u, Rightarrow, shift left=1.5ex, "\substack{\text{[KS10,} \\ \text{Thm~4.4]} \\ \text{$X$ regular}}"] 
\ar[d, Rightarrow, shift left=1.5ex, "\substack{\text{[KS10,} \\ \text{Lem~3.5]}}"]
\\
&
&
\fbox{\begin{tabular}{c}
chain \\
tame
\end{tabular}}
\ar[u, Rightarrow, shift left=1.5ex, "\substack{\text{[KS10,} \\ \text{Thm~4.4]} \\ \text{$X$ regular}}"]
\end{tikzcd}
\]

\subsection{Review of Galois categories}
\label{sec:Gal cats}

We recall the definition of a Galois category from 
\stacks[Definition]{0BMY}.  These axioms are stronger than  but ultimately equivalent to the original axioms from \cite[V, \S4]{SGA1}. In particular, we use the term \emph{connected object} in the sense of \stacks{0BMY}: an object $X$ is \textbf{connected}  if it is not the initial object and every monomorphism $Y \to X$ is an isomorphism or $Y$ is the initial object. 

Unlike \stacks{0BMY}, we do not consider a fixed fibre functor to be part of the datum of a Galois category.

\begin{defi}
    \label{defi:Gal cat}
    A \textbf{Galois category} is a category $\cC$ that admits a \textbf{fibre functor} $F \colon \cC \to \Sets$ such that the following axioms hold. 
    \begin{enumerate}[(G\arabic*)]
        \item 
        The category $\cC$ has finite limits and finite colimits.
        \item 
        Every object of $\cC$ is a finite (possibly empty) coproduct of connected objects. 
        \item 
        The functor $F$ takes values in finite sets.
        \item 
        The functor $F$ is exact 
        and conservative. 
    \end{enumerate}
    A \textbf{morphism} between Galois categories is an exact functor.    
    Recall that a functor $F$
    \begin{itemize}
        \item 
        is called \textbf{exact} if it commutes with finite limits and finite colimits, and 
        \item 
        it is \textbf{conservative} if a morphism $f$ is an isomorphism if and only if $F(f)$ is an isomorphism. 
    \end{itemize} 
\end{defi}

The main theorem of Galois categories \cite[V~Th\'eor\`eme 4.1]{SGA1}, see also \stacks{0BMR} and \stacks{0BN4}, states the following. The group
\[
    \pi_1(\cC,F) =  \Aut(F \colon \cC \to \Sets),
\]
which is called the \textbf{fundamental group of $\cC$ with base point $F$}, 
is a profinite group (with topology induced by declaring stabilisers of points $x \in F(X)$ to be open). Moreover, the natural enrichment
\[
    F \colon \cC \longrightarrow \pi_1(\cC,F)\text{-}\sets
\]
is an equivalence of $\cC$ with the category of finite sets with continuous action of $\pi_1(\cC,F)$. Furthermore, by \cite[V Proposition~6.1]{SGA1} and similarly \stacks{0BN5}, composition with an exact functor transforms fibre functors into fibre functors. In fact this property of a functor between Galois categories is equivalent to being a morphism.

As in \cref{defi:tame cover of schemes X over S}, we will frequently encounter a full subcategory $\cC'$ of a Galois category $\cC$ and would like to deduce that $\cC'$ is Galois with respect to the restriction of the fibre functor. 

\begin{lem} 
\label{lem:Galois-subcategory-criterion} 
    Let $\cC$ be a Galois category with fibre functor $F \colon \cC\to \mathrm{Set}$, and let $\cC'$ be a full
    subcategory of $\cC$ containing $1_{\cC}$. Then, $\cC'$ is a Galois category with fibre functor $F|_{\cC'}$ if and only if the following conditions are satisfied:
    \begin{enumerate}[label=(\roman*)]
        \item 
        \label{lemitem:fullGalCat1}
        if $X'\to X$ is an epimorphism and $X'$ is isomorphic to an object of $\cC'$, then so is $X$,
        \item 
        \label{lemitem:fullGalCat2}
        an object of $\cC$ is isomorphic to an object of  $\cC'$ if and only if all of its connected components are,
        \item 
        \label{lemitem:fullGalCat3}
        the product of two objects in $\cC'$ is isomorphic to an object of $\cC'$.
    \end{enumerate}
\end{lem}

\begin{proof}
    For simplicity, we may first replace $\cC'$ by the strictly full subcategory of $\cC$ of all objects isomorphic to an object of $\cC'$. Moreover, the fibre functor yields an equivalence of $\cC$ with $\pi\text{-}\sets$, the category of finite sets with continuous $\pi$-action for the profinite group $\pi = \pi_1(\cC,F)$. We may therefore replace $\cC$ by $\pi\text{-}\sets$. 

    If $\cC'$ is Galois, then clearly 
    \labelcref{lemitem:fullGalCat1}--\labelcref{lemitem:fullGalCat3} hold. So we now suppose conditions \labelcref{lemitem:fullGalCat1}--\labelcref{lemitem:fullGalCat3} hold. 
    For every object $X$ of $\cC'$ and any point $x \in F(X)$ the stabiliser in $\pi$ is denoted by $U_{X,x}$. This is an open subgroup of $\pi$ and the $\pi$-set $\pi/U_{X,x}$ is isomorphic to the connected component of $X$ containing $x$. The object $\pi/U_{X,x}$ belongs to $\cC'$ by 
    \labelcref{lemitem:fullGalCat2}. The family of subgroups $U_{X,x}$ is closed under finite intersections by \labelcref{lemitem:fullGalCat3}, 
    under passing to a larger open subgroup by \labelcref{lemitem:fullGalCat1}, and under conjugation by arbitrary elements $\gamma \in \pi$ since $U_{X,\gamma(x)} = \gamma U_{X,x} \gamma^{-1}$.

    The intersection $N$ of all $U_{X,x}$ is a closed normal subgroup of $\pi$, and $\cC'$ is contained in the full subcategory $\pi/N\text{-}\sets$ of $\pi\text{-}\sets$. We claim that $\cC' \simeq \pi/N\text{-}\sets$. Indeed, since by \labelcref{lemitem:fullGalCat2} both are determined by their connected objects it suffices to see that any connected $X$ in $\pi/N\text{-}\sets$ actually belongs to $\cC'$. Such an object is isomorphic to a $\pi/U$ for an open subgroup $U$ that contains $N$. By the definition of $N$ as an intersection of the $U_{X,x}$ and compactness, there are finitely many $X_i$ and $x_i$ such that $U$ contains the intersection of the $U_{X_i,x_i}$. But then, $U$ itself is of the form $U_{X,x}$ for some $X$ and $x$. This shows that $\pi/U$  belongs to $\cC'$.
\end{proof}

\begin{lem}
\label{lem:filtered colimits of Galois categories}
   Let $\cC_i$ be a filtered system of Galois categories such that for $i < j$ the functor $\cC_i \to \cC_j$ is exact. Let $\cC = \varinjlim_i \cC_i$ be the colimit in the $(2,1)$-category of small categories\footnote{This colimit exists by  \cite[3.74]{Kelly:EnrichedCats}.}, and let $F\colon \cC \to \Sets$ be a functor such that each composition $F_i\colon \cC_i \to \cC \to \Sets$ is a fibre functor. 
   
   Then $\cC$ is a Galois category with fibre functor $F$, the functors $\cC_i \to \cC$ are exact, and the natural map 
   \[
   \pi_1(\cC,F) \longrightarrow \varprojlim_i \pi_1(\cC_i, F_i)
   \]
   is an isomorphism of profinite groups.
\end{lem}

\begin{proof}
    The filtered system $(\cC_i,F_i)$ yields a cofiltered system of profinite groups 
    $\pi_i = \pi_1(\cC_i,F_i)$. Using the $F_i$ we may replace $\cC_i$ by $\pi_i$-$\sets$ and $F_i$ by just the forgetful functor into the category of sets. The colimit $\cC$ can then be computed as the category whose objects are finite sets with an action of $\pi_i$ for $i \gg 0$ (equivalently some $i$ or compatibly for all large enough $i$). Morphisms in $\cC$ are given by 
    \[
    \Hom_{\cC}(X,Y) = \varinjlim_i \Hom_{\cC_i}(X,Y)= \varinjlim_i \Hom_{\pi_i}(X,Y).
    \]
    This category agrees with the category of finite sets with continuous $\pi = \varprojlim_i \pi_i$-action, and $F$ enriches to an equivalence $F \colon \cC \isomto \pi\text{-}\sets$.  
\end{proof}

\subsection{The tame fundamental group of a scheme}
\label{sec:tame pi1}

We come back to the Galois theory of tame covers of an $S$-scheme $X$.

\begin{prop}
\label{prop:Fett is Galois for schemes}
    Let $X$ be a non-empty connected scheme, and let $X\to S$ be a morphism of schemes. Then $\FEtt_{X/S}$ is a Galois category and the inclusion $\FEtt_{X/S} \subseteq \FEt_X$ is exact.
\end{prop}

\begin{proof}
We know from \stacks[Lemma]{0BNB} that $\FEt_X$ is a Galois category. Moreover, the conditions of \cref{lem:Galois-subcategory-criterion} are satisfied by the inclusion $\FEtt_{X/S}\subseteq \FEt_X$. 
\end{proof}

\begin{cor}
\label{cor:exact functor tame covers}
    The functor $g^\ast \colon \FEtt_{X/S} \to \FEtt_{X'/S'}$  
    is an exact functor of Galois categories if $X$ and $X'$ are connected and non-empty.
\end{cor}

\begin{proof}
    This is simply the restriction of the exact functor $g^\ast \colon \FEt_X \to \FEt_{X'}$.
\end{proof}

\begin{defi}
\label{defi:tamepi1 schemes}
    Let $X$ be a connected $S$-scheme, and let $\bar x \to X$ be a geometric point. We denote by $F_{\bar x}$ the fibre functor on $\FEtt_{X/S}$ implicit in \cref{prop:Fett is Galois for schemes} that is the restriction of the fibre functor `fibre at $\bar x$' from $\FEt_X$.     
    \begin{enumerate}
        \item 
        The \textbf{tame fundamental group} of $X$ (relative to $S$) is the profinite fundamental group  
        \[
        \pitame(X/S,\bar x) = \pi_1(\FEtt_{X/S}, F_{\bar x}) = \Aut(F_{\bar x} \colon \FEtt_{X/S} \to \Sets)
        \]
        of the Galois category $\FEtt_{X/S}$. 
        \item
        In view of \cref{cor:exact functor tame covers} a    commutative square
        \[ 
        \begin{tikzcd}
            X'\ar[d] \ar[r,"g"] & X\ar[d] \\
            S'\ar[r] & S
        \end{tikzcd}
        \]
        with $X$ and $X'$ connected together with geometric points $\bar x' \to X'$ and $\bar x = g(\bar x')$ induces in a functorial way a group homomorphism
        \[
        g_\ast \colon \pitame(X'/S',\bar x')  \la \pitame(X/S,\bar x).
        \]
\end{enumerate}
\end{defi}

\begin{rmks}
    \begin{enumerate}
        \item 
        For morphisms $X \to S' \to S$, the fully faithful inclusions of Galois categories $\FEtt_{X/S}  \subseteq \FEtt_{X/S'}  \subseteq\FEt_X$ yield canonical surjective homomorphisms
        \[
        \pi_1(X,\bar x) \surj \pitame(X/S',\bar x) \surj \pitame(X/S,\bar x).
        \]
        These are just the maps $\id_{X,\ast}$, and if $X \to S$ is the 
        identity $\id\colon X \to X$, then the canonical map is an isomorphism
        \[
        \pi_1(X,\bar x) \isomto \pitame(X/X,\bar x).
        \]
     
        \item 
        The tame fundamental group of $X$ relative to $S$ was considered in \cite[\S5]{HubnerSchmidt2021:TameSiteSchemea} in the locally noetherian case and defined as the pro-discrete monodromy group of local systems on the locally connected tame site $\Spa(X,S)_{\rt}$ in the spirit of the \emph{pro-groupe fondamentale \'elargi}, see \cite[X~\S6]{SGA3:2}.  Upon profinite completion, we recover our $\pitame(X/S,\bar x)$, see  also \cite[\S9]{Hubner2021:AdicTameSite}.      
    \end{enumerate}
\end{rmks}

\begin{rmk}
For a connected and separated scheme of finite type~$X$ over an integral, pure-dimensional, and excellent scheme~$S$ and a geometric point $\bar{x} \to X$, the curve tame covers form a Galois category and give rise to the curve tame fundamental group of~$X$.
If, in addition,~$X$ is normal, we can form the chain tame, divisor tame, and valuation tame fundamental groups.
If~$X$ is regular with an snc compactification~$\ov X$, we furthermore have the Grothendieck-Murre tame fundamental group at our disposal, which we can identify with the Kummer étale fundamental group of~$\ov X$ endowed with the compactifying log structure defined by~$X$. 

Under the various hypotheses discussed in \cref{sec:adic tame discrete case} all of the above fundamental groups are isomorphic to the tame fundamental group $\pi_1^t(X/S,\bar{x})$.
We can thus reinterpret properties of some variant of a tame fundamental group proved in the literature as results about the (adic) tame fundamental group.
\end{rmk}

\subsection{Descent for Galois categories and graphs of groups}
\label{sec:descent for Gal cats}

We review the content of \cite{Stix2006:vanKampen}, a profinite version of \cite{HaefligerComplexesGroups}, with emphasis on the application to finite generation results. 

\subsubsection*{Complex of Galois categories}
Let $\Delta_{\leq 2}$ be the category with three objects $[n] = \{0,\ldots,n\}$ for $0 \leq n \leq 2$. Morphisms in $\Delta_{\leq 2}$ are \emph{strictly} monotone maps of sets. We will use the following notation for maps in $\Delta_{\leq 2}$:
\[
\begin{array}{rll}
    v_i \colon & [0] \to [2] & \qquad \text{ image equals $\{i\}$ and $i \in \{0,1,2\}$,} \\[1ex]
    \pr_{i} \colon & [0] \to [1] & \qquad \text{ image equals $\{i\}$ and $i \in \{0,1\}$,} \\[1ex] 
    \pr_{ij} \colon & [1] \to [2] & \qquad  \text{ image equals $\{i,j\}$ and $i,j \in \{0,1,2\}$.}
\end{array}
\]

A \textbf{$2$-complex of sets} is a functor $E \colon \Delta_{\leq 2}^\op  \to \Sets$, 
i.e.~essentially a diagram in the category of sets of the form
\[
E_\bullet = \Big[\begin{tikzcd}
E_0 & \ar[l,shift right=0.7ex] \ar[l,shift left=0.7ex] E_1 & \ar[l,shift right=1ex] \ar[l,shift left=1ex] \ar[l] E_2  
\end{tikzcd}\Big]
\]
with the usual simplicial identities among compositions of maps. 
We also consider $E_\bullet$ as a category: objects of $E_\bullet$ are the elements of the sets $E_i$ and morphisms are built from $\partial\colon [i] \to [j]$ in $\Delta_{\leq 2}$ and $t \in E_j$ with $s = E(\partial)(t)$ as  morphisms $\partial^\ast\colon s \to t$. Composition in $E_\bullet$ is derived from composition in $\Delta_{\leq 2}$. 

\begin{defi}
    We say that $E_\bullet$ is \textbf{connected} if the topological realisation $|E_\bullet|$ is connected. 
\end{defi}

Let $E_\bullet$ be a $2$-complex of sets. Let $\cG_\bullet \to E_\bullet$ be a fibred category in Galois categories, i.e. for all $s \in E_i$ the category $\cG_s = \cG_\bullet(s)$ is a Galois category and the pull-back functors are exact. The associated $2$-complex of multi-Galois categories (see \cite[V~\S9]{SGA1})
also denoted by~$\cG_\bullet$ is 
\[
\cG_\bullet = \left[\begin{tikzcd}
\cG_0 \ar[r,shift right=0.7ex,"{\pr_1^\ast}",swap] \ar[r,shift left=0.7ex,"\pr_0^\ast"]  &  \cG_1 \ar[r, shift left=1ex,"{\pr_{ij}^\ast}"] \ar[r,shift right=1ex] \ar[r] & \cG_2 
\end{tikzcd}\right]
\]
where $\cG_i = \prod_{s \in E_i} \cG_s$. Here there are three functors $v_i^\ast \colon \cG_0 \to \cG_2$, $i=0,1,2$. Each occurs in two ways as the composition $v_i^\ast \simeq  \pr_{jk}^\ast \circ \pr_\ell^\ast$ for all $j,k,\ell$ and $i = j$ for $\ell = 0$ while $i = k$ for $\ell = 1$. 
The $2$-complex $\cG_\bullet$ comes equipped with natural invertible transformations $v_i^\ast \simeq  \pr_{jk}^\ast \circ \pr_\ell^\ast$ as part of the datum implicitly. All these transformations will be denoted by $\alpha$ without causing too much confusion. 

\begin{ex}
\label{ex:FETdescent}
    Let $f\colon  X' \to X$ be a morphism of schemes. We set $X'' = X' \times_X X'$ and $X'''=X' \times_X X' \times_X X'$, and assume that $X'$, $X''$ and $X'''$ are locally connected, i.e.\ are the disjoint union of their connected components. Then 
    \[
    E_\bullet = \big[
    \begin{tikzcd} 
    \pi_0(X') & \ar[l,shift right=0.7ex] \ar[l,shift left=0.7ex] \pi_0(X'')  & \ar[l,shift right=1ex] \ar[l,shift left=1ex] \ar[l] \pi_0(X''')
    \end{tikzcd}
    \big]
    \]
    is a $2$-complex of sets. The fibred category $\FEt$ over schemes induces a fibred category over $E_\bullet$ that leads to the $2$-complex 
    \[
    \begin{tikzcd} 
    \FEt_{X'} \ar[r,shift left=0.7ex,"{\pr_0^\ast}"] \ar[r,shift right=0.7ex,"\pr_1^\ast",swap]   &  \FEt_{X''}  \ar[r, shift left=1ex,"{\pr_{ij}^\ast}"] \ar[r,shift right=1ex] \ar[r] & \FEt_{X'''} 
    \end{tikzcd}
    \]
    of multi-Galois categories. 
\end{ex}

\subsubsection*{Descent} 
We now recall the definition of the category $\Desc(\cG_\bullet)$ of descent data as the equaliser category of a fibred category in Galois categories $\cG_\bullet \to E_\bullet$  (compare with \cite[Definition~4.1]{Stix2006:vanKampen}):
\begin{enumerate}
    \item 
    Objects of $\Desc(\cG_\bullet)$ are pairs $(X,\ph)$ with an object $X$ of $\cG_0$ and an isomorphism in $\cG_1$
    \[
    \ph \colon \pr_0^\ast(X) \isomto \pr_1^\ast(X).
    \]
    The isomorphism $\ph$ is referred to as the \textbf{gluing isomorphism}. 
    \item
    The isomorphism $\ph$ must satisfy the cocycle relation in $\cG_2$ (well defined if the necessary identifications are made using the natural isomorphisms $\alpha$):
    \[
    \pr_{02}^\ast(\ph) = \pr_{12}^\ast(\ph) \circ \pr_{01}^\ast(\ph).
    \]
    We refer to \cite[Definition~4.1]{Stix2006:vanKampen} for details.  
    \item 
    Morphisms $(X,\ph) \to (Y,\psi)$ in $\Desc(\cG_\bullet)$ are morphisms $X \to Y$ in $\cG_0$ that commute with $\ph$ and $\psi$ in $\cG_1$, i.e. we have
    \[
    \pr_1^\ast(f) \circ \ph = \psi \circ \pr_0^\ast(f).
    \]
\end{enumerate}

\begin{ex}
For the $2$-complex of Galois categories coming from the fibred category $\FEt$ over schemes as in \cref{ex:FETdescent} we introduce the short-hand notation $\Desc(X'/X, \FEt)$ for the category of descent data. 

Fix a scheme $S$ and denote the category of schemes over $S$ by $\Sch/S$. We can consider tame covers relative $S$ as a fibred category $\FEtt_{-/S} \to \Sch/S$ in view of the base change stability of tameness. 
For a map of $S$-schemes $X'\to X$ as in \cref{ex:FETdescent} we obtain a full subcategory of tame descent data
\[
    \Desc(X'/X, \FEtt_{-/S}) \subseteq \Desc(X'/X, \FEt).
\]
\end{ex}

\begin{prop}
\label{prop:DD is galois for connected E}
    Let $E_\bullet$ be a connected $2$-complex of sets and $\cG_\bullet \to E_\bullet$ a fibred category in Galois categories. Then $\Desc(\cG_\bullet)$ is a Galois category.
\end{prop}

\begin{proof}
    A proof can be found in 
    \cite[Propositions~2.4,~4.4]{Stix2006:vanKampen}. We sketch the crucial steps. Finite limits and colimits are constructed `simplexwise'.  
    An object $(X,\ph)$ consists of objects $X_v$ for every vertex $v \in E_0$ together with isomorphisms 
    \[
    \ph_e\colon  \pr_0^\ast(X_{\pr_0(e)}) \isomto \pr_1^\ast(X_{\pr_1(e)})
    \]
    in $\cG_e$ for every edge $e \in E_1$. Every object $X_v$ has a degree defined as the size of its fibre $F_v(X_v)$ under any fibre functor $F_v$ of $\cG_v$. The isomorphisms $\ph_e$ show that the degree is locally constant on $|E_\bullet|$, hence constant. Similarly, a map $f \colon (X,\ph) \to (Y, \psi)$ has fibrewise, i.e. for $F_v(f) \colon F_v(X_v) \to F_v(Y_v)$, image and preimage of locally constant size. In particular, $f$ is an isomorphism if and only if for one vertex $v$ the fibrewise map $F_v(f)$ is bijective. This shows that for any $v \in E_0$ the forgetful functor $\Desc(\cG_\bullet) \to \cG_v$ followed by a fibre functor of $\cG_v$ is a fibre functor for $\Desc(\cG_\bullet)$. The rest is formal. 
\end{proof}

\begin{defi}
    Let $\cC$ be a Galois category and $E_\bullet$ a connected $2$-complex of sets. We consider the constant fibred category $\cC \times E_\bullet \to E_\bullet$ with fibre $\cC$, i.e.~the pull-back functors $\partial^\ast\colon  \cC_s \to \cC_t$ are all the identity $\cC \to \cC$, and all natural transformations $\alpha$ are given by the identity.  
    
    Sending an object $Z$ of $\cC$ to the constant descent datum $(X,\ph)$ with $X_v = Z$ for all $v \in E_0$ and gluing isomorphism $\ph = \id$ defines a functor
    \[
    \cC \la \Desc(\cC \times E_\bullet).
    \]
    Since $E_\bullet$ is connected, this functor is fully faithful (but might fail to be an equivalence in general).

    A functor $f_0^\ast\colon  \cC \times E_\bullet \to \cG_\bullet$ over $E_\bullet$ induces a functor $\Desc(\cC \times E_\bullet) \to \Desc(\cG_\bullet)$; precomposing with the constant descent datum yields a functor
    \[
    f^\ast \colon \cC \la \Desc(\cG_\bullet).
    \]
    We say that $f_0^\ast$ satisfies \textbf{descent} (resp.\ \textbf{effective descent}) if $f^\ast$ is fully faithful (resp.\ an equivalence).

    In geometric settings (see the example below), when $f_0^*$ (and so $f^*$) come from some morphism $f$ of schemes or rigid spaces, we will also be saying that "$f$ is of (effective) descent".
\end{defi}

\begin{ex}
In our guiding \cref{ex:FETdescent} coming from a fibred category over schemes, the property of (effective) descent concerns properties of the functors
\begin{equation}
\label{eq:CechDescentSettingForFEt}
    f^\ast \colon \FEt_X \la \Desc(X'/X, \FEt) \quad \textrm{ and } \quad f^\ast \colon \FEtt_{X/S} \la \Desc(X'/X, \FEtt_{-/S})
\end{equation}
induced by a morphism of schemes $f\colon X'\to X$.
\end{ex}
We now gather some known results about descent concerning these categories.
\begin{fact}\label{fact:eff-decent-maps-for-FEt}
    Let $f \colon X' \to X$ be a surjective map of schemes.
\begin{enumerate}
    \item 
    (\cite[IX Corollaire 3.3+Proposition 2.4]{SGA1}) If $f$ is quasi-compact and universally submersive, then $f$ satisfies  descent for $\FEt$.
    \item 
    (\cite[Corollary~5.18]{RydhSubmersions}) If $f$ is  universally topologically submersive and $X, X'$ are noetherian, then $f$ satisfies effective descent for $\FEt$.
    \item 
    (\cite[IX Proposition~4.1]{SGA1}+\cite[IX Th\'eor\`eme~4.12]{SGA1}) When $f$ is an fpqc covering\footnote{Our presentation of effective descent is geared towards schemes with discrete $\pi_0$ -- this case is sufficient for our purposes. In general, fibre products of fpqc maps can often lead outside of this case, e.g.\ $\Spec(k^{\rm sep}\otimes_k k^{\rm sep})$.} or proper and of finite presentation, then $f$ satisfies effective descent for $\FEt$.
\end{enumerate}
\end{fact}
As  $\FEtt_{X/S} \subset \FEt_X$ and $\Desc(X'/X, \FEtt_{-/S}) \subset \Desc(X'/X, \FEt)$ are full subcategories, full faithfulness of functors is preserved and gives the following corollary.
\begin{lem}\label{lem:f-desc-for-FEt-and-FEtt}
    If $f \colon X' \to X$ satisfies descent for $\FEt$, then it satisfies descent for $\FEtt_{-/S}$. In particular, this holds in  the situations listed in \Cref{fact:eff-decent-maps-for-FEt}.
\end{lem}
Note that we do not claim \emph{effective} descent to be satisfied in the last corollary. In fact, the effective descent property for $\FEt^t_{-/S}$ will often fail even if it holds for $\FEt$.

\begin{ex}
We consider the simplest example where the last assertion fails, maybe the embodiment of all examples: the Artin-Schreier cover $\wp \colon X' = \bA^1_K \to X = \bA^1_K$, $\wp(x) = x^p - x$ over $S = \Spec(K)$, a field of characteristic $p$. The Artin-Schreier cover considered as an object of $\FEt_X$ has a base change to $X'$ that splits completely. It thus yields a descent datum with respect to the fibred category $\FEtt_{-/S}$. This descent datum is not effective, for the only cover it can descend to is the Artin-Schreier cover itself, and that cover is not a tame cover.
\end{ex}

This problem can be fixed by adding extra assumptions that only hold in some specific situations.

\begin{lem}\label{lem:f-eff-desc-for-FEt-and-FEtt}
    Assume that $f\colon X'\to X$ satisfies effective descent for $\FEt$. Then it satisfies effective descent for $\FEt^t_{-/S}$ if and only if for every finite \'etale $Y\to X$ whose base change $Y'=Y\times_X X'\to X'$ is tame relative to $S$, we have that $Y\to X$ is tame relative to $S$. This holds when for every $x\in X$, there exists a point $x'\in X'$ such that $f(x')=x$ and $k(x')=k(x)$, or more generally such that $k(x')/k(x)$ is tame with respect to all valuations on $k(x')$ with centre on $S$.
\end{lem}

\subsubsection*{Van Kampen formula} 
Let $\cG_\bullet \to E_\bullet$ be a fibred category in Galois categories over a $2$-complex of sets. For each simplex $s \in E_i$ we choose a base point (i.e.\ a fibre functor) $F_s \colon \cG_s \to \Sets$.
For $\partial(t) = s$ in $E_\bullet$ and the resulting exact functor $\iota_{st} \colon \cG_s \to \cG_t$ we choose a path, i.e. a natural isomorphism of fibre functors  $\gamma_{st} \colon F_t \circ \iota_{st} \to F_s$. This determines a group homomorphism 
\begin{equation}
\label{eq:boundary map 2-complex of pi1s}
    \partial_\ast \colon \pi_1(\cG_t ,F_t)\la \pi_1(\cG_s,F_s).
\end{equation}

Note that the datum of the natural transformations $\alpha$ that are implicit in $\cG_\bullet$ leads to a flag $vef$ in $E_\bullet$ consisting of a face $f \in E_2$, an edge $e = \pr_{jk}(f) \in E_1$ and a vertex $v = \pr_i(e) \in E_0$ (the indices $i,j,k$ are part of the datum!) to an element $\alpha_{vef} \in \pi_1(\cC_v,F_v)$ that makes composition commutative:
\[
\begin{tikzcd}
\pi_1(\cG_f,F_f) \ar[rr,"\pr_{jk,\ast}"] \ar[d,"v_{\ell,\ast}",swap] && \pi_1(\cG_e,F_e) \ar[d,"\pr_{i,\ast}"] \\
\pi_1(\cG_v,F_v) \ar[rr,"\alpha_{vef}(-)\alpha_{vef}^{-1}",swap] && \pi_1(\cG_v,F_v)
\end{tikzcd}
\]
where $\ell$ is the appropriate index that is determined by $v_\ell = \pr_{jk} \circ \pr_i$.
The above endows $E_\bullet$ with group data and allows one to define a Galois category of local systems, see \cite[\S4.3]{Stix2006:vanKampen}, equivalent to $\Desc(\cG_\bullet)$. The van Kampen formula 
\cite[Corollary~3.3]{Stix2006:vanKampen} computes the fundamental group of  $\Desc(\cG_\bullet)$ in terms of the $2$-complex $E_\bullet$ endowed with its group data. 

\begin{prop}
\label{prop:DD-fgfp}
    Let $\cG_\bullet \to E_\bullet$ be a category fibred in Galois categories over a connected $2$-complex. 
    \begin{enumerate}[ref=(\arabic*)]
        \item
        \label{propitem:DD-fg}
        Suppose that $\pi_1(\cG_v,F_v)$ is  finitely generated for all $v \in E_0$, and that $E_i$ is finite for $i \leq 1$. 
        Then $\pi_1(\Desc(\cG_\bullet),F_v)$ is finitely generated.
        \item
        \label{propitem:DD-fp}
        Suppose that $\pi_1(\cG_v,F_v)$ is finitely presented for all $v \in E_0$, that $\pi_1(\cG_e,F_e)$ is  finitely generated for all $e \in E_1$, and that $E_i$ is finite for all $i \leq 2$.
        Then $\pi_1(\Desc(\cG_\bullet),F_v)$ is finitely presented.
    \end{enumerate}
\end{prop}
\begin{proof}
This is stated in \cite[IX Corollaire~5.2,~5.3]{SGA1} with a proof that requires a ``common'' base point for all Galois categories involved and leaves the general case to the reader as an exercise. The remaining ``exercise'' can be found in the van Kampen formula in \cite[Corollary~3.3]{Stix2006:vanKampen}.  
\end{proof}

\begin{cor}
\label{cor:DD-fgfp}
       Let $\cG_\bullet \to E_\bullet$ be a fibred category over a connected $2$-complex, fibred in Galois categories. Let $\cC$ be a Galois category with fibre functor $F$, and let $f\colon  \cC \times E_\bullet \to \cG_\bullet$  be an exact functor over $E_\bullet$.  
       \begin{enumerate}[ref=(\arabic*)]
           \item 
            \label{coritem:DD-fg}
           If $f$ satisfies descent and $\cG_\bullet$ satisfies the assumptions of \cref{prop:DD-fgfp} \labelcref{propitem:DD-fg}, then $\pi_1(\cC,F)$ is finitely generated.
           \item 
            \label{coritem:DD-fp}
           If $f$ satisfies effective descent and $\cG_\bullet$ satisfies the assumptions of \cref{prop:DD-fgfp} \labelcref{propitem:DD-fp}, then $\pi_1(\cC,F)$ is finitely presented.
       \end{enumerate}
\end{cor}
\begin{proof}
    A choice of a path $F \simeq F_v \circ f$ induces a homomorphism $\pi_1(\Desc(\cG_\bullet,F_v)) \to \pi_1(\cC,F)$. In the first case, this homomorphism is surjective, and in the second case it is an isomorphism. Then apply \cref{prop:DD-fgfp}.
\end{proof}

\subsubsection*{Graph of groups} 

In a $2$-complex of sets $E_\bullet$ the set $E_2$ might be empty. The $2$-complex in reality is then a graph. In this case the multi-Galois category $\cG_2$ for a fibred category in Galois categories is the category with just one object and one morphism. This leads in the definition of $\Desc(\cG_\bullet)$ to simplifications: the cocycle condition becomes automatic. 

We can truncate any given $2$-complex $E_\bullet$ by replacing $E_2$ with the empty set to obtain $E_{\le 1}$. This yields a right adjoint to the inclusion functor with counit $E_{\leq 1} \to E_\bullet$. Similarly, any fibred category $\cG_\bullet \to E_\bullet$ in Galois categories can be truncated by restriction to obtain a fibred $\cG_{\leq 1} \to E_{\leq 1}$, fibred in Galois categories. 

\begin{prop}
\label{prop:truncated DD}
    Let $\cG_\bullet \to E_\bullet$ be a fibred category in Galois categories over a connected $2$-complex.
    \begin{enumerate}
        \item 
        The natural forgetful functor `forget the cocycle condition'
        \[
        \Desc(\cG_\bullet) \la \Desc(\cG_{\leq 1})
        \]
        is fully faithful.
        \item
        A functor $\cC \times E_\bullet \to \cG_\bullet$ satisfies descent if and only if the restriction $\cC \times E_{\leq 1} \to \cG_{\leq 1}$ satisfies descent.
    \end{enumerate}
\end{prop}
\begin{proof}
    Obvious.
\end{proof}

Combining \cref{prop:truncated DD} with \cref{cor:DD-fgfp} shows that finite generation of fundamental groups depends only on truncation of descent data at level $1$. Furthermore, since an exact functor between Galois categories is fully faithful if and only if the induced homomorphism of fundamental groups is surjective \cite[V Proposition~6.9]{SGA1}, we can spell out concretely when an object $(X,\ph)$ of $\Desc(\cG_{\leq 1})$ is connected in terms of this equivalence. 

The underlying $2$-complex in the $1$-truncated setting is actually a graph $|E_{\leq 1}|$. It is endowed with profinite groups and thus becomes a graph of groups (cf.\ \cite[I.4.4 Definition 8]{Serre:trees}) as follows. For every simplex $s \in E_i$ we use the group $\pi_1(\cG_s,F_s)$ and for every edge $e \in E_1$ with boundary vertices $v_i = \pr_i(s) \in E_0$ we obtain group homomorphisms
\[
\pi_1(\cG_{v_0},F_{v_0}) \xleftarrow{\pr_{0,\ast}} 
\pi_1(\cG_e,F_e)
\xrightarrow{\pr_{1,\ast}} \pi_1(\cG_{v_1},F_{v_1}) 
\]
where we have implicitly also fixed paths for change of base points as in \labelcref{eq:boundary map 2-complex of pi1s}. A \textbf{local system} $M = (M_s)_{s \in E_{\leq 1}}$ on this graph of groups now consists of a finite set $M_s$ for every vertex $s \in E_i$, $i=0,1$ with a continuous action by $\pi_1(\cG_s,F_s)$ together with, for every map $\partial \colon e \to v$ in the category $E_{\leq 1}$ a monodromy map $m_{\partial} \colon M_e \to M_v$ which is bijective and equivariant along the homomorphism $\partial_\ast \colon \pi_1(\cG_e,F_e) \to \pi_1(\cG_v,F_v)$. Local systems form a category with the obvious notion of morphism. If, moreover, $|E_{\leq 1}|$ is connected, then local systems form a Galois category equivalent to $\Desc(\cG_{\leq 1})$ by \cite[Proposition~4.4]{Stix2006:vanKampen}. 

Let $T$ be a maximal tree in the graph $|E_{\leq 1}|$. The functor $F_T$ on $\Desc(\cG_{\leq 1})$ that assigns the set of connected components of the restriction to $T$ of the underlying local system $M = (M_s)_{s \in E_{\leq 1}}$ yields a group homomorphism $\pi_1(|E_{\leq 1}|,T)^\wedge \to \pi_1(\Desc(\cG_{\leq 1}),F_T)$. Here $\pi_1(|E_{\leq 1}|,T)^\wedge$ denotes the profinite completion of the topological fundamental group of the graph $|E_{\leq 1}|$ with the maximal tree $T$ as a base point. Furthermore, for every vertex $v \in E_0$, the inclusion $v \in T$ induces a natural transformation of fibre functors $F_v \isomto F_T$. This leads to group homomorphisms $\pi_1(\cG_v,F_v) \to \pi_1(\Desc(\cG_{\leq 1}),F_T)$ that together combine to a natural group homomorphism
\[
    \bigg(\underset{v \in E_0}{\resizebox{.3cm}{!}{$\ast$}} \pi_1(\cG_v,F_v)\bigg)  \ast \pi_1(|E_{\leq 1}|,T)^\wedge \surj \pi_1(\Desc(\cG_{\leq 1}))  .
\]
By \cite[Proposition~2.7]{Stix2006:vanKampen} this is surjective. In particular, this proves \cref{prop:DD-fgfp} \labelcref{propitem:DD-fg}.

To a local system $M$ we can associate a graph $M_{\leq 1} = \big[\begin{tikzcd}
M_0 & \ar[l,shift right=0.7ex] \ar[l,shift left=0.7ex]  M_1 
\end{tikzcd}\big]$ over $E_{\leq 1}$ by 
\[
M_{\leq 1} = 
\Big[\begin{tikzcd}
\amalg_{v \in E_0} M_v & \ar[l,shift right=0.7ex] \ar[l,shift left=0.7ex] \amalg_{e \in E_1} M_e 
\end{tikzcd}\Big] 
\la
E_{\leq 1} = \Big[\begin{tikzcd}
E_0 & \ar[l,shift right=0.7ex] \ar[l,shift left=0.7ex] E_1 
\end{tikzcd}\Big]
\]
with source map $\pr_0$ and target  map  $\pr_1$ defined on $M_e$ by the monodromy maps
\[
M_{\pr_0(e)} \xleftarrow{m_{\pr_0}} M_e  \xrightarrow{m_{\pr_1}} M_{\pr_1(e)}
\]
of the local system $M$.
Recall that the $M_s$ for $s \in E_i$ are $\pi_1(\cG_s,F_s)$-sets and that the set of connected components $\pi_0(M_s)$ is the set of orbits for this group action. Associated to the graph $M_{\leq 1}$ there is the graph 
\[
\pi_0(M_{\leq 1}) = 
\Big[\begin{tikzcd}
\amalg_{v \in E_0} \pi_0(M_v) & \ar[l,shift right=0.7ex] \ar[l,shift left=0.7ex] \amalg_{e \in E_1} \pi_0(M_e) 
\end{tikzcd}\Big]
\]
of connected components. We have a factorisation
\[
M_{\leq 1} \la \pi_0(M_{\leq 1}) \la E_{\leq 1}.
\]

\begin{lem} \label{lem:local-system-connected}
    The object of $\Desc(\cG_{\leq 1})$ corresponding to the local system $M$ on the associated graph of groups is connected if and only if the associated graph $\pi_0(M_{\leq 1})$ of connected components is connected.
\end{lem}

\begin{proof}
    Since $\Desc(\cG_{\leq 1})$ is equivalent to the Galois category of local systems, we may focus on connectedness of $M$. 
    Passing from $M$ to $\pi_0(M_{\leq 1})$ is a functor that preserves disjoint unions, and $\pi_0(M_{\leq 1})$ is empty if and only if $M$ is empty. Conversely, any decomposition $M = A \sqcup B$ arises by a decomposition $M_s = A_s \sqcup B_s$ as $\pi_1(\cG_s,F_s)$-sets for every simplex $s$ of $E_{\leq 1}$ that are compatible with monodromy operators.  
    This always arises as the preimage of a decomposition $\pi_0(M_s) = \ov{A}_s \sqcup \ov{B}_s$ for every simplex $s$ of $E_{\leq 1}$. 
    But this agrees with the datum of a decomposition of 
    $\pi_0(M_{\leq 1})$ as a graph.
\end{proof}

\subsection{Finite generation and finite presentation}

We end \cref{sec:pit schemes} with a general finite generation result for the tame fundamental group of connected varieties over an algebraically closed field. The proof proceeds by reduction to the case of a variety admitting a smooth projective snc compactification, in which case the result is known due to \cite[Theorem~1.1]{EsnaultKindler2016:LefschetzTheoremsTamely} and the case of curves \cite[XIII, Corollaire 2.12]{SGA1}.
We also give a sufficient condition for finite presentation which relies on \cite[Theorem~1.1]{EsnaultShustermanSrinivas2022:FinitePresentationTame} instead. 

\begin{thm} \label{thm:naked pi1tame fg}
    Let $X$ be a connected scheme of finite type over an algebraically closed field $k$. Then $\pitame(X/k,\bar x)$ is finitely generated. 
\end{thm}

\begin{proof}
Let $h_1 \colon X_1 \to X$ be the map from the disjoint union $X_1 = \coprod U$ of an affine open covering of $X$ by finitely many affine Zariski open $U$. 
Let further $h_2\colon X_2 \to X_1$ be the normalisation map of $X_1$, and let $X_2 = \coprod Z_\alpha$ be the decomposition into (finitely many) connected components. 
By \cite[Theorem~4.1]{deJong1996:SmoothnessSemistabilityAlterations} there exist  alterations $f_\alpha\colon  Y_\alpha \to Z_\alpha$ with a smooth projective compactification
$j\colon  Y_\alpha \inj \ov{Y}_\alpha$ such that the boundary $D_\alpha = \ov{Y}_\alpha - Y_\alpha$ is normal crossing. 
Set $Y$ equal to $\coprod_\alpha Y_\alpha$ and let $h_3: Y \to X_2$ be the map $\amalg f_\alpha$. 
The composition $h = h_1\circ  h_2 \circ h_3 \colon Y \to X$ is quasi-compact and universally submersive. So $h$ is of effective descent for $\FEt$ by \cref{fact:eff-decent-maps-for-FEt}.
As observed in \Cref{lem:f-desc-for-FEt-and-FEtt}, it follows that $h$ is of  descent for $\FEtt_{-/k}$.
Therefore the map
\[
h_\ast\colon  \pi_1\big(\Desc(X^\nu/X,\FEtt_{-/k}), F_{\bar y}\big) \surj \pitame(X/k,\bar x)
\]
is surjective (where $F_{\bar y}$ is the fibre functor at a geometric point $\bar y$ of $X^\nu$ above $\bar x$). 
By \cite[Proposition~3.2]{EsnaultShustermanSrinivas2022:FinitePresentationTame}, the group $\pitame(Y_\alpha/k,\bar y) = \pitame(\ov{Y}_\alpha,D_\alpha; \bar y)$ is finitely generated for every $\alpha$.  
As $\pi_0(Y)$ and $\pi_0(Y \times_X Y)$ are finite sets,
it follows by \cref{cor:DD-fgfp} \labelcref{coritem:DD-fg}
that $\pitame(X/k,\bar x)$ is finitely generated. 
\end{proof}

\cref{thm:naked pi1tame fg} recovers the finite generation part of \cite[Th\'eor\`eme~4.11]{Orgogozo2003:AlterationsGroupeFondamental}:

\begin{cor}
    The prime-to-$p$ quotient of the \'etale/tame fundamental group $\pi_1(X)^{\wedge p'}$ of any connected scheme of finite type over an algebraically closed field of char $p>0$ is finitely generated.
\end{cor}

\begin{proof}
    The category of all étale covers of~$X$ whose connected components have Galois closures of degrees prime to~$p$ is a full subcategory of $\FEtt_{X/K}$ satisfying the conditions of \cref{lem:Galois-subcategory-criterion}.
    The corresponding fundamental group is thus a quotient of $\pitame(X/K)$ and inherits the property of being finitely generated.
\end{proof}

\begin{thm} \label{thm:naked pi1tame fp}
    Let $X$ be a connected scheme of finite type over an algebraically closed field $k$. Suppose that there exists a morphism of finite type $f\colon Y\to X$ such that 
    \begin{enumerate}[(i)]
        \item 
        $f$ is universally topologically submersive (e.g.\ surjective and either proper or flat);
        
        \item 
        for every $x\in X$ there exists a point $y\in Y$ such that $f(y)=x$ and $k(y)=k(x)$, or more generally such that the extension $k(y)/k(x)$ is tame with respect to all valuations on $k(x)$ which are trivial on $k$;
        
        \item 
        there exists a projective snc compactification $Y \hookrightarrow \overline{Y}$ over $k$.
    \end{enumerate}
     Then $\pitame(X/k)$ is finitely presented.
\end{thm}

\begin{proof}
Condition (i) and \Cref{fact:eff-decent-maps-for-FEt} guarantee that the morphism $f$ satisfies effective descent for finite \'etale covers. Condition (ii) and \Cref{lem:f-eff-desc-for-FEt-and-FEtt} ensure that the same is true for $\FEtt_{-/k}$, i.e.\ finite \'etale covers tame over $k$. For every connected component $Z$ of $Y$, $Y\times_X Y$, or $Y\times_X Y\times_X Y$, the group $\pitame(Z/k)$ is finitely generated by \cref{thm:naked pi1tame fg}. Moreover, condition (iii) and \cite[Theorem~1.1]{EsnaultShustermanSrinivas2022:FinitePresentationTame} together imply that for every connected component $Z$ of $Y$, the group $\pitame(Z/k)$ is finitely presented. We are now able to apply \cref{cor:DD-fgfp}~\labelcref{coritem:DD-fp}.
\end{proof}

\begin{rmk}
    We mention a few further results for variants of tame fundamental groups of schemes.
    \begin{itemize}
        \item 
        The K\"unneth formula holds for $\pitame$ if both factors are smooth connected varieties over an algebraically closed field and $\pitame$ is computed with respect to curve tame covers. This was recently proved by de~Jong in \cite[Theorem A.1]{deJong2026:KuennethForPitameAppendixJLR2026} and was previously known for varieties with good compactifications (the boundary is a normal crossing divisor) by  
        \cite[Th\'eor\`eme~5.1]{Orgogozo2003:AlterationsGroupeFondamental}. 
        \item 
        Invariance of $\pitame$ under geometric (separably closed) base change of base fields for varieties with good compactification: \cite[Remarque~5.3]{Orgogozo2003:AlterationsGroupeFondamental}. See also \cite{Landesman}.
    \end{itemize}
\end{rmk}

\section{Tame fundamental groups of log schemes}
\label{s:fg-pi1-log}
In this section, we deal with tame fundamental groups of log schemes. For our application to tame fundamental groups of rigid spaces in \cref{s:fg-pi1-rig}, we need to consider certain log schemes which are not necessarily fs, but whose log structure is in a certain sense finitely presented relative to a base divisible valuative monoid. We introduced them in \cite{AHLS-Log} and called them \emph{log schemes of type $\typeVd$}. In the same paper, we extended the theory of the Kummer \'etale fundamental group to such log schemes. The relevant results are reviewed in \cref{ss:log-review}. 

The goal of this section is to introduce the \emph{tame} Kummer \'etale fundamental group $\pitame(X/S, \bar x)$ of a connected log scheme $X$ of type $\typeVd$ relative to a base scheme $S$. It is defined using Kummer \'etale covers which become prime-to-$p$ after pull-back to every strictly henselian valued field with a centre on $S$ (see \cref{def:tame-log}). The main result \cref{thm:pi1tame log is fg} states that if $X$ is of finite type over an algebraically closed field $k$, then $\pitame(X/k, \bar x)$ is finitely generated.

This finiteness result will be obtained in three steps. The first step, the case of trivial log structure, was proved in \cref{thm:naked pi1tame fg}. We next treat the case where the log structure is locally constant by means of the map to the underlying scheme endowed with trivial log structure (see \cref{cor:pi1tame locally constant log is fg}). As is well known, the intuition for this `forget log map' is that of a fibration in tori, and this works well here even for log schemes of type $\typeVd$. The general case is then obtained by formal gluing along the log stratification.

In the final \S\ref{ss:surgery} we explain how the tame fundamental group of certain smooth log schemes over a log point can be computed by a ``surgical procedure'' along the strict normalisation map. This will enable us to show that it is finitely presented in certain cases (\cref{cor:surgery-fp}) and to compute some examples.

\subsection{Review of logarithmic geometry beyond fs} 
\label{ss:log-review}

As we explained in the introduction, dealing with semistable reduction and tamely ramified covers over general valuation rings necessitates the extension of logarithmic geometry beyond the case of fs log schemes. This has been handled in the companion paper \cite{AHLS-Log}, where in particular it was observed that for the most natural definition of Kummer \'etale maps, the underlying maps of schemes might not be of finite type in general. Fortunately for us, this difficulty does not occur if one considers log schemes of type $\typeVd$ as defined in \cite[\S 5]{AHLS-Log} and recalled below. In particular, the theorems and definitions reviewed here are not exactly the general ones stated in \cite{AHLS-Log}, but their simplified versions in the case of log schemes of type $\typeVd$.

Following general conventions (see \cite{Kato1989:LogarithmicStructures,Ogus,KatoToricSingularities} in addition to \cite{AHLS-Log}), the underlying scheme of a log scheme $X$ is denoted by $\underline{X}$, and $\overline{\cM}_{X}$ is the quotient $\cM_X/\cO_X^\times$. We will only consider log schemes which \'etale locally admit a chart by a saturated (but not necessarily finitely generated) monoid.  

Recall that a monoid $V$ is called {\bf valuative} if it is integral and if $V^\gp = V \cup (-V)$. We call $V$ {\bf divisible} if the associated group $V^\gp$ is divisible. 

\begin{defi} \label{def:log-sch-type-Vd}
    A saturated monoid is of {\bf type $\typeVd$} if it is finitely generated over a divisible valuative submonoid. A saturated log scheme is of {\bf type $\typeVd$} if \'etale locally it admits a chart by a monoid of type $\typeVd$.
\end{defi}

Trivially, every fs monoid is of type $\typeVd$, and hence every fs log scheme is of type $\typeVd$. As we shall see, another natural source of such log schemes are semistable schemes over a valuation ring $K^+$ with algebraically closed field of fractions, endowed with the standard log structure (see \cref{ss:log-formal-schemes}). In this case, the base divisible valuative monoid is the non-negative part $\Gamma_K^+$ of the value group of $K$.

\begin{defi}[{\cite[Definition~2.4.4 and Lemma~2.6.1]{AHLS-Log}}] \label{def:Kummer-etale-monoid}
    Fix a set of primes $\Sigma$. Let $P$ be a monoid of type $\typeVd$ and let $\vtheta\colon P\to Q$ be a homomorphism to a saturated monoid $Q$. We say that $\vtheta$ is {\bf Kummer \'etale} if it is injective and exact, and if the quotient $Q^\gp/P^\gp$ is a finite group of order prime to all $p\in \Sigma$. We say that $\vtheta \colon P\to Q$ is {\bf smooth} (resp.\ {\bf \'etale}) if $Q$ is finitely generated over $P$ and the kernel and the torsion part of the cokernel (resp.\ the kernel and the cokernel) of $P^{\rm gp}\to Q^{\rm gp}$ are finite groups of order prime to all $p\in \Sigma$.
\end{defi}

We showed \cite[Lemma~2.6.2]{AHLS-Log} that if $P$ is of type $\typeVd$ and if $P\to Q$ is Kummer \'etale, then $Q$ is finitely generated as a $P$-set, and in particular finitely generated (in fact, even finitely presented) as a monoid over $P$. Consequently, $Q$ is of type $\typeVd$ as well, and the morphism of schemes $\Spec(\bZ[Q])\to \Spec(\bZ[P])$ is finite and of finite presentation. 

\begin{defi} \label{def:Kummer-etale-map}
    Let $X$ be a log scheme of type $\typeVd$ and let $f\colon Y\to X$ be a morphism of saturated log schemes. 
    \begin{enumerate}[(a)]
        \item 
        We say that $f$ is {\bf Kummer \'etale} if \'etale locally on $X$ and $Y$ it admits a chart by a Kummer \'etale morphism of saturated monoids $P\to Q$ for $\Sigma$ the set of primes non-invertible on $X$ such that the induced morphism of schemes
        \[ 
            \underline{Y} \la \underline{X}\times_{\Spec(\bZ[P])} \Spec(\bZ[Q])
        \]
        is \'etale (see \cite[Definition~3.5.1]{AHLS-Log}).
        \item 
        We say that $f$ is {\bf finite Kummer \'etale} if it is Kummer \'etale and if its underlying morphism of schemes is finite (see \cite[Definition~4.2.1 and Corollary~5.1.2]{AHLS-Log}).
    \end{enumerate}
\end{defi}

We note that if $X$ is a log scheme of type $\typeVd$ and if $Y\to X$ is Kummer \'etale, then $Y$ is of type $\typeVd$ as well, and the underlying morphism of schemes $\underline{Y}\to \underline{X}$ is locally quasi-finite and locally of finite presentation \cite[Proposition~5.1.2]{AHLS-Log}. Consequently, if $\underline{X}$ is locally noetherian, then so is $\underline{Y}$. In particular, if moreover $Y$ is quasi-compact, then $Y$ has a finite number of connected components.

A {\bf log geometric point} of a log scheme $X$ is a morphism of log schemes $\bar{x}\to X$ where $\bar{x} = \Spec(P\to k)$ is the spectrum of a separably closed field $k$ endowed with a log structure charted by a monoid $P$ such that $P^\gp$ is $n$-divisible for every $n\geq 1$ invertible in $k$. Every non-empty Kummer \'etale log scheme over such an $\bar{x}$ has a section.

\begin{thm}[{\cite[Theorem~4.3.2]{AHLS-Log}}] \label{thm:FKEt-Galois}
    Let $X$ be a connected log scheme of type $\typeVd$ and let $\bar{x}\to X$ be a log geometric point. Then the category $\FEt_X$ of finite Kummer \'etale morphisms $Y\to X$ is a Galois category on which $\bar{x}$ induces a fibre functor $F_{\bar{x}}$. 
\end{thm}

\begin{defi}[{\cite[Definition~4.3.3]{AHLS-Log}}] \label{def:FKEt-pi1}
    Let $X$ be a connected log scheme of type $\typeVd$ and let $\bar{x}\to X$ be a log geometric point. The fundamental group of the Galois category $\FEt_X$ endowed with the fibre functor $F_{\bar{x}}$ is called the {\bf Kummer \'etale fundamental group} of $X$, and is denoted by $\pi_1(X, \bar{x})$.
\end{defi}

Finite Kummer \'etale maps can be defined locally on the base. More precisely, for a log scheme $X$ of type $\typeVd$, the association $U\mapsto \FEt_U$ forms a stack in the (strict) \'etale topology, and even in the Kummer \'etale topology \cite[Corollary~4.3.7]{AHLS-Log}. 
Another useful fact is its ``topological invariance''.

\begin{prop}[{\cite[Proposition~4.2.4(d)]{AHLS-Log}}] \label{prop:FKEt-top-inv}
    Let $i\colon X_0\to X$ be a strict morphism between log schemes of type $\typeVd$ whose underlying morphism of schemes is a universal homeomorphism. Then the induced functor
    \[ 
        i^*\colon \FEt_X \la \FEt_{X_0}
    \]
    is an equivalence. Consequently, if $X$ is connected and $\bar x\to X$ is a log geometric point, we have an isomorphism
    \[ 
        i_* \colon \pi_1(X_0, \bar{x}) \isomto \pi_1(X, \bar{x}).
    \]
\end{prop}

\begin{defi}
\label{defi:pi1-of-monoid}
    For a (sometimes implicit) choice of a separably closed field $k$, we define $\widehat{\bZ}'(1) = \varprojlim_n \mu_n(k)$ to be the Tate module of $k^\times$.
    For a monoid $P$ we set
    \[
        \pi_1(P) = \Hom(P^{\rm gp}, \widehat{\bZ}'(1)).
    \]
\end{defi}

The notation $\pi_1(P)$ is motivated by the following result.

\begin{prop} \label{prop:FKEt-pi1-local}
    Let $X = \Spec(P\overset{\alpha}{\to} A)$ where $P$ is a monoid of type $\typeVd$ and $A$ is a strictly henselian local ring with residue field $k$. Let $x$ be the closed point of $\underline{X}$, and let $\bar{x}$ be a log geometric point above $x$. In this case, we have $\overline{\cM}_{X,x} = P/\alpha^{-1}(A^\times)$.  Then there is a natural isomorphism of profinite groups
    \[ 
        \pi_1(X, \bar{x}) \simeq \pi_1(\overline{\cM}_{X,x}).
    \]
    In particular, $\pi_1(X, \bar{x})$ is a free $\hZ'$-module of finite rank, so abelian, torsion-free, of order coprime 
    to the residue characteristic at $\bar{x}$, and finitely presented as a profinite group.
\end{prop}
\begin{proof}
All but the group theoretic properties of $\pi_1(\overline{\cM}_{X,x})$ have been proven in \cite[Proposition~4.3.1]{AHLS-Log}.
Since the monoid $\overline{\cM}_{X,x}$ is of type $\typeVd$ and sharp, the group 
$\overline{\cM}_{X,x}^{\rm gp}$ is an extension of a divisible group and a finitely generated free abelian group $A$. Since a divisible group does not admit non-zero maps to $\widehat{\bZ}'(1)$, we deduce that
$\pi_1(\overline{\cM}_{X,x})$ is isomorphic to $\Hom(A,\hZ'(1))$, a free $\widehat{\bZ}'$-module of finite rank.

That $\pi_1(\overline{\cM}_{X,x})$ is finitely presented as a profinite group follows from 
\cref{lem:abelian-fg-implies-finitely-presented} below.
\end{proof}

\begin{lem} \label{lem:abelian-fg-implies-finitely-presented}
    Every finitely generated abelian profinite group $A$ is finitely presented as a (potentially nonabelian) profinite group.
\end{lem}

\begin{proof}
    Let $\psi \colon \hZ^n \surj A$ be a surjection with kernel $B$. Since for every prime $\ell$, pro-$\ell$ completion $(-)^{\wedge \ell}$ is exact on abelian profinite groups, the group $B^{\wedge \ell}$ is the kernel of $\bZ_\ell^n \surj A^{\wedge \ell}$. Since $\bZ_\ell$ is a principal ideal domain, the $\bZ_\ell$-module $B^{\wedge \ell}$ is free of rank at most $n$ and can thus be generated by $n$ elements. We fix a surjection $\bZ_\ell^n \surj B^{\wedge \ell}$ for every $\ell$. Combining these in the product over all $\ell$ we obtain a surjection $\hZ^n \surj B$, and so a finite presentation $\hZ^n \to \hZ^n \to A \to 0$ as $\hZ$-module. 

    To obtain a presentation as a profinite group we take generators $x_i$ for $i=1,\ldots,n$ that we map to the images $\psi(e_i) \in A$ of the standard generators $e_i \in \hZ^n$. As relations among the $x_i$ it suffices to take the pairwise commutators $[x_i,x_j]$ and profinite words representing the finitely many generators of $B \subseteq \hZ^n$. 
\end{proof}

We give two further results on the structure of finite Kummer \'etale maps.

\begin{lem}[{See \cite[Lemma~4.2.7]{AHLS-Log}}] 
\label{lem:fKet-local-structure}
    Let $Y\to X$ be a finite Kummer \'etale map and let $P\to \cM(X)$ be a chart by a~saturated monoid $P$. Then \'etale locally on $X$ there exists a finite collection of Kummer \'etale maps $P\to Q_j$, for $j=1,\ldots, r$, and an isomorphism over $X$ 
    \[
        Y \simeq \coprod_{j=1}^r X\times_{\Spec(P\to \bZ[P])}\Spec(Q_j\to\bZ[Q_j]).
    \]
\end{lem}

\begin{cor} \label{cor:fKet-local-structure}
    In the situation of \cref{lem:fKet-local-structure}, if $X$ is quasi-compact, then there exists a Kummer \'etale homomorphism of monoids $P\to Q$ such that in the (saturated) base change diagram
    \[ 
     \begin{tikzcd}
         Y_Q \ar[r] \ar[d] & X_Q \ar[d] \ar[r] & \Spec(Q \to \bZ[Q]) \ar[d] \\
         Y \ar[r] & X \ar[r] & \Spec(P \to \bZ[P])
     \end{tikzcd}
    \]
    the map $Y_Q\to X_Q$ is strict \'etale. 
\end{cor}

\begin{proof}
By \cref{lem:fKet-local-structure} and the quasi-compactness of $X$ there exists an \'etale covering $\{X_i\to X\}_{i=1}^n$ such that the assertion of said lemma holds over $X_i$. Let us denote the resulting monoids by $Q_{ij}$, for $j=1, \ldots, r(i)$. To obtain the assertion of the corollary, it suffices to take $P\to Q$ such that each of these $P\to Q_{ij}$ factors as $P\to Q_{ij}\to Q$. This is clearly possible, as we can take $Q$ to be the saturation of $P$ in the pushout of $\bigoplus P^{\rm gp}\to \bigoplus Q_{ij}^{\rm gp}$ along the sum map $\bigoplus P^{\rm gp}\to P^{\rm gp}$ (which is an extension of $\bigoplus (Q_{ij}^{\rm gp}/P^{\rm gp})$ by $P^{\rm gp}$).  
\end{proof}

\subsection{The tame fundamental group of a log scheme}

We now introduce a notion of tameness for finite Kummer \'etale maps of log schemes $Y\to X$ relative to a base scheme $S$, which lets us define the tame fundamental group $\pitame(X/S,\ov{x})$ if $X$ is connected. In the case of schemes, tameness ultimately is defined in terms of extensions of strict henselisations being of degree prime-to-$p$, see \cref{defi:tame-extension-of-valued-fields}. We take the same approach in the case of log schemes, replacing the Galois group of a strictly henselian valued field $(K,K^+)$ with the category of finite Kummer \'etale covers of the log scheme $\Spec(K)$ endowed with the log structure pulled back from $X$. 

\begin{defi} \label{def:prime-to-p-obj}
    Let $\cC$ be a Galois category and let $p\geq 1$ be a natural number. We say that an object $Y$ of $\cC$ is {\bf prime-to-$p$} if there exists a finite collection of Galois objects $Y_i$ of degree\footnote{The degree of an object of a Galois category is the cardinality of its image under a fibre functor.} prime to $p$ and a surjection $\coprod Y_i\to Y$. 
\end{defi}

\begin{rmk} \label{rmk:prime-to-p-obj}
    Let $\Gamma$ be a profinite group. The category of prime-to-$p$ objects in $\Gsets{\Gamma}$ coincides with the category of $\Gamma'$-sets where $\Gamma \surj \Gamma'$ is the prime-to-$p$ completion of $\Gamma$, and for an open subgroup $H\subseteq \Gamma$, the object $\Gamma/H$ is prime-to-$p$ if and only if $H$ is open in the prime-to-$p$ topology \cite[\S 3.1]{RibesZalesskii}. It follows that the full subcategory $\cC'\subseteq\cC$ consisting of prime-to-$p$ objects is a Galois subcategory in general. Further, every map of Galois categories preserves prime-to-$p$ objects, and for a Galois subcategory $\cC_0\subseteq \cC$, an object of $\cC_0$ is prime-to-$p$ if and only if it is prime-to-$p$ when considered as an object of $\cC$. Moreover, for a pair of objects $Y,Z\in\cC$ with $Y$ prime-to-$p$ and non-empty (i.e.\ non-initial), the object $Z$ is prime-to-$p$ if and only if the product $Y\times Z$ is prime-to-$p$.
\end{rmk}

\begin{lem} \label{lem:wreath}
    Let $Z\to Y$ be a morphism in a Galois category $\cC$ and let $p\geq 1$ be a natural number. Consider the following conditions.
    \begin{enumerate}[(i)]
        \item $Z$ is prime-to-$p$.
        \item $Y$ is prime-to-$p$.
        \item for every connected component $W$ of $Y$, the base change $Z\times_Y W$ is prime-to-$p$ as an object of the slice category $\cC_{/W}$ (which is a Galois category).
    \end{enumerate}
    Then (ii)\&(iii)$\Rightarrow$(i)$\Rightarrow$(iii), and if $Z\to Y$ is surjective, then also (i)$\Rightarrow$(ii).
\end{lem}

\begin{proof}
We only need to show (ii)\&(iii)$\Rightarrow$(i). For this, we may assume that $Y=W$ and $Z$ are connected. Write $\cC = \Gsets{\Gamma}$ for a profinite group $\Gamma$; we may assume $Y=\Gamma/H$ for a subgroup $H\subseteq \Gamma$ which is open in the prime-to-$p$ topology on $\Gamma$, and $Z = \Gamma/K$ for an open subgroup $K\subseteq H$. In this case, \cite[Lemma~3.1.4]{RibesZalesskii} implies that $K$ is open in the prime-to-$p$ topology on $H$ if and only if it is open in the prime-to-$p$ topology on $\Gamma$.
\end{proof}

\begin{defi} \label{def:tame-log}
    Let $X$ be a saturated log scheme and let $\underline{X}\to S$ be a map of schemes. 
    We say that a finite Kummer \'etale map $Y\to X$ is {\bf tame relative to $S$} if for every commutative square
    \begin{equation} \label{eqn:tame-log-test-sq}
        \begin{tikzcd}
            \Spec(K)\ar[r] \ar[d] & \underline{X}\ar[d] \\
            \Spec(K^+) \ar[r] & S
        \end{tikzcd}
    \end{equation}
    where $(K, K^+)$ is a strictly henselian valued field and, with $T$ as an abbreviation of $\Spec(K)$ with the pull-back log structure from $X$, the pull-back $Y_T \in \FEt_{T}$ is a prime-to-$p$ object where $p$ is the residue characteristic exponent of $K^+$. 
    
    The full subcategory of Kummer \'etale covers $\FEt_X$ consisting of those which are tame relative to $S$ is denoted by $\FEtt_{X/S}$.
\end{defi}

\begin{lem} \label{lem:tame-log-properties}
    Let $X$ be a saturated log scheme and let $\underline{X}\to S$ be a morphism of schemes. 
    \begin{enumerate}[(a)]
        \item \label{lemitem:tame-log-properties-basic} The full subcategory $\FEtt_{X/S}\subseteq \FEt_X$ is stable under fibre product and subquotients. For finite Kummer \'etale maps $Z\to Y\to X$ with $Y\to X$ tame relative to $S$, we have that $Z\to X$ is tame relative to $S$ if and only if $Z\to Y$ is tame relative to $S$.
        \item \label{lemitem:tame-log-properties-Galois} If $X$ is connected, then $\FEtt_{X/S}$ is a Galois subcategory of $\FEt_X$.
        \item \label{lemitem:tame-log-properties-bc} Given a map of saturated log schemes $X'\to X$ and a commutative square 
        \begin{equation} \label{eqn:tame-log-properties-bc-square}
            \begin{tikzcd}
                \underline{X}'\ar[r] \ar[d] & \underline{X} \ar[d] \\
                S' \ar[r] & S,
            \end{tikzcd}
        \end{equation}
        if $Y\to X$ is tame relative to $S$, then $Y'=Y\times_X X'\to X'$ is tame relative to $S'$. 
        \item \label{lemitem:tame-log-properties-no-log-str} If $X$ has trivial log structure, then $Y\to X$ is tame relative to $S$ if and only if it is tame relative to $S$ in the sense of \cref{defi:tame cover of schemes X over S}.
        \item \label{lemitem:tame-log-properties-tame-local} Let $U\to X$ be a strict \'etale map which is a covering for the Nisnevich topology.
        Then $Y\to X$ is tame relative to $S$ if and only if the base change $Y_U\to U$ is tame relative to $S$.
        \item \label{lemitem:tame-log-properties-strict} A strict finite \'etale map $Y\to X$ is tame relative to $S$ if and only if the underlying map of schemes is tame relative to $S$ in the sense of \cref{defi:tame cover of schemes X over S}. 
        \item \label{lemitem:tame-log-properties-XequalsS} If $\underline{X}=S$ then $\FEtt_{X/S} = \FEt_X$.
        \item \label{lemitem:tame-log-properties-XQ} Let $P\to Q$ be a Kummer \'etale map of saturated monoids with respect to the set of primes non-invertible on $S$ and let $X\to \Spec(P\to \bZ[P])$ be a chart (a strict morphism of log schemes). Set
        \[
            X_Q =  X\times_{\Spec(P\to\bZ[P])} \Spec(Q\to\bZ[Q]) .
        \]
        Then $Y\to X$ is tame relative to $S$ if and only if the base change $Y_Q=Y\times_X X_Q\to X_Q$ is. 
    \end{enumerate}
\end{lem}

\begin{proof}
\ref{lemitem:tame-log-properties-basic} The first two assertions follow directly from the analogous statements for the categories of prime-to-$p$ objects of $\FEt_T$ for test objects $T$ as in the definition. 

For the last one, suppose first that $Z\to Y$ is tame relative to $S$. We will show that $Z\to X$ is tame relative to $S$. Consider a test square \eqref{eqn:tame-log-test-sq} and $T=\Spec(K)$ with log structure pulled back from $X$. We must check that $Z_T\in \FEt_T$ is prime-to-$p$ (with $p$ the residue characteristic exponent of $K^+$), which by \cref{lem:wreath} is equivalent to the fact that for every connected component $W$ of $Y_T$ (with the induced log structure), the preimage $Z_W = Z_T\times_{Y_T} W$ is prime-to-$p$ as an object of $\FEt_W$. Let thus $W\subseteq Y_T$ be a connected component and let $T'=W_{\rm red}$ with the induced log structure. By topological invariance \cite[Proposition~4.2.4(d)]{AHLS-Log}, we have $\FEt_W \simeq \FEt_{T'}$, so the second condition is equivalent to $Z_{T'}\in\FEt_{T'}$ being prime-to-$p$. On the other hand, since $Y\to X$ is integral, so is $T'\to W\to Y_T\to T$. Since $T'$ is connected and reduced and $T=\Spec(K)$, we have $T'=\Spec(K')$ for an algebraic extension $K'$ of $K$. Choosing a valuation subring $(K')^+\subseteq K'$ extending $K$, we obtain a test square of the form \eqref{eqn:tame-log-test-sq} but for $Z\to Y$ relative to $S$. As $Z\to Y$ is tame relative to $S$, the object $Z_{T'}\in \FEt_{T'}$ is prime-to-$p$ as desired.

Conversely, suppose $Z\to X$ is tame relative to $S$ and consider a test square 
\[
    \begin{tikzcd}
        \Spec(K)\ar[r] \ar[d] & \underline{Y}\ar[d] \\
        \Spec(K^+) \ar[r] & S.
    \end{tikzcd}
\]
This square induces by composition with $Y\to X$ a test square as in \eqref{eqn:tame-log-test-sq}. Let $T'$ (resp.\ $T$) be $\Spec(K)$ with log structure pulled back from $Y$ (resp.\ $X$); we have a map $T'\to T$ which is the identity on the underlying scheme. Then $Z_{T'}=Z\times_{Y} T'$ is a sub-object of $Z\times_X T' = (Z\times_X T)\times_T T'$. By assumption, $Z\times_X T$ is a prime-to-$p$ object of $\FEt_T$, and hence its pull-back to $\FEt_{T'}$ is prime-to-$p$.

\ref{lemitem:tame-log-properties-Galois} 
Follows from \ref{lemitem:tame-log-properties-basic} and \cref{lem:Galois-subcategory-criterion}.

\ref{lemitem:tame-log-properties-bc} 
Consider a test square 
    \[
        \begin{tikzcd}
            \Spec(K)\ar[r] \ar[d] & \underline{X}'\ar[d] \\
            \Spec(K^+) \ar[r] & S'
        \end{tikzcd}
    \]
and let $T'$ be $\Spec(K)$ with log structure pulled back from $X'$. By stacking with \eqref{eqn:tame-log-properties-bc-square} we obtain a test square for $X/S$, and let $T$ be $\Spec(K)$ with log structure coming from $X$. We have a map $T'\to T$ which is the identity on the underlying scheme. Now, the pull-back $(Y')_{T'}$ can be expressed as $Y_T\times_T T'$. Since $Y\to X$ is tame relative to $S$, the object $Y_T$ of $\FEt_T$ is prime-to-$p$, and hence so is its base change to $T'$.

\ref{lemitem:tame-log-properties-no-log-str} 
Clear from the definitions.

\ref{lemitem:tame-log-properties-tame-local} 
Follows since in this case 
every test square \eqref{eqn:tame-log-test-sq} lifts to $U$.

\ref{lemitem:tame-log-properties-strict} 
In this case $Y\to X$ is the pull-back of the \'etale map $\underline{Y}\to\underline{X}$ under $X\to\underline{X}$. Consider a test square \eqref{eqn:tame-log-test-sq}. We have a Galois subcategory $\FEt_{\underline{T}}\subseteq \FEt_T$, and hence $Y_T = (\underline{Y})_{\underline T}\times_{\underline{T}} T$ is prime-to-$p$ as an object of $\FEt_T$ if and only if $(\underline{Y})_{\underline T}$ is prime-to-$p$ as an object of $\FEt_{\underline{T}}$. The equivalence follows formally from this.  

\ref{lemitem:tame-log-properties-XequalsS} 
Consider a test square \eqref{eqn:tame-log-test-sq} and let $T^+$ be $\Spec(K^+)$ endowed with the log structure pulled back from $X$ (which makes sense since $\underline{X}=S$). We can base change $Y\to X$ to $T^+$ before base changing further to $T$. As $K^+$ is strictly henselian, by \cref{prop:FKEt-pi1-local} we see that $\pitame(T^+)$ is prime to $p$. So the base change of $Y$ to $T^+$, and hence to $T$, is prime to $p$.

\ref{lemitem:tame-log-properties-XQ} We may assume that $X=\Spec(K)$ and $S=\Spec(K^+)$ for a strictly henselian valued field $(K,K^+)$.
Now $X_Q\to X$ is a non-empty prime-to-$p$ object (when connected, it is Galois of degree $[Q^{\rm gp}:P^{\rm gp}]$ which is prime to $p$), so $X_Q\to X$ is tame.
By \ref{lemitem:tame-log-properties-basic} we see that $Y_Q\to X_Q$ is tame relative to $S$ if and only if $Y_Q\to X$ is.  Thus $Y\to X$ is prime-to-$p$ in $\FEt_X$ if and only if $Y_Q = Y \times_X X_Q$ is prime-to-$p$, see  \cref{rmk:prime-to-p-obj}.
\end{proof}

\begin{defi} \label{def:FKEt-pi1-tame}
    Let $X$ be a connected log scheme and let $\underline{X}\to S$ be a morphism of schemes. For a log geometric point $\overline{x}$ of $X$, we denote by $\pitame(X/S, \overline{x})$ the fundamental group of the Galois category $\FEtt_{X/S}$ with respect to the fibre functor at $\ov{x}$, called the {\bf tame fundamental group} of $X$ relative to $S$. 
\end{defi}

By \cref{lem:tame-log-properties}~\ref{lemitem:tame-log-properties-bc}, a commutative square \eqref{eqn:tame-log-properties-bc-square} with connected log schemes $X$ and $X'$ induces a continuous map $\pitame(X'/S',\ov x)\to\pitame(X/S,\ov x)$. This map is surjective if $X'\to X$ is an isomorphism, and $\pitame(X/X,\ov{x})\simeq \pi_1(X,\ov{x})$. If $X$ has trivial log structure, we have $\pitame(X/S,\ov{x})\simeq \pitame(\underline{X}/S,\ov{x})$.

\subsection{Log stratifications}
\label{ss:log-strat}

For our results about the tame Kummer \'etale fundamental group we will need to work with log stratifications. This needs some extra care as the monoids charting our log structures might not be finitely generated. Fortunately, things work well if the underlying scheme is noetherian (and with a global chart). 

Recall that a {\bf face} of an integral monoid $P$ is a submonoid $F\subseteq P$ such that $x+y\in F$ implies $x, y\in F$. A {\bf prime ideal} of $P$ is a subset $\fp\subseteq P$ such that $P + \fp = \fp$ and $x+y\in \fp$ implies $x\in \fp$ or $y\in \fp$. The quotient $P/F$ is the quotient of the localisation $(P+F^{\rm gp})/F^{\rm gp} \subseteq P^{\rm gp}/F^{\rm gp}$. Faces of $P$ are in an inclusion-reversing bijection $F \mapsto P \setminus F$ with prime ideals of $P$. The set of all prime ideals of $P$ is denoted by $\Spec(P)$, which we endow with the topology generated by the sets
\[
    D(x) = \{\fp \in \Spec(P) \ | \ x \notin \fp\} \qquad (x\in P).
\]
The space $\Spec(P)$ is a $T_0$-space. 

We proved in \cite[Corollary~2.3.3]{AHLS-Log} that if $\vtheta\colon P\to Q$ is an sfp (e.g.\ smooth) morphism of saturated monoids, then the fibres of the map 
\[ 
    \Spec(Q)\la \Spec(P), \quad \fq \mapsto \vtheta^{-1}(\fq)
\]
are finite; if $P\to Q$ is Kummer \'etale, this map is a homeomorphism. In particular, if $P$ is a monoid of type $\typeVd$, i.e.\ finitely generated over a divisible valuative submonoid $V$, 
then the fibres of $\Spec(P)\to \Spec(V)$ are finite. Hence, $P$ has finitely many faces if $V$ does. For any monoid $P$, we have a natural continuous map
\begin{equation} \label{eqn:SpecZP-to-SpecP} 
    \Spec(\bZ[P]) \la \Spec(P), \quad \fp \mapsto \fp\cap P.
\end{equation}

Consider a log scheme $X$ that admits a global chart $\alpha\colon P\to \cM(X)$ where $P$ is a saturated monoid. The chart corresponds to a strict map
\[ 
    X \la \Spec(\bZ[P])
\]
which we compose with \labelcref{eqn:SpecZP-to-SpecP} to obtain a continuous map
\[ 
    \ov\alpha\colon X \la \Spec(P), \quad x \mapsto P \setminus \alpha^{-1}_x(\cO_{X,x}^\times)
\]
where $\alpha_x$ is the composition $P\to \cM(X) \to \cM_{X,x}$. The fibres of $\ov\alpha$ are called the {\bf log strata} of~$X$. They are locally closed if $P$ has finitely many faces. With no such assumption on $P$, we have the following result in the noetherian case.

\begin{lem} \label{lem:noeth-faces}
    Let $X$ be a log scheme admitting a global chart by a monoid $P$ of type $\typeVd$ (or more generally which is sfp over a finitely generated valuative submonoid). Suppose that the underlying scheme of $X$ is noetherian. Then the image of $\ov\alpha\colon X\to \Spec(P)$ is finite, and its fibres are locally closed subsets of $X$. 
\end{lem}

\begin{proof}
Let $V\subseteq P$ be a divisible valuative submonoid such that $P$ is finitely generated over $V$. Since the fibres of the map $\Spec(P)\to \Spec(V)$ are finite, to show the first assertion it suffices to show that the image of $X\to\Spec(V)$ is finite. 
Since $X$ is noetherian, it has finitely many irreducible components, and therefore it suffices to treat the case when $X=\Spec(R)$ is the spectrum of a noetherian domain $R$. We shall see that in this case the image of $X\to \Spec(V)$ consists of at most two points. 

Suppose that the image of $\Spec(R)\to \Spec(V)$ contains three points, corresponding to prime ideals $\fp_0,\fp_1,\fp_2$ of $V$. Since the prime ideals of $V$ are totally ordered, we may assume that $\fp_0\subsetneq \fp_1\subsetneq \fp_2$ and can choose elements $x\in \fp_1\setminus \fp_0$ and $y\in \fp_2\setminus \fp_1$. 

We claim that $x - ny$ belongs to $V$ for all $n\in\bZ$. Indeed, it suffices to show that $-x + ny$ does not belong to $V$. Let $F_i = V\setminus\fp_i$ denote the respective faces, and let $W = V/F_1$, which is a sharp valuative monoid. The quotient map $V\to W$ sends $y$ to zero and $x$ to a non-zero element. This means that the image of $-x+ny$ in $W^{\rm gp}$ does not belong to $W$, and so $-x+ny$ does not belong to $V$. This proves the claim.

Let $\theta\colon V\to R$ denote the monoid map under consideration. Then $\theta(x)$ is infinitely divisible by $\theta(y)$ in $R$ due to  the claim and
\[
\theta(x) = \theta(y)^n \cdot \theta(x-ny).
\]
Let $\fq_i\subseteq R$  be a prime over $\fp_i\subseteq P$, for $i=1,2,3$. Since the element $\theta(y) \in \fq_2$ is not invertible in the local ring $R_{\fq_2}$, by Krull's intersection theorem we obtain $\bigcap_{n \geq 1} \theta(y)^nR_{\fq_2} = 0$ in $R_{\fq_2}$. 
So $\theta(x)$ maps to $0$ in $R_{\fq_2}$. Since $R$ is a domain, we get that $\theta(x) = 0$. This contradicts the fact that $\theta(x)\notin \fq_0$.

For the second assertion, the image $T = \ov\alpha(X)\subseteq \Spec(P)$ is a finite $T_0$-space, and hence for every $y\in T$, the subspace $\{y\}\subseteq T$ is locally closed. Therefore the fibres of $\ov\alpha\colon X\to T$ are locally closed subsets of $X$. 
\end{proof}

\subsection{Locally constant log structure}
\label{ss:loc-const}

We call a log structure $\cM_X$ on a scheme $X$ {\bf locally constant} if the sheaf $\overline{\cM}_X = \cM_X/\cM_X^\times$ is locally constant (in the \'etale topology). Our next goal is to relate the Kummer \'etale fundamental group of a log scheme to that of the underlying scheme in case the log structure is locally constant (see \cref{prop:torus-fibration-pi1} and \cref{cor:pi1tame locally constant log is fg}). This will be a key ingredient in our proof of finite generation in \S\ref{ss:strat}.

\begin{lem} \label{lem:ovalpha--const}
    Let $X$ be a log scheme of type $\typeVd$ and $\alpha\colon P\to \cM(X)$ a chart where $P$ is a monoid of type $\typeVd$, and $\ov\alpha\colon X\to \Spec(P)$ is the induced map. Suppose that $\ov{\alpha}$ maps $X$ to a single point of $\Spec(P)$ corresponding to a face $F\subseteq P$. 
    
    Then the sheaf $\ov{\cM}_X$ is the constant \'etale sheaf associated to the quotient monoid $P/F$. In particular, the log structure on $X$ is locally constant. 
\end{lem}

\begin{proof}
By definition of $\ov\alpha$, for every $x\in X$, the preimage of $\cO_{X,\ov{x}}^\times$ under $\alpha_{\ov x}\colon P\to \cO_{X,\ov{x}}$ is equal to $F$. Thus $\cM_{X,\ov{x}} \simeq P\oplus_F \cO_{X,\ov{x}}^\times$, the map $P\to \cM_X(X)\to \ov{\cM}_X(X)$ factors through $P/F$ and induces a map of sheaves $\underline{P/F}_X\to \ov{\cM}_X$ which is an isomorphism on stalks and hence an isomorphism. 
\end{proof}

Combining the above with the results on log stratifications in the previous section, we obtain the following that we will use in the proof by noetherian induction of \cref{thm:pi1tame log is fg}.

\begin{cor} \label{cor:generically-log-constant}
    Let $X$ be a non-empty log scheme of type $\typeVd$ whose underlying scheme is noetherian. Then there exists a non-empty open subset $U\subseteq X$ on which the log structure is locally constant. 
\end{cor}
\begin{proof}
    The question of finding a Zariski open $U$ on which $\overline{\cM}_X|_U$ is locally constant is \'etale local on $X$. So we may assume that $X$ admits a global chart $\alpha\colon P\to \cM(X)$ as in \cref{lem:ovalpha--const}.
    By \cref{lem:noeth-faces} there is a finite log stratification, and since $\Spec(P)$ is $T_0$, there is a non-empty open stratum $U$. 
    Then the log structure on $U$ is locally constant by \cref{lem:ovalpha--const}.
\end{proof}

We now turn our attention to a log scheme $X$ of type $\typeVd$ with a locally constant log structure.
We moreover assume that $X$ is connected and fix a log geometric point $\tilde{x} \to X$ above a geometric point $\ov{x} \to \underline{X}$. We endow $\ov{x}$ with the induced log structure. We have a homomorphism
\[
     x_\ast \colon \pi_1(\overline{\cM}_{X,\ov x}) = \pi_1(\ov x,\tilde x) \la \pi_1(X, \tilde x).
\]
We think of the morphism forgetting the log structure
\[
    \ep\colon X \la \underline{X}
\]
as a (locally constant) fibration whose fibres are tori with character group $\overline{\cM}_{X,\ov x}^\gp$. Our next goal is to show right exactness of the associated homotopy sequence when the log structure is locally constant. This will rely on the following lemma. 

\begin{lem}
\label{lem:strict-locus-clopen}
    Let $X$ be a log scheme of type $\typeVd$ with a locally constant log structure. Let $Y\to X$ be a Kummer \'etale map. Then the log structure on $Y$ is locally constant and the set 
    \[ 
        \{ y\in Y \,:\, \ov{\cM}_{X,f(\ov y)}\isomto \ov{\cM}_{Y,\ov y} \text{ for some/every geom.\ pt. $\ov y\to \underline{Y}$ above $Y$}\}
    \]
    (called the {\bf strict locus} of $Y\to X$) is both open and closed in $Y$.
\end{lem}

\begin{proof}
The assertion is \'etale local on $X$ and we may replace $Y$ by a direct summand as a Kummer \'etale cover of $X$. We may therefore assume that $X$ admits a global chart $\alpha \colon P \to \cM(X)$ where $P$ is a monoid of type $\typeVd$. By \cref{lem:ovalpha--const} there is a face $F$ of $P$ and an isomorphism $\underline{P/F}_X \isomto \overline{\cM}_X$.

Using the chart lifting property \cite[Proposition~3.5.8]{AHLS-Log}, working locally and replacing $Y$ with a clopen subset we may assume that $Y\simeq X\times_{\Spec(\bZ[P])}\Spec(\bZ[Q])$ for a Kummer \'etale map of monoids $P\to Q$ (with respect to the set of primes non-invertible on $X$). 

We may further assume that $X$ is reduced. Note that in this case the map $\ov\alpha : X \to \Spec(P)$ factors through the subscheme (log stratum of $\Spec(\bZ[P])$)
\[ 
    \Spec(P\to \bZ[F^\gp])
\]
where the map is the inclusion $F\hookrightarrow F^\gp$ and sends everything else to zero (it factors set-theoretically, but since $X$ is reduced, it factors scheme-theoretically). Since for a strict map $g\colon W\to Z$ we have $g^*\ov{\cM}_Z\isomto \ov{\cM}_W$, we may assume that $X = \Spec(P\to \bZ[F^\gp])$, in which case 
\[ 
    Y_{\rm red} = \left(X\times_{\Spec(\bZ[P])}\Spec(\bZ[Q]))\right)_{\rm red} \simeq \Spec(Q\to \bZ[G^\gp])
\]
where $G\subseteq Q$ is the face generated by $F$ (note that since $P\to Q$ is Kummer \'etale, we have $F = P\cap G$). Thus $\overline{\cM}_Y$ is the constant sheaf with value $Q/G$, and $Y$ has locally constant log structure. Moreover, if $P/F\isomto Q/G$, the strict locus is everything, and otherwise it is empty. The strict locus is therefore open and closed.
\end{proof}

We shall combine \cref{lem:strict-locus-clopen} with the following result about the strict locus.

\begin{lem} \label{lem:criterion strict}
    Let $X$ be a log scheme of type $\typeVd$ with a locally constant log structure, let $Y\to X$ be a Kummer \'etale map, and let $\ov x$ be a geometric point of $\underline{X}$.  Then the set of points in the geometric fibre $h^{-1}(\ov x)$ above $\ov x$ at which $h$ is strict agrees with the projection under $h^{-1}(\tilde{x}) \to h^{-1}(\ov x)$ of the 
    fixed points of the natural action of $\pi_1(\ov x,\tilde x)$ on the log geometric fibre $h^{-1}(\tilde x)$.
\end{lem}

\begin{proof}
We may work \'etale locally on $X$ and thus assume that $X$ has a chart by a monoid $P$ of type $\typeVd$.
We may further replace $Y \to X$ by a connected component of its base change to $\ov x = \Spec(P\to k)$. By the structure of Kummer \'etale covers over a strict henselian local ring (\cref{prop:FKEt-pi1-local} or \cite[Lemma~4.2.7]{AHLS-Log}) we may assume that $Y = \Spec(Q\to k\otimes_{\bZ[P]}\bZ[Q])$ for a Kummer \'etale map of monoids $P \to Q$. Let $F \subseteq P$ be the preimage of $k^\times$ and $G\subseteq Q$ the unique face above $P$ (we have $\Spec(Q)\isomto\Spec(P)$ as $P\to Q$ is Kummer \'etale). Then $Y\to X$ is strict if and only if $P/F\isomto Q/G$, which translates into the trivial action of $\pi_1(\ov{x},\tilde{x}) = \pi_1(P/F)$ on $h^{-1}(\tilde x) = Y_{\tilde x}$.
\end{proof}

\begin{prop} \label{prop:torus-fibration-pi1}
    Let $X$ be a connected log scheme of type $\typeVd$ with a locally constant log structure. 
    The map $\ep\colon X \to \underline{X}$ to the underlying scheme induces an exact sequence
    \[
        \begin{tikzcd}
            \pi_1(\ov \cM_{X,x}) \ar[r] & \pi_1(X,\bar x) \ar[r,"\ep_\ast"] & \pi_1(\underline{X},\bar x) \ar[r] & 1.
        \end{tikzcd}
    \]
\end{prop}

\begin{proof}
An \'etale cover $\underline{Y} \to \underline{X}$ induces a strict Kummer \'etale map $Y \to X$. Therefore its base change $Y$ along $\ep$ simply endows $\underline{Y}$ with the pull-back log structure. In particular, the base change of a connected \'etale cover remains connected. This proves that $\ep_\ast$ is surjective. 

To prove exactness at $\pi_1(X, \bar x)$ we must show that a connected \'etale cover $h\colon  Y \to X$ with a section above $\bar x$ (the point being endowed with the pull-back log structure) actually already descends to an \'etale  cover of $\underline{X}$. By \cref{lem:criterion strict}, the existence of the section shows that the locus $S \subseteq Y$ where $Y \to X$ is strict is non-empty. The strict locus $S$ is both open and closed by \cref{lem:strict-locus-clopen}.
Since $Y$ is connected, we must have $S=Y$ and $Y \to X$ is a strict map. Hence $Y \to X$ descends as an \'etale cover $\underline{Y} \to \underline{X}$ and the proof is complete.
\end{proof}

We deduce an analogous result for tame fundamental groups.

\begin{cor} \label{cor:torus-fibration-tame}
    Let $X$ be a connected log scheme of type $\typeVd$ with a locally constant log structure. 
    Let $\underline{X}\to S$ be a morphism of schemes. Then the sequence of tame Kummer \'etale fundamental groups induces an exact sequence
    \[
        \begin{tikzcd}
            \pi_1(\ov \cM_{X,x}) \ar[r] & \pitame(X/S,\bar x) \ar[r,"\ep_\ast"] & \pitame(\underline{X}/S,\bar x) \ar[r] & 1.
        \end{tikzcd}
    \]
\end{cor}

\begin{proof}
Consider the commutative diagram
\[
    \begin{tikzcd}
        \pi_1(\ov \cM_{X,x}) \ar[d,equal] \ar[r] & \pi_1(X,\bar x) \ar[r,"\ep_\ast"]\ar[d,two heads]  & \pi_1(\underline{X},\bar x) \ar[r]\ar[d,two heads] & 1 \\
        \pi_1(\ov \cM_{X,x}) \ar[r,"\iota_{x,\ast}"] & \pitame(X/S,\bar x) \ar[r,"\ep_\ast"] & \pitame(\underline{X}/S,\bar x) \ar[r] & 1.
    \end{tikzcd}
\]
The top row is exact by \cref{prop:torus-fibration-pi1}. The vertical arrows are surjective. Thus the image of $\iota_{x,\ast}$ is normal. Moreover, the map $\varepsilon_*$ in the bottom sequence is surjective. 

To prove that the bottom row is exact, it remains to check that a finite $\pitame(X/S,\ov{x})$-set $M$ on which $\pi_1(\ov \cM_{X,x})$ acts trivially comes from a $\pitame(\underline{X}/S,\bar x)$-set. Translating in terms of covering spaces, let $Y\to X$ be the corresponding finite Kummer \'etale map. By exactness of the top sequence, we see that $Y$ is in the image of $\FEt_{\underline{X}}$, i.e.\ $Y\to X$ is strict. Then \cref{lem:tame-log-properties}~\ref{lemitem:tame-log-properties-strict} implies that $\underline{Y}\to\underline{X}$ is tame relative to $S$, so that $Y\to X$ is in the image of $\FEtt_{\underline{X}/S}$.
\end{proof}

\begin{cor}
\label{cor:pi1tame locally constant log is fg}
    Let $X$ be a connected log scheme of type $\typeVd$ whose underlying scheme is of finite type over an algebraically closed field $k$. Suppose that $X$ has locally constant log structure.
    \begin{enumerate}
        \item 
        The group $\pitame(X/k,\bar x)$ is finitely generated.
        \item 
        The group $\pitame(X/k,\bar x)$ is finitely presented if and only if $\pitame(\underline{X}/k,\bar x)$ is finitely presented.
    \end{enumerate}
\end{cor}

\begin{proof}
    By \cref{prop:FKEt-pi1-local} and 
    \cref{lem:abelian-fg-implies-finitely-presented},
    the image of 
    $\pi_1(\ov \cM_{X,x})$ in $\pitame(X/k,\bar x)$ is a finitely presented profinite group.  
    Thus, by 
    \cite[Proof of Claim~2.7]{EsnaultShustermanSrinivas2022:FinitePresentationTame}, the group 
    $\pitame(X/k,\bar x)$ is finitely generated (resp.\ finitely presented) if and only if $\pitame(\underline{X}/k,\bar x)$ has the same property. 

    In particular, because $\pitame(\underline{X}/k,\bar x)$ is finitely generated by \cref{thm:naked pi1tame fg}, we also know that $\pitame(X/k,\bar x)$ is finitely generated.    
\end{proof}

\subsection{Finite generation} 
\label{ss:strat}

We shall use \cref{cor:pi1tame locally constant log is fg} (the case of locally constant log structure) as a bootstrap, combined with a formal gluing argument, to prove the following theorem, which is the main result of this section. 

\begin{thm} \label{thm:pi1tame log is fg}
    Let $k$ be an algebraically closed field.
    Let $X$ be a non-empty connected log scheme of type $\typeVd$
    whose underlying scheme is of finite type over $k$.
    Then $\pitame(X/k,\bar x)$ is finitely generated.
\end{thm}

The main difficulty in the proof of \cref{thm:pi1tame log is fg} (deferred to the end of this subsection) hides in the proof of the following proposition, in which a ``formal gluing'' argument is employed. In order to avoid dealing with log rigid analytic spaces, we only formulate it in the affine case.

\begin{prop} \label{prop:pi1tame-strat}
    Let $X=\Spec(R)$ be a non-empty connected affine log scheme of type $\typeVd$ which is of finite type over a field $k$. Let $U=\Spec(R[1/f])$ be a distinguished affine open and let $Z \subseteq X$ be a closed subscheme with $|Z| = |X|\setminus |U|$. We equip both $U$ and $Z$ with the induced log structure. 
    
    Suppose that for every connected component $W$ of either $U$ or $Z$, the tame fundamental group $\pitame(W/k)$ is finitely generated.     
    Then 
    $\pitame(X/k, \bar x)$ 
    is finitely generated. 
\end{prop}
\begin{proof}
Let $T = \Spec(\widehat{R})$ where $\widehat{R} = \varprojlim R/f^{n+1}$ is the $f$-adic completion of $R$. The inclusion $Z \inj T$ induces a bijection $\pi_0(Z) \isomto \pi_0(T)$ of connected components. For $\alpha \in \pi_0(Z)$ corresponding to the connected component $Z_\alpha$  we denote by $T_\alpha$ the connected component of $T$ containing $Z_\alpha$. 

More generally, if $T'\to T$ is a finite morphism, and $Z' = Z\times_T T'$, then $T'$ is the spectrum of a ring which is complete with respect to the ideal $f\cdot \cO_{T'}$, and thus $\pi_0(Z')\isomto\pi_0(T')$ is also a bijection. Since the underlying maps of finite Kummer \'etale covers are finite (as remarked in \cref{def:Kummer-etale-map}, this follows from \cite[Corollary~5.1.2]{AHLS-Log} because $X$ is of type $\typeVd$), this formally implies that, if we endow $T$ with the log structure pulled back from $X$, so that $Z \inj T$ is strict, then for every $\alpha$ and a log geometric point $\overline{z}\to Z_\alpha$, the map
\[
    \pitame(Z_\alpha/k, \overline{z})\la\pitame(T_\alpha/k,\overline{z})
\]
is surjective. In particular, $\pitame(T_\alpha/k,\overline{z})$ is finitely generated.

Let us also consider $T^\circ = T\times_X U = \Spec(\widehat{R}[1/f])$ that fits into a cartesian square
\[ 
    \begin{tikzcd}
        T^\circ \ar[r,"s"] \ar[d,"c",swap] & U \ar[d] \\
        T \ar[r] & X.
     \end{tikzcd}
\]
Because $U$, $T$ and $T^\circ$ are noetherian, passing to connected components defines a finite graph 
\[
    E_\bullet = \Big[\begin{tikzcd}
        \pi_0(T) \sqcup \pi_0(U) & \ar[l,"c",shift left=0.7ex] \ar[l,"s",swap,shift right=0.7ex] \pi_0(T^\circ) & \ar[l,shift left=1ex] \ar[l,shift right=1ex] \ar[l] \varnothing  
    \end{tikzcd}\Big].
\]

Let us temporarily work with a more general $X$ that is not necessarily connected. Let $\idem(S)$ denote the set of idempotents in $\cO(S)$ for a scheme $S$. By \cite[Theorem~2.2]{ArtinFormalModuliII}, the sequence 
\[ 
    \begin{tikzcd}
        0 \ar[r] &\cO(X) \ar[r] &  \cO(T)\times \cO(U) \ar[r,"-"] &\cO(T^\circ),
    \end{tikzcd}
\]
is exact. Hence, the 
idempotents $\idem(X)$ of $X$ are the equaliser of 
\[
\begin{tikzcd}
        \idem(T) \times \idem(U) \ar[r,"c^\ast",shift left=0.7ex] \ar[r,"s^\ast",swap,shift right=0.7ex] 
        & \idem(T^\circ) .
    \end{tikzcd}
\]
This is equivalent to the map 
\[
\rH^0(X,\bZ/2\bZ) \isomto \rH^0(|E_\bullet|,\bZ/2\bZ)
\]
being an isomorphism. It follows that the natural map $\pi_0(|E_\bullet|) \la \pi_0(X)$
is bijective. In particular, if we again assume that $X$ is connected, then the graph $E_\bullet$ is also connected.

We now define a fibred category in Galois categories $\cG_\bullet$ over $E_\bullet$ by associating 
\begin{itemize}
    \item 
    $\FEtt_{T_\alpha/k}$ to $\alpha \in \pi_0(T)$, and
    \item $\FEtt_{U_\beta/k}$ to $\beta \in \pi_0(U)$ describing the connected component $U_\beta \subseteq U$, and likewise
    \item 
    $\FEtt_{T^\circ_\gamma/k}$ to $\gamma \in \pi_0(T^\circ)$ describing the connected component $T^\circ_\gamma \subseteq T^\circ$. 
    \item 
    The functors $c^*$ and $s^*$ of $\cG_\bullet$ are induced by the maps 
    $T^\circ_\gamma\to T_{c(\gamma)}$ and $T^\circ_\gamma\to U_{s(\gamma)}$. 
\end{itemize}
Because $E_\bullet$ is connected, \cref{prop:DD is galois for connected E} shows that $\Desc(\cG_\bullet)$ is a Galois category. Since $E_\bullet$ is finite and the connected components of $T$ and of $U$ have finitely generated tame fundamental groups, \cref{prop:DD-fgfp} shows that $\pi_1(\Desc(\cG_\bullet))$ is finitely generated. 

Next, we have an obvious pull-back functor between Galois categories
\begin{equation} \label{eqn:tame-desc}
    f^*\colon \FEtt_{X/k} \la \Desc(\cG_\bullet)
\end{equation}
that is exact and thus a morphism of Galois categories. It remains to show that 
the functor \eqref{eqn:tame-desc} is fully faithful. This means that we have to show that for a connected tame cover $X' \to X$, the induced object $f^*(X'\to X)$ in $\Desc(\cG_\bullet)$ is still connected. We use the criterion \cref{lem:local-system-connected} and need to show that with $T' = T \times_X X'$, and $U' = U \times_X X'$ and $T'^\circ = T^\circ \times_X X'$ the graph
\[
E'_\bullet = \Big[\begin{tikzcd}
        \pi_0(T') \sqcup \pi_0(U') & \ar[l,"c",shift left=0.7ex] \ar[l,"s",swap,shift right=0.7ex] \pi_0(T'^\circ) & \ar[l,shift left=1ex] \ar[l,shift right=1ex] \ar[l] \varnothing  
    \end{tikzcd}\Big]
\]
is connected. This follows from the bijection $\pi_0(|E'_\bullet|) \isomto  \pi_0(X')$
we proved above, because $X'$ is connected. This completes the proof. 
\end{proof}

\begin{proof}[Proof of \cref{thm:pi1tame log is fg}]
By descent \cite[Proposition~4.2.4]{AHLS-Log} and by an application of \cref{cor:DD-fgfp}~\ref{coritem:DD-fg}, we may work \'etale locally on $X$ (we have seen this strategy used in the proof of \cref{thm:naked pi1tame fg}). We may therefore assume that $X=\Spec(P \to R)$ is affine with a global chart by a monoid $P$ of type $\typeVd$.

\Cref{cor:generically-log-constant} shows that in the situation of \cref{thm:pi1tame log is fg}, there exists a non-empty open subset $U\subseteq X$ on which the log structure is locally constant. Shrinking $U$, we may assume that $U=\Spec(R[1/f])$ is a distinguished affine open. By \cref{cor:pi1tame locally constant log is fg}, the tame fundamental group of every connected component of $U$ is finitely generated. By noetherian induction, the proof now follows from \cref{prop:pi1tame-strat}.
\end{proof}

\subsection{Surgery, finite presentation, and examples}
\label{ss:surgery}

In this subsection, we show how the tame fundamental groups of certain log schemes (a class including the special fibres of strictly semistable formal schemes) can be computed by a surgical procedure of decomposing said scheme into its irreducible components and partially stripping off its log structure. This strategy is enabled by the following descent result, generalised beyond the fs setting in \cref{cor:descent-beyond-fs}.

\begin{thm}[{\cite[Theorem~3.2.25]{Stix2002:Thesis}}] \label{thm:Stix-descent}
    Let $f\colon S'\to S$ be an exact morphism of fs log schemes whose underlying map of schemes is finitely presented, proper, and surjective. Then $f$ satisfies universal effective descent for finite Kummer \'etale maps. 
\end{thm}

\begin{cor} \label{cor:descent-beyond-fs}
    Let $f\colon S'\to S$ be a strict morphism of saturated log schemes whose underlying map of schemes is finitely presented, proper, and surjective. Then $f$ satisfies universal effective descent for finite Kummer \'etale maps.
\end{cor}

\begin{proof}
The assertion is (strict) \'etale local on $S$, and hence we may assume that $S$ has a chart by a saturated monoid $P$. Write $P=\varinjlim P_\lambda$ with $P_\lambda$ fs, let $S_\lambda$ be $\underline{S}$ with log structure induced by $P_\lambda$, and let $S'_\lambda$ be $\underline{S}'$ with log structure so that $S'_\lambda\to S_\lambda$ strict. Then by \cite[Proposition~4.2.6]{AHLS-Log} we have
\[
    \FEt_S \simeq \varinjlim \FEt_{S_\lambda}, \qquad \FEt_{S'} \simeq \varinjlim \FEt_{S'_\lambda}
\]
(and the same for $S'\times_S S'$ and $S'\times_S S'\times_S S'$). The maps $S'_\lambda \to S_\lambda$ satisfy the assumptions of \cref{thm:Stix-descent}, and the result follows formally from this. 
\end{proof}

\begin{rmk}
    It is likely that the assertion of \cref{cor:descent-beyond-fs} holds more generally for $f$ exact but not necessarily strict. However, one would need to argue why the maps $S'_\lambda\to S_\lambda$ will be exact, proper, and surjective for $\lambda\gg 0$.
\end{rmk}

We are interested in tame fundamental groups of log schemes which are smooth over a generalised log point. Because of the tameness condition, we have to assume that charts exist Zariski locally, not just \'etale locally. The precise assumptions and notation are listed below.

\begin{setup} \label{setup:surgery}
Consider the following situation: 
\begin{itemize}
    \item 
    $k$ is a field,
    
    \item 
    $S = \Spec(M\to k)$ for a saturated monoid $M$ which is finitely generated over a sharp divisible valuative submonoid $V\subseteq M$ and a map $\alpha\colon M\to k$ with $\alpha^{-1}(k^\times) = \{0\}$ (the key cases of interest are $M=V$ and $V=0$), in particular $M$ is sharp,
    
    \item 
    $X\to S$ is a smooth morphism admitting a chart Zariski locally in the following sense: Zariski locally on $X$ there exists a smooth monoid morphism $\theta\colon M\to P$ (with respect to the set of primes non-invertible in $k$) and a strict \'etale map
    \[
        X \la \Spec(P\to k[P]\otimes_{k[M]} k) = \Spec\big(P\to k[P]/(\theta(M_+))\big)
    \]
    where $M_+ = M\setminus M^\times = M \setminus 0$.
    We assume moreover that $X$ is quasi-compact and connected.
     
    \item $\nu\colon X^{(0)}\to X$ is the normalisation map on the underlying schemes, endowed with the pull-back log structure,
    
    \item $X^{(a)}$ is the $(a+1)$-fold fibre product of $X^{(0)}\to X$, endowed with the pull-back log structure,
    
    \item $U^{(a)}\subseteq X^{(a)}$ is the open subset where $\overline{\cM}_{X^{(a)}}$ (which coincides with the pull-back of $\overline{\cM}_X$ to $X^{(a)}$) is locally constant ($U^{(a)}$ is open and dense by \cref{cor:generically-log-constant}),
    
    \item $X^{(a)}_\star$ is the scheme $\underline{X}{}^{(a)}$ endowed with the log structure induced by the open subset $U^{(a)}$.
\end{itemize}
\end{setup}

\begin{rmk}
As we shall learn in \cref{lem:surgery-properties}, there are natural maps $\varepsilon^{(a)}\colon X^{(a)}\to X^{(a)}_\star$ which intuitively are torus bundles as in \S\ref{ss:loc-const}, and the log schemes $X^{(a)}_\star$ are log smooth over $k$ with no log structure, i.e.\ toroidal embeddings. The following diagram then illustrates how in our setup the log scheme $X$ is disassembled into the log schemes $X_\star^{(a)}$. 
\[
    \begin{tikzcd}
        X \ar[dd,"\text{smooth}",swap] & X^{(0)} \ar[l,swap,"\nu"] \ar[d,"\varepsilon^{(0)}"] & \ar[l,shift right=0.7ex] \ar[l,shift left=0.7ex] X^{(1)} \ar[d,"\varepsilon^{(1)}"] \ar[d] & \ar[l,shift right=1ex] \ar[l,shift left=1ex] \ar[l] X^{(2)} \ar[d] \ar[d,"\varepsilon^{(2)}\text{ ``torus bundle''}"] & \cdots\\
        & X^{(0)}_\star \ar[dr] & X^{(1)}_\star \ar[d] & X^{(2)}_\star \ar[dl,"\text{smooth}"] & \cdots \\
        S = \Spec(M\to k) \ar[rr] & & \Spec(k)
    \end{tikzcd}
\]
\end{rmk}

\begin{ex}[Semistable curve] \label{ex:semistable-log-curve}
    Consider the above setup with $M=V$, let $\pi\in V$ be a nonzero element, and let $P= P_2(\pi) = V[e_1,e_2]/(e_1+e_2=\pi)$ be the ``standard semistable'' monoid as in \cite[Example~2.4.5]{AHLS-Log}. Let $X = \Spec(P\to k[x_1,x_2]/(x_1x_2))$ where the map sends $e_i$ to $x_i$ and all of $V\setminus \{0\}$ to zero. In this case, we have 
    \begin{align*}
        X^{(0)} & = \Spec(P\to k[x_1]) \sqcup \Spec(P\to k[x_2]), \\
        X^{(1)} & = X^{(0)} \sqcup \Spec(P\to k) \sqcup \Spec(P\to k),
    \end{align*}
    and 
    \[
        X^{(0)}_\star  = \Spec(\bN\to k[x_1]) \sqcup \Spec(\bN\to k[x_2]),
        \qquad 
        X^{(1)}_\star = X^{(0)}_\star \sqcup \Spec(k) \sqcup \Spec(k).
    \]
    We notice that $X^{(a)}_\star$ is log smooth over $k$ (with no log structure). This is the case in general.
\end{ex}

\begin{lem} \label{lem:strata-closures}
     In the Zariski local situation of \cref{setup:surgery}, where we consider the log scheme  
    \[
         W = \Spec(P\to k[P]\otimes_{k[M]} k) = \Spec\big(P\to k[P]/(\theta(M_+))\big),
    \]  
    let $F \subseteq P$ be a face in the image of $W \to \Spec(P)$. Then the following hold.
    \begin{enumerate}[(a)]
        \item  \label{lemitem:strata-closures1}
        The monoid $F$ is fs and $0 \to F$ is smooth.
        
        \item \label{lemitem:strata-closures2}
        The log scheme $Z_F = \Spec(P \to k[F])$, where $x \in P \setminus F$ maps to $0$ and $x \in F$ maps to $x \in k[F]$, is the closure of the log stratum of $W$ associated to $F$.
        
        \item  \label{lemitem:strata-closures3}
        The log scheme $Z_{F,\ast} = \Spec(F \to k[F])$ is log smooth over $\Spec(k)$ with trivial log structure, and it agrees with the log scheme structure on $\underline{Z}_F$ induced by the open subset that is the log stratum associated to $F$.
        
        \item  \label{lemitem:strata-closures4}
        The underlying scheme $\underline{Z}_F = \Spec(k[F])$ is normal.
    \end{enumerate}
\end{lem}
\begin{proof}
    \labelcref{lemitem:strata-closures1} The image of $F$ in $\Spec(M)$ by the map $\theta: \Spec(P) \to \Spec(M)$ must be the image of $\Spec(M \to k) \to \Spec(M)$ which is the face $0$. In other words $0 = \theta^{-1}(F) = F \cap M$. The morphism $0 = \theta^{-1}(F) \to F$ is sft by \cite[Lemma~2.3.2]{AHLS-Log}, and hence $F$ is an fs monoid by \cite[Proposition~2.2.9]{AHLS-Log}.    

    Since $M$ is sharp and saturated, $M^\gp$ is torsion free. Hence torsion of $F^\gp$ injects into $P^\gp/M^\gp$ which only contains torsion of order prime to $p$ because $M \to P$ is smooth. It follows that $0 \to F$ is smooth. 

    \labelcref{lemitem:strata-closures2}
    The map $P \to k[F]$ induces a surjective homomorphism $k[P]/(M_+) \surj k[F]$ and thus a strict closed immersion $Z_F \inj W$. The associated vanishing ideal is generated by $P \setminus F$, which is the defining ideal for the closure of the stratum associated to $F$.

    \labelcref{lemitem:strata-closures3} This follows from \labelcref{lemitem:strata-closures1}.

    \labelcref{lemitem:strata-closures4}
    By \cite[Proposition I.3.4.1]{Ogus}, $\underline{Z}_F$ is normal since $F$ is fs. 
\end{proof}

\begin{lem} \label{lem:surgery-properties}
    In the situation of \cref{setup:surgery}, the following statements hold. 
    \begin{enumerate}[(a)]
        \item \label{lemitem:surgery-finite-type}
        The underlying scheme of $X$ is of finite type over $k$.
        
        \item \label{lemitem:surgery-finite-normalisation}
        The map $\nu\colon X^{(0)}\to X$ is finite and surjective, and for every point $x \in X^{(0)}$ the residue fields $k(x) = k(\nu(x))$ agree. 
        
        \item \label{lemitem:surgery-tameness}
        A finite Kummer \'etale map $Y\to X$ is tame relative to $k$ if and only if its base change $Y^{(0)} = Y\times_X X^{(0)}\to X^{(0)}$ is tame relative to $k$. 
                
        \item \label{lemitem:surgery-fs-and-smooth}
        The log scheme $X^{(a)}_\star$ is fs and smooth over $k$ (with trivial log structure), and $U^{(a)}$ is the locus where the log structure on $X^{(a)}_\star$ is trivial.

        \item \label{lemitem:surgery-relative-strip-off-log}
        There exists a unique morphism of log schemes $\varepsilon^{(a)}\colon X^{(a)} \to X^{(a)}_\star$ inducing the identity on the underlying schemes. 
    \end{enumerate}
\end{lem}

\begin{proof}
These assertions are Zariski local on $X$, so we may assume that there exists a smooth morphism of saturated monoids $M\to P$ (necessarily injective) and a strict \'etale map  
\[ 
    X \la W = \Spec(P\to k[P]/(M_+)), \qquad M_+ = M\setminus 0 \subseteq P.
\]
By \cite[Corollary~2.5.18]{AHLS-Log} we see that $P$ is finitely presented over $M$, hence $W$ and consequently $X$ are of finite type, showing \labelcref{lemitem:surgery-finite-type}. In particular, the normalisation map $\nu$ is finite, showing the first part of  \labelcref{lemitem:surgery-finite-normalisation}.

In order to show the claim on the residue fields in \labelcref{lemitem:surgery-finite-normalisation}, we first explicate the map $\nu\colon X^{(0)}\to X$ in the Zariski local situation $W$ used in the previous paragraph and studied in \cref{lem:strata-closures}. 
The irreducible components of~$W$ correspond to the connected components of the $Z_F \inj W$ associated to maximal faces $F$ in the image of $W \to \Spec(P)$. These $Z_F$ are all normal and thus $W^{(0)}$, the normalisation of $W$, is the disjoint union of these maximal $Z_F$. Since $X \to W$ is strict \'etale, the same holds for the normalisation $X^{(0)} = X \times_W W^{(0)}$: this is the disjoint union of the connected components of the $X_F \inj X$ which are the base changes of $Z_F \inj W$ along the strict \'etale $X \to W$; they can also be described as the closures of the open strata of the log stratification of $X$.

Since $X^{(0)}$ is the union of the irreducible components of $X$, the assertion about the residue fields in \labelcref{lemitem:surgery-finite-normalisation} follows at once, and assertion \labelcref{lemitem:surgery-tameness} is also an immediate consequence. 

For assertion \labelcref{lemitem:surgery-fs-and-smooth}, we note first that for faces $F_1$ and $F_2$ in the image of $W \to \Spec(P)$, the intersection $F_{12} = F_1 \cap F_2$ is a face, and naturally 
\[
Z_{F_1} \cap Z_{F_2} = Z_{F_1} \times_W Z_{F_2} = Z_{F_{12}}.
\]
Consequently, we have the equality $X_{F_{12}} = X_{F_1} \cap X_{F_2}$. It follows that $X^{(a)}$ is also a disjoint union of normal subschemes of $X$ of the form $X_F$ for some faces $F$ of $P$. This makes $X^{(a)}_\star$ a disjoint union of subschemes $X_{F,\star}$ which are strict \'etale over $Z_{F,\star}$ and therefore fs and log smooth over $\Spec(k)$ with trivial log structure. 
 
The map $X^{(a)}\to X^{(a)}_\star$ in \labelcref{lemitem:surgery-relative-strip-off-log} is locally on $X_F$ given by  the map $X_F\to X_{F,\star}$ induced by the obvious map 
\[
(F \to k[F]) \la (P \to k[F])
\]
of prelog rings. In order to construct the map globally we need to establish its uniqueness locally. 
Since, by definition, $\cM_{X^{(a)}_\star}$ is a subsheaf of $\cO_{X^{(a)}}$, it suffices to show that $\alpha\colon \cM_{X^{(a)}}\to\cO_{X^{(a)}}$ maps the preimage of $\cM_{X^{(a)}_\star}\subseteq \cO_{X^{(a)}}$ isomorphically onto $\cM_{X^{(a)}_\star}$. This is an \'etale  local question to be checked in a geometric point $\bar{x}$ of $X_F$.  An element $(q,u)$ of the stalk 
\[
\cM_{X^{(a)},\bar x} = P\oplus_{\alpha_{\bar x}^{-1}(\cO_{X,(\bar x)}^\times)} \cO_{X,(\bar x)}^\times
\]
maps to a nonzero element of $\cO_{X,(\bar x)}$ if and only if $q\in F$. Thus, the preimage in question is $F\oplus_{\alpha_{\bar x}^{-1}(\cO_{X,(\bar x)}^\times)} \cO_{X,(\bar x)}^\times$, which maps isomorphically onto the stalk $\cM_{X^{(a)}_\star, \bar x}$. This shows \labelcref{lemitem:surgery-relative-strip-off-log} locally, and thus constructs the map in question globally together with the claimed uniqueness.
\end{proof}

\begin{rmk}
    The proof of \labelcref{lemitem:surgery-relative-strip-off-log} used the special local form of $X^{(a)}$, and the map $\varepsilon$ might not exist or be non-unique in simple situations. For an example exhibiting non-uniqueness, consider the diagonal in $\bA^2 = \Spec(\bN^2\to k[\bN^2])$, i.e.\ $X=\Spec(\bN^2\to k[x])$ where $\bN^2\to k[x]$ sends both generators $e_0,e_1$ to $x$. In this case, $U=D(x)$ and $X_\star = \Spec(\bN\to k[x])$ where $\bN\to k[x]$ sends the generator $e$ to $x$. There are two different maps $\varepsilon_i\colon X \to X_\star$, for $i=0,1$, induced by the identity on $k[x]$ and the two inclusions $\bN\to\bN^2$ sending $e\mapsto e_i$.
\end{rmk}

Thanks to \cref{lem:surgery-properties}~\ref{lemitem:surgery-finite-normalisation} and \ref{lemitem:surgery-tameness} and \cref{cor:descent-beyond-fs}, the van Kampen formula of \cref{sec:descent for Gal cats} yields (up to a choice of base points and paths) an isomorphism
\[ 
    \pitame(X/k) \simeq \pitame(\Desc(X^{(0)}/X), \FEtt_{-/k}).
\]
The next lemma allows us to strip off part of the log structure on the connected components of $X^{(a)}$ appearing implicitly in the above formula. It is a ``partially compactified'' version of \cref{cor:torus-fibration-tame}.

\begin{lem} \label{lem:right-exact-Xa}
    Let $Y$ be a connected component of $X^{(a)}$ and let $Y_\star$ be the corresponding component of $X^{(a)}_\star$ (the underlying schemes of $X^{(a)}$ and $X^{(a)}_\star$ are the same). Let $\bar\eta$ be a geometric generic point of $Y$. Then we have an exact sequence
    \[
        \begin{tikzcd}
            \pi_1(\overline{\cM}_{Y,\bar\eta}) \ar[r] & \pitame(Y/k,\bar\eta)\ar[r] & \pitame(Y_\star/k,\bar\eta)\ar[r] & 1.
        \end{tikzcd}
    \]
    The group on the left is isomorphic to $\widehat{\mathbb{Z}}'(1)^r$ for some $r\geq 0$. 
\end{lem}

\begin{proof}
Let $U=U^{(a)}\cap Y$ be the locus where $\overline{\cM}_Y$ is locally constant, let $j\colon U\to Y$ be the inclusion, and let $U_\star\subseteq Y_\star$ be the corresponding open. Then $U_\star=\underline{U}$ is the locus where the log structure on $Y_\star$ is trivial, and since $Y_\star$ is log regular by \cref{lem:surgery-properties}~\labelcref{lemitem:surgery-fs-and-smooth}, the open $U_\ast$ is connected and the map $\pitame(U_\star/k)\to \pitame(Y_\star/k)$ is an isomorphism by log purity as recalled later in \cref{thm:GR-log-Abhyankar}. We have a commutative diagram of profinite groups
\[
    \begin{tikzcd}
            \pi_1(\overline{\cM}_{U,\bar\eta}) \ar[r] \ar[d,equal] & \pitame(U/k,\bar\eta)\ar[r] \ar[d] & \pitame(U_\star/k,\bar\eta)\ar[r] \ar[d,"\simeq\text{ by purity}"] & 1 \\
            \pi_1(\overline{\cM}_{Y,\bar\eta}) \ar[r] & \pitame(Y/k,\bar\eta)\ar[r] & \pitame(Y_\star/k,\bar\eta)\ar[r] & 1
    \end{tikzcd}
\]
in which the top row is exact by \cref{cor:torus-fibration-tame}. We claim that the middle vertical arrow $\pitame(U/k)\to\pitame(Y/k)$ is surjective. If this is the case, then an easy diagram chase shows that the bottom row is exact as well, completing the proof.

The assertion that $\pitame(U/k)\to\pitame(Y/k)$ is surjective is equivalent to the full faithfulness of $\FEtt_{Y/k}\to\FEtt_{U/k}$. Since both source and target are full subcategories of $\FEt$, we can ignore the tameness condition and show the full faithfulness of $\FEt_Y\to\FEt_U$. This translates to: if $Y'\in \FEt_Y$ is connected, then $U' = Y'\times_Y U$ is connected. In order to show this, it is enough to show that $U'$ is dense in $Y'$ and that $Y'$ is geometrically unibranch. We can check these properties by an étale local calculation. Locally, $X$ has the form $\Spec(P\to k[P]/(M_+))$ and $Y$ is a connected component of $\Spec(P\to k[F])$ for a face $F \subseteq P$ with $F \cap M = 0$, see the proof of \cref{lem:surgery-properties}. Then $Y'$ locally is a connected component of  
\[
    \Spec(Q \to k[F]\otimes_{k[P]} k[Q])
\]
for some Kummer \'etale map $P \to Q$ (\cref{lem:fKet-local-structure}). Let $G$ be the face of $Q$ corresponding to $F$ under the bijection $\Spec(Q) \to \Spec(P)$. Then the reduction of $Y'$ is locally of the form $\Spec(Q \to k[G])$. Now $F \to G$ is Kummer \'etale, hence $G$ is also fs and smooth over $0$.
Thus, the reduction of $Y'$ is normal, and hence $Y'$ is geometrically unibranch. The same calculation shows that $U'$ is dense as its underlying topological space locally equals that of $\Spec(k[G^{\rm gp}])$ in the above presentation.
\end{proof}

\begin{cor} \label{cor:surgery-fp}
    In the situation of \cref{setup:surgery}, suppose that $k$ is algebraically closed and that the smooth scheme $U^{(0)}$ (with trivial log structure) admits a projective toroidal  compactification. Then $\pitame(X/k)$ is finitely presented. 
\end{cor}

\begin{proof}
By the van Kampen formula of \S\ref{sec:descent for Gal cats}, since each $X^{(a)}$ (for $a\leq 2$) has finitely many connected components, it suffices to show that for every connected component $Y$ of some $X^{(a)}$, the group $\pitame(Y/k)$ is finitely generated (true by our general result \cref{thm:pi1tame log is fg}), and finitely presented if $a=0$. By \cref{lem:right-exact-Xa}, this group sits inside an exact sequence 
\[
    \begin{tikzcd}
        \pi_1(\overline{\cM}_{Y,\bar\eta}) \ar[r] & \pitame(Y/k,\bar\eta)\ar[r] & \pitame(Y_\star/k,\bar\eta)\ar[r] & 1.
    \end{tikzcd}
\]
Arguing as in the proof of \cref{cor:pi1tame locally constant log is fg}, we see that $\pitame(Y/k)$ is finitely presented if and only if $\pitame(Y_\star/k)$ is. On the other hand, $U=U^{(0)}\cap Y_\star$ (which is a connected component of $U^{(0)}$) is the locus where the log structure on $Y_\star$ is trivial, and hence by log purity (\cref{thm:GR-log-Abhyankar}) we have $\pitame(Y_\star/k)\simeq \pitame(U/k)$ (see the proof of \cref{lem:right-exact-Xa}). By our assumption, $U$ has a projective toroidal compactification $\overline{U}$. Applying a suitable log blowup \cite{Niziol}, we may assume that $\overline{U}$ is an snc compactification, in which case the assertion follows from \cref{thm:naked pi1tame fp}.
\end{proof}

\begin{ex}[Semistable curve, continued] \label{ex:semistable-log-curve-2}
    Consider the situation in \cref{ex:semistable-log-curve}, and assume that $k$ is algebraically closed. Let $Y_i = \Spec(P \to k[x_i])$, for $i = 1,2$, denote the two connected components of $X^{(0)}$.
    Since $M=V$ is divisible and equal to the stalk of $\overline{\cM}$ at the generic points of $X^{(0)}$, for $a=0$ the group $\pi_1(\overline{\cM}_{Y_i,\bar\eta})$ in \cref{lem:right-exact-Xa} is trivial.  This identifies both $\pitame(Y_1/k)$ and $\pitame(Y_2/k)$ canonically with $\pitame(\Spec(\bN \to k[x])/k) = \widehat{\mathbb{Z}}'(1)$. As both maps $Y_i \to X$ are monomorphisms, we can by \cite[Proposition~4.3]{Stix2006:vanKampen} simplify the category of descent data to `ordered descent data' where we fix an ordering of the connected components of $X^{(0)}$ and only consider those components of $X^{(a)}$ arising from fibre products for ordered $(a+1)$-many factors. Here this means we are left with just $Y_{12} = Y_1 \times_X Y_2 = \Spec(P \to k)$. \cref{prop:FKEt-pi1-local} yields the formula $\pitame(Y_{12}/k) = \pi_1(P) = \widehat{\mathbb{Z}}'(1)$.   
    By van Kampen, we have a pushout diagram
    \[
        \begin{tikzcd}
            \pitame(Y_{12}/k) \ar[r] \ar[d] & \pitame(Y_1/k) \ar[d] \\
            \pitame(Y_2/k) \ar[r] & \pitame(X/k)
        \end{tikzcd}
    \]
    which thanks to what we know equals
    \[
        \begin{tikzcd}
            \widehat{\mathbb{Z}}'(1) \ar[r] \ar[d] & \widehat{\mathbb{Z}}'(1) \ar[d] \\
            \widehat{\mathbb{Z}}'(1) \ar[r] & \pitame(X/k).
        \end{tikzcd}
    \]
    The top and left maps are isomorphisms (a special case of \cref{lem:pi1XF} below), and we conclude that 
    \[
        \pitame(X/k) \simeq \widehat{\mathbb{Z}}'(1).  
    \]
\end{ex}

The following lemma (whose inclusion was prompted by a question of Paul Alexander Helminck, see \cref{ex:helminck}) will allow us to treat more general examples of the toric kind. 

\begin{lem} \label{lem:pi1XF}  
     In the Zariski local situation of \cref{setup:surgery}, assume $k$ is algebraically closed and $P$ is sharp. Let $F \subseteq P$ be a face with $F \cap M = 0$, and denote 
     \[
     Z_F = \Spec(P\to k[F])
     \]
     as in \cref{lem:strata-closures}. Then $Z_F$ is connected and the map $Z_0 = \Spec(P\to k) \inj Z_F$, where the subscript $0$ refers to the face $0\subseteq P$, induces isomorphisms
     \[
        \pi_1(P) = \pi_1(Z_0) = \pitame(Z_0/k)\la \pitame(Z_F/k)
    \]
    (since these groups are abelian, we can omit base points). 
\end{lem}

\begin{proof}
Since $P$ is sharp, also all its faces $F \subseteq P$ are sharp, and $F^\gp$ does not contain torsion. Since $\Spec(k[F^\gp])$ is dense in $\underline{Z}_F$, the log scheme $Z_F$ is connected. 

Let us now show that $\FEtt_{Z_F/k}\to \FEtt_{Z_0/k} = \FEt_{Z_0}$ is essentially surjective. By \cref{prop:FKEt-pi1-local}, every connected cover of $Z_0$ is of the form $\Spec(Q\to k[Q]\otimes_{k[P]} k)$ for some Kummer \'etale map $P\to Q$. This map is the base change to $Z_0$ of the Kummer \'etale map
\[
    Y_F(Q) = \Spec(Q \to k[Q] \otimes_{k[P]} k[F]) \la Z_F.
\]
This map is Galois as an object of $\FEt_{Z_F}$ with Galois group $\Hom(Q^\gp/P^\gp,\mu_\infty(k))$ 
of degree prime to $p$, and hence is tame relative to $k$. 

To show full faithfulness, we need to show that if $Y \to Z_F$ is a connected cover which is tame relative to $k$, then $Y_0 = Y \times_{Z_F} Z_0$ is connected. To this end, consider the map 
\[
    (Y_{\rm red})_\star \la Z_{F,\star} = \Spec(F \to k[F]).
\]
By \cref{lem:fKet-local-structure}, $Y$ is \emph{locally on $Z_F$} isomorphic to the disjoint union of some $Z_F(Q)$'s. Since $(Z_F(Q)_{\rm red})_\star \simeq \Spec(G\to k[G])$ where $G \subseteq Q$ is the saturation of $F$ in $Q$ (or the unique face above $F$), see the computation in the proof of \cref{lem:right-exact-Xa}, it follows that $(Y_{\rm red})_\star$ is locally on $Z_{F,\star}$ the disjoint union of $\Spec(G \to k[G])$ for some Kummer \'etale maps $F \to G$. We deduce from this analysis that $(Y_{\rm red})_\star \to Z_{F,\star}$ is Kummer \'etale and tame. 
On the other hand, by purity as in the proof of \cref{lem:right-exact-Xa} we compute 
\[
\pitame(Z_{F,\star}/k) = \pitame(\Spec(k[F^{\rm gp}])/k) = \pi_1(F).
\]
It follows that 
$(Y_{\rm red})_\star \simeq \Spec(G\to k[G])$
for some Kummer \'etale map $F\to G$. The fibre of $(Y_{\rm red})_\star$ above $\underline{Z}_0=\Spec(k)$ consists of a single (possibly non-reduced) point, and hence is connected. However, the log schemes $(Y_{\rm red})_\star$ and $Y$ share their underlying topological space, and we conclude that $Y_0$ is connected.
\end{proof}

\begin{cor} \label{cor:Paul}
    In the Zariski local situation of \cref{setup:surgery}, where we consider the log scheme  
    \[
         W = \Spec(P\to k[P]\otimes_{k[M]} k) = \Spec\big(P\to k[P]/(\theta(M_+))\big),
    \]  
    assume $k$ is algebraically closed and $P$ is sharp. Then $W$ is connected and the inclusion $Z_0 \inj W$ (notation as in \cref{lem:strata-closures}) induces an isomorphism
    \[
        \pi_1(P) = \pi_1(Z_0) = \pitame(Z_0/k) \isomto \pitame(W/k). 
    \]   
\end{cor}

\begin{proof}
Since $P$ is sharp, also all its faces $F \subseteq P$ are sharp. Therefore the closures of log strata $Z_F \inj W$ are connected by \cref{lem:pi1XF}. Since all $Z_F$ contain the minimal log stratum $Z_0 = \Spec(P \to k)$, the scheme $W$ is connected. 

We apply the van Kampen formula with ordered descent data given by the  connected components $Z_F = \Spec(P \to k[F])$ of $W^{(0)}$ (with $F \subseteq P$ maximal faces in the image of $W \to \Spec(P)$, see the proof of \cref{lem:surgery-properties}). We deduce from \cref{lem:pi1XF} that all connected components of $W^{(a)}$, for all $a\geq 0$, have fundamental groups canonically isomorphic to $\pi_1(P)$.  Since all fibre products of ordered factors are again of the form $Z_F$ and so connected, the underlying $2$-complex for ordered descent data is the $2$-truncation of a full simplex. Moreover, all constituents have tame fundamental group canonically isomorphic to $\pi_1(P) = \pitame(Z_0/k)$ which is abelian (thus base points do not matter), so we can choose the paths so that the annoying datum $\alpha_{ijk}^{(f)}$ in \cite[Corollary~5.3]{Stix2006:vanKampen} is the identity in all cases. Thus, the van Kampen formula simplifies to just~$\pi_1(P)$.
\end{proof}

\section{Vertical compactifications of adic spaces}
\label{s:compactifications}

This section is self-contained and serves as a preparation for \cref{s:fg-pi1-rig}.
Its goal is to study the universal vertical compactification $\Spa(X,S)$ of an adic space $X$ locally of finite type over a base adic space $S$, which exists under some mild hypotheses.
For rigid spaces this recovers the universal compactification from Huber's book \cite[\S 5.1]{HuberBook}, and for a map of discretely ringed adic spaces $X_0^{\rm ad}\to S_0^{\rm ad}$ induced by a map of schemes $X_0\to S_0$ one obtains the space $\Spa(X_0, S_0)$ we have already encountered in \cref{sec:adic tame discrete case}. Later in \cref{s:fg-pi1-rig}, we will need to consider a situation which is of neither kind, which is the map of adic spaces induced by an adic morphism of formal schemes $\fX\to\Spf(K^\circ)$.   The general construction, reviewed in \cref{ss:univ-comp}, is due to Solodov \cite{Solodov-vertical-compactification}.
However, in \cref{sec:root-triples} we give an alternative description of this compactification in terms of root triples. 

As we shall explain in \cref{ss:tame-pi1-rig}, in order to obtain a notion of (relative) tameness in rigid geometry with good finiteness properties, one should study finite \'etale morphisms $Y \to X$ whose extension to the universal vertical compactification $\Spa(X,S)$ satisfies a pointwise tameness condition. 

To a reader who is not interested in diving into the adic space structure of the universal vertical compactification, we recommend focusing on the definition of root triples in \cref{sec:root-triples}.
In most of the arguments in \cref{s:fg-pi1-rig} it is enough to know that the underlying set of $\Spa(X,S)$ is the set of root triples.

\subsection{Review of adic spaces}
\label{ss:adic-spaces}

We recall some general facts and terminology concerning adic spaces.
The types of adic spaces we encounter in this article are rigid spaces, their formal models, and discretely ringed adic spaces.
Accordingly, our focus will be on these spaces and their interplay.

\subsubsection*{Adic spaces}

For a Huber pair $(A, A^+)$ we denote by $A^\circ\subseteq A$ the subring of power-bounded elements and by $A^{\circ\circ}\subseteq A^\circ$ the ideal of topologically nilpotent elements. Then $(A, A^\circ)$ is a Huber pair as well, and we abbreviate $\Spa(A, A^\circ)$ to $\Spa(A)$.

If $(A, A^+)\to (B, B^+)$ is a homomorphism of Huber pairs, we write $\Spa(B, A^+)$ for the adic space $\Spa(B, B^+_{\min})$ where $B^+_{\min}\subseteq B$ is the smallest ring of integral elements containing the image of $A^+$, i.e.\ the integral closure of the sub-$A^+$-algebra of $B$ generated by $B^{\circ\circ}$.
Note that a continuous valuation of~$B$ that takes values less than or equal to~$1$ on~$A^+$ is already less than or equal to~$1$ on~$B^+_{\min}$.

An {\bf affinoid field} is a Huber pair $(K, K^+)$ where $K$ is a field and $K^+$ is a valuation ring of $K$, equipped with the valuation topology.
The property of $(K,K^+)$ being a Huber pair forces~$K^+$ to be {\bf microbial} meaning that it admits an element $\pi$ such that $K = K^+[1/\pi]$.
Every such element is called a {\bf pseudouniformiser}.
The residue field of $K^+$ is denoted by $K^\succ$ and called the {\bf specialisation field} of $K$. 

\subsubsection*{Specialisation}

Next, we discuss specialisation relations on adic spaces. A specialisation $x\leadsto x'$ on a topological space $X$ is a pair of points $x,x'\in X$ such that $x'\in\overline{\{x\}}$. 

If $x'$ is the specialisation of a point $x$ of an adic space $X$, we have a cospecialisation map $\cO_{X,x'}\to \cO_{X,x}$.
Then $x'$ is a \textbf{vertical (secondary)} specialisation of $x$ if and only if this map is a local homomorphism, so that it induces a homomorphism of residue fields $k(x')\to k(x)$. 
If $X$ is analytic, then every specialisation on $X$ is vertical. 

For a point $x\in X$ the following are equivalent:
\begin{enumerate}
    \item $x$ has no vertical generalisations,
    \item $k(x)^+ = k(x)^\circ$ (powerbounded elements),
    \item \begin{enumerate}[(i)]
            \item $x$ is analytic and $k(x)^+$ is of rank one or
            \item $k(x)$ has the trivial valuation (meaning $k(x)^+=k(x)$).
          \end{enumerate}
\end{enumerate}
We call a point satisfying the above equivalent conditions {\bf vertically maximal}.
For a point~$x$ of an adic space~$X$ we denote by~$x^\circ$ its {\bf maximal vertical generalisation}. 
It is the point corresponding to the valuation ring $k(x)^\circ$ of~$k(x)$.
Sometimes we will just call a point without vertical generalisations~$x^\circ$, without specifying any point~$x$.

\begin{lem} \label{lem:adic-preserves-vertical-maximal}
    Let $f \colon Y \to X$ be an adic morphism of adic spaces.
    Then for any point $y \in Y$ we have $f(y)^\circ = f(y^\circ)$.
    In particular,~$f$ takes vertically maximal points to vertically maximal points.
\end{lem}

\begin{proof}
    For every adic homomorphism of Huber rings $A \to B$, the image of~$A^\circ$ is contained in~$B^\circ$ (compare \cite[Lemma~1.8~(i)]{HuberContinuousValuations}).
    This implies that
    \[
        \widehat{k(f(y)^\circ)}^+ = \widehat{k(f(y))}^\circ \subseteq \widehat{k(y)}^\circ = \widehat{k(y^\circ)}^+
    \]
    and consequently, 
    \[
        \widehat{k(f(y)^\circ)}^+ \subseteq \ \widehat{k(y^\circ)}^+ \cap \widehat{k(f(y))} = \widehat{k(f(y^\circ))}^+.
    \]
    The inclusion of valuation rings translates to a vertical specialisation relation $f(y^\circ) \leadsto f(y)^\circ$.
    The maximality of~$f(y)^\circ$ (with respect to vertical generalisation) forces the equality $f(y^\circ) = f(y)^\circ$.
\end{proof}

There is a similar characterisation for horizontal specialisations: $x'$ is a {\bf horizontal (primary) specialisation} if and only if $\cO_{X,x'}^+\to\cO^+_{X, x}$ is a local homomorphism, so that we get a homomorphism of specialisation fields $k(x')^\succ \to k(x)^\succ$.
Finally, a {\bf generalised horizontal specialisation} is a specialisation $x \leadsto x'$ that is either horizontal or factors as $x \leadsto y \leadsto x'$ such that~$y$ and~$x'$ are trivial valuations, and $x \leadsto y$ is horizontal.

\subsubsection*{Rigid spaces}

A {\bf non-archimedean field} is a field $K$ which is complete with respect to a rank one valuation. Its valuation ring $K^+ = K^\circ$ is thus $\pi$-adically complete with respect to every pseudouniformiser $\pi$ of $K^\circ$.

Let $K$ be a non-archimedean field. A {\bf rigid space} over $K$ is an adic space locally of finite type over $\Spa(K)$. So, a rigid space over $K$ is locally of the form $\Spa(A)$ for an affinoid $K$-algebra $A$. We shall denote the category of rigid spaces over $K$ by $\cat{Rig}_K$.

\subsubsection*{Formal schemes}

For our subsequent discussion of formal models and the specialisation map we shall rely on \cite{FujiwaraKato}, using \cite[Theorem~II~A.5.1]{FujiwaraKato} to translate results to the language of adic spaces.

Let $K$ be a non-archimedean field with valuation ring $K^+=K^\circ$ and pseudouniformiser $\pi$. If $A$ is an algebra topologically of finite type over $K^\circ$ (that is, a quotient of $K^\circ\langle T_1, \ldots, T_n\rangle$ for some $n\geq 1$, endowed with the $\pi$-adic topology), then $A_K = A[1/\pi]$ is an affinoid $K$-algebra and $A_K^\circ$ is the integral closure of the (image of) $A$ in $A_K$. We call a formal scheme over $K^\circ$ locally of finite type if it is locally of the form $\Spf(A)$ for a $K^\circ$-algebra $A$ topologically of finite type. We denote the category of such formal schemes over $K^\circ$ by $\cat{FSch}^{\rm lft}_{K^\circ}$. 

In \cite[\S~4]{Huber94:Generalization} Huber explains how the category of locally noetherian formal schemes embeds fully faithfully into the category of adic spaces by sending $\Spf(A)$ to $\Spa(A)$.
Being noetherian was needed because he had established sheafiness of~$A$ only in the noetherian setting (see \cite[Theorem~2.2]{Huber94:Generalization}).
Today we have \cite{ZavyalovSheafy} at hand that settles sheafiness also if~$A$ is only strongly rigid noetherian (e.g. topologically of finite type over a microbial valuation ring).
Consequently, we have a fully faithful functor
\[
    (-)^{\rm ad} \colon \cat{FSch}^{\rm lft}_{K^\circ} \longrightarrow \cat{AdSpc}_{K^\circ}
\]
to the category of adic spaces that are locally of finite type over~$K^\circ$.

\subsubsection*{Rigid generic fibre}

There exists a unique functor 
\[
    (-)_{\rm rig} \colon \cat{FSch}^{\rm lft}_{K^\circ} \longrightarrow \cat{Rig}_K,
\]
called the {\bf rigid generic fibre} functor, which sends open coverings to open coverings, and such that 
\[ 
    \Spf(A)_{\rm rig} \simeq \Spa(A_K).
\]
If~$\cX$ is a formal scheme locally of finite type over~$K^\circ$, then~$\cX_{\rm rig}$ is the generic fibre of the morphism $\cX^{\rm ad} \to \Spa (K^\circ)$.
Being a fibre product of adic spaces, it preserves open and closed immersions, finite maps, étale maps, etc.

\subsubsection*{The specialisation map}

Let $A$ be an algebra topologically of finite type over $K^\circ$. For any locally topologically ringed space $X$, we have a functorial isomorphism $\Hom(X, \Spf(A))\simeq \Hom(A, \cO(X))$ (continuous homomorphisms). Thus, the map $A\to A_K^\circ$ induces a map of locally topologically ringed spaces
\[ 
    {\rm sp}_A \colon (\Spf(A)_{\rm rig}, \cO^+) = (\Spa(A_K),  \cO^+) \longrightarrow (\Spf(A), \cO). 
\]
This construction respects basic open subsets: the preimage of $\{f\neq 0\}=\Spf(A\langle f^{-1}\rangle)$ is $\{|f|\geq  1\} = \Spa(A\langle f^{-1}\rangle_K$ and ${\rm sp}_A$ restricts to ${\rm sp}_{A\langle f^{-1}\rangle}$.  Hence it globalises: for every formal scheme $\fX$ locally of finite type over $K^\circ$ there exists a unique map of topologically locally ringed spaces
\[
    {\rm sp}_{\fX} \colon (\fX_{\rm rig}, \cO^+)\longrightarrow \fX
\]
such that for every map $\Spf(A)\to \fX$ of formal schemes locally of finite type over $K^\circ$, the resulting square 
\[ 
    \begin{tikzcd}
        (\Spa(A_K), \cO^+) \ar[d] \ar[r,"{\rm sp}_A"] & \Spf(A) \ar[d] \\
        (\fX_{\rm rig}, \cO^+) \ar[r,"{\rm sp}_\fX",swap] & \fX.
    \end{tikzcd}
\]
commutes.
Furthermore, we have ${\rm sp}_{\fX,*}(\cO) = \cO_\fX[1/\pi]$, and ${\rm sp}_{\fX,*}(\cO^+)$ coincides with the integral closure of (the image of) $\cO_{\fX}$ in $\cO_{\fX}[1/\pi]$.

\begin{defi}[{\cite[Definition~A.1]{AchingerLaraYoucis2022}}] \label{defi:eta-normal}
    A formal scheme $\fX$ locally of finite type over~$K^\circ$ is 
    {\bf \mbox{$\eta$-normal}} if $\cO_{\fX}$ is $\pi$-torsion free and integrally closed in $\cO_{\fX}[1/\pi]$, or equivalently if the map $\cO_\fX \isomto {\rm sp}_{\fX*}(\cO^+)$ is an isomorphism.
\end{defi}

\begin{rmk} \label{rmk:sp-facts}
We shall use the following facts about ${\rm sp}={\rm sp}_\fX\colon X=\fX_{\rm rig}\to \fX$.
\begin{enumerate}
    \item On the level of points, the specialisation map can be described as follows. Let $x\in X$ and let $\Spf(A)\subseteq\fX$ be an affine neighbourhood of ${\rm sp}(x)$, so that $\Spa(A_K)\subseteq X$ is an affinoid neighbourhood of $x$. Then $k(x)^+$ is a $\pi$-adically separated valuation ring, and the map $A\to k(x)^+$ induces a map of formal schemes 
    \[
        \Spf(k(x)^+) \longrightarrow \Spf(A) \subseteq \fX
    \]
    which sends the closed point of $\Spf(k(x)^+)\simeq \Spec(k(x)^+/\sqrt{(\pi)})$ to ${\rm sp}(x)$.
    \item If $\fX$ is {\bf admissible} (i.e. flat over $K^\circ$), then ${\rm sp}$ is surjective \cite[Proposition II 3.1.5]{FujiwaraKato}. 
    \item If $\fX$ is moreover $\eta$-normal, then every generic point $x\in \fX$ has a unique preimage $x'\in X$ under ${\rm sp}$. Moreover, we have
    \[
        \cO_{\fX, x} = \cO^+_{X, x'} = k(x')^+,
    \]
    and hence $k(x) = k(x')^\succ$. In particular, $\cO_{\fX,x}$ is a rank one valuation ring. These assertions follow by combining \cite[Proposition~2.2]{BhattHansen} with \cite[Lemma~A.12 and Proposition~A.15]{AchingerLaraYoucis2022}. 
\end{enumerate}
\end{rmk}

\subsubsection*{Raynaud's equivalence}

The rigid generic fibre functor sends the class $W$ of {\bf admissible formal blow-ups} (completed blowups in open coherent ideals) to isomorphisms, and induces by \cite[Theorems~II~A.5.1,~A.5.2]{FujiwaraKato} an equivalence
\[ 
    \cat{FSch}^{\rm ft}_{K^\circ}[W^{-1}] \isomto \cat{Rig}_K^{\rm qcqs}
\]
between formal schemes of finite type over $K^\circ$ with admissible formal blow-ups formally inverted and quasi-compact and quasi-separated rigid spaces over $K$. Both categories are equivalent to $\cat{FSch}^{\rm adm}_{K^\circ}[W^{-1}]$ where $\cat{FSch}^{\rm adm}_{K^\circ}$ is the category of admissible formal schemes over $K^\circ$. 

\subsubsection*{Formal models}

Let $X$ be a rigid space over $K$. A {\bf formal model} (resp.\ {\bf admissible formal model}) of $X$ is a formal scheme $\fX$ locally of finite type over $K^\circ$ (resp.\ admissible) together with an identification $\fX_{\rm rig}\simeq X$. We denote by $M(X)$ (resp.\ $M_{\rm adm}(X)$) the category of formal models (resp.\ admissible formal models) of $X$. Then $M(X)$ is a cofiltering category and $M_{\rm adm}(X)$ is (equivalent to) a cofiltering poset. The specialisation maps ${\rm sp}_\fX\colon X\to \fX$ for $\fX\in M(X)$ commute with the maps in $M(X)$ and induce an isomorphism of locally ringed spaces
\[
    (X, \cO_X^+)\isomto \varprojlim_{\fX\in M(X)} \, \fX. 
\]
In particular, for every $x\in X$ we have $\cO^+_{X,x} = \varinjlim_{\fX\in M(X)} \cO_{\fX, {\rm sp}_{\fX}(x)}$.
In fact, the admissible formal models $M_{\rm adm}(X)$ are cofinal in~$M(X)$ so it suffices to take the direct limit over~$M_{\rm adm}(X)$.
Moreover, if we start with an admissible formal model~$\cX$, we obtain a cofinal class of formal models of~$X$ by taking all admissible formal blow-ups of~$\cX$, see \cite[Chapter~II,\S~3.2 and A.5]{FujiwaraKato}.

\subsection{Vertical compactifications of adic spaces}
\label{ss:univ-comp}

In \cite[\S5.1]{HuberBook} Huber constructs the universal compactification for a given morphism $f\colon Y \to X$ of analytic adic spaces subject to the following hypotheses.
According to his general assumptions, the spaces~$X$ and~$Y$ have to be locally noetherian and, moreover, the morphism~$f$ is required to be locally of ${}^+$weakly finite type\footnote{Locally of the form $\Spa(B,B^+)\to\Spa(A,A^+)$ where $A\to B$ is topologically of finite type and where $B^+$ is the smallest integral subring containing a finitely generated $A^+$-algebra, see \cite[Definition 1.2.1]{HuberBook}.} , separated, and taut\footnote{A locally spectral topological space is {\bf taut} if it is quasi-separated and every quasi-compact open has quasi-compact closure. A spectral map $f\colon Y\to X$ is {\bf taut} if the preimage of every taut open of $X$ is taut. Examples include qcqs and partially proper morphisms of adic spaces. See \cite[\S 5.1]{HuberBook}.}.
The situation is quite different from the scheme setting, where we normally do not have an initial object among all compactifications.
What makes analytic adic spaces special in this respect is the fact that all specialisations are vertical.
The universal compactification is constructed by adding all possible vertical specialisations; there is no choice involved.

If we want to drop the hypothesis that~$X$ and~$Y$ are analytic, we have to account for horizontal specialisations entering the picture.
As a consequence, we can no longer expect the existence of a universal compactification.
However, we can construct the \emph{universal vertical compactification} as in \cite{Solodov-vertical-compactification} that can be thought of as adding only \emph{vertical} specialisations.
In this way, we obtain a \emph{vertically partially proper} morphism, which is a morphism that satisfies a valuative criterion for properness only for vertical specialisations.
Let us start by explaining this notion.

Recall from \cite[Definitions~8.26 and~8.54]{Wedhorn-adicspaces} that a Huber ring~$A$ is {\bf stably sheafy} if every Huber pair $(B,B^+)$ such that $B$ is topologically of finite type over~$A$ is sheafy.
An adic space is {\bf stable} if it is locally of the form $\Spa(A,A^+)$ for a stably sheafy Huber ring~$A$.

\begin{defi}
    Let~$S$ be a stable adic space and $f \colon X \to S$ a morphism locally of weakly finite type. 
    \begin{enumerate}[(i)]
        \item The morphism~$f$ is {\bf vertically specialising} at a point $x \in X$ if every vertical specialisation $f(x) \rightsquigarrow s'$ in~$S$ can be lifted to a specialisation $x \rightsquigarrow x'$ in~$X$.
        If~$f$ is vertically specialising at every point, we call it {\bf vertically specialising}.
        \item We call~$f$ {\bf universally vertically specialising} if every base change of~$f$ along an adic  morphism from a stable adic space is vertically specialising.
    \end{enumerate}
\end{defi}

In the above definition the lifted specialisation $x \rightsquigarrow x'$ is not required to be vertical.
However, if a lift exists, we can always modify it to a \emph{vertical} specialisation by \cite[Proposition~3.4]{Solodov-vertical-compactification}.
We further remark that in the definition of universally vertically specialising morphisms we test after base change via an \emph{adic} morphism because then we can guarantee that the respective fibre product exists.

\begin{lem}[Valuative criterion for being universally vertically specialising] \label{valuative-criterion-separatednes}
    Let~$S$ be a stable adic space and $f \colon  X \to S$ a morphism locally of weakly finite type.
    Then~$f$ is universally vertically specialising if and only if for every solid arrow diagram 
    \begin{equation} \label{eqn:valuative-criterion-test-square} 
        \begin{tikzcd}
            \Spa(k, k_2^+) \ar[r] \ar[d] & X \ar[d] \\
            \Spa(k, k_1^+) \ar[r] \ar[ur,dashed] & S,
        \end{tikzcd}
    \end{equation}
    where $(k,k_1^+)$ and~$(k,k_2^+)$ are affinoid fields with $k_1^+ \subseteq k_2^+$, there is a diagonal dashed arrow making the diagram commutative.
\end{lem}

\begin{proof}
    Analogous to \cite[Proposition~1.3.8]{HuberBook}, see also \cite[Remark~3.5]{Solodov-vertical-compactification}. 
\end{proof}

For simplicity we define vertical separatedness directly in terms of a valuative criterion.

\begin{defi}
    Let~$S$ be a stable adic space and $f \colon  X \to S$ a morphism that is locally of weakly finite type.
    Then~$f$ is {\bf vertically separated} if it is quasi-separated and for every solid arrow square (\ref{eqn:valuative-criterion-test-square}) there is at most one dashed arrow making the diagram commutative.
\end{defi}

As is to be expected, separatedness implies vertical separatedness (see the argument in \cite[Proposition~1.3.8]{HuberBook}) and a universally closed morphism is universally specialising and thus universally vertically specialising.

\begin{defi}
    Let~$S$ be a stable adic space.
    A morphism $f \colon  X \to S$ is {\bf vertically partially proper} if it is locally of ${}^+$weakly finite type, vertically separated and universally vertically specialising.

    It is {\bf vertically proper} if it is vertically partially proper and quasi-compact.
\end{defi}

Let~$S$ be a stable adic space.
It follows from \cref{valuative-criterion-separatednes} and the definition of vertical separatedness that a quasi-separated morphism $f \colon X \to S$ that is locally of ${}^+$weakly finite type is vertically partially proper if and only if for every solid arrow square (\ref{eqn:valuative-criterion-test-square}) there is a unique dashed arrow.

\begin{defi}
    Let~$S$ be a stable adic space and $f \colon  X \to S$ a morphism locally of ${}^+$weakly finite type.
    A {\bf vertical compactification} of~$f$ is a commutative triangle
    \[
     \begin{tikzcd}
         X  \ar[rr,open,"j"] \ar[dr,"f"']   &       & \ov{X}    \ar[dl,"\bar{f}"]  \\
                                            & S
     \end{tikzcd}
    \]
    of adic spaces in which $j$ is a locally closed embedding and~$\bar{f}$ is vertically partially proper.
    We call~$\overline{X}$ the {\bf universal vertical compactification} if, in addition, it is initial among all morphisms over~$S$ to adic spaces that are vertically partially proper over~$S$.
\end{defi}

\begin{rmk}
    In \cite{Solodov-vertical-compactification}, Solodov requires vertically partially proper morphisms to be separated (and not just vertically separated).
    Accordingly, at first sight, his universal vertical compactifications satisfy a slightly weaker universal property that only takes test objects $Y \to S$ that are vertically partially proper \emph{and} separated.
    However, the universal vertical compactification he constructs also satisfies the universal property without the separatedness assumption.
    Below we will retrace his arguments and explain why this is the case.
\end{rmk}

Generalising \cite[Lemma~5.1.7]{HuberBook}, Solodov gives sufficient conditions for a morphism $\ov X \to S$ to be the universal compactification of $X \to S$.
Here, we give a slightly stronger version because in our terminology vertical partial properness does not include separatedness (only vertical separatedness).

\begin{prop}[\cite{Solodov-vertical-compactification}, Lemma~3.19~iv)] \label{prop:check-univprop-comp}
    We consider a commutative triangle of adic spaces
    \[
        \begin{tikzcd}
            X \ar[r,"j",open] \ar[d,"f"]    & \ov X \ar[dl,"\bar{f}"]   \\
            S
        \end{tikzcd}
    \]
    with the following properties:
    \begin{itemize}
        \item $\bar{f}$ is vertically partially proper
        \item $j$ is a quasi-compact open immersion,
        \item $X$ is closed under generalised horizontal specialisations in~$\ov X$,
        \item every point of~$\ov X$ is a vertical specialisation of a point of~$X$,
        \item $\cO_{\ov X} \to j_*\cO_X$ is an isomorphism of sheaves of topological rings.
    \end{itemize}
    Then the above commutative triangle exhibits~$\ov X$ as the universal vertical compactification of~$X$ over~$S$ and this assertion even holds after base change to some open subspace of~$S$.
\end{prop}

\begin{proof}
    Solodov's proof in \cite[Lemma~3.19~iv)]{Solodov-vertical-compactification} basically just combines items~i) and~iii) of his Lemma~3.19, which only require vertical separatedness.
\end{proof}

In \cref{sec:root-triples} we will identify the universal vertical compactification~$\overline{X}$ of $f \colon X \to S$ (if it exists, and under mild supplementary geometric assumptions) with an explicitly defined space $\Spa(X,S)$. 
But before having established that $\Spa(X,S)$ is indeed the universal vertical compactification of~$f$, we will exclusively use~$\overline{X}$ to denote the universal vertical compactification.

In order to ensure the existence of the universal vertical compactification of $f \colon X \to S$, some mild additional assumptions about~$S$ and~$X$ are necessary.
The condition is called \emph{(weak) square completeness}, a concept introduced in \cite[Definition~3.13]{Huebner2025:ProperBaseChange} and further differentiated in \cite[Definition~3.20]{Solodov-vertical-compactification}.
\begin{defi}
    An adic space~$X$ is called {\bf weakly square complete} if for three given points $x$, $x_h$, and $x_v$ in~$X$ such that~$x_h$ is a horizontal and~$x_v$ a vertical specialisation of~$x$, there is a point $x' \in X$ that is a horizontal specialisation of~$x_v$ and a vertical specialisation of~$x_h$.
     \[
      \begin{tikzcd}
       x_v	\ar[r,rightsquigarrow]							& x'			\\
       x	\ar[r,rightsquigarrow]	\ar[u,rightsquigarrow]	& x_h	\ar[u,rightsquigarrow]
      \end{tikzcd}
     \]
     If the point~$x'$ is unique for any choice of~$x$, $x_h$, and~$x_v$, we call~$X$ {\bf square complete}.
\end{defi}

Examples of square complete adic spaces include affinoid spaces, analytic spaces, formal schemes, and spaces of the form $\Spa(X,S)$ for a morphism of schemes $X \to S$. For an example of a non \emph{weakly} square-complete discretely ringed adic space see \cite[Example~3.12]{Huebner2025:ProperBaseChange}. 
We are now ready to state the existence theorem of the universal vertical compactification.
Its construction has been worked out in \cite{Solodov-vertical-compactification}.

\begin{thm}[{\cite[Theorem~3.29]{Solodov-vertical-compactification}}] \label{thm:Solodov}
    Let~$S$ be a square complete and stable adic space and $f\colon  X \to S$ a separated and taut morphism that is locally of ${}^+$weakly finite type and such that~$X$ is weakly square complete.
    Then~$f$ admits a universal vertical compactification.
\end{thm}

The author, Ronald Solodov, closely follows Huber's line of argumentation for the construction of the universal compactification of an analytic adic space.
Note that even in the analytic setting, Solodov's result is more general than Huber's as it does not make any noetherianness assumptions.

\smallskip

In the remaining part of this section we will make the construction of the universal compactification more explicit.
Let us start with the case of a morphism of affinoid adic spaces.

\begin{prop}[{\cite[Theorem~3.14]{Solodov-vertical-compactification}}] \label{prop:affinoid-compactification}
    Let $f\colon X\to S$ be a morphism of ${}^+$-weakly finite type between affinoid adic spaces $X = \Spa(A,A^+)$ and $S = \Spa(R,R^+)$ where $S$ is stable.
    Then the universal vertical compactification of~$f$ is given by
    \[
     \bar{f} \colon \overline{X} = \Spa(A,R^+) \longrightarrow \Spa(R,R^+) = S.
    \]   
\end{prop}

In the general case, Solodov (following Huber) first covers~$S$ by affinoid open subspaces~$S_i$ and shows that universal vertical compactifications over~$S_i$ can be glued together to a universal vertical compactification over~$S$.
This uses the last part of \cref{prop:check-univprop-comp}.

Consequently, it suffices to construct~$\ov X$ in case~$S$ is affinoid.
Now Solodov covers~$X$ with affinoid subspaces~$X_i$ and aims at gluing the respective universal compactifications~$\overline{X}_i$.
Here, a major technical difficulty arises because we cannot expect the spaces~$\overline{X}_i$ to be open in the universal compactification~$\overline{X}$ of~$X$.
Huber overcomes these difficulties by finding appropriate open subspaces~$U_i$ of~$\overline{X}_i$ that are expected to be open in~$\overline{X}$ but are large enough to cover~$\overline{X}$.
The universal compactification~$\overline{X}$ is then constructed by gluing the subspaces~$U_i$.
Solodov generalises this approach for his construction of the universal vertical compactification.
As a result, we are guaranteed the existence of the universal vertical compactification.
In addition, we extract from its proof the existence of the open subsets~$U_i$ mentioned above:

\begin{lem} \label{lem:interiors-cover}
    In the setting of \cref{thm:Solodov}, assume that~$S$ is affinoid. 
    \begin{enumerate}[(1)]
        \item \label{lemitem:interior-open}
        Let $U \subseteq X$ be an open affinoid subspace and consider the injective map $\ov{U} \to \ov{X}$ of universal vertical compactifications. Let $\ov{U}^\circ$ be the subset of $\ov{U}$ consisting of all points $y \in \ov{U}$ such that for any vertical generalisation~$x$ of~$y$ in $\ov{X}$ we either have $x \in U$ or $x \notin X$. 
        Then $\ov{U}^\circ$ is an open subset of $\ov{X}$ such that $\ov{U}^\circ \cap X = U$.

        \item \label{lemitem:interior-open-covering}
        Let $X = \bigcup_i U_i$ be an open covering by affinoid subspaces. Then the subsets $\ov{U_i}^\circ$ as constructed in \labelcref{lemitem:interior-open} associated to $U_i$ form an open covering of $\ov{X}$.
    \end{enumerate}
\end{lem}

\begin{proof}
    \labelcref{lemitem:interior-open} In \cite{Solodov-vertical-compactification}, Solodov considers the universal vertical compactification~$\ov U$ (there denoted by~$U_c$) of an affinoid open $U \subseteq X$.
    He constructs a certain subspace $U_d \subseteq \ov U$ containing~$U$ that turns out to be open.
    In addition, an inclusion of open affinoids $U \subseteq V \subseteq X$ induces an open immersion $U_d \to V_d$, see \cite[Lemma~3.27]{Solodov-vertical-compactification}.
    Looking at the definition of~$U_d$ we identify it with $\ov{U}^\circ$ as the subset of~$\ov U$. The identification $\ov{U}^\circ \cap X = U$ follows from the definition of~$\ov{U}^\circ$.

    \labelcref{lemitem:interior-open-covering}
    In the end~$\ov X$ is obtained by gluing the spaces~$\ov{U_i}^\circ = U_{i,d}$ (in Solodov's notation) constructed from the members~$U_i$ of the open covering.
    Then the spaces~$\ov{U_i}^\circ$ form an open covering of~$\ov X$ by construction.   
\end{proof}

In the end, Solodov checks that his space~$\ov X$ satisfies the universal property using \cite[Lemma~3.19~iv)]{Solodov-vertical-compactification}.
With the updated version \cref{prop:check-univprop-comp} of this criterion we ensure that we can use any vertically partially proper morphism as test object (it does not need to be separated).

\subsection{Root triples} 
\label{sec:root-triples}

Let us sketch an alternative description of the universal vertical compactification that allows for an explicit definition of the points of $\overline{X}$.
In what follows, we will write down a set $\Spa(X,S)$ and endow it with a topology and a structure sheaf.
This construction works without any hypotheses on~$X$ or~$S$.
Our final goal will be to identify this space $\Spa(X,S)$ under the assumptions of \cref{thm:Solodov} with the universal vertical compactification~$\overline{X}$, in this way also showing that it carries the structure of an adic space. 

\begin{defi}
    Let $f\colon X\to S$ be a morphism of adic spaces. A \textbf{point of $X$ with centre on $S$} is a commutative square of the form
    \begin{equation} \label{eqn:pp-test-square} 
        \begin{tikzcd}
            \Spa(k, k^\circ) \ar[r,"u"] \ar[d] & X \ar[d] \\
            \Spa(k, k^+) \ar[r,"v",swap] & S
        \end{tikzcd}
    \end{equation}
    where $(k, k^+)$ is an affinoid field with ring of powerbounded elements $k^\circ$ (so $(k, k^\circ)$ is an affinoid field as well), and where the left vertical map is induced by the inclusion $k^+\subseteq k^\circ$ and~$u$ is adic. We denote such a datum by $(k^+, u, v)$ for simplicity. 

    We deem two points $(k_i^+, u_i, v_i)$, for $i=1,2$, of $X$ with centre on $S$ \textbf{equivalent} if there exists a third square with $(k^+, u, v)$ and commutative diagrams
    \[ 
        \begin{tikzcd}
            \Spa(k_i, k^\circ_i) \ar[r] \ar[d] \ar[rr,"u_i",bend left=15] & \Spa(k, k^\circ) \ar[r,"u",swap] \ar[d] & X \ar[d]  \\
            \Spa(k_i, k^+_i) \ar[r] \ar[rr,"v_i",swap,bend right=15] & \Spa(k, k^+)\ar[r,"v"] & S
        \end{tikzcd}
    \]
    for $i=1,2$ such that $(k_i,k_i^+)/(k,k^+)$ are extensions of valuation rings, i.e.\ $k^+=k_i^+\cap k$.
\end{defi}

By \cite[Lemma~4.3]{Solodov-vertical-compactification} this defines an equivalence relation.
In every equivalence class we can identify a distinguished representative given by a so called root triple.
Before we can define root triples, we need to recall centres of valuations.
For a pair $(x, V)$ consisting of a point $x\in X$ and a valuation subring $V\subseteq k(x)^+$, by a {\bf centre of $(x, V)$ on $S$} we shall mean a vertical specialisation $s\in S$ of $f(x)$ with the property that $k(s)^+ = V\cap k(s)$.
The centres of $(x,V)$ on~$S$ are in one to one correspondence with the morphisms $\Spa(k(x),V) \to S$ making the diagram
\[
 \begin{tikzcd}
     \Spa(k(x),k(x)^+)  \ar[r]  \ar[d]  & X \ar[d]  \\
     \Spa(k(x),V)       \ar[r,dotted]          & S
 \end{tikzcd}
\]
commutative.

\begin{defi} \label{defi:root-triple}
    A {\bf root triple} of $X/S$ is a triple $x = (x^\circ,k(x)^+,x^\succ)$ for a point $x^\circ \in X$ without vertical generalisations and a valuation subring 
    \[
        k(x)^+ \subseteq k(x^\circ)^+ = k(x^\circ)^\circ
    \]
    with centre $x^\succ \in S$. For a root triple~$x$ as above we use the notation $k(x)$ to denote $k(x^\circ)$.
    Moreover, we sometimes write $x^\succ$ also for the induced map $\Spa(k(x),k(x)^+) \to S$.
\end{defi}

A root triple $x = (x^\circ,k(x)^+,x^\succ)$ naturally defines a point of~$X$ with centre on~$S$
    \[  
        \begin{tikzcd}
            \Spa(k(x), k(x)^\circ) \ar[r] \ar[d] & X \ar[d] \\
            \Spa(k(x), k(x)^+) \ar[r,"x^\succ"] & S
        \end{tikzcd}
    \]
and indeed every equivalence class of points of~$X$ with centre on~$S$ has a unique representative given by a root triple (see \cite[Lemma~4.2]{Solodov-vertical-compactification}).
In particular, this ensures that the equivalence classes form a set.

\begin{defi}
    We denote the set of equivalence classes of points of~$X$ with centre on~$S$ by $\Spa(X,S)$.
\end{defi}

In what follows, we shall often implicitly identify a point $x \in \Spa(X,S)$, which is a priori an equivalence class of squares, with the corresponding root triple and we write $x = (x^\circ,k(x)^+,x^\succ)$.

\begin{ex} \label{ex:comp-affinoid}
    For the identity morphism $X \to X$ we obtain a natural identification of sets
    \[
     X \overset{\sim}{\longrightarrow} \Spa(X,X), \qquad x \mapsto (x^\circ,k(x)^+,x).
    \]
    The inverse sends a root triple $(x^\circ,k(x)^+,x^\succ)$ to~$x^\succ$.
\end{ex}

The formation of $\Spa(X,S)$ enjoys the following functoriality.
If
    \[ 
        \begin{tikzcd}
            X'\ar[d] \ar[r,"f_X"] & X\ar[d] \\
            S'\ar[r,"f_S"] & S
        \end{tikzcd}
    \]
is a commutative square of adic spaces such that~$f_X$ is adic, we obtain a natural morphism
\[
 \Spa(f_X,f_S) \colon \Spa(X',S') \longrightarrow \Spa(X,S)
\]
by sending a point $(k^+,u,v)$ of~$X$ with centre on~$S$ to $(k^+,f_X \circ u,f_S \circ v)$.
\begin{equation} \label{diagam:functoriality-comp}
 \begin{tikzcd}
     \Spa(k,k^\circ)    \ar[r,"u"]  \ar[d]  & X'    \ar[r,"f_X"]    \ar[d]  & X \ar[d]  \\
     \Spa(k,k^+)        \ar[r,"v"]          & S'    \ar[r,"f_S"]            & S.
 \end{tikzcd}
\end{equation}
This construction is compatible with the equivalence relation, and thus indeed defines a map $\Spa(X',S') \to \Spa(X,S)$.
In terms of root triples it sends $x' = (x'^\circ,k(x')^+,x'^\succ)$ to
\[
    f(x') = (f_X(x'^\circ),k(x')^+ \cap k(f(x'^\circ)),f_S(x'^\succ)).
\]
Note that this implicitly uses \cref{lem:adic-preserves-vertical-maximal} to ensure that $f(x'^\circ)$ is vertically maximal.

If~$f_X$ is not adic, we still get a map $\Spa(f_X,f_S)$ as above but its description is a bit more complicated.
Starting with a point $(k^+,u,v)$ we can consider the diagram~(\ref{diagam:functoriality-comp}) but $f_X \circ u$ might not be adic, in which case it does not directly define a point of~$X$ with centre on~$S$.
The morphism $f_X \circ u$ can only possibly not be adic when $(k,k^\circ)$ is analytic but the image point~$x$ of $f_X \circ u$ is not analytic.
In this case we replace the diagram with
\[
    \begin{tikzcd}
        \Spa(k,k)   \ar[r,"u"]  \ar[d]  & X \ar[d]  \\
        \Spa(k,k^+) \ar[r,"v"]          & S,
    \end{tikzcd}
\]
where the affinoid fields in the left column are now endowed with the discrete topology.
Since~$x$ is nonanalytic, the morphisms are still well defined and the above diagram defines a point of~$X$ with centre on~$S$.
We declare this point to be the image of $(k^+,u,v)$ in $\Spa(X,S)$.

The root triple description is obtained as follows:
We first note that we always have $k(f(x'^\circ)^\circ) = k(f(x'^\circ))$ because $f(x'^\circ)^\circ \ne f(x'^\circ)$ can only happen if these points are nonanalytic, in which case the residue fields are discrete.
We send $x' = (x'^\circ,k(x')^+,x'^\succ)$ to the root triple $x$ where $x^\circ = f(x'^\circ)^\circ$, $k(x)^+ = k(x')^+ \cap k(f(x'^\circ))$, and $x^\succ = f_S(x'^\succ)$.

\begin{lem} \label{lem:Spa-injective}
    If $f_X \colon X' \to X$ is a locally closed embedding and $f_S \colon S' \to S$ is separated, the map $\Spa(f_X,f_S)$ is injective.
\end{lem}

\begin{proof}
    Suppose that $x_1$ and $x_2$ are two points of $\Spa(X',S')$ that map to the same point in $\Spa(X,S)$.
   We identify~$x_i$ with the corresponding root triple $x_i = (x_i^\circ,k(x_i)^+,x_i^\succ)$ and consider the diagrams
    \[
 \begin{tikzcd}
     \Spa(k(x_i),k(x_i)^\circ)    \ar[r]        \ar[d]  & X'    \ar[r,"f_X"]    \ar[d]  & X \ar[d]  \\
     \Spa(k(x_i),k(x_i)^+)        \ar[r,"x_i^\succ"]          & S'    \ar[r,"f_S"]            & S.
 \end{tikzcd}
\]
Since~$f_X$ is injective, we obtain $x_1^\circ = x_2^\circ$.
In particular, the upper line is the same in both diagrams.
We denote by~$x^\circ$ the image in~$X$ of $x_1^\circ = x_2^\circ \in X'$.
The morphism~$f_X$ being a locally closed embedding, the residue field extension $k(x_i)/k(x)$ induces an isomorphism on completions.
In particular, $k(x_i)^+$ is uniquely determined by $k(x_i)^+ \cap k(x)$, which is the same for $i=1,2$.
Finally, the morphisms~$x_1^\succ$ and~$x_2^\succ$ coincide by the valuative criterion for separatedness applied to the square
\[
    \begin{tikzcd}
        \Spa(k(x_1),k(x_1)^\circ) = \Spa(k(x_2),k(x_2)^\circ)   \ar[r]  \ar[d]  & S' \ar[d,"f_S"]    \\
        \Spa(k(x_1),k(x_1)^+) = \Spa(k(x_2),k(x_2)^+)           \ar[r]  & S,
    \end{tikzcd}
\]
where the lower map is $f_S \circ x_1^\succ = f_S \circ x_2^\succ$.
\end{proof}

As a next step we want to equip the set $\Spa(X,S)$ of equivalence classes of points of~$X$ with centre on~$S$ with a topology.
In the affinoid case this is straightforward.

\begin{lem} \label{lem:Spa-barX-affinoid}
    If $X = \Spa(A,A^+)$ and $S = \Spa(R,R^+)$ are affinoid, there is a natural identification (as sets)
    \[
    \Spa(X,S) = \Spa(A,R^+)
    \]
    and in case $S$ is stable and $X \to S$ is ${}^+$weakly of finite type, $\Spa(A,R^+)$ is the universal vertical compactification of~$X$ over~$S$.
\end{lem}

\begin{proof}
    An element of $\Spa(X,S)$ is defined by a root triple $x = (x^\circ,k(x)^+,x^\succ)$.
    Giving the pair $(x^\circ,k(x)^+)$ is equivalent to giving a continuous valuation $v$ on~$A$.
    The centre~$x^\succ$ is unique in the affinoid case.
    Its existence is equivalent to the valuation~$v$ taking values less than or equal to~$1$ on~$R^+$.
    The last assertion of the lemma follows from  \cite[Theorem~3.14]{Solodov-vertical-compactification}.
\end{proof}

By \cref{lem:Spa-barX-affinoid} the set $\Spa(X,S)$ for affinoid spaces~$X$ and~$S$ naturally carries the structure of an adic space.
In particular, it is endowed with a topology.
It is clear from the construction that commutative squares
\[
 \begin{tikzcd}
     X' \ar[r]  \ar[d]  & X \ar[d]  \\
     S' \ar[r]          & S
 \end{tikzcd}
\]
of \emph{affinoid} adic spaces induce morphisms of adic spaces $\Spa(X',S') \to \Spa(X,S)$.

For the general case, we use the following construction for the definition of a topology.
For all possible commutative squares
\begin{equation} \label{test-square-comp-topology}
 \begin{tikzcd}
     U  \ar[r,open] \ar[d]  & X \ar[d]  \\
     T  \ar[r,open]         & S,
 \end{tikzcd}
\end{equation}
where~$U$ and~$T$ are affinoid and the horizontal maps are open immersions, we consider the associated morphisms $\Spa(U,T) \to \Spa(X,S)$. 
By \cref{lem:Spa-injective} this map is injective, and we can thus treat $\Spa(U,T)$ as a subset of $\Spa(X,S)$.
Now we endow $\Spa(X,S)$ with the final topology for these inclusions.
This means that a subset of $\Spa(X,S)$ is open if and only if its intersection with $\Spa(U,T)$ is open for every square as above.
The first sanity check is the following:

\begin{lem}
    If $X = \Spa(A,A^+)$ and $S = \Spa(R,R^+)$ are affinoid, the above defined topology on $\Spa(X,S)$ coincides with the topology of the adic space $\Spa(A,R^+)$ via the identification from \cref{lem:Spa-barX-affinoid}.
\end{lem}

\begin{proof}
    We consider a subset $\cU \subseteq \Spa(X,S)$.
    A square (\ref{test-square-comp-topology}) comes from a square
    \[
        \begin{tikzcd}
            (B,B^+)         & (A,A^+) \ar[l]    \\
            (D,D^+) \ar[u]  & (R,R^+) \ar[l] \ar[u]
        \end{tikzcd}
    \]
    of Huber pairs and thus induces a morphism of affinoid adic spaces $\Spa(B,D^+) \to \Spa(A,R^+)$ whose underlying map of sets identifies with $\Spa(U,T) \to \Spa(X,S)$.
    Consequently, if~$\cU$ is open in $\Spa(X,S)$ with the topology coming from the structure as an affinoid adic space, then $\cU \cap \Spa(U,T)$ is open.

    Conversely, suppose that $\cU \cap \Spa(U,T)$ is open for all squares~(\ref{test-square-comp-topology}).
    This applies in particular to $\Spa(U,T) = \Spa(X,S)$.
    Hence~$\cU$ is open in $\Spa(X,S)$ with its adic space topology.
\end{proof}

Suppose that the hypotheses of \cref{thm:Solodov} are satisfied for $f \colon X \to S$.
We define a map
\[
    \phi \colon \Spa(X,S) \la \ov X
\]
by sending a root triple $x = (x^\circ,k(x)^+,x^\succ)$ to the unique centre of~$k(x)^+$ in~$\ov X$ lying over~$x^\succ$ (which exists and is unique by the valuative criterion for vertical partial properness \cite{Bauer:valuative-criteria}).
We obtain a commutative diagram
\[
 \begin{tikzcd}
     X  \ar[rr] \ar[dr,"\iota"]  \ar[ddr,"f"']   &           & \ov X \ar[ddl]    \\
     & \Spa(X,S) \ar[d]  \ar[ur,"\phi"] \\
     & S.
 \end{tikzcd}
\]
Here, the map $\iota \colon X \to \Spa(X,S)$ sends~$x$ to $(x^\circ,k(x)^+,f(x))$ (alternatively, identify~$X$ with $\Spa(X,X)$ as in \cref{ex:comp-affinoid} and use the natural map $\Spa(X,X) \to \Spa(X,S)$).
The commutativity of the upper triangle is ensured by the vertical separatedness of $\ov X \to S$.

\begin{prop}[\cite{Solodov-vertical-compactification}, Theorem~4.8] \label{prop:comp-homeo}
    Under the hypotheses of \cref{thm:Solodov}, the above defined map~$\phi$ is a homeomorphism.
\end{prop}

By the proposition the map
$\iota \colon X = \Spa(X,X) \longrightarrow \Spa(X,S)$ is a topological open embedding.
We endow $\Spa(X,S)$ with a structure sheaf by setting 
\[
\cO_{\Spa(X,S)} = \iota_* \cO_X.
\]
The computation of the stalk of $\iota_*\cO_X$ at a point $x \in \Spa(X,S)$ relies on the following lemma.

\begin{lem} \label{lem:stalks-comp}
    Suppose $f \colon X \to S$ satisfies the hypotheses of \cref{thm:Solodov}.
    Let $x \in \Spa(X,S)$ be a point and denote by~$x'$ the minimal vertical generalisation of~$x$ lying in~$X$ (so if $x' \leadsto y \leadsto x$ are vertical specialisations and $x' \ne y$, then~$y$ is not contained in~$X$).
    Then the natural homomorphism $(\iota_*\cO_X)_x \to (\iota_*\cO_X)_{x'} = \cO_{X,x'}$ is an isomorphism.
\end{lem}

\begin{proof}
    We need to show the following:
    For every open neighbourhood $U' \subseteq X$ of~$x'$ there is an open neighbourhood $U \subseteq \Spa(X,S)$ of~$x$ such that $U \cap X \subseteq U'$.

    The subset $\Spa(U',S) \subseteq \Spa(X,S)$ contains~$x$.
    Moreover, by the definition of~$x'$, $x$ is contained in the subset $U \subseteq \Spa(U',S)$ consisting of all points $\bar{z}$ such that for any vertical generalisation $z$ of~$\bar{z}$ we either have $z \in U'$ or $z \notin X$.
    In particular, $U \cap X \subseteq U'$.
    Via the homeomorphism $\Spa(X,S) \cong \ov X$ from \cref{prop:comp-homeo} we can apply \cref{lem:interiors-cover} to $\Spa(X,S)$ telling us that~$U$ is an open neighbourhood of~$x$ in $\Spa(X,S)$.
\end{proof}

\cref{lem:stalks-comp} tells us that $(\Spa(X,S),\iota_*\cO_X)$ is a locally topologically ringed space.
We turn it into a locally $v$-ringed space by defining valuations on the residue fields of the stalks of $\iota_*\cO_X$.
For a point $x = (x^\circ,k(x)^+,x^\succ)$ in $\Spa(X,S)$ the point~$x'$ from \cref{lem:stalks-comp} is of the form $x' = (x^\circ,k(x')^+,x'^\succ)$ where~$k(x')^+$ is the smallest valuation ring satisfying $k(x)^+ \subseteq k(x')^+ \subseteq k(x^\circ)$ and having a centre on~$X$.
By \cref{lem:stalks-comp} we have an equality of residue fields $k(x) = k(x')$ and these are dense in $k(x^\circ)$.
Now we define the valuation ring at~$x$ to be $k(x)^+ \cap k(x)$.

We would like to check that the locally $v$-ringed space $\Spa(X,S)$ is indeed an adic space, i.e.~locally isomorphic to the adic spectrum of a Huber pair.

\begin{thm} \label{thm:Solodov-comp}
    Suppose~$X$ is weakly square complete,~$S$ is square complete and stable, and that the morphism $f \colon X \to S$ is locally of ${}^+$weakly finite type, separated, and taut.
    Then $\Spa(X,S)$ is an adic space isomorphic to the universal vertical compactification~$\overline{X}$ of~$X$ over~$S$.
\end{thm}

\begin{proof}
    We constructed a natural homeomorphism $\phi \colon \Spa(X,S) \to \ov X$ by sending $(x^\circ,k(x)^+,x^\succ)$ to the unique centre of $k(x)^+$ on~$\ov X$ lying over~$x^\succ$ (see \cref{prop:comp-homeo}).
    On the one hand, the universal vertical compactification comes with an open immersion $i \colon X \to \overline{X}$ such that $i_*\cO_X = \cO_{\overline{X}}$.
    On the other hand we defined the structure sheaf on $\Spa(X,S)$ by $\iota_*\cO_X$.
    Via the homeomorphism $\phi$ the two sheaves $i_*\cO_X$ and $\iota_*\cO_X$ are identified.
    It remains to check that the valuations on the residue fields agree.
    For a point $x = (x^\circ,k(x)^+,x^\succ)$ in $\Spa(X,S)$ the valuation ring on $k(x)$ is $k(x)^+ \cap k(x)$ and the same is true for $\phi(x)$.
    This gives both assertions of the theorem.
\end{proof}

\begin{rmk} \label{rmk:hypotheses-compactification}
    The hypotheses of \cref{thm:Solodov-comp} are quite mild.
    Let us discuss the major settings of interest to this paper when they are satisfied.
    First of all, the assumption that~$S$ be sheafy just tells us that the structure presheaf is a sheaf and this continues to hold after base change.
    Tautness is also a minor restriction.
    \begin{enumerate}[(i)]
        \item Suppose that~$S$ and hence also~$X$ are analytic.
              Then~$S$ and~$X$ are automatically square complete.
              Our main application here is to rigid spaces, where $S = \Spa(K)$ is the spectrum of a non-archimedean field and $X \to S$ is locally of finite type.
              These spaces are always stable.
        \item Let us now look at the case of formal schemes, i.e. where $X \to S$ comes from a morphism of formal schemes.
              Also in this situation square completeness is automatic once we require~$X$ to be separated:
              If $x \leadsto v$ is a vertical specialisation and $x \leadsto h$ a horizontal specialisation, we can find an affinoid neighbourhood of~$v$ of the form $\Spa(A)$.
              Then also $x \in \Spa(A)$ and for every intermediate valuation ring $k(x)^+ \subseteq V \subseteq k(x)^\circ$ there is a corresponding horizontal specialisation in $\Spa(A)$.
              By the separatedness assumption, the horizontal specialisations of $x \in X$ are uniquely determined by a valuation ring~$V$ as above, so in particular $h \in \Spa(A)$.
              Since affinoids are square complete, we can find $z \in \Spa(A)$ such that $v \leadsto z$ horizontally and $h \leadsto z$ vertically.

              In \cite{ZavyalovSheafy} Zavyalov shows that strongly rigid-noetherian Huber pairs are sheafy.
              In particular, formal schemes that are locally the formal spectrum of such a ring give rise to stable adic spaces (sending $\Spf(A)$ to $\Spa(A)$).
              This applies to formal schemes that are locally of topologically finite type over a complete microbial valuation ring (e.g. over~$K^\circ$ for a non-archimedean field~$K$).
              So formal models of rigid spaces over a non-archimedean field satisfy the assumptions of \cref{thm:Solodov-comp}.
        \item For discretely ringed adic spaces stability is not an issue.
              Square completeness is satisfied for adic spaces of the form $\Spa(Y,T)$ for a morphism of schemes $Y \to T$.
              Moreover, tautness is certainly satisfied for spaces coming from schemes of finite type over a noetherian base.
              Our main application is to adic spaces of the form $X = \Spa(Y,k)$, where~$k=K^\succ$ is the residue field of a non-archimedean field $(K,K^\circ)$ and $Y$ is a scheme of finite type over~$k$.
              The base~$S$ is usually $S = \Spa(k)$.
    \end{enumerate}
\end{rmk}

We have already noted that for a morphism of \emph{analytic} adic spaces the universal vertical compactification is actually the universal compactification and coincides with Huber's construction in the noetherian setting.
In the opposite corner in the category of adic spaces there are located the discretely ringed adic spaces, and indeed for this type of adic spaces the universal vertical compactification recovers a familiar construction from \cite[\S~3.1]{Temkin2011:RelativeRiemannZariskiSpacesa}:

\begin{ex} \label{example-Spa-schemes}
        Let $f_0\colon X_0\to S_0$ be a morphism of schemes and let 
    \[
        f\colon X=\Spa(X_0, X_0)\la \Spa(S_0, S_0)=S
    \]
    be the associated morphism of discretely ringed adic spaces. Then $\Spa(X, S)$ is isomorphic to the discretely ringed adic space $\Spa(X_0, S_0)$.
    Indeed, by \cref{defi:Spa for schemes}, points of $\Spa(X_0, S_0)$ are triples $(x_0, V, \varphi)$ where $x_0\in X_0$, $V$ is a valuation subring of $k(x_0)$, and $\varphi \colon \Spec(V)\to S_0$ is a map of schemes making the square
\[ 
    \begin{tikzcd}
        \Spec(k(x_0)) \ar[r] \ar[d] & X_0 \ar[d] \\
        \Spec(V)\ar[r,"\varphi"] & S_0
    \end{tikzcd}
\]
commute.
The point $x_0$ lifts to a unique point $x^\circ$ of $X=\Spa(X_0, X_0)$ without vertical generalisations by endowing $k(x_0)$ with the trivial valuation.
Then $k(x^\circ) = k(x^\circ)^+ = k(x_0)$ and we set $k(x)^+ := V$.
Then $\varphi$ induces a map $\Spa(k(x^\circ), V)\to S$.
We denote by $x^\succ \in S$ the image of the closed point.
This determines a root triple $x = (x^\circ,k(x)^+,x^\succ)$, i.e. a point of $\Spa(X,S)$.
The resulting map $\Spa(X_0,S_0) \to \Spa(X,S)$ is bijective and the following observation shows that it respects the adic space structure by reducing to the affine case:
For every test square
\[
    \begin{tikzcd}
        U_0 \ar[r]  \ar[d]  & X_0   \ar[d] \\
        T_0 \ar[r]          & S_0
    \end{tikzcd}
\]
for affine schemes~$U_0$ and~$T_0$ with open immersions $U_0 \to X_0$ and $T_0 \to S_0$, the induced morphism of adic spaces $\Spa(U_0,T_0) \to \Spa(X_0,S_0)$ is an open immersion.
\end{ex}

\subsection{Functoriality of the vertical compactification}

We end this section by exploring the functoriality of the universal vertical compactification.

\begin{prop} \label{prop:Spa-functoriality}
Suppose that we are given a commutative square
\[
 \begin{tikzcd}
     Y     \ar[r,"f_X"]   \ar[d]  & X   \ar[d]  \\
     T      \ar[r,"f_S"]          & S
 \end{tikzcd}
\]
of adic spaces such that $f_S$ is adic and the vertical morphisms satisfy the assumptions of \cref{thm:Solodov-comp}.
Then there is a unique morphism of adic spaces $f \colon \Spa(Y,T) \to \Spa(X,S)$ making, as the dashed arrow, the diagram below commute.
\[
 \begin{tikzcd}[column sep= 1cm]
     Y          \ar[r,"f_X"]   \ar[d]  & X   \ar[d]  \\
     \Spa(Y,T)  \ar[r,dashed,"f"] \ar[d] & \Spa(X,S) \ar[d] \\
     T          \ar[r,"f_S"]          & S
 \end{tikzcd}
\]
In terms of root triples,~$f$ maps a point $y = (y^\circ,k(y)^+,y^\succ)$ to $(f_X(y^\circ)^\circ,k(y)^+ \cap k(f_X(y^\circ)^\circ),f_S(y^\succ))$ (i.e. the underlying map of sets is the one described in \cref{sec:root-triples}).
\end{prop}

\begin{proof}
    We consider the solid arrow diagram
    \[
    \begin{tikzcd}
    & Y \ar[dl] \ar[d] \ar[r,"f_X"] & X \ar[d] \\
    \Spa(Y,T) \ar[r,dashed] \ar[dr] & \Spa(X,S) \times_S T \ar[r,"\pr"] \ar[d,"\rm vpp"] & \Spa(X,S) \ar[d,"\rm vpp"] \\
    & T \ar[r,"f_S"] & S.
    \end{tikzcd}
    \]
    The fibre product exists and $\Spa(X,S) \times_S T \to T$ is vertically partially proper by our assumption that $T \to S$ is adic.
    There is a unique dashed arrow by the universal property of the universal vertical compactification~$\Spa(Y,T)$.
    Composing it with the projection $\pr$,
    we obtain the desired morphism.
    
    By the valuative criterion for vertical partial properness \cite{Bauer:valuative-criteria}, the dotted map sends $y \in \Spa(Y,T)$ to the unique centre of $k(y)^+$ on $\Spa(X,S) \times_S T$ lying over~$T$.
    The projection to $\Spa(X,S)$ then maps it to the unique centre of $k(y)^+$ on $\Spa(X,S)$ lying over $f_S(y^\succ)$.
    This centre identifies with the root triple $(f_X(y^\circ)^\circ,k(y)^+ \cap k(f_X(y^\circ)^\circ),f_S(y^\succ))$.
\end{proof}

\begin{prop} \label{prop:Spa-iso-vpp}
    Suppose that the hypotheses of \cref{thm:Solodov} are satisfied for $X \to T$ and $T \to S$ is vertically partially proper.
    Then $\Spa(X,T) \to T \to S$ is the universal vertical compactification of~$X$ over~$S$.
\end{prop}

\begin{proof}
    We first note that $\Spa(X,T) \to S$ is vertically partially proper, so it is a vertical compactification of~$X$ over~$S$.
    In order to check universality, we consider a solid arrow commutative diagram
    \[
        \begin{tikzcd}
             X  \ar[drr,bend left]   \ar[dr,dashed] \ar[d]  \\
             \Spa(X,T) \ar[r,dashed] \ar[dr] & Y_T \ar[r] \ar[d,"\rm vpp"] \ar[r] & Y \ar[d,"\rm vpp"] \\
             & T \ar[r] & S,
        \end{tikzcd}
    \]
    where~$Y \to S$ is a vertically partially proper morphism and~$Y_T = Y \times_S T$.
    The universal property of the fibre product gives the diagonal dashed arrow and then the universal property for $\Spa(X,T)$ gives the horizontal dashed arrow.
    Composing with $Y_T \to Y$ we obtain the unique map required in the universal property of the universal vertical compactification of~$X$ over~$S$.
\end{proof}

What we have seen so far in this subsection relies entirely on the universal property of the universal vertical compactification.
We did not have to know anything about the construction or further geometric properties.
This will change with the following proposition where we will actually argue using the geometry of the spaces involved.

\begin{prop} \label{prop:comp-fibre-product}
    We consider a commutative diagram
        \[ 
            \begin{tikzcd}
                Y'\ar[r] \ar[d] & Y \ar[d] \\
                X'\ar[r] \ar[d] & X \ar[d] \\
                S'\ar[r]  & S
            \end{tikzcd}
        \]
    of adic spaces with $Y' \simeq X' \times_X Y$.
    Assuming that all vertical morphisms to~$S$ and~$S'$ satisfy the hypotheses of \cref{thm:Solodov-comp} and the horizontal morphisms  are adic, it follows that the natural morphism
    \[
     \Spa(Y',S') \longrightarrow \Spa(X',S') \times_{\Spa(X,S)} \Spa(Y,S)
    \]
    is an isomorphism of adic spaces.
\end{prop}

\begin{proof}
    The fibre product $\Spa(X',S') \times_{\Spa(X,S)} \Spa(Y,S)$ exists because $\Spa(Y,S) \to \Spa(X,S)$ is locally of ${}^+$weakly finite type (being vertically partially proper) and $\Spa(X',S') \to \Spa(X,S)$ is adic because $X' \to X$ is adic.
    By functoriality of $\Spa(-,-)$ as in \cref{prop:Spa-functoriality}, we obtain the natural morphism
    \[
       \varphi \colon \Spa(Y',S') \longrightarrow \Spa(X',S') \times_{\Spa(X,S)} \Spa(Y,S).
    \]
    Using the explicit description via root triples, we check that~$\varphi$ is bijective.
    Now, it suffices to show that $\varphi$ is locally an open immersion.

    We first reduce to the case that $S$ and $S'$ are affinoid as follows. 
    In the situation of \cref{thm:Solodov-comp}, the hypotheses of \cref{prop:check-univprop-comp} 
    are satisfied (see the proof of \cref{prop:check-univprop-comp} in \cite[Theorem~3.29]{Solodov-vertical-compactification}).
    This implies that not only is $\Spa(X,S)$ the universal compactification of~$X$ over~$S$, but also for every open subspace $U \subseteq S$ the base change $\Spa(X,S) \times_S U$ is the universal compactification of $X \times_S U$ over~$U$.
    Applying this reasoning to $X \to S$, $X' \to S'$, $Y \to S$, and $Y' \to S'$, we may now assume that $S$ and $S'$ are affinoid. 

    We choose compatible affinoid open covers $\{U_i\}_{i \in I}$, $\{U'_j\}_{j \in J}$, and $\{V_k\}_{k \in K}$ of~$X$,~$X'$, and~$Y$, respectively, such that for all $i \in I$ we have
    \[
        U_i \times_X X' = \bigcup_{j \to i} U'_j, \qquad U_i \times_X Y = \bigcup_{k \to i} V_k.
    \]
    We can apply \cref{lem:compatible-covers} below to the squares
    \[
        \begin{tikzcd}
            X' \ar[r] \ar[d] & X \ar[d] \\
            \Spa(X',S')  \ar[r] & \Spa(X,S)
        \end{tikzcd} \qquad
             \begin{tikzcd}
            Y \ar[r] \ar[d] & X \ar[d] \\
            \Spa(Y,S)  \ar[r] & \Spa(X,S)
        \end{tikzcd}
    \]
    and the open subsets $\ov U_i^\circ \subseteq \Spa(X,S)$ from \cref{lem:interiors-cover}, noting that $\ov U_i^\circ \cap X = U_i$.
    Consequently, for fixed~$i$, the open subsets 
    \begin{itemize}
        \item 
        $\widetilde{U}'_j := (\ov U^\circ_i \times_{\Spa(X,S)} \Spa(X',S')) \cap \ov U'^\circ_j$ for $j \to i$ and
        \item 
        $\widetilde{V}_k := (\ov U^\circ_i \times_{\Spa(X,S)} \Spa(Y,S)) \cap \ov V_k^\circ$ for $k \to i$
    \end{itemize}  
    cover $\ov U^\circ_i \times_{\Spa(X,S)} \Spa(X',S')$ and $\ov U^\circ_i \times_{\Spa(X,S)} \Spa(Y,S)$, respectively.
    Moreover $\widetilde{U}'_j \cap X' = U'_j$ and $\widetilde{V}_k \cap Y = V_k$ and for $j \to i$ and $k \to i$, the images of~$\widetilde{U}'_j$ and $\widetilde{V}_k$ are contained in~$\ov U^\circ_i$.
    Hence, the fibre products $\widetilde{U}'_j \times_{\ov U^\circ_i} \widetilde{V}_k$ cover $\Spa(X',S') \times_{\Spa(X,S)} \Spa(Y,S)$ and we can apply \cref{lem:compatible-covers} again, this time to the square
    \[
        \begin{tikzcd}
            Y' \ar[r,"\id"] \ar[d] & X' \times_X Y \ar[d] \\
            \Spa(Y',S') \ar[r,"\varphi"] & \Spa(X',S') \times_{\Spa(X,S)} \Spa(Y,S)
        \end{tikzcd}
    \]
    with the open subsets $\widetilde{U}'_j \times_{\ov U^\circ_i} \widetilde{V}_k$, and the trivial open cover of the affinoid $V'_{j,k} = U'_j \times_{U_i} V_k$ by itself (for all $(j,k) \in J \times_I K$).
    In this way we construct an open cover of $\Spa(Y',S')$ by subsets $\widetilde{V}'_{j,k}$  lying in $\Spa(V'_{j,k},S')$ and fitting into the diagram
        \[
        \begin{tikzcd}
            & \Spa(V'_{j,k},S') \arrow[rr,"\varphi_{j,k}"] \arrow[dd]
            && \Spa(U'_j,S') \times_{\Spa(U_i,S)} \Spa(V_k,S) \arrow[dd]
            \\
            \widetilde{V}'_{j,k} \arrow[rr,dashed, crossing over]  \arrow[ur,open] \arrow[dr,open]
            &&  \widetilde{U}'_j \times_{\ov{U}_i^\circ} \widetilde{V}_k \arrow[ur,open] \arrow[dr,open]
            \\
            & \Spa(Y',S') \arrow[rr,"\varphi"] 
            && \Spa(X',S') \times_{\Spa(X,S)} \Spa(Y,S) 
        \end{tikzcd}   
    \]
    where all diagonal maps are open immersions. The square with the morphisms $\ph$ and $\varphi_{j,k}$ commutes by the functoriality of $\Spa(-,-)$, hence both morphisms agree on $\widetilde{V}'_{j,k}$.
    Since the subsets~$\widetilde{V}'_{j,k}$ cover $\Spa(Y',S')$, if we knew that~$\varphi_{j,k}$ is an isomorphism, we could conclude that~$\varphi$ is a bijective open immersion, hence an isomorphism.
    
    So finally, we need to treat the case where $S = \Spa(R,R^+)$, $S' = \Spa(R',R'^+)$, $X = \Spa(A,A^+)$, $Y = \Spa(B,B^+)$, and $X' = \Spa(A',A'^+)$ are affinoid.
    Then
    \begin{align*}
        \Spa(X',S') \times_{\Spa(X,S)} \Spa(Y,S) &=  \Spa(A',R'^+) \times_{\Spa(A,R^+)} \Spa(B,R^+) \\
        & = \Spa(A' \otimes_A B,R'^+) 
        = \Spa(Y',S'). \qedhere
    \end{align*}    
\end{proof}

\begin{lem} \label{lem:compatible-covers}
    Let $T$ be an affinoid adic space. Let $Z \to T$ be a morphism of adic spaces satisfying the hypotheses of \cref{thm:Solodov-comp}.
    Suppose we are given a diagram
    \[
        \begin{tikzcd}
            Z   \ar[r,"\varphi_0"]  \ar[d,open] & W_0 \ar[d,open] \\
            \Spa(Z,T) \ar[r,"\varphi"]                  & W
        \end{tikzcd}
    \]
    of adic spaces, with an open immersion $W_0 \to W$ and an open subspace $U \subseteq W$. 
    If $\varphi_0^{-1}(U \cap W_0) = \bigcup_{j \in J} V_j$ is a covering by open affinoid subspaces $V_j \subseteq Z$, then, with the open $\ov V_j^\circ \subseteq \Spa(Z,T)$ as in \Cref{lem:interiors-cover} we have 
    \[
        \varphi^{-1}(U) \subseteq \bigcup_{j \in J} \ov V_j^\circ.
    \]    
\end{lem}
\begin{proof}
    For~$z \in \varphi^{-1}(U)$ let $z_0 \leadsto z$ be the minimal vertical generalisation of~$z$ contained in~$Z$.
    Then $\varphi(z_0) = \varphi_0(z_0) \in U \cap W_0$.
    We choose $j \in J$ with $z_0 \in V_j$.
    Then~$z$ is contained in~$\ov V_j^\circ$.
\end{proof}

\begin{cor} \label{cor:functoriality-compactification}
    Let $X \to S$ be as in \cref{thm:Solodov-comp}.
    For any adic morphism $S' \to S$ with $S'$ stable and square complete, the base change $\Spa(X,S)\times_S S'$ is the universal vertical compactification of $X \times_S S'$ over~$S'$.
    In other words, $\Spa(X\times_S S',S')$ is an adic space with a natural isomorphism
    \[
     \Spa(X \times_S S',S') \overset{\sim}{\longrightarrow} \Spa(X,S)\times_S S'.
    \]
\end{cor}

\begin{proof}
We want to apply \cref{prop:comp-fibre-product} to the diagram
\[
    \begin{tikzcd}
        X \times_S S'\ar[r] \ar[d] & X\ar[d] \\
        S'\ar[r] \ar[d] & S \ar[d] \\
        S'\ar[r] & S.
    \end{tikzcd}
\]
For doing so we still need to check that $X \times_S S' \to S'$ satisfies the hypotheses of \cref{thm:Solodov-comp}, i.e. that $X \times_S S'$ is weakly square complete.
This follows from the subsequent \cref{lem:square-complete-stable}.
\end{proof}

\begin{lem} \label{lem:square-complete-stable}
    Let
    \[
        \begin{tikzcd}
            X'  \ar[r]  \ar[d]  & X \ar[d]  \\
            S'  \ar[r]          & S
        \end{tikzcd}
    \]
    be a cartesian diagram of adic spaces, where~$S' \to S$ is adic, $S$ is square complete, and~$X$ and~$S'$ are weakly square complete.
    Then~$X'$ is weakly square complete.
\end{lem}

\begin{proof}
    We consider a diagram
         \[
      \begin{tikzcd}
       x'_v									\\
       x'	\ar[r,rightsquigarrow]	\ar[u,rightsquigarrow]	& x'_h	
      \end{tikzcd}
     \]
     in~$X'$, where $x' \leadsto x'_v$ is a vertical and $x' \leadsto x'_h$ a horizontal specialisation.
     We need to find $x'_{vh} \in X'$, a horizontal specialisation $x'_v \leadsto x'_{vh}$, and a vertical specialisation $x'_h \leadsto x'_{vh}$ completing the square.
     This is possible once we find a common affinoid neighbourhood   of~$x'$, $x'_h$, and~$x'_v$, as affinoids are square complete.
     
     Since~$X$ and~$S'$ are weakly square complete, the image diagrams in~$X$ and~$S'$ can be completed to squares
          \[
      \begin{tikzcd}
       x_v	\ar[r,rightsquigarrow]							& x_{vh}			\\
       x	\ar[r,rightsquigarrow]	\ar[u,rightsquigarrow]	& x_h	\ar[u,rightsquigarrow]
      \end{tikzcd} \qquad
      \begin{tikzcd}
       s'_v	\ar[r,rightsquigarrow]							& s'_{vh}			\\
       s'	\ar[r,rightsquigarrow]	\ar[u,rightsquigarrow]	& s'_h	\ar[u,rightsquigarrow]
      \end{tikzcd}
     \]
     with vertical and horizontal specialisations according to the definition of weak square completeness.
     Since~$S$ is even square complete, the images of both squares in~$S$ are the same.
     Choosing small enough affinoid neighbourhoods $U$, $V'$, $V$ of~$x_{vh}$, $s'_{vh}$, and their image $s_{vh}$ in~$S$, respectively, the fibre product $U' = U \times_{V} V'$ is an affinoid open in $X'$ containing $x'$, $x'_v$ and $x'_h$.
\end{proof}

\begin{rmk}
    In \cref{prop:comp-fibre-product} and \cref{cor:functoriality-compactification} the assumption that the relevant morphisms be adic is not only needed to ensure that the fibre product exists.
    For instance, the need for $S' \to S$ to be adic in \cref{cor:functoriality-compactification} is illustrated by the following example.
    Let $S'$ be the adic spectrum of a non-archimedean field $(K,K^\circ)$.
    We denote by~$S$ the spectrum of the same valued field $(K,K^\circ)$ but equipped with the discrete topology.
    The resulting morphism $S' \to S$ is not adic.
    For $X = \Spa(K,K)$ (with discrete topology on~$K$) the universal vertical compactification of~$X$ over~$S$ equals
    \[
     \Spa(X,S) = \Spa(K,K^\circ) = S.
    \]
    Its base change to~$S'$ equals~$S'$.
    But $X' = X \times_S S'$ is empty, so $\Spa(X,S) \times_S S'$ cannot be the universal vertical compactification of~$X'$ over~$S'$.
\end{rmk}

\begin{cor} \label{cor:base-change-comp}
    Let~$K$ be a non-archimedean field with residue field~$k$ and~$\fX$ a taut and separated formal scheme locally of finite type over~$K^\circ$.
    Let $\fX_\rig$ be its rigid generic fibre and~$\fX_k$ its special fibre.
    Then
    \begin{align*}
     \Spa(\fX^{\rm ad},K^\circ)_\eta &= \Spa(\fX_\rig,K),   \\
     \Spa(\fX^{\rm ad},K^\circ)_k &= \Spa(\fX_k,k).
    \end{align*}
\end{cor}

\begin{proof}
    This follows by applying \cref{cor:functoriality-compactification} to the base change of $\fX^{\rm ad} \to \Spa(K^\circ)$ via $\Spa(K) \to \Spa(K^\circ)$ and $\Spa(k) \to \Spa(K^\circ)$, respectively.
\end{proof}

\begin{ex} \label{ex:comp-disk}
    We illustrate the identifications from  \cref{cor:base-change-comp} by going through the constructions when
    \[
        \fX = \bA^1_{K^\circ} =  \Spf (K^\circ \langle T \rangle)
    \]
    is the formal affine line.
    Its generic fibre is the unit disc
    \[
        \fX_\eta = \bD_K = \Spa (K \langle T \rangle),
    \]
    whose compactification $\Spa(\bD_K,K)$ has one additional point~$x_{1^+}$ corresponding to the valuation of $K \langle T \rangle$ sending~$T$ to a value that is infinitesimally bigger than~$1$.
    It can be written as the composition of the Gauß valuation (of radius~$1$) with the valuation on its specialisation field $k(T)$ corresponding to $\infty \in \bP^1_k$.
    The special fibre is the affine line
    \[
        \fX_k = \bA^1_k = \Spec (k[T])
    \]
    with universal vertical compactification $\Spa(k[T],k)$.
    Also here we get one additional point~$x_{\infty}$, namely the valuation of $k(T)$ corresponding to the point $\infty \in \bP^1_k$.
    Now according to \cref{cor:base-change-comp}, the universal vertical compactification $\Spa(\bA^1_{K^\circ},K^\circ)$ has the two additional points~$x_{1^+}$ and~$x_\infty$.
    Note that~$x_\infty$ is a horizontal specialisation of~$x_{1^+}$ and if we were to extend the specialisation map from \cref{ss:adic-spaces} to universal vertical compactifications, we would send~$x_{1^+}$ to~$x_\infty$.
\end{ex}

\section{Tame fundamental groups of rigid spaces}
\label{s:fg-pi1-rig}

In this section, we define and study the tame fundamental group $\pitame(X/K)$ of a connected rigid-analytic space relative to a base non-archimedean field $K$ (see \S\ref{ss:adic-spaces} for our conventions and basic facts regarding rigid geometry). A finite \'etale cover $f\colon Y\to X$ is said to be tame relative to $K$ if for every $y\in Y$, the extension $k(y)/k(f(x))$ is tamely ramified with respect to every valuation subring $V$ of $k(y)$ containing $K^\circ$ (see \cref{def:tame-rig}). As explained in \S\ref{sec:root-triples}, if $X$ is separated and taut, such pairs $(y,V)$ are the points of the universal compactification $\overline{Y} = \Spa(Y, K)$, and $Y\to X$ is tame relative to $K$ if and only if for the induced finite \'etale map $\overline{Y}\to \overline{X}$ the extensions of residue fields are tamely ramified (see \cref{prop:tame-S-compactification}). Finite \'etale covers of $X$ which are tame relative to $K$ form a Galois category with fundamental group~$\pitame(X/K)$.

Our main technical result (\cref{cor:Abhy-infty-equiv}, proved in \S\ref{ss:log-abhyankar}--\ref{ss:tameness-comparison}) states that if $K$ is either discretely valued or algebraically closed, then for a connected strictly semistable formal scheme $\fX$ over $K^\circ$ with generic fibre $X$ and special fibre $\fX_k$, we have an isomorphism
\[
    \pitame(X/K) \simeq \pitame(\fX^{\rm log}_k/k)
\]
where $\fX^{\rm log}$ is the log formal scheme obtained by endowing $\fX$ with the log structure induced by the generic fibre (called the ``standard log structure'' below), and $\fX^{\rm log}_k$ is the special fibre endowed with the log structure induced from $\fX^{\rm log}$. Then $\fX^{\rm log}_k$ is a log scheme of type $\typeVd$ over $k$ and the group on the right is the one studied in \S\ref{s:fg-pi1-log}. In particular, combined with \cref{thm:pi1tame log is fg} this implies immediately that $\pitame(X/K)$ is finitely generated if $\fX$ is quasi-compact and $k$ is algebraically closed. 

Using an auxiliary result about the behaviour of $\pitame(X/K)$ with respect to field extensions (\S\ref{ss:K-vs-C}), some descent assertions (\S\ref{ss:alterations-v-descent}), and uniformisation results of Gabber (\cref{prop:h-locally-regular}) and Temkin (\cref{thm:Temkin-unif}), we can drop the semistable reduction hypothesis, as well as the restrictions on $K$, and show that $\pitame(X/K)$ is finitely generated if $X$ is qcqs and if the tame Galois group $\pitame(K)$ is finitely generated (\cref{thm:fg-pi1-rig}). Employing \cref{cor:surgery-fp} we also show that $\pitame(X/K)$ is finitely presented under more restrictive assumptions (\cref{thm:fp-sometimes}). We conclude this section with a few examples (see \cref{ex:polydisc}--\cref{ex:helminck}).

\subsection{The tame fundamental group of a rigid space}
\label{ss:tame-pi1-rig}

We give the definition of the tame fundamental group for arbitrary morphisms of adic spaces, before specialising to the case of rigid-analytic varieties. First, recall the following definition from \cite{Hubner2021:AdicTameSite}. 

\begin{defi}[{Absolute tameness}] \label{def:tame-rig-abs}
    Let $f\colon Y\to X$ be an \'etale morphism of adic spaces. We say that $f$ is {\bf tame} if for every $y\in Y$ with image $x\in X$,  the finite separable extension of valued fields 
    \[
        (k(y), k(y)^+)/(k(x), k(x)^+)
    \]
    is tamely ramified, see \S\ref{sec:tame extensions valued fields}.
\end{defi}

We extend this definition by making it relative to a base adic space $S$.

\begin{defi}[Relative tameness] \label{def:tame-rig}
    Let $X\to S$ be a morphism of adic spaces and let $Y\to X$ be a finite \'etale morphism. We say that $Y\to X$ is \textbf{tame relative to $S$} if for every root triple 
    \[
        y= (y^\circ, k(y)^+, y^\succ) \in \Spa(Y, S)
        \qquad \text{(see~\cref{defi:root-triple})}
    \]
    with image $x = (x^\circ, k(x)^+, x^\succ)\in \Spa(X, S)$, the finite separable extension of valued fields $(k(y), k(y)^+)/(k(x), k(x)^+)$ is tamely ramified. 
\end{defi}

\begin{rmks} \label{rmks:relative-tameness}
    \begin{enumerate}[(1)]
        \item Note that we defined tameness relative to~$S$ only for \emph{finite} étale morphisms. So with this notion we will be able to define a tame fundamental group of~$X$ relative to~$S$ (see \cref{def:tamepi1-rigid}) but we do not talk about a tame site of~$X$ relative to~$S$. We could define tameness relative to $S$ for a general étale morphism just as above but this is not a useful concept. For instance, $\Spa(Y,S) \to \Spa(X,S)$ fails to be étale if $Y \to X$ is étale but not finite. Moreover, many of the properties proved in \cref{lem:properties-tameness-relative-S} below are not satisfied if $Y\to X$ is not finite. 
        \item In the setting of \cref{def:tame-rig}, any point $y \in Y$ gives rise to the root triple $(y^\circ,k(y)^+,y^\succ)$ where~$y^\succ$ is the image of~$y$ in~$S$. So if $Y \to X$ is tame relative to~$S$, it is also (absolutely) tame. For tameness relative to~$S$ there might be additional tameness conditions as there can exist root triples that do not arise from points of~$Y$. 
        \item \label{enumitem:relative-tameness-test-square} A finite \'etale $Y\to X$ is tame relative to $S$ if and only if the following condition holds: for every point of $\Spa(X, S)$ represented by a square \labelcref{eqn:pp-test-square}, the pull-back of $Y$ to $\Spa(k, k^\circ)$ is a finite \'etale map. Since $\Spa(k, k^\circ)$ and $\Spa(k, k^+)$ have the same finite \'etale covers (namely, finite separable $k$-algebras), we obtain a finite \'etale cover of $\Spa(k, k^+)$, and we want it to be tame with respect to $k^+$. 
    \end{enumerate}
\end{rmks}

\begin{lem} \label{lem:properties-tameness-relative-S}
    Let $Y\to X\to S$ be maps of adic spaces, with $Y\to X$ finite \'etale. 
    \begin{enumerate}[(a)]
        \item \label{lemitem:sorite-tameness-relative-S1}
        If $X=S$, then $Y\to X$ is tame relative to $S$ if and only if it is tame (in the sense of \cref{def:tame-rig-abs}).
        \item \label{lemitem:sorite-tameness-relative-S2}
        Given a map $S\to T$, if $Y\to X$ is tame relative to $T$, then it is tame relative to $S$. The converse holds if $S\to T$ is vertically partially proper.
        \item 
        \label{lemitem:sorite-tameness-relative-S3}
        Given a surjective finite \'etale $Z\to Y$, the map $Z \to X$ is tame relative to~$S$ if and only if $Z\to Y$ and $Y\to X$ are tame relative to $S$. 
        \item 
        \label{lemitem:sorite-tameness-relative-S4}
        Given a commutative diagram
        \[ 
            \begin{tikzcd}
                Y'\ar[r] \ar[d] & Y \ar[d] \\
                X'\ar[r] \ar[d] & X \ar[d] \\
                S'\ar[r]  & S
            \end{tikzcd}
        \]
        where $Y' = Y\times_X X'$, if $Y\to X$ is tame relative to $S$, then $Y'\to X'$ is tame relative to~$S'$.
    \end{enumerate}    
\end{lem}

\begin{proof}
\labelcref{lemitem:sorite-tameness-relative-S1} In this case $\Spa(X, S) = X$ and $\Spa(Y,S) = Y$.
The first part of \labelcref{lemitem:sorite-tameness-relative-S2} is a special case of \labelcref{lemitem:sorite-tameness-relative-S4} (take  $X=X'$), and the second follows from the fact that $\Spa(X, S)\to \Spa(X, T)$ is a bijection if $S\to T$ is vertically partially proper (see \cref{prop:Spa-iso-vpp}).
Part \labelcref{lemitem:sorite-tameness-relative-S3} follows from the standard fact that for algebraic extensions of valued fields $M/L/K$, the extension $M/K$ is tame if and only if $M/L$ and $L/K$ are tame. Finally, part \labelcref{lemitem:sorite-tameness-relative-S4} is clear by the characterisation in \cref{rmks:relative-tameness}~\ref{enumitem:relative-tameness-test-square}.
\end{proof}

We can understand tameness relative to~$S$ in terms of the universal vertical compactification.

\begin{prop} \label{prop:tame-S-compactification}
    Let $Y\to X\to S$ be maps of adic spaces.
    Suppose that $Y\to X$ is finite \'etale and that $X\to S$ satisfies the assumptions of \cref{thm:Solodov} and \cref{thm:Solodov-comp}. Then the map $Y\to S$ also satisfies said assumptions, and the induced map (see \cref{prop:Spa-functoriality}) between universal vertical compactifications relative to $S$
        \[ 
            \bar{f}\colon \Spa(Y,S)\longrightarrow \Spa(X,S)
        \]
    is finite \'etale. The map $f$ is tame relative to $S$ if and only if the map $\bar{f}$ is tame (in the absolute sense of \cref{def:tame-rig-abs}).
\end{prop}

\begin{proof}
We first convince ourselves that~$\bar{f}$ is finite étale.
We start with the affinoid case.
If~$S$ and~$X$ are affinoid, so is~$Y$, and we can write
\[
 S = \Spa(R,R^+), \qquad X = \Spa(A,A^+), \qquad Y = \Spa(B,B^+),
\]
for complete Huber pairs $(R,R^+)$, $(A,A^+)$, and $(B,B^+)$.
Then
\[
 \Spa(X,S) = \Spa(A,R^+) = \Spa(A,A^+_{\min}), \qquad \Spa(Y,S) = \Spa(B,R^+) = \Spa(B,B^+_{\min}),
\]
(see \cref{prop:affinoid-compactification} and \cref{rmk:hypotheses-compactification}) where~$A^+_{\min}$ and~$B^+_{\min}$ are the smallest rings of integral elements in~$A$ and~$B$, respectively, containing~$R^+$.
Since $Y \to X$ is finite étale, $B$ is finite étale over~$A$ and~$B^+$ is the integral closure of~$A^+$ in~$B$.
The integral closure of~$A^+_{\min}$ in~$B$ is a ring of integral elements in~$B$ and must therefore equal~$B^+_{\min}$.
We conclude that~$\bar{f}$ is finite étale in the affinoid case.

In order to deduce the general case, we cover~$S$ by affinoid opens~$S_i$.
Then by \cref{cor:functoriality-compactification}, 
\[
 \Spa(X \times_S S_i,S_i) \simeq \Spa(X,S) \times_S S_i
\]
and these spaces form an open covering of $\Spa(X,S)$.
The restriction of~$\bar{f}$ to $\Spa(X \times_S S_i,S_i)$ is the natural map
\[
 \Spa(Y \times_S S_i,S_i) \la \Spa(X \times_S S_i,S_i)
\]
and since the property of being finite étale is local on the target, we are reduced to the case where~$S$ is affinoid.

Now we cover~$X$ by open affinoids~$X_i$.
We obtain a surjective family
\[
 \{\Spa(X_i,S) \la \Spa(X,S)\}_i
\]
of subspaces.
By \cref{lem:interiors-cover}, the open subspaces $\Spa(X_i,S)^\circ \subseteq \Spa(X_i,S)$ defined there comprise an open cover of $\Spa(X,S)$.
Using \cref{prop:comp-fibre-product} we identify the restriction of~$\bar{f}$ to $\Spa(X_i,S)$ with the morphism
\[
 \Spa(X_i \times_X Y,S) \longrightarrow \Spa(X_i,S).
\]
By the affinoid case this morphism is finite étale and stays so when restricted to~$\Spa(X_i,S)^\circ$.
As the open subsets~$\Spa(X_i,S)^\circ$ cover $\Spa(X,S)$, this shows that~$\bar{f}$ is finite étale.

The last assertion about~$f$ being tame relative to~$S$ if and only if~$\bar{f}$ is tame follows immediately by the description of the points of $\Spa(Y,S)$ as root triples.
\end{proof}

By a {\bf geometric point} of an adic space $X$ we shall mean a map of adic spaces $\Spa(L,L^+)\to X$ where $(L, L^+)$ is an affinoid field with $L$ separably closed, see \cite[\S 3.2]{AchingerLaraYoucis2023:GeometricArcsFundamental}. Note that the space $\Spa(L, L^+)$ might have more than one point.

\begin{lem} 
\label{lemma:adic-tame-galois}
    Let $X\to S$ be a morphism of adic spaces. Suppose that $X$ is connected and let $F_{\ov x}$ be the fibre functor associated to a geometric point  $\ov{x}$ of $X$. 

    Then $\FEtt_{X/S}$ is a Galois category with fibre functor $F_{\ov{x}}$. 
\end{lem}

\begin{proof}
This follows from the criterion of \cref{lem:Galois-subcategory-criterion} and \cref{lem:properties-tameness-relative-S}. 
\end{proof}

\begin{defi} \label{def:tamepi1-rigid}
\begin{enumerate}
    \item 
    In the situation of \cref{lemma:adic-tame-galois}, we denote by 
    $\pitame(X/S, \ov{x})$
    the fundamental group of the Galois category $\FEtt_{X/S}$ with base point $F_{\ov{x}}$, and call it the \textbf{tame fundamental group} of $X$ relative to $S$. 
    \item
    In the special case $X = S$ we call $\pitame(X, \ov{x}) = \pitame(X/X, \ov{x})$ the \textbf{(absolute) tame fundamental group} of $X$. The latter invariant has been introduced in \cite[\S9]{Hubner2021:AdicTameSite}.
    \item    
    If $X$ is a rigid-analytic space over a non-archimedean field $K$, we write $\pitame(X/K, \ov{x})$ for the tame fundamental group of $X$ relative to $\Spa(K)$. 
\end{enumerate}
\end{defi}

\begin{rmk} \label{rmk:Berkovich-covers}
    In \cite[\S6.3]{Berkovich:EtaleCohomology} Berkovich studies a variant of tame covers of Berkovich spaces over a non-archimedean field of residue characteristic $p > 0$.
    In the language of adic spaces, the (connected) covers he considers are (connected) finite étale covers $Y \to X$ whose Galois closure $Y' \to X \in \FEt_X$ satisfies the following: for any point $y' \in Y'$ with image $x \in X$, the degree $[\widehat{k(y')}:\widehat{k(x)}]$ is prime to~$p$.
    Every such cover is tame relative to~$K$ in our sense.
    However, a nontrivial strongly étale $p$-cover is tame relative to~$K$ but not tame in Berkovich's sense.
    Berkovich's covers form a Galois subcategory of $\FEtt_{X/K}$ by~\cref{lem:Galois-subcategory-criterion} and the corresponding fundamental group~$\pi_1^B(X/K)$ is a quotient of~$\pitame(X/K)$.
\end{rmk}

Since for $X\to S\to T$ we have full subcategories 
\[
    \FEtt_{X/T}\subseteq  \FEtt_{X/S}\subseteq \FEtt_X\subseteq \FEt_X,
\]
we obtain continuous surjections of profinite groups
\[ 
    \pi_1(X, \overline{x}) \surj \pitame(X, \overline{x}) \surj \pitame(X/S,\overline{x}) \surj \pitame(X/T,\overline{x}).
\]
In particular, for a rigid-analytic space $X$ over $K$, we obtain surjections
\begin{equation} \label{eqn:pi1comp-maps}
    \pi_1(X, \overline{x}) \surj \pitame(X,\overline{x}) \surj \pitame(X/K,\overline{x}).
\end{equation}
Neither of these maps is an isomorphism in general, as the following example shows. 

\begin{ex} \label{ex:unit-disc-Cp}
    Let $K = \mathbb{C}_p$ be a completed algebraic closure of $\mathbb{Q}_p$ and let 
    \[ 
        X = \Spa(K\langle T\rangle)
    \]
    be the affinoid unit disc over $K$. In the sequence of surjections \eqref{eqn:pi1comp-maps}, neither of the maps is an isomorphism, and only the right-most group is finitely generated (in fact trivial, see \cref{ex:polydisc}). 
    
    Let $\mathfrak{X} = \Spf(K^\circ \langle T\rangle)$ be the standard affine formal model of $X$ with special fibre $\mathbb{A}^1_k$ where $k=K^\succ$ is the residue field of $K$. Every finite \'etale map $Y_0\to \mathbb{A}^1_k$ lifts uniquely to a finite \'etale $\mathfrak{Y}\to \mathfrak{X}$. Its generic fibre $Y=\mathfrak{Y}_{\rm rig}\to \mathfrak{X}_{\rm rig}=X$ is a tame map \emph{relative to $X$}, and $Y$ is connected if and only if $Y_0$ is (since $\mathfrak{Y}$ is normal, cf.~\cite[Lemma A.2+Corollary A.16]{AchingerLaraYoucis2022}). We obtain a surjection
    \[ 
        \pitame(X,\overline{x}) \surj \pi_1(\mathbb{A}^1_k, \overline{x}).
    \]
    Since the target is not finitely generated, neither are $\pitame(X, \overline{x})$ and $\pi_1(X, \overline{x})$. As we shall prove later in \S\ref{ss:tameness-comparison}, the cover $Y\to X$ is tame \emph{relative to $K$} if and only if $Y_0 \to \mathbb{A}^1_k$ is tame relative to $k$ (in the sense of \cref{defi:tame cover of schemes X over S}), which happens only if $Y_0 \to\mathbb{A}^1_k$ is trivial.

    To be explicit, the finite \'etale covers
    \[ 
        Y_\lambda = \Spa(K\langle T, U\rangle/(U^p-U-\lambda T)) \longrightarrow \Spa(K\langle T\rangle) = X
    \]
    for $\lambda\in\mathbb{C}_p$ with $|\lambda|=1$ give rise to $\mathbb{F}_p$-torsors such that $Y_\lambda\neq Y_{\lambda'}$ if $|\lambda-\lambda'|=1$. They are all tame relative to $X$, so $\Hom(\pitame(X,\overline{x}),\bF_p)$ is infinite and therefore $\pitame(X,\overline{x})$ is not finitely generated. On the other hand, the $Y_\lambda \to X$ are not tame relative to $K$.  
    To see the latter, consider, as in \cref{ex:comp-disk}, the unique boundary point of the universal compactification 
    \[
        \Spa(X,K)\setminus X = \{\xi\}.
    \]
    Then the corresponding maximal point $\xi^\circ$ is the Gau\ss\ point. Its specialisation field $k(\xi^\circ)^\succ$ is the function field $k(T)$ of $\mathbb{A}^1_k$. If $v_\infty$ is the valuation on $k(T)$ corresponding to the point at infinity, then the valuation corresponding to $\xi$ is the composite of the two. In the cover $Y_\lambda \to X$ the Gau\ss\ point is inert with residue field extension $k(U)/k(T)$, the function field of the Artin-Schreier extension defined by $U^p-U - \bar \lambda T = 0$, where $0 \not=\bar \lambda \in k$ is the reduction of $\lambda$. It follows that $v_\infty$ is wildly ramified in this residual representation, hence the composite $\xi$ is also not tame by \cref{composition-tame}.

    In order to construct a finite \'etale $Y\to X$ which is not even tame relative to $X$, consider the finite \'etale cover of the disc described by the equation $U^p - p^{1/2}U  - T = 0$. Here the specialisation field extension at the Gau\ss\ point is purely inseparable.    
\end{ex}

\subsection{Behaviour with respect to field extensions} 
\label{ss:K-vs-C}

To deal with a non-algebraically closed base field $K$, we need to investigate how the tame fundamental group $\pitame(X/K)$ of a geometrically connected rigid space $X$ over $K$ compares to its geometric variant $\pitame(X_C/C)$ where $C$ is a completed algebraic closure of $K$. This relationship is slightly subtle, as the analogue of the fundamental sequence is not exact in general. Fortunately, it is always almost exact, as we show in \cref{prop:weak-fundamental-exact-seq} below. The case when $X$ is a point is covered by the following lemma. 

\begin{lem} \label{lem:structure of pitame na field}
    Let $(K,K^+)$ be a henselian valued field of residue characteristic $p \ge 0$.
    We fix an algebraic closure of $(K,K^+)$ giving rise to a geometric point $\bar x$ of $S=\Spa(K,K^+)$.
    We use the notation $\pitame(K)$ to denote the Galois group $\Gal(K^t/K)$.
    Then there is a canonical isomorphism
        \[
        \pitame(S,\bar x) \isomto \pitame(K).
        \]
\end{lem}

\begin{proof}
    The assertion follows immediately from the definition since the relevant space of root triples $\Spa(S,S) = S$ contains only coarsenings of the valuation defining $K^+$. 
\end{proof}

Let now $X$ be a geometrically connected qcqs rigid analytic space over a non-archimedean field $K$ and let $C$ be a completed algebraic closure of $K$. Pick a geometric point $\ov x$ of $X_C$. We are interested in the exactness of the following sequence of profinite groups 
\begin{equation} \label{eqn:fundamental-ex-seq}
    \begin{tikzcd}
        \pitame(X_C/C, \ov x)\ar[r] & \pitame(X/K,\ov x) \ar[r] & \pitame(K)\ar[r] & 1.
    \end{tikzcd}
\end{equation}
The subtlety surrounding this sequence which we hinted at in the first paragraph stems from the fact that for a wildly ramified finite separable extension $K'/K$, the induced finite \'etale morphism $X_{K'}\to X$ might be tame relative to $K$ (see \cref{ex:Raynaud inspired example,ex:moregeneral-twisted-nodal-curve} below). 
In order to address this issue, let us consider the full subcategory $\cC_{X/K}$ of $\FEt_K$ of all (spectra of) finite \'etale algebras $A/K$ such that 
\[
    X_A = X \times_{\Spa(K)} \Spa(A)\la X
\]
is tame relative to $K$. This contains $\FEtt_{K}$ but might a priori be larger if $X(K)=\emptyset$. The category $\cC_{X/K}$ is a Galois category whose fundamental group we denote by $\pi_1(\cC_{K/X})$. Clearly we have surjections
\[
\pi_1^\et(K) \surj \pi_1(\cC_{K/X}) \surj \pitame(K).
\]

\begin{prop} 
\label{prop:weak-fundamental-exact-seq}
    Let $X$ be a geometrically connected qcqs rigid analytic space over a non-archimedean field $K$. Let $C$ be the completion of an algebraic closure of $K$.  Let $\bar x$ be a geometric point of the base change $X_C$. We will also denote by $\bar x$ its image in $X$ and compute $\pi_1(\cC_{K/X}) \surj \pitame(K)$ with respect to the separable closure of $K$ in $\bar x$. 
    \begin{enumerate}[ref=(\arabic*)]
        \item 
        \label{propitem:weak-fes ct version}
        The sequence of profinite groups
        \begin{equation}
            \label{eq:homotopy exact sequence ct version}
            \begin{tikzcd}
                \pitame(X_C/C, \bar x)\ar[r] & \pitame(X/K, \bar x) \ar[r] & \pi_1(\cC_{K/X}) \ar[r] & 1           
            \end{tikzcd}
        \end{equation}        
        is exact.
        \item \label{propitem:finite kernel}
        The kernel of $\pi_1(\cC_{K/X}) \surj \pitame(K)$ is finite.
        \item 
        \label{propitem:weak-fes}
        If we assume furthermore that $X(K^t) \not= \varnothing$, then 
        $\pi_1(\cC_{K/X})\isomto \pitame(K)$ and
        \begin{equation}
            \label{eq:homotopy exact sequence}
            \begin{tikzcd}
                \pitame(X_C/C, \bar x)\ar[r] & \pitame(X/K, \bar x) \ar[r] & \pitame(K) \ar[r] & 1   
            \end{tikzcd}
        \end{equation}        
        is exact.
    \end{enumerate}
\end{prop}

\begin{proof}
    \labelcref{propitem:weak-fes ct version}
    The map $\pitame(X/K,\bar x) \surj \pi_1(\cC_{K/X})$ is surjective, because $X$ is assumed to be geometrically connected, hence $L \mapsto X_L$ maps connected objects to connected objects. 

    The composition $\pitame(X_C/C, \bar x) \to \pitame(X/K,\bar x) \surj \pi_1(\cC_{K/X})$ is the zero homomorphism, because for every $L$ in $\cC_{X/K}$ the cover $C \times_K X_L = \coprod_{L \inj C} X_C \to X_C$ is completely split. 

    We proceed as in \cite[X]{SGA1}. It remains to show that for every open subgroup $U \subseteq \pitame(X/K,\bar x)$ corresponding to a connected tame cover $X' \to X$ such that the projection  $X_C \to X$ admits a lift $X_C \to X'$, there is a separable extension $L/K$ such that $X' \simeq X_L$ as tame covers of $X$. This follows at once from
    \cite[Proposition~2.13]{deJong95:FundamentalGroupsofAnalytic}
    (which is stated for Berkovich spaces). Note that $L$ automatically belongs to $\cC_{X/K}$.
    
    \labelcref{propitem:finite kernel} The Galois category $\cC_{X/K}$ also makes sense if $X$ is not necessarily geometrically connected over $K$. Moreover, for a $K$-morphism $Y \to X$ of rigid analytic varieties, $\cC_{X/K}$ is contained in $\cC_{Y/K}$. In order to show \labelcref{propitem:finite kernel} we may replace $X$ by $Y$.
    
    Let $y = \Spa(K') \to X$ be a classical point with $K'/K$ finite. It thus suffices to show that there is a field extension $M$ of $K$ that is a finite extension of $K^t$ such that all fields $L$ from $\cC_{y/K}$ embed into $M$. A field $L$ belongs to $\cC_{y/K}$ if and only if $L \otimes_K K'$ is a product of tame extensions of $K'$. Being a subfield of those factors, $L$ is contained in $K'^t = K' K^t$, see  \cref{cor:pitame injective for fields}, which is finite over $K^t$. This verifies \labelcref{propitem:finite kernel} and assertion \labelcref{propitem:weak-fes} follows at once from the above proof if we pick $K' \subseteq K^t$.     
\end{proof}

\begin{cor} \label{cor:fg-alg-closed}
    In the situation of \cref{prop:weak-fundamental-exact-seq}, suppose that $\pitame(X_C/C,\bar x)$ and $\pitame(K)$ are finitely generated. Then $\pitame(X/K,\bar x)$ is finitely generated as well.
\end{cor}

\begin{proof}
By \cref{prop:weak-fundamental-exact-seq}~\labelcref{propitem:finite kernel} also 
$\pi_1(\cC_{K/X})$ is finitely generated. The claim thus follows from the short exact sequence from
\cref{prop:weak-fundamental-exact-seq}~\labelcref{propitem:weak-fes ct version}.
\end{proof}

\begin{rmk}
    The conclusions of \cref{prop:weak-fundamental-exact-seq} and \cref{cor:fg-alg-closed} also hold for a geometrically connected scheme of finite type over $K$ with tameness relative to $K^\circ$.
\end{rmk}

We finish this subsection by giving two examples of geometrically connected rigid spaces $X$ over $K$ for which $\pi_1(\cC_{X/K})\neq \pitame(K)$ and consequently for which \eqref{eqn:fundamental-ex-seq} is not exact. 

\begin{ex}[{due to Wittenberg based on ideas of Raynaud from  \cite[\S9]{Raynaud1970:SpecialisationFoncteurPicard}}] \label{ex:Raynaud inspired example}
    Let $K$ be a complete discretely valued field with residue field $k$ of characteristic $p>0$ and let $E$ be an elliptic curve over $K^\circ$ with a $p$-torsion point $x\in E(K^\circ)$ with nonzero image in $E(k)$. Translation by $x$ gives a free action of $\bZ/p\bZ$ on $E$ and an \'etale $\bZ/p\bZ$-torsor $E\to E/(\bZ/p\bZ)=F$.   
    Let $L/K$ be a wildly ramified $\mathbb{Z}/p\mathbb{Z}$-extension. We set
    \[
        X = \left(E\times_{\Spec(K^\circ)} \Spec(L^\circ)\right)\bigg/ (\bZ/p\bZ)
    \]
    with the quotient by the free diagonal action of $\bZ/p\bZ$, using the translation-by-$x$ action on $E$ and the Galois action on $L^\circ$. 
    So $X \to F$ is an $E/F$-twisted version of $F_{L^0} \to F$. Since the twisting $\bZ/p\bZ$-torsor $E\to F$ becomes trivial after pull-back to $E$, so does $X \to F$ and hence we have a cartesian square
    \[
        \begin{tikzcd}
            E_{L^0} \ar[r] \ar[d] & E \ar[d] \\
            X \ar[r] & F.
        \end{tikzcd}
    \]
    In particular, $E_{L^0}\to X$ is \'etale. On the other hand, the generic fibre of this map is the projection $E_L = X_L\to X_K$. To see this, note that $X_L=E_L\to F_K$ is a connected $(\bZ/p\bZ)^2$-torsor and that $X_K$ and $E_K$ are non-isomorphic degree $p$ intermediate coverings (as $X_K$ is a non-trivial twisted form of $E_K$), which implies that $E_L \simeq X_K\times_{F_K} E_K$, as desired.
    We conclude that the map of generic fibres $X_L\to X_K$ is tame relative to $K^\circ$, and that the induced map of rigid analytic spaces $X_L^{\rm an}=(X_K^{\rm an})_L\to X_K^{\rm an}$ is tame relative to $K$, showing that $\pi_1(\cC_{X_K^{\rm an}/K})\neq\pitame(K)$.
\end{ex}

\begin{ex}[Twisted nodal curve] \label{ex:moregeneral-twisted-nodal-curve}
Let $K$ be a non-archimedean field with a wildly ramified cyclic separable extension $L/K$. Let $\tau \in \Gal(L/K)$ be a generator. Let $X=\Spa(A)$ where
\[
    A= \{f\in K\langle T\rangle\,:\, f(0) = \tau(f(1))\}, 
\]
which is an affinoid curve with a single node with residue field $L$. We claim that $X_L\to X$ is tame relative to $K$. 
One computes that 
\[
    A\otimes_K L \simeq \left\{(f_\sigma) \in {\textstyle \prod_{\sigma \in \Gal(L/K)}} L[X] \,:\, f_{\tau\sigma}(1)=\tau(f_{\sigma}(0)), \forall \sigma \in \Gal(L/K) \right\}.
    \]
This means that $X$ is geometrically connected over $K$ and $X_L$ agrees with a chain of affinoid disks $D_\sigma$ indexed by $\sigma \in \Gal(L/K)$ where $0 \in D_{\sigma}$ is glued transversally and over $L$ with $1 \in D_{\tau\sigma}$. The projection $X_L \to X$ is a finite Nisnevich cover that splits itself, hence it is tame relative to $\Spa(K)$, although $L/K$ is not. 
\end{ex}

\subsection{\texorpdfstring{Alterations and $v$-descent}{Alterations and v-descent}}
\label{ss:alterations-v-descent}

The goal of this subsection is twofold. First, using the good topological properties of qcqs analytic adic spaces (\cref{lem:submersive}), we prove a general descent statement for finite generation of the (tame) fundamental group under surjections (\cref{cor:descent-of-fg-fett}). We then show, using Gabber's uniformisation results, that every rigid space admits a qcqs surjection from a regular rigid space. 

\begin{lem} 
\label{lem:submersive}
    Let $Y\to X$ be a surjective map between qcqs analytic adic spaces. Then $Y\to X$ is topologically submersive, i.e.\ $X$ has the quotient topology. If $Y\to X$ is moreover of finite type, then this holds universally for base-changes along all maps of qcqs analytic adic spaces $X'\to X$. 
\end{lem}

\begin{proof}
As in this setting ``surjectivity is universal'', we are quickly reduced to just proving that $Y\to X$ is topologically submersive. For this, see \cite[Lemma~2.5]{ScholzeECoD} and the remark of \emph{loc.~cit.} that all maps between analytic adic spaces are generalising, see \cite[Lemma~1.1.10(v)]{HuberBook}. 
\end{proof}

Let $Y\to X$ be a map between qcqs rigid spaces. Then pull-back defines a rigid analogue of \labelcref{eq:CechDescentSettingForFEt}, namely a functor 
\[
    \FEt_X\la \Desc(Y/X, \FEt).
\]

\begin{lem} 
\label{lem:univ-top-subm-descent}
    Let $f\colon Y\to X$ be a surjective map between qcqs rigid spaces. Then the pull-back functor
    \[ 
       f^\ast \colon  \FEt_X\la \Desc(Y/X, \FEt)
    \]
    is fully faithful, and the same holds with $\FEtt$ or $\FEtt_{-/S}$ in place of $\FEt$.
\end{lem}

\begin{proof} 
As in \Cref{lem:f-desc-for-FEt-and-FEtt}, if $f^\ast$ is fully faithful for $\FEt$, then $f^\ast$ is also fully faithful for $\FEtt$ and $\FEtt_{-/S}$.
So we have to prove only the assertion for $\FEt$. In this case, the result follows from \cref{lem:submersive} exactly as in \cite[Exp.\ IX, Corollaire 3.3]{SGA1}.
\end{proof}

\begin{rmk}[Effective descent]
In the situation of \cref{lem:univ-top-subm-descent}, if the base field $K$ is of positive or mixed characteristic, one can use the theory of diamonds \cite{ScholzeECoD} to show that    
    \[ 
       f^\ast \colon  \FEt_X\la \Desc(Y/X, \FEt)
    \]
is an equivalence. (We will not use this observation in this paper.)

The argument for this is as follows. In both cases, $X$ is an analytic adic space over $\bZ_p$. By \cite[Lemma 15.6]{ScholzeECoD}, the associated diamond $X^{\diamond}$ is locally spatial and $\FEt_X \simeq \FEt_{X^\diamond}$. Analogous facts hold for $Y$, $Y\times_X Y$, and $Y\times_X Y \times_X Y$ (note that diamantification commutes with fibre products). By Lemma~15.6 of \emph{loc.~cit.}, the map $|Y^\diamond| \to |X^\diamond|$ is still a surjective map of qcqs spaces and by Lemma~12.11 of \emph{loc.~cit.}, this implies that the map of diamonds $f^\diamond \colon Y^\diamond \to X^\diamond$ is surjective, so a v-covering. Any descent datum of diamonds along $f^\diamond$ is effective. By Proposition 10.11 of \emph{loc.~cit.}, the gluing is finite \'etale over $X^\diamond$ and so comes from $\FEt_X$, as desired. We thank Ben Heuer for supplying this argument. 
\end{rmk}

\begin{cor} \label{cor:descent-of-fg-fett}
    Let $Y\to X$ be a surjective map between qcqs rigid spaces over $K$. 
    If for every connected component $Y'$ of $Y$ the fundamental group $\pitame(Y'/K)$ is finitely generated, then for every connected component $X'$ of $X$ the fundamental group $\pitame(X'/K)$ is finitely generated.
\end{cor}

\begin{proof}
Replacing $Y\to X$ with its base change to a connected component of $X$, we may assume that $X=X'$ is connected. \Cref{lem:univ-top-subm-descent} then produces a surjection
\[ 
    \pi_1(\Desc(Y/X), \FEtt_{-/K}) \surj \pitame(X/K).
\]
The source of this map is finitely generated by \cref{prop:DD-fgfp}, and hence so is $\pitame(X/K)$. 
\end{proof}

Rigid spaces are ``$h$-locally regular'' is our next result in \cref{prop:h-locally-regular} below.

\begin{defi}[{Relative analytification, see \cite[Proposition~3.8]{Huber94:Generalization}}] \label{def:relative-analytification}
    Let $R$ be an affinoid $K$-algebra, let $S_0=\Spec(R)$, and let $X_0\to S_0$ be a morphism locally of finite type. Let $S=\Spa(R)$. The {\bf relative analytification} is the morphism of rigid spaces $X=X\times_{S_0} S\to S$. 
\end{defi}

In general, relative analytification preserves open immersions and open coverings (which allows us to globalise the above picture), proper maps \cite[Lemma~5.7.3]{HuberBook}, surjective maps \cite[Lemma~3.9]{Huber94:Generalization} etc.

\begin{prop} 
\label{prop:h-locally-regular}
    Let $X$ be a qcqs rigid space over a perfect non-archimedean field $K$. Then there exists a surjective map of rigid spaces $Y\to X$ with $Y$ qcqs and smooth over $K$.
\end{prop}

This proposition is essentially a result by Berkovich \cite[Theorem 1.3.1]{Berkovich:VanishingCyclesFormalSchemes} but in the adic setting, and in fact it follows from the Berkovich version.
Indeed, any qcqs adic space is taut and, by  \cite[Proposition 8.3.1+Remark 8.3.2]{HuberBook} (cf.\ also \cite[Theorem 1.6.1]{Berkovich:EtaleCohomology}), the category of qcqs rigid spaces over $K$ is equivalent to compact hausdorff strictly $K$-analytic (Berkovich) spaces. Let us denote the quasi-inverse functors realising this equivalence by $(-)^{\mathrm{Berk}}$ and $(-)^{\mathrm{rig}}$. By \cite[Theorem 1.3.1]{Berkovich:VanishingCyclesFormalSchemes}\footnote{The statement of \cite[Theorem 1.3.1]{Berkovich:VanishingCyclesFormalSchemes} gives more information, but rig-smoothness is obtained as a step in the proof.}, there is a ``rig-smooth'' (see \cite[p.~335]{Berkovich:SmoothLocallyContrII}) $Y_0$ with a compact map $Y_0 \to X^\mathrm{Berk}$. So, $Y_0$ is compact and there is a finite set of affinoid subdomains $\{U_i\}$ covering $Y_0$ such that each $U_i$ admits a quasi-\'etale map to $(\mathbb{A}^n_K)^{\rm Berk}$.
By 
\cite[Lemma 1.6.2]{Berkovich:EtaleCohomology} 
and
\cite[Proposition 8.3.1+Remark 8.3.2]{HuberBook} again,
the corresponding rigid space $Y = Y_0^{\mathrm{rig}}$ is qcqs and locally admits an \'etale map to $(\mathbb{A}^n_K)^{\rm an}$. By \cite[Corollary 1.6.10]{HuberBook}, $Y$ is smooth over $K$.

For convenience, let us sketch the central part of the proof of \cref{prop:h-locally-regular} directly in the adic case.

\begin{proof}
Replacing $X$ by the constituents of a finite affinoid open covering, we may assume that $X=\Spa(A)$ is affinoid with $A$ an affinoid $K$-algebra. Let $X_0=\Spec(A)$, which is an excellent scheme. By Gabber's uniformisation theorem \cite[Th\'eor\`eme~1.1]{TravauxGabberIX}, there exists a separated morphism of finite presentation $Y'_0\to X_0$ where $Y'_0$ is regular and which is an $h$-covering. By \cite[Corollary~10.4]{SuslinVoevodsky}, the $h$-covering admits a refinement of normal form, so there exists a proper surjection $X'_0\to X_0$, an open covering $X'_0 = \bigcup_{i=1}^r U_{0,i}$, and a commutative square
\[ 
    \begin{tikzcd}
        \coprod U_{0,i} \ar[d] \ar[r] & Y'_0 \ar[d] \\
        X'_0 \ar[r] & X_0.
    \end{tikzcd}
\]
Analytifying the above square in a relative sense, we obtain a commutative square of rigid spaces 
\[ 
    \begin{tikzcd}
        \coprod U_i \ar[d] \ar[r] & Y' \ar[d] \\
        X'\ar[r] & X
    \end{tikzcd}
\]
where $X'\to X$ is a proper  (cf.~\cite[Lemma~5.7.3]{HuberBook}) surjection and $\coprod U_i \to X'$ is a Zariski open covering, and where $Y'$ is regular (by which we mean that the local rings at classical points are regular).
Indeed, one checks (using e.g.\ \cite[Satz 2.1]{Kopf}) that relative analytification preserves completed local rings of classical points. Cf.\ also \cite[Lemma 2.6.3]{Berkovich:EtaleCohomology} for the Berkovich spaces incarnation of this result and \cite[I \S 4.1, Proposition 2]{Bosch} for a proof of the same statement when the target is $\Spa(K)$. Then one checks that regularity together with perfectness of $K$ implies smoothness (cf.\ \cite[\S 5]{Berkovich:VanishingCyclesAnalytic}).

Since $X'\to X$ is proper, the rigid space $X'$ is quasi-compact and separated. Thus, there exist quasi-compact opens $V_i\subseteq U_i$ such that $X'=\bigcup_{i=1}^r V_i$. To see this, take a refinement of $\coprod U_i \to X'$ by affinoids, hence quasi-compact opens, and let $V_i$ be the union of those opens inside $U_i$. The image of $\coprod _{i=1}^r V_i\to Y'$ is quasi-compact, and hence contained in a quasi-compact open $Y\subseteq Y'$, which is also quasi-separated because $Y'\to X$ is separated and $X$ is affinoid. The map $Y\to X$ is surjective by construction. 
\end{proof}

The following example explains a subtle point in the above proof. 

\begin{ex}[Punctured blowup]
Let $X = \Spa(K\langle x, y\rangle)$ be the affinoid bidisc, and let $Y'\to X$ be the blowup of the ideal $(x, y)$. Let $P \in Y'$ be a $K$-point in the exceptional divisor, and let $Y = Y'\setminus\{P\}$. Then the map $Y\to X$ is surjective, but no quasi-compact open $U\subseteq Y$ surjects onto $X$. Suppose otherwise, and consider in suitable coordinates the chart $\Spa(K\langle x,t\rangle)\subseteq Y'$ with $t=y/x$ in which $p$ has coordinates $(0,0)$. Then for $n\gg 0$, the open ball $B_n = \{|x|, |t|<1/n\}\subseteq Y'$ is disjoint from $U$, and therefore the points $(x,0) \in X$ with $0<|x|<1/n$ are not in the image of $U\to X$.  

Note that if $Y'_0\to X_0=\Spec(K\langle x, y\rangle)$ is the blowup of $(x,y)$, and $Y_0 = Y'_0\setminus\{P\}$, then $Y\to X$ is the analytification of $Y_0\to X_0$. However, $Y_0\to X_0$ is not an $h$-covering even though it is surjective. The refinement of normal form used in the proof of \cref{prop:h-locally-regular} is needed precisely to ensure that after taking analytifications, the map will remain surjective after passing to a suitable quasi-compact open.
\end{ex}

\subsection{Review of log formal schemes}
\label{ss:log-formal-schemes}

In this subsection, we discuss the log geometry of formal schemes needed for the treatment of semistable formal models of rigid spaces in \S\ref{ss:log-abhyankar}. See \cite[\S 7]{Fersi} for some foundations of log formal schemes (though in lesser generality than needed here).

As before, we fix a complete rank one valuation ring $K^\circ$ with fraction field $K$ and pick a pseudouniformiser $\pi\in K^\circ$. All formal schemes over $K^\circ$ are assumed to be $\pi$-adic, i.e.\ such that $\pi$ generates an ideal of definition. Logarithmic structures can be defined on any ringed topos, and by a {\bf log formal scheme} we mean a formal scheme $\mathfrak{X}$ endowed with a log structure 
\[
\alpha\colon\cM_\fX\to\cO_\fX.
\]

The data of a $\pi$-adic formal scheme $\fX$ is equivalent to the data of a sequence of schemes $\fX_n = \fX\otimes K^\circ/\pi^{n+1}$ over $K^\circ/\pi^{n+1}$ together with identifications
\[ 
    \fX_n \simeq \fX_{n+1} \otimes_{K^\circ/\pi^{n+2}} K^\circ/\pi^{n+1}.
\]
The following basic result extends this equivalence to log structures.

\begin{lem} \label{lem:log-fsch-equiv}
    Let $\cM_\fX$ be a log structure on a $\pi$-adic formal scheme $\fX$, and let $\cM_{\fX_n}$ be the induced log structure on the scheme $\fX_n$. Then
    \[
    \cM_{\fX_n} = \cM_\fX/U_n
    \quad \text{where} \quad
    U_n = \ker(\cO_{\fX}^\times \to \cO_{\fX_n}^\times) = 1+\pi^{n+1}\cO_\fX
    \]
    and
    \[ 
        \cM_\fX = \varprojlim \cM_{\fX_n} = \varprojlim \cM_\fX/U_n.
    \]
    Conversely, given a sequence of log structures $\cM_{\fX_n}$ on $\fX_n$ and identifications $\cM_{\fX_{n+1}}|_{\fX_{n}} \simeq \cM_{\fX_n}$, the above formula defines a log structure on $\fX$ whose restriction to $\fX_n$ equals $\cM_{\fX_n}$. 
\end{lem}

\begin{proof}
The preimage of $\cO_{\fX_n}^\times\subseteq\cO_{\fX_n}$ under $\alpha_n\colon \cM_{\fX}\to\cO_\fX\to\cO_{\fX_n}$ equals $\cO_\fX^\times$. Therefore
\[
    \cM_{\fX_n} = \cM_\fX \oplus_{\alpha_n^{-1}(\cO_{\fX_n}^\times)} \cO_{\fX_n}^\times = \cM_{\fX} \oplus_{\cO_{\fX}^\times} (\cO_{\fX}^\times/U_n) = \cM_\fX/U_n,
\]
showing the first assertion. The second follows from the proof of \cite[Proposition~7.8]{Fersi}. 
We omit the proof of the final assertion.
\end{proof}

We are primarily interested in log structures which locally admit a chart, in particular smooth log formal schemes over $K^\circ$. For a monoid $P\to A$ mapping into a $\pi$-adic ring $A$, we denote by $\Spf(P\to A)$ the log scheme $\Spf(A)$ endowed with the log structure charted by $P$. 

\begin{defi}
    \begin{enumerate}[(a)]
        \item A log formal scheme $\fX$ is {\bf quasi-coherent} (resp.\ {\bf integral}, {\bf saturated}, {\bf fine}, {\bf fs}) if it is \'etale locally of the form $\Spf(P\to A)$ for a monoid (resp.\ integral, saturated, fine, fs monoid) $P$ mapping into a $\pi$-adic ring $A$.
        \item (See \cite[Definition~3.5.1]{AHLS-Log} for the scheme case.) Let $f\colon \fY\to\fX$ be a morphism of saturated adic log formal schemes over $K^\circ$. We say that $f$ is {\bf smooth} (resp.\ {\bf \'etale}, {\bf Kummer \'etale}) if for every strict \'etale map $\fX' = \Spf(P\to A)\to\fX$ with a saturated monoid $P$, \'etale locally on $\fY' = \fY\times_\fX\fX'$ there exists a  smooth (resp.\ \'etale, resp.\ Kummer \'etale) morphism of monoids $P\to Q$ (with $\Sigma$ the set of primes non-invertible on $X$) such that $\fY'\to\fX'$ factors as
    \[
        \fY\la \Spf\left(Q\to (A\otimes_{\bZ[P]}\bZ[Q])^\wedge_\pi\right) \to \Spf(P\to A) = \fX'
    \]
    where the first map is strict and smooth (resp.\ \'etale, resp. \'etale). Here, $(-)^\wedge_\pi$ denotes the $\pi$-adic completion.  
    \end{enumerate}
\end{defi}

Arguing as in the scheme case \cite[\S 3.5]{AHLS-Log}, one shows these three classes of morphisms (smooth, \'etale and Kummer \'etale) are closed under composition and base change. We warn the reader, however, that, since smooth maps of monoids are not always finitely generated,  the underlying maps of formal schemes might not be locally topologically of finite type in general. Below we shall work exclusively with Kummer \'etale maps between smooth log formal schemes over $\Spf(K^\circ)^{\rm log}$ where $K$ is algebraically closed or discretely valued, in which case there is no problematic behaviour. 

\begin{defi} \label{def:std-log-str}
    The log structure ``induced by the generic fibre'' is the key example of a log structure on a formal scheme.
    Let $\fX$ be a $\pi$-adic formal scheme over $K^\circ$. 
    \begin{enumerate}
        \item The subsheaf
        \[ 
            \cM_\fX^{\rm std} = \cO_\fX \times_{\cO_\fX[1/\pi]} \cO_\fX[1/\pi]^\times \quad\subseteq\quad \cO_\fX
        \]
        is a log structure, called the {\bf standard log structure}.
        \item We denote by $\fX^{\rm log}$ the log formal scheme $(\fX, \cM_\fX^{\rm std})$, and by $\fX^{\rm log}_k$ its special fibre $\fX_k$ with the induced log structure $\cM_\fX^{\rm std}|_{\fX_k}$. 
        \item A log structure $\cM$ on $\fX$ is {\bf vertical} if the image of $\cM\to\cO_\fX$ lies in $\cM_\fX^{\rm std}$, i.e.\ if the image of every local section of $\cM$ in $\cO_\fX[1/\pi]$ is invertible.
    \end{enumerate}
\end{defi}

Intuitively, being vertical means that the induced log structure on the rigid generic fibre $\fX_{\rm rig}$ is trivial.  Thus, $\cM_\fX^{\rm std}$ is the maximal vertical log structure. An important special case is the log formal scheme $\Spf(K^\circ)^{\rm log}$, which admits a chart by the monoid 
    \[
        V = K^\circ \cap K^\times \la K^\circ,
    \]
    cf.\ \cite[Lemma~4.4.1]{AHLS-Log}. Note that $V$ is a valuative monoid, which is moreover divisible if $K$ is algebraically closed. Thus $\Spf(K^\circ)_k^{\rm log} = \Spec(V\to k)$ is of type $\typeVd$ if $K$ is algebraically closed or if $K$ is discretely valued.

\begin{lem} \label{lem:spf-Kcirc-satisfies364}
    The valuation map $\nu \colon K^0 \cap K^\times \to \Gamma_K^+$ admits a splitting $\sigma$, which induces an isomorphism of log formal schemes
    \[
    \Spf(K^\circ)^{\log} \isomto \Spf(\Gamma_K^+ \xrightarrow{\sigma} K^\circ).
    \]
\end{lem}
\begin{proof} 
    The exact sequence 
    \[
        \begin{tikzcd}
            1 \ar[r] & K^{\circ\times} \ar[r] & K^\times \ar[r,"\nu"] & \Gamma_K\ar[r] & 0.
        \end{tikzcd}
    \]
   splits because by assumption $\Gamma_K \simeq \bZ$ is free or $K^{\circ \times}$ is divisible. The restriction of a splitting induces a splitting $\sigma \colon \Gamma_K^+ \to K^\circ \cap K^\times$. 
\end{proof}

The most important special case of a log smooth formal scheme over $\Spf(K^\circ)^{\rm log}$ is $\fX^{\rm log}$ for a semistable formal scheme $\fX$ over $K^\circ$. Before we show this, let us recall the definition. 

\begin{defi} \label{def:ss-formal-sch}
    A formal scheme $\fX$ over $K^\circ$ is {\bf semistable} if \'etale locally on $\fX$ there exists a pseudouniformiser $\pi\in K^\circ$ and an \'etale map
    \[
        \fX \la \Spf\left( K^\circ\langle T_1,\dots, T_n\rangle/(\pi - T_1\cdot\ldots\cdot T_n) \right).
    \]
    We say that $\fX$ is {\bf strictly semistable} if such a map exists Zariski locally on $\fX$. 
\end{defi}

\begin{ex} \label{ex:standard-log-structure-Kcirc}
    Here is an example showing that our definition of a strictly semistable formal scheme covers the usual cases. For $n \geq s > 0$, the morphism
    \[
    \Spf(K^\circ\langle X_1,\ldots, X_s, X_{s+1}^\pm, \ldots, X_n^\pm\rangle/(\pi - X_1 \cdot \ldots \cdot X_s)) \to \Spf( K^\circ\langle T_1,\dots, T_n\rangle/(\pi - T_1\cdot\ldots\cdot T_n))
    \]
    sending $X_i \mapsto T_i$ for $i \not= s$ and $X_s \mapsto T_s \cdot(T_{s+1} \cdot \ldots \cdot T_n)$ is an open immersion. In particular, its source is strictly semistable. 

    Note that the case $s=1$ covers the formally smooth (good reduction) case as well.
\end{ex}

We define the standard semistable monoid in $n$ variables relative to a monoid $M$ and $\gamma \in M$ (cf.\ \cite[Example~2.4.5]{AHLS-Log}) as the monoid $P_n(\gamma)$ together with a smooth map 
\[
M \la P_n(\gamma) = M[e_1,\ldots,e_n]/(e_1 + \ldots + e_n = \gamma).
\]

\begin{lem} \label{lem:strictly-semistable-formal-scheme-reduces-to-setup364}
Let $M = \Gamma_K^+$ and choose a splitting $\sigma$ as in \cref{lem:spf-Kcirc-satisfies364}. 
Let $\fX$ be a (strictly) semistable formal scheme over $K^\circ$. Then the standard log structure on $\fX$ admits \'etale  locally (resp.\ Zariski locally), with a pseudouniformiser $\pi$ that satisfies 
$\sigma(\nu(\pi)) = \pi$, charts of the form 
\[
    \begin{tikzcd}
        \fX^{\log} \arrow[rr,"\text{strict \'etale}"] \arrow[d]
        && \Spf\left(P_n(\nu(\pi)) \xrightarrow{\alpha} K^\circ\langle T_1,\dots, T_n\rangle/(\pi - T_1\cdot\ldots\cdot T_n) \right) \arrow[d] \\
        \Spf(K^\circ)^{\log}  \arrow[rr,"\sim"] & &  
        \Spf(M \xrightarrow{\sigma} K^\circ) 
    \end{tikzcd}
\]
where $\alpha(e_i) = T_i$, and $\alpha$ agrees with $\sigma$ on $M = \Gamma_K^+$.
Consequently, in the strictly semistable case the special fibre $\fX^{\log}_k$ is described by \cref{setup:surgery}. 
\end{lem}

\begin{proof}
    By \cref{def:ss-formal-sch} it suffices to understand a global chart for the standard log structure on $\fX = \Spf\left( K^\circ\langle T_1,\dots, T_n\rangle/(\pi - T_1\cdot\ldots\cdot T_n) \right)$. For the chosen splitting $\sigma$ the pseudouniformiser $\pi$ in the equation defining $\fX$ might differ from $\sigma(\nu(\pi))$ by a unit in $K^\circ$. After scaling one of the coordinates appropriately, we may assume that indeed $\sigma(\nu(\pi)) = \pi$. The rest is obvious.     
\end{proof}

Let us recall Temkin's uniformisation result, to be used in the proof of \cref{ss:tamepi1-rig-fgfp}.

\begin{thm}[{Temkin \cite[Corollary~3.3.2]{Temkin2017:AlteredLocalUniformization}}] \label{thm:Temkin-unif}
    Let $X$ be a smooth qcqs rigid space over a non-archimedean field $K$. Then there exists a finite extension $L$ of $K$, an affine strictly semistable formal scheme $\fY$ over $L^\circ$, and a surjective \'etale map of rigid $L$-spaces $\fY_\rig\to X_L$.
\end{thm}

The following proposition lists some good properties of certain smooth log formal schemes over $\Spf(K^\circ)^{\rm log}$.

\begin{prop} \label{prop:log-sm-vertical}
    Suppose that $K$ is algebraically closed or discretely valued. Let $\fX$ be a log formal scheme over $\Spf(K^\circ)^{\rm log}$.
    \begin{enumerate}[(a)]
        \item \label{propitem:log-sm-vertical-a}
        If $\fX$ is smooth over $\Spf(K^\circ)^{\rm log}$, then its underlying formal scheme is normal and flat over $K^\circ$, in particular $\eta$-normal.
        
        \item \label{propitem:log-sm-vertical-b}
        If $\fX$ is smooth and vertical over $\Spf(K^\circ)^{\rm log}$, then $\cM_\fX$ coincides with the standard log structure $\cM^{\rm std}_\fX$.
        
        \item \label{propitem:log-sm-vertical-c}
        If the underlying formal scheme of $\fX$ is semistable over $K^\circ$ and $\cM_\fX$ is the standard log structure, then $\fX$ is smooth and vertical over $\Spf(K^\circ)^{\rm log}$. 
        
        \item \label{propitem:log-sm-vertical-d}
        If $\fX$ is smooth and vertical over $\Spf(K^\circ)^{\rm log}$ and $\fY\to \fX$ is Kummer \'etale, then $\fY$ is smooth and vertical over $\Spf(K^\circ)^{\rm log}$.
    \end{enumerate}
\end{prop}

\begin{rmk}
We expect that if $K$ is algebraically closed and $\fX$ is log smooth and vertical over $\Spf(K^\circ)^{\rm log}$, then there exists a log blow-up $\fX'\to\fX$ such that the underlying scheme of $\fX'$ is semistable and $\cM_{\fX'}$ is the standard log structure.
\end{rmk}

\begin{proof}
For log schemes over $K^\circ$, these assertions are proved in \cite[\S5]{AHLS-Log}. Here we show how to deduce from them the log formal scheme versions.

\ref{propitem:log-sm-vertical-a}
Working locally, we may assume that $\fX$ is the $\pi$-adic completion of an affine charted smooth log scheme $Y = \Spec(B)$ over $\Spec(V\to K^\circ)$, so that $\fX = \Spf(\widehat{B})$ where $\widehat{B}$ is the $\pi$-adic completion of $B$. By \cite[Lemma 5.3.4]{AHLS-Log}, log smoothness implies that the ring $B$ is normal and flat over $K^\circ$. 
By \cite[Lemma B.6 + Remark B.7]{GabberZavyalov}, the map $B \to \widehat{B}$ is ind-smooth. In particular, it is flat. Moreover, normality persists under smooth maps (see \stacks[Lemma]{033C}) and cofiltered limits (see \stacks[Lemma]{037D}), so we get that $\widehat{B}$ is normal and flat over $K^\circ$, as desired.

\ref{propitem:log-sm-vertical-b} Working locally we may assume that there exists a smooth homomorphism $M\to P$ and a strict \'etale map $X\to \Spf(P\to \widehat{B})$ where $\widehat{B}$ is the $\pi$-adic completion of $B = K^\circ\otimes_{\bZ[M]}\bZ[P]$. We may replace $X$ with the target of this map, in which case $X$ is the $\pi$-adic completion of the smooth log scheme $X_0 = \Spec(P\to B)$. Note that under our assumptions on $K$, the underlying scheme of $X_0$ is of finite presentation over $K^\circ$. 
Moreover, for every strict \'etale $U\to X$ there exists a strict \'etale $U_0\to X_0$ whose $\pi$-adic completion is $U\to X$ (see \cite[1.7.3]{HuberBook}). This enables us to work \'etale locally on $X$. 

By the scheme-theoretic result \cite[Proposition~5.3.1]{AHLS-Log}, we know that every $f\in B$ which satisfies $fg = \pi^n$ is locally of the form $\alpha(p)u$ for $p\in P$ and $u\in B^\times$, where $\alpha\colon P\to B$ is the given map. We want to prove the analogous result for $\widehat{B}$. Let $f\in \widehat{B}$ satisfy $fg = \pi^n$ for some $g\in \widehat{B}$ and $n\geq 0$ and let $J = f\widehat{B}$. Since $\pi^n\in J$ and $\widehat{B}/\pi^n = B/\pi^n$, there exists an ideal $J_0\subseteq B$ such that $J_0\cdot \widehat{B} = J$. Since $B\to B[1/\pi]\times\widehat{B}$ is faithfully flat (by  \cite[Chapter~0, Proposition~8.2.17]{FujiwaraKato}), and since $J_0\cdot B[1/\pi] = B[1/\pi]$ and $J_0\cdot \widehat{B} = J = (f)$ are both invertible ideals, by flat descent for line bundles we see that $J_0$ is an invertible ideal. Localising on $X$ and $Y$, we may assume that $J_0 = (f_0)$ is principal. Since $J_0$ contains $\pi^n$, we can write $\pi^n = f_0g_0$, and hence (shrinking $X$ further) $f_0 = \alpha(p)u_0$ for some $p\in P$ and $u_0\in B^\times$. Since $f \widehat{B} = f_0\widehat{B}$, there exists a unit $u_1\in \widehat{B}^\times$ such that $f=u_1 f_0$. Then $f = \alpha(p)u$ where $u=u_0u_1\in \widehat{B}^\times$. 

\ref{propitem:log-sm-vertical-c} This can be deduced directly from the scheme version, \cite[Corollary~5.3.7]{AHLS-Log}.

\ref{propitem:log-sm-vertical-d} Since Kummer \'etale maps are smooth and smooth maps are closed under composition, we have that $\fY$ is smooth. It remains to show $\fY$ is vertical. We may assume that $\fY = \fX\times_{\Spf(P\to \bZ[P])}\Spf(Q\to\bZ[Q])$ for a Kummer \'etale map $P\to Q$. We must show that the image $g$ of every $q\in Q$ in $\cO_\fY[1/\pi]$ is invertible. Since $P\to Q$ is Kummer \'etale, it is injective and we have $nq=p\in P$ for some $n\geq 1$ invertible in $K^\circ$. Since $\fX$ is vertical, the image $f$ of $p$ in $\cO_\fX[1/\pi]$ is invertible. But $g^n=f$, and it follows that $g$ is invertible as well.
\end{proof}

\begin{prop} \label{prop:FEt-formal-vs-sp-fibre}
    Let $\fX$ be a saturated log formal scheme over $K^\circ$. Then the inclusion $\fX_k\to \fX$ induces an equivalence
    \[
        \FEt_\fX \isomto \FEt_{\fX_k}.
    \]
\end{prop}

\begin{proof}
Let $\fX_n$ be the reduction of $\fX$ modulo $\pi^{n+1}$. It is clear from the definition that if $\fY\to\fX$ is finite Kummer \'etale, then so are $\fY_n\to\fX_n$ and $\fY_k\to \fX_k$, so that our functor is well-defined. By \cref{prop:FKEt-top-inv}, the restriction functors $\FEt_{\fX_n}\to\FEt_{\fX_k}$ are equivalences. On the other hand, by \cref{lem:log-fsch-equiv}, giving a map of formal log schemes $\fY\to\fX$ is the same as giving a compatible family of maps $\fY_n\to \fX_n$. This implies that the functor is fully faithful, and hence it remains to prove that it is essentially surjective \'etale locally on $\fX$. We may thus assume that $\fX=\Spf(P\to A)$ for a saturated monoid $P$. Let $Y\to \fX_k$ be an object of $\FEt_{\fX_k}$. By the local structure in \cref{lem:fKet-local-structure} we may assume that there exists a finite collection of Kummer \'etale monoid maps $P\to Q_j$, $j=1,\dots, r$, such that $Y = \coprod_{j=1}^r \Spec(Q_j \to (A\otimes k)\otimes_{\bZ[P]}\bZ[Q_j])$. We may then define 
\[
    \fY = \coprod_{j=1}^r \Spf(Q_j \to (A\otimes_{\bZ[P]}\bZ[Q_j])^\wedge_\pi) \la \Spf(P\to A) = \fX
\]
which is an object of $\FEt_\fX$ whose restriction to $\fX_k$ equals $Y\to\fX_k$.
\end{proof}

\subsection{Tame covers and Kummer covers (I): purity}
\label{ss:log-abhyankar}

In this subsection, $K$ will be either algebraically closed  or 
    discretely valued.
Our goal is to prove that for a semistable formal scheme $\fX$ over $K^\circ$, a finite \'etale cover $Y\to X$ of its generic fibre $X=\fX_{\rig}$ is tame (in the absolute sense, i.e.\ relative to $X$) if and only if it extends to a Kummer \'etale cover of $\fX^{\rm log}$ ($\fX$ endowed with the standard log structure). In order to do this, we shall first establish a version for schemes over~$K^\circ$, which in turn will ultimately rely on the classical log purity (a.k.a.\ log Abhyankar's lemma) recalled below. This is particularly delicate in the non-noetherian setting (i.e.\ if $K$ is algebraically closed), as we need to perform a rather intricate approximation argument to be able to apply classical log purity. Let us mention here that the case of $\Spec(K^+)$ for an arbitrary valuation ring $K^+$ has been handled in \cite[Corollary~4.4.9]{AHLS-Log}, identifying the tame Galois group of $K$ with the Kummer \'etale fundamental group of $\Spec((K^+\setminus 0)\to K^+)$. 

To state log purity, we need the following terminology. If $U\to X$ is an \'etale morphism of schemes and $\zeta\in X$ is a point such that $\cO_{X,\zeta}$ is a valuation ring, we say that $U\to X$ is {\bf tame above $\zeta$} if the map $\Spa(U, X)\to \Spa(X, X)$ is tame at every $u\in \Spa(U, X)$ mapping to the point $x\in \Spa(X, X)$ corresponding to the valuation induced by $\zeta$.
Explicitly, we are testing tameness of $\Spa(U,X) \to \Spa(X,X)$ at all preimages of the root triple $(\eta, \cO_{X,\zeta},\zeta)$ of $\Spa(X,X)$ where $\eta\in X$ is the image of the generic point of $\cO_{X,\zeta}$.

A natural setting where tameness above~$\zeta$ comes up is when $X$ is the underlying scheme of a regular log scheme and $U \to X$ factors through a finite étale morphism $U \to X^*$ (the locus of trivial log structure).
Since~$X$ is normal by \cite[Theorem~4.1]{KatoToricSingularities}, the local rings at the generic points of $X \setminus X^*$ are valuation rings and it makes sense to ask whether $U \to X$ is tame above these points.

\begin{thm}[{Classical log purity, \cite[Theorem~13.3.43]{GabberRameroFoundations}, \cite[Theorem~B]{Mochizuki}}] 
\label{thm:GR-log-Abhyankar}
    Let $X$ be a regular log scheme, and 
    let $X^* \subseteq X$ be the open subset where the log structure is trivial. 
    
    Then the restriction functor $\FEt_X\to \FEt_{X^*}$ induces an equivalence with the full subcategory of $\FEt_{X^*}$ consisting of \'etale covers that are tame above every generic point of $X\setminus X^*$.
\end{thm}

Our first result is an extension of the above theorem to semistable log schemes over $K^\circ$, i.e.\ those which \'etale locally admit an \'etale morphism to $\Spec(K^\circ[T_1,\dots,T_n]/(\pi- T_1\cdot\ldots\cdot T_s))$ for some $n \geq s \geq 0$ and some pseudouniformiser $\pi\in K^\circ$. To this end, we first need to check that the local rings of the generic points of the special fibre are valuation rings. This follows from the lemma below, in which we use the following terminology analogous to \cref{defi:eta-normal}: a scheme $X$ locally of finite type over a valuation ring $K^+$ is {\bf $\eta$-normal} if it is locally of the form $\Spec(A)$ where $A$ is a finitely generated and flat $K^+$-algebra which is integrally closed in $A\otimes_{K^+} K$ where $K$ is the field of fractions of $K^+$. We note that being $\eta$-normal is an \'etale local property thanks to \stacks[Lemma]{03GG}.

\begin{lem} \label{lem:local-ring-val-ring}
    Let $(K,K^+)$ be a valued field with residue field $k = K^\succ$ and pseudouniformiser~$\pi$.
    Let $X$ be a reduced scheme that is of finite type, flat and $\eta$-normal over~$K^+$.
    Then for every generic point~$\beta$ of the special fibre~$X_k$, the local ring~$\cO_{X,\beta}$.
    Moreover, the map $K^+\to \cO_{X,\beta}$ is an extension of valuation rings with finite ramification index (see \stacks[Definition]{0ASG}). 
\end{lem}

\begin{proof}
We discuss two special cases and deduce from them the general statement.
First, we treat the case where $X = \bA^n_{K^+} = \Spec (K^+[T_1,\ldots,T_n])$ and $\beta$ is the unique generic point of the special fibre corresponding to the prime ideal $\mathfrak{m}_{K^+}\cdot K^+[T_1,\ldots,T_n]$.
The field of fractions of $\cO_{X,\beta}$ is $K(T_1, \dots, T_n)$, and a fraction $f/g$ with $f,g\in K^+[T_1,\dots, T_n]$ belongs to $\cO_{X,\beta}$ if and only if we can find a representation with $g\notin \mathfrak{m}_{K^+} \cdot K^+[X_1, \ldots, X_n]$.
Since every element of $K^+[T_1,\ldots,T_n]$ can be written as a product $x h$ where $x\in K^+$ and $h\notin \mathfrak{m}_{K^+}[T_1,\ldots,T_n]$ (thus $h\in \cO_{X,\beta}^\times$), it follows that $\cO_{X,\beta}$ is a valuation ring and that  $K^+\to \cO_{X,\beta}$ is an extension of valuation rings with valuation index one. 

Second, we treat the case where~$K^+$ is henselian and~$X$ is finite over~$K^+$.
Then $\cO_{X,\beta}$ is reduced and finite flat over~$K^+$.
In this situation $\eta$-normality implies normality, so $\cO_{X,\beta}$ is normal and this already implies that it is a valuation ring.
By the finiteness of~$\cO_{X,\beta}$ over~$K^+$, it follows that $K^+\to \cO_{X,\beta}$ is an extension of valuation rings with finite valuation index.

For the general case, since the assertion is local on $X$, we may assume that $X = \Spec(A)$ is affine and $X_k$ is irreducible, so $\beta$ is the unique generic point of $X_k$.
By the Noether normalisation lemma, there exist $f_1,\dots, f_n\in A$ such that the map $f=(f_1,\dots, f_n)\colon X\to \bA^n_{K^+}$ induces a finite map $f_k\colon X_k\to \bA^n_k$ which sends $\beta$ to the generic point $\gamma$.
Let
\[
    f_\beta \colon X_\beta^h \la \bA_{K^+,\gamma}^{n,h}
\]
be the induced map of henselisations.
In the first paragraph we have seen that the localisation $\bA_{K^+,\gamma}^n$ -- and hence also its henselisation -- is the spectrum of a valuation ring extending $K^+$ with finite ramification index.
The second paragraph deduces that the same holds for~$X_\beta^h$.
Since a local ring is a valuation ring if and only if its henselisation is, we conclude that $\cO_{X,\beta}$ is a valuation ring extending $K^+$ with finite ramification index.
\end{proof}

Henceforth in this section we use the following notation. If $Z$ is a scheme over $K^\circ$, we denote by $Z^*$ its generic fibre $Z\otimes_{K^\circ} K$. We use the natural log scheme variants of \cref{def:std-log-str}.  Our version of \cref{thm:GR-log-Abhyankar} over $K^\circ$ is the following.    

\begin{thm}[Purity, scheme version] \label{thm:scheme-Abhyankar}
  Let $X$ be a semistable scheme over $K^\circ$ and let $Y^*\to X^* = X_K$ be a finite \'etale map.  Assume that $K$ is either algebraically closed or discretely valued. 
 Let $Y\to X$ be the normalisation of $X$ in $Y^*$. The following are equivalent:
  \begin{enumerate}[(a)]
    \item \label{propitem:Abhy-tame} 
        $Y^*\to X^*$ is tame relative to $X$,
    \item \label{propitem:Abhy-tame-boundary}
        $Y^*\to X^*$ is tame above every generic point of the special fibre $X_k$,
    \item \label{propitem:Abhy-ket}
        the morphism $Y^{\rm log}\to X^{\rm log}$ is Kummer \'etale.
  \end{enumerate}
\end{thm}

\noindent\emph{Proof (part one).} 
\label{proof:scheme-Abhyankar}
Assertion \labelcref{propitem:Abhy-tame} obviously implies assertion \labelcref{propitem:Abhy-tame-boundary}. Assume \labelcref{propitem:Abhy-ket} and show that it implies \labelcref{propitem:Abhy-tame}. For any $x \in \Spa(X^*,X)$ we endow $S = \Spec(k(x)^+)$ with the standard log structure. The resulting  map of log schemes $\tilde{x} = \Spec(k(x)^+ \setminus 0 \to k(x)^+) \to X$ yields a pull-back of $Y \to X$ to $\tilde{x}$ that is a Kummer \'etale cover of $\tilde{x}$.
By  \cite[Corollary 4.4.9]{AHLS-Log} this corresponds to a tamely ramified extension of $k(x)$, and this shows \labelcref{propitem:Abhy-tame}.

For \labelcref{propitem:Abhy-tame-boundary}$\Rightarrow$\labelcref{propitem:Abhy-ket}, suppose first that $K$ is discretely valued. In this case, $\Spec(K^\circ)$ is log regular for the standard log structure. By \cite[Corollary~5.3.7]{AHLS-Log} (the scheme version of  \cref{prop:log-sm-vertical}(c)) the map $X\to \Spec(K^\circ)$ is log smooth for the standard log structures on source and target. Therefore by \cite[Theorem~8.2]{KatoToricSingularities} the log scheme $X$ is log regular as well. By log purity (\cref{thm:GR-log-Abhyankar}) $Y^*\to X^*$ extends to a Kummer \'etale $Y'\to X$. It remains to check that $Y'=Y$, for which we note that the log structure on $Y'$ is also the standard one, and again $Y'$ is log regular (being smooth over $X$). Thus in particular $Y'$ is normal \cite[Theorem~4.1]{KatoToricSingularities}, and hence $Y=Y'$.

It remains to show \labelcref{propitem:Abhy-tame-boundary}$\Rightarrow$\labelcref{propitem:Abhy-ket} for $K$ algebraically closed. In this case, log purity is not readily available, but we shall ultimately be able to apply it using approximation techniques. 

\hphantom{a} \hfill \emph{...the proof will resume after \cref{lem:generic-points-ugggh}. $\square$}  

\medskip

The approximation part of the proof of \cref{thm:scheme-Abhyankar} will rely on the following lemma. 

\begin{lem}[Zavyalov] \label{lem:Zavyalov}
    Let $K$ be an algebraically closed non-archimedean field, and let $\pi\in K$ be a pseudouniformiser. Then $K^\circ$ is isomorphic to the filtered colimit $\varinjlim R_\lambda$ of subalgebras $R_\lambda\subseteq K^\circ$ such that
    \begin{enumerate}[(i)]
        \item $\pi \in R_\lambda$;
        \item $R_\lambda$ is a strictly henselian local ring and $R_\lambda\to K^\circ$ is a local homomorphism;
        \item $R_\lambda$ is excellent and regular;
        \item the divisor $V(\pi)_{\rm red}\subseteq \Spec(R_\lambda)$ has strict simple normal crossings.
    \end{enumerate}
\end{lem}

\begin{proof}
We apply \cite[Lemma~A.2]{Zavyalov} and replace the resulting subrings of $K^\circ$ with their respective strict henselisations with respect to the map to the residue field $k$.
\end{proof}

Let us henceforth fix a pseudouniformiser $\pi$ of $K^\circ$. We will consider schemes $Z$ defined over subalgebras $R\subseteq K^\circ$ containing $\pi$. In this case, we shall extend our previous notation and denote by $Z^*$ the open subset $D(\pi)\subseteq Z$ (this agrees with the previous use since $K= K^\circ[1/\pi]$).

\begin{lem}  
\label{lem:straightforward-approx}
    Let $K$ be an algebraically closed non-archimedean field. Let $X=\Spec(A)$ be an affine $K^\circ$-scheme equipped with an \'etale map
    \[ 
        X\longrightarrow W = \Spec(K^\circ[T_1, \dots, T_n]/(\pi- T_1\cdot\ldots\cdot T_s)).
    \]
    Let $X^* = X_K = \Spec(A[1/\pi])$ be its generic fibre and let $Y^*=\Spec(B)\to X^*$ be a finite \'etale map. Let $\{R_\lambda; \lambda \in I\}$ be a cofiltering system of subalgebras of $K^\circ$ containing $\pi$ such that $K^\circ=\varinjlim R_\lambda$, 
    and let
    \[
        W_\lambda = \Spec(R_\lambda[T_1, \dots, T_n]/(\pi- T_1\cdot\ldots\cdot T_s)).
    \]
    Then for some $\lambda\in I$ there exist an \'etale map $X_\lambda=\Spec(A_\lambda)\to W_\lambda$ and a finite \'etale map $Y_\lambda=\Spec(B_\lambda)\to \Spec(A_\lambda[1/\pi]) = X^*_\lambda$ whose base changes to $K^\circ$ are $X\to W$ and $Y^*\to X^*$, respectively.
\end{lem}

\begin{proof}
Since \'etale morphisms of affine schemes are of finite presentation, this follows from \cite[8.8.2 and 8.10.5(viii,x)]{EGAIV3} and \stacks[Lemma]{07RP}.
\end{proof}

We now fix a cofiltering system $\{R_\lambda\}$ as in \cref{lem:Zavyalov} and assume that there is an initial index $\lambda_0$ for which the claim of \cref{lem:straightforward-approx} holds. Then we define $X_\lambda = \Spec(A_\lambda) \to W_\lambda$ and $Y_\lambda = \Spec(B_\lambda) \to X_\lambda^*$ by base change from the objects for the initial $\lambda_0$. 

\begin{lem}
\label{lem:locus of tameness}
     Assume that $K$ is algebraically closed. 
Let $\beta_1, \ldots, \beta_m \in X$ be such that $\cO_{X,\beta_i}$ is a valuation ring for $i =1, \ldots, m$. Let $Y \to X$ be the normalisation of a finite \'etale map $Y^* \to X^*$ that is tame at $\beta_1, \ldots, \beta_m$. Then there exist a $\lambda$ and an open $U \subseteq X_\lambda$ containing the images of $\beta_1, \ldots, \beta_m$ and such that, for every $\zeta \in \Spa(X_\lambda^*, X_\lambda)$ centred in $U$ (i.e. $\zeta \in \Spa(U^*,U)$), the cover $Y_\lambda \to X_\lambda$ is tame at $\zeta$. 
\end{lem}

\begin{proof}
Applying \cref{lem:tameness-conditions}~\ref{lemitem:tame-abhyankar} to the valuation rings $\cO_{X,\beta_i}$ and spreading out, we see that for each $i=1,\dots, m$ there exist
\begin{itemize}
    \item 
    an affine scheme $U_i$ with an \'etale map $U_i\to X$ whose image contains $\beta_i$,
    \item 
    elements $g_1,\dots, g_r\in \cO(U_i)$ whose images in $\cO(U_i^*)$ are invertible,
    \item 
    an integer $n\geq 1$ which is invertible on $X$;
    \item 
    a finite set $J_i$, and  
    \item 
    a surjective map of \'etale covers of $U_i^*$
    \[
        V_i = \coprod_{j \in J_i} 
        \Spec(\cO(U_i)[1/\pi][x_1,\dots, x_r]/(x_1^{n} - g_1, \dots, x_r^{n} - g_r)) \la Y^*\times_{X^*} U^*_i. 
    \]
\end{itemize}
Picking $\lambda$ large enough, we may assume that the $U_i\to X$, the functions $g_j\in \cO(U_i)$, and the map $V_i\to Y^*\times_{X^*}U_i^*$ are respectively the base changes to $X$ of \'etale maps $U_{i,\lambda}\to X_\lambda$, functions $g_j\in \cO(U_{i,\lambda})$ whose images in $\cO(U_{i,\lambda}^*)$ are invertible, and a surjection of \'etale covers of $U_{i,\lambda}^*$ from $V_{i,\lambda}$ (defined by the same formula over $U_{i,\lambda}$) onto the base change of $Y^*_\lambda$. It now suffices to take $U \subseteq X_\lambda$ to be the union of the images of the \'etale maps $U_{i,\lambda}\to X_\lambda$. Indeed, for every map $\Spec(V)\to U$ from a valuation ring $V$ which sends the generic point to $U^*$, the base change of $Y^*_\lambda \to X^*_\lambda$ to ${\rm Frac}(V)$ will be tamely ramified by the characterisation of  \cref{lem:tameness-conditions}~\ref{lemitem:tame-abhyankar}.
\end{proof}

\begin{lem} \label{lem:generic-points-ugggh}
    In the situation fixed after  \cref{lem:straightforward-approx}, let $\bar x\to X$ be a geometric point. Then for every $\lambda$, every open subset of the strict henselisation $X_{\lambda,(\bar x)}$ containing the images of the generic points of $X_{(\bar x)}\setminus (X_{(\bar x)})^*$ contains all generic points of $X_{\lambda,(\bar x)} \setminus (X_{\lambda,(\bar x)})^*$.
\end{lem}

\begin{proof}
Since $X\to W$ is \'etale, the induced map $X_{(\bar x)}\to W_{(\bar x)}$ is an isomorphism. We may therefore assume $X=W$ and $X_\lambda=W_\lambda$. Let us find the generic points of $W\setminus W^*$. Since $R_\lambda$ is a regular local ring in which $V(\pi)_{\rm red}$ is an snc divisor, there exist local parameters $x_1,\dots, x_d$ (where $d=\dim(R_\lambda)$), integers $a_1,\dots,a_r\geq 1$ with $1 \leq r\leq d$ and a unit $u \in R_{\lambda}^\times$ such that
\[
    \pi=u \cdot x_1^{a_1}\cdot\ldots\cdot x^{a_r}_r. 
\]
In this case, $W\setminus W^*=V(\pi)$ is the underlying topological space of the spectrum of 
\[
    R_\lambda[T_1,\dots, T_n]/(\pi - T_1\cdot\ldots\cdot T_s,\pi) = R_\lambda[T_1,\dots, T_n]/(x_1^{a_1}\cdot\ldots\cdot x^{a_r}_r, T_1\cdot\ldots\cdot T_s),
\]
and the minimal primes of this ring 
are of the form 
\[
\mathfrak{p}_{ij} = (x_i, T_j), \qquad i \leq r, \ j \leq s
\]
(if $s=0$, then this simplifies to $\mathfrak{p}_i=(x_i)$). The minimal primes of the ring 
\[
K^\circ[T_1,\dots, T_n]/(\pi, T_1\cdot\ldots\cdot T_s)
\]
are of the form 
\[
\mathfrak{q}_j = \sqrt{(\pi, T_j)},  \qquad j \leq s
\]
(if $s=0$ then this is just $\mathfrak{q} = \sqrt{(\pi)}$). Note that $x_i$ is contained in the maximal ideal of $K^\circ$ (since $R_\lambda \to K^\circ$ is local) that can be described as $\sqrt{(\pi)}$. We see that $\mathfrak{p}_{ij}\subseteq \mathfrak{q}_j$, resp.~$\mathfrak{p}_i \subseteq \mathfrak{q}$ if $s=0$. In other words, every generic point of $W_\lambda\setminus W^*_\lambda$ specialises to the image of a generic point of $W\setminus W^*$, and thus every open subset of $W_\lambda$ containing the images of the generic points of $W\setminus W^*$ contains all generic points of $W_\lambda\setminus W^*_\lambda$.

The computation shows that the irreducible components $V(\mathfrak{p}_{ij})$, resp.~$V(\mathfrak{p}_i)$ if $s=0$, of $W_\lambda\setminus W^*_\lambda$ are regular and in particular geometrically unibranch. It follows that the irreducible components of $W_{\lambda,(\bar x)} \setminus (W_{\lambda,(\bar x)})^*$ are in bijection with the irreducible components of $W_\lambda\setminus W^*_\lambda$ which pass through the image of $\bar x$ in $W_\lambda$, i.e.\ with those $\mathfrak{p}_{ij}$ such that $x_j(\bar x)=0=T_i(\bar x)$. Analogously, the irreducible components of $W_{(\bar x)}\setminus W^*_{(\bar x)}$ correspond to $\mathfrak{q}_i$ such that $T_i(\bar x) = 0$. Since $R_\lambda\to K^\circ$ is local and $\pi(\bar x)=0$, we have $x_j(\bar x)=0$ for all $j\leq r$. (The argument in the case $s=0$ is similar.) Again, it follows from this enumeration that every generic point $\zeta_{ij}$ of $W_{\lambda,(\bar x)} \setminus (W_{\lambda,(\bar x)})^*$ 
specialises to the image of a generic point $\beta_i$ of $W_{(\bar x)}\setminus W^*_{(\bar x)}$, and the result follows.
\end{proof}

We can now finish the proof of \cref{thm:scheme-Abhyankar} started earlier on page~\pageref{proof:scheme-Abhyankar}.

\begin{proof}[Proof of \cref{thm:scheme-Abhyankar} (part two)]
Assume that $K$ is algebraically closed and \ref{propitem:Abhy-tame-boundary} holds; we wish to show \ref{propitem:Abhy-ket} holds. Since the question is \'etale local on $X$, we pick a geometric point $\bar x\to X$ and show that it holds after base change to the strict henselisation  $X_{(\bar x)}$. We may assume that $X$ is affine and that there exists a pseudouniformiser $\pi$ and an \'etale map 
\[
    X\la W=\Spec(K^\circ[T_1,\dots, T_n]/(\pi-T_1\cdot\ldots\cdot T_s)).
\]
Apply \cref{lem:Zavyalov} and \cref{lem:straightforward-approx} to obtain the systems $Y^*_\lambda\to X_\lambda^*\to X_\lambda\to W_\lambda$ over $S_\lambda=\Spec(R_\lambda)$, as in the situation fixed after \cref{lem:straightforward-approx}. Apply \cref{lem:locus of tameness} with $\beta_i$ corresponding to the components of $V(\pi)$ to obtain a specific $\lambda$ and the open subset $U$, and then apply \cref{lem:generic-points-ugggh} to the preimage of $U$ in $X_{\lambda,(\bar x)}$. We deduce that the base change of $Y_\lambda\to X_\lambda$ to $(X_\lambda)_{(\bar x)}$ is tame at the generic points of $X_{\lambda,(\bar x)} \setminus (X_{\lambda,(\bar x)})^*$. Since $X_{\lambda,(\bar x)}$ is log regular with respect to the log structure induced by the open subset $(X_{\lambda,(\bar x)})^*$, we deduce from the classical purity \cref{thm:GR-log-Abhyankar} that the normalisation $Y_\lambda\to X_\lambda$ is Kummer \'etale at $\bar x\to X_\lambda$ (where we endow every scheme $Z\in \{S_\lambda, W_\lambda, X_\lambda, Y_\lambda \dots\}$ with the log structure induced by the open subset $Z^* = D(\pi)$). 

Let $Y'\to X$ be the map of log schemes obtained by base change of $Y_\lambda\to X_\lambda$ along the map of log schemes $X\to X_\lambda$. The map $Y'\to X$ is Kummer \'etale, and hence $Y'$ is smooth over $S$. It follows from \cite[5.3.1, 5.3.4, and 5.3.7]{AHLS-Log} that $Y'$ is $\eta$-normal, that $Y'\to X$ is finite, and that the log structure on $Y'$ is the one induced by the open subset $D(\pi)\subseteq Y'$. But then $Y'=Y$, and the result is proved.
\end{proof}

In order to deduce a version of purity for semistable formal schemes, we need an algebraisation-type result. 

\begin{lem}[{\cite[Corollaire, p.~578 + Remarque~2(c), p.~588]{Elkik}}] \label{Elkik-theorem} 
    Let $A$ be a ring and $\pi\in A$ a nonzerodivisor such that $(A, (\pi))$ is a henselian pair. Let $\widehat{A}$ be the $\pi$-adic completion of $A$. Then the morphism $A[1/\pi]\to \widehat{A}[1/\pi]$ induces an equivalence between finite \'etale $A[1/\pi]$-algebras and finite \'etale $\widehat{A}[1/\pi]$-algebras.
\end{lem}

\begin{cor} \label{finite-etale-Huber-pair}
    Let $(A,A^+)$ be a uniform Tate Huber pair, and let $\pi \in A^+$ be a topologically nilpotent unit.
    For every finite \'etale morphism $Y \to \Spa(A,A^+)$, we can find a homomorphism of Huber pairs
    $(A,A^+) \la (A_1,A_1^+)$
    such that 
    \begin{enumerate}[(i)]
    \item 
    $A^+ \to A_1^+$ is \'etale and $A_1 = A_1^+ \otimes_{A^+} A$, 
    \item 
    the induced map $A^+/\pi A^+\isomto A_1^+/\pi A_1^+$ is  an isomorphism, and 
    \item 
    there is a finite \'etale homomorphism of Huber pairs $(A_1,A_1^+) \la (B_1,B_1^+)$ and a cartesian square
    \[ 
    \begin{tikzcd}
        \Spa(B_1,B_1^+) \ar[d] \ar[r,"\sim"] & Y \ar[d] \\
        \Spa(A_1,A_1^+) \ar[r,"\sim"] & \Spa(A,A^+) . 
    \end{tikzcd}
    \]
    (Note that by (i) and (ii) the bottom map and therefore also the top map are isomorphisms.)
    \end{enumerate}
\end{cor}

\begin{proof}
    Combining \cite[1.4.4]{HuberBook} with \cite[Corollary~1.7.3(iii)]{HuberBook}, we know that finite étale morphisms to $\Spa(A,A^+)$ are in one to one correspondence with finite étale homomorphisms
    \[
     (\widehat{A},\widehat{A}^+) \longrightarrow (B,B^+)
    \]
    of Huber pairs.
    So $Y \to \Spa(A,A^+)$ comes from a Huber pair $(B_3,B_3^+)$ over $(\widehat{A},\widehat{A}^+)$ such that $\widehat{A} \to B_3$ is finite étale and~$B_3$ is the integral closure of~$\widehat{A}^+$ in~$B_3^+$.
    We consider the henselisation $(A^+_h,\pi)$ of the pair $(A^+,\pi)$ and set $A_h=A^+_h[1/\pi]$.
    We thus have homomorphisms of Huber pairs
    \[
     (A,A^+) \longrightarrow (A_h,A^+_h) \longrightarrow (\widehat{A},\widehat{A}^+) 
    \]
    and both of them induce isomorphisms on completions, hence on adic spectra.
    Moreover, this completion is just the $\pi$-adic completion of the ring of integral elements (here we use the uniformity assumption).
    By the generalised version of Elkik's theorem (\cref{Elkik-theorem}) the finite \'etale homomorphism
    \[
     \widehat{A} \longrightarrow B_3
    \]
    is the base change of a unique finite \'etale homomorphism
    \[
     A_h \longrightarrow B_2.
    \]
    Now we can write $A^+_h$ as a colimit of étale homomorphisms
    \[
     A^+ \longrightarrow A_1^+
    \]
    inducing an isomorphism modulo~$\pi$.
    Inverting~$\pi$ we obtain a presentation of~$A_h$ as a colimit of \'etale $A$-algebras~$A_1$.
    The finite \'etale map $A_h \to B_2$ comes via base change from some finite \'etale map
    \[
     A_1 \longrightarrow B_1.
    \]
    Taking~$B_1^+$ to be the integral closure of~$A_1^+$ in~$B_1$, we obtain the desired finite \'etale homomorphism
    \[
     (A_1,A_1^+) \longrightarrow (B_1,B_1^+). \qedhere
    \]
\end{proof}

In order to state the version of \cref{thm:scheme-Abhyankar} for semistable formal schemes, we first recall the construction of the normalisation of a formal scheme $\fX$ in a finite cover $Y$ of its generic fibre, see \cite[Lemma~4.1]{AchingerLaraYoucis2022}. Let $\fX$ be an $\eta$-normal (see \cref{defi:eta-normal}) tft formal scheme over $K^\circ$ with generic fibre $X = \fX_{\rm rig}$, and let $f\colon Y\to X$ be a finite morphism of rigid spaces with $Y$ reduced. Then there exists a unique (up to a unique isomorphism) $\eta$-normal formal model $\fY$ of $Y$ together with a finite map $\mathfrak{f}\colon \fY\to \fX$ with $\mathfrak{f}_\rig = f$. We call it the {\bf $\eta$-normalisation} of $\fX$ in $Y$. It can be characterised uniquely by the following two properties.
\begin{enumerate}[(i)]
    \item Its formation commutes with \'etale base change (see \cite[Appendix~A]{AchingerLaraYoucis2022}).
    \item If $\fX = \Spf(\mathcal{A})$ for a tft $K^\circ$-algebra $\mathcal{A}$, and $Y = \Spa(B)$ for a finite algebra $B$ over $A = \mathcal{A}_K$, then $\fY = \Spf(\mathcal{B})$ where $\mathcal{B}$ is the integral closure of $\mathcal{A}$ in $B$. 
\end{enumerate}
Since every semistable formal scheme is $\eta$-normal (see \cref{prop:log-sm-vertical}), the above construction applies in our situation of interest. We warn the reader that the construction of $\mathfrak{Y}$ (or more precisely the finiteness of $\mathcal{A}\to\mathcal{B}$ in (ii)) relies on our assumptions on $K$, see \cite[Remark~4.3]{AchingerLaraYoucis2022}.

The analogue of \cref{lem:local-ring-val-ring} for formal schemes is the fact (featured in \cref{rmk:sp-facts}) that for an $\eta$-normal tft formal scheme over $K^\circ$ and a generic point $\beta$ of $\fX_k$, the local ring $\cO_{\fX,\beta}$ is a rank one valuation ring. If $X$ is its rigid generic fibre, let us call a point $\zeta\in X$ a {\bf boundary point} of $\fX$ if its specialisation ${\rm sp}(\zeta)$ is a generic point of the special fibre $\fX_k$. Then the set of boundary points maps bijectively onto the set of generic points of $\fX_k$. Moreover, if $\beta = {\rm sp}(\zeta)$, then $\cO_{\fX,\beta}=\cO_{X,\zeta}^+=k(\zeta)^+$, and hence $k(\zeta)^\succ = k(\beta)$.

\begin{thm}[Purity, formal scheme version] \label{thm:Abhyankar-formal-more-precise}
  Let $\fX$ be a semistable formal scheme over $K^\circ$ and let $X=\fX_{\rm rig}$.  Assume that $K$ is either algebraically closed or discretely valued. 
 Let $Y\to X$ be a finite \'etale map and let $\fY\to \fX$ be the $\eta$-normalisation of $\fX$ in $Y$. The following are equivalent:
  \begin{enumerate}[(a)]
    \item \label{propitem:f-Abhy-tame} 
        $Y\to X$ is tame relative to $X$,
    \item \label{propitem:f-Abhy-tame-boundary}
        $Y\to X$ is tame above every boundary point of $\fX$, 
    \item \label{propitem:f-Abhy-ket}
        the morphism $\fY^{\rm log}\to\fX^{\rm log}$ is Kummer \'etale.
  \end{enumerate}
\end{thm}

\begin{proof}
The implication \ref{propitem:f-Abhy-tame}$\Rightarrow$\ref{propitem:f-Abhy-tame-boundary} is obvious and  \ref{propitem:f-Abhy-ket}$\Rightarrow$\ref{propitem:f-Abhy-tame} follows as in the scheme case (see the proof of \cref{thm:scheme-Abhyankar}). It remains to show that  \ref{propitem:f-Abhy-tame-boundary} implies \ref{propitem:f-Abhy-ket} by reducing to the case of schemes. Since the assertion is \'etale local on $\fX$, we may assume that $\fX$ is affine and \'etale over 
\[
    \mathfrak{W} = \Spf(\widehat{\mathcal{R}}), \qquad \mathcal{R} = K^\circ[T_1, \ldots, T_n]/(\pi - T_1\cdot\ldots\cdot T_s),
\]
where $\mathcal{R}$ is given the $\pi$-adic topology. By \cite[Corollary~1.7.3(iii)]{HuberBook} we find an \'etale ring homomorphism $\mathcal{R}\to \mathcal{A}$ such that $\fX = \Spf(\mathcal{A})$ (where $\mathcal{A}$ is endowed with the $\pi$-adic topology). Using \cref{finite-etale-Huber-pair} we can replace $\mathcal{A}$ with an \'etale $\mathcal{A}$-algebra with the same formal spectrum so that $Y\to X$ is induced by a finite \'etale morphism $\mathcal{A}[1/\pi]=A\to B$.

Consider the morphism of schemes of finite type over $K^\circ$
\[ 
    \Spec(B) \la \Spec(\mathcal{A}).
\]
Here $\Spec(\mathcal{A})$ is \'etale over $\Spec(\mathcal{R})$ and hence semistable, and $\Spec(B)\to \Spec(A) = \Spec(\mathcal{A}[1/\pi])$ is finite \'etale. Moreover, since $|\Spec(\mathcal{A}/\pi \mathcal{A})| = |\fX|$, the boundary points of $\fX$ are in bijection with the generic points of $\Spec(\mathcal{A}/\pi \mathcal{A})$. 

Let $\beta \in \Spec(\mathcal{A})$ be a generic point of $\Spec(\mathcal{A}/\pi \mathcal{A})$. Then $\cO_{\Spec(\mathcal{A}), \beta}$ is a valuation ring by \cref{lem:local-ring-val-ring}. We claim that $\Spec(B)\to \Spec(A)$ is tame above $\beta$ in the sense of \cref{thm:scheme-Abhyankar}. We have an extension of $\pi$-adically separated valuation rings 
\[ 
    \cO_{\Spec(\mathcal{A}), \eta} \la \cO_{\fX, \eta} = \cO_{X, \nu}^+ = k(\nu)^+.
\]
Here $\eta\in \fX$ denotes ``the same'' point on the formal scheme, and $\nu\in X = \fX_{\rm rig}$ is its unique preimage under the specialisation map. This map of valuation rings induces an isomorphism on $\pi$-adic completions, as for every $f\in \mathcal{A}$ the map $\mathcal{A}[f^{-1}]\to \widehat{\mathcal{A}}\langle f^{-1}\rangle = \widehat{\mathcal{A}[f^{-1}]}$  induces an isomorphism mod $\pi^n$ for all $n$. Passing to the colimit over all $f\in \mathcal{A}$ with $f(\beta)\neq 0$ we obtain the assertion. Therefore, tame extensions of $\cO_{\Spec(\mathcal{A}), \eta}$ correspond to tame extensions of $k(\nu)^+$.  

We are now in position to apply \cref{thm:scheme-Abhyankar} to $\Spec(B) \to \Spec(\mathcal{A})$, obtaining that if $\mathcal{B}$ is the integral closure of $\mathcal{A}$ in $B$, then $\Spec(\mathcal{B})\to\Spec(\mathcal{A})$ is Kummer \'etale with respect to the standard log structures. Let $\fY' = \Spf(\mathcal{B})$. Then $\fY'\to \fX$ is Kummer \'etale, and hence by \cref{prop:log-sm-vertical} we have that $\fY'$ is $\eta$-normal and it has the standard log structure. Clearly we have $\fY'_\rig = \Spa(B, \mathcal{B}) = Y$, and hence $\fY'\simeq \fY$. Thus $\fY\to \fX$ is Kummer \'etale, as desired.
\end{proof}

\begin{cor}
\label{cor:generic fibre and eta normalisation}
    Let $\fX$ be a semistable formal scheme over $K^\circ$ endowed with the standard log structure, and let $X=\fX_{\rig}$ be the rigid generic fibre.  Assume that $K$ is either algebraically closed or discretely valued. 
    The functor $\fY \mapsto Y = \fY_{\rm rig}$ induces an equivalence of Galois categories
    \[
    \FEt_\fX \longrightarrow \FEtt_{X/X} =\FEtt_X
    \]
    between finite Kummer \'etale covers of the log scheme $\fX$ and \'etale covers of $X$ which are tame (relative to $X$, see \cref{def:tame-rig}).
    The inverse functor is given by $\eta$-normalisation endowed with the standard log structure.
\end{cor}

\begin{proof}
Combining the four assertions of \cref{prop:log-sm-vertical},
if $\fY\to \fX$ is a Kummer \'etale map, then $\fY$ is $\eta$-normal and its log structure is the standard log structure. It now follows from \cref{thm:Abhyankar-formal-more-precise}~\labelcref{propitem:f-Abhy-ket}$\Rightarrow$\labelcref{propitem:f-Abhy-tame} that the generic fibre functor sends $\FEt_\fX$ into $\FEtt_{X}$.  Similarly, it follows from \cref{thm:Abhyankar-formal-more-precise}~\labelcref{propitem:f-Abhy-tame} $\Rightarrow$\labelcref{propitem:f-Abhy-ket} that the $\eta$-normalisation functor sends $\FEtt_{X}$ to $\FEt_\fX$. The two functors are mutual inverses. 
\end{proof}

\subsection{Comparison of tameness conditions}
\label{ss:tameness-comparison-pre}

The goal of this subsection is to prove an auxiliary result comparing tameness at infinity on the special and generic fibres of a formal scheme locally of finite type over~$K^\circ$. A finite \'etale map $f\colon \fY\to \fX$ of formal schemes locally of finite type over $K^\circ$ induces a finite \'etale map of $k$-schemes $f_k\colon \fY_k\to \fX_k$ and a finite \'etale map $f_{\rm rig}\colon \fY_{\rm rig}\to \fX_{\rm rig}$ of rigid spaces over $K$. The question is whether 
\[
    \text{$f_k$ is tame relative to $k$} \quad\overset{?}{\Leftrightarrow}\quad \text{$f_{\rm rig}$ is tame relative to $K$.}
\]
In \cref{prop:tameness-comparison-etale}, we confirm this under the assumption that the formal scheme $\fX$ is $\eta$-normal and that the scheme $\fX_k$ has the \emph{cdh uniformisation property} (\cref{def:cdh-unif}). This result is the key ingredient in the subsequent \S\ref{ss:tameness-comparison}, in which we prove a similar statement regarding finite Kummer \'etale covers of log formal schemes. 

\begin{rmk}
Let us explain the principal mechanism used to compare these two tameness conditions.
If we want to check tameness of $\fY_\rig \to \fX_\rig$ relative to~$K$, we need to consider a root triple $y = (y^\circ,k(y)^+,y^\succ)$ of~$\fY_\rig$ with image $x = (x^\circ,k(x)^+,x^\succ)$ in~$\fX_\rig$ and see whether the extension of valued fields $(k(y),k(y)^+)/(k(x),k(x)^+)$ is tamely ramified.
The valuation ring~$k(x)^+$ is the composition of $k(x^\circ)^+$ and a valuation ring~$V_x$ on the specialisation field~$k(x^\circ)^\succ$ and similarly for~$k(y)^+$ (call the corresponding valuation ring~$V_y$).
Since $\fY_\rig \to \fX_\rig$ is the generic fibre of the \'etale map $\fY\to \fX$, we already know that $(k(y^\circ),k(y^\circ)^+)/(k(x^\circ),k(x^\circ)^+)$ is unramified (and in particular tame).
But tameness behaves well with respect to composite valuations (\cref{composition-tame}), so $(k(y),k(y)^+)/(k(x),k(x)^+)$ is tame if and only if $(k(y^\circ)^\succ,V_y)/(k(x^\circ)^\succ,V_x)$ is tame.

The valued field $(k(x^\circ)^\succ,V_x)$ defines a point~$x_1$ on $\Spa(\fX_k,k)$ by restricting~$V_x$ to the residue field $k(\spe x^\circ)$ of the image of~$x^\circ$ under the specialisation map $\spe \colon X \to \fX_k$.
In a similar way we define the point $y_1  \in \Spa(\fY_k,k)$. 

If $\fY_k \to \fX_k$ is tame relative to~$k$, the extension of valued fields $(k(y_1),k(y_1)^+)/(k(x_1),k(x_1)^+)$ is tamely ramified.
The tricky question is whether this implies that $(k(y^\circ)^\succ,V_y)/(k(x^\circ)^\succ,V_x)$ is tame.
For the converse direction we have to answer the even more difficult question: does the tameness of $(k(y^\circ)^\succ,V_y)/(k(x^\circ)^\succ,V_x)$ imply the tameness of $(k(y_1),k(y_1)^+)/(k(x_1),k(x_1)^+)$?
As a direct conclusion this is not true in general. 
There might be a nontrivial residue extension $k(x^\circ)^\succ/k(x_1)$ and the wild ramification of $k(y_1)/k(x_1)$ might be killed when base changing the extension to $k(x^\circ)^\succ$.
However, we will use that under an extra assumption (the cdh uniformisation property) it is enough to test tameness only at points of a special type (\cref{tameness-support-generic-point}).
For these points we can ensure that $k(x^\circ)^\succ = k(x_1)$.

We start this subsection with the construction of these special points.
\end{rmk}

In the two statements below (\cref{lem:valuation-on-regular} and \cref{lem:valuation-on-cdh-unif}), we use the following terminology. Let $\xi\leadsto x$ be a specialisation of points on a scheme $X$ and let $k(\xi)^+\subseteq k(\xi)$ be a valuation subring. We say that $k(\xi)^+$ has {\bf residue field $k(x)$} if the map $\cO_{X,x}\to \cO_{X,\xi}\to k(\xi)$ has image in $k(\xi)^+$, the map $\cO_{X,x}\to k(\xi)^+$ is local, and induces an isomorphism of residue fields $k(x)\simeq k(\xi)^\succ$. Equivalently, there exists a (necessarily unique) map $\Spec(k(\xi)^+)\to X$ extending $\Spec(k(\xi))\to X$, mapping the closed point $s$ to $x$, and inducing an isomorphism $k(x)\simeq k(\xi)^\succ =k(s)$. 

\begin{lem} \label{lem:valuation-on-regular}
    Let $\xi\leadsto x$ be a specialisation on a scheme $X$ and let $Z=\overline{\{\xi}\}$ with the reduced subscheme structure. Suppose that the local ring $\cO_{Z,x}$ is regular. Then there exists a valuation subring $k(\xi)^+\subseteq k(\xi)$ with residue field $k(x)$.
\end{lem}

\begin{proof}
We may replace $X$ with $\Spec(\cO_{Z,x})$ and thus reduce to the case where $X=Z$ is the spectrum of a regular local ring $R=\cO_{X,x}$ with generic point $\xi$ and closed point $x$. Let $f_1, \ldots, f_d$ be a regular sequence of parameters in $R$, which yields a sequence of specialisations
\[
 \xi = x_0 \rightsquigarrow x_1 \rightsquigarrow \cdots \rightsquigarrow x_d = x,
\]
where $x_i$ corresponds to the prime ideal $(f_1, \dots, f_i)$ and each $\cO_{\overline{\{x_{i-1}\}},x_i}$ is a discrete valuation ring with uniformiser $f_i$, fraction field $k(x_{i-1})$, and residue field $k(x_i)$.
Let $k(\xi)^+$ be the valuation ring of $k(\xi)$ obtained as the composition of all these discrete rank one valuation rings.
It is discrete of rank~$d$ and its residue field $k(\xi)^\succ$ equals $k(x_d) = k(x)$.
\end{proof}

We can generalise the above result to certain non-regular schemes. For this we need the following definition.

\begin{defi} \label{def:cdh-unif}
    A quasi-excellent scheme $X$ has the {\bf cdh uniformisation property} if for every $x\in X$ and every irreducible component $Z\subseteq X$ containing $x$ (with the reduced scheme structure) there exists a dominant finite type morphism $f\colon Z'\to Z$ with $Z'$ integral and regular, and a point $x'\in f^{-1}(x)\subseteq Z'$ with $k(x)=k(x')$.
\end{defi}

It is easy to see that the cdh uniformisation property can be checked Zariski locally.
Moreover, if $Y\to X$ is a smooth map and $X$ has the cdh uniformisation property, then so does $Y$.
However, the property is not étale local as the following example shows.

\begin{ex} \label{ex:nodal-curve-noncdh}
    Consider the nonsplit nodal cubic over $\mathbb{R}$:
    \[
        X = \Spec(\mathbb{R}[x,y]/(y^2 + x^2 - x^3))
    \]
    and let $P=(x=y=0)\in X$ be the node. The normalisation $X'\to X$ is given by $X'=\Spec(\mathbb{R}[t])$ where $x=t^2+1$ and $y=t^3+t=tx$, and the point $P'=(t^2+1=0)\in X'$ is the unique preimage of $P$. Then $X$ does not have the cdh uniformisation property. Indeed, suppose that $f\colon Y\to X$ is a dominant finite type map with $Y$ regular and $Q\in Y$ satisfies $f(Q)=P$ and $k(Q)=k(P)=\mathbb{R}$. Since $Y$ is normal, $f$ factors through $X'$, and then $Q$ maps to $P'$, so that $k(Q)\supseteq k(P')=\mathbb{C}$, contradiction. However, the base change $X_\mathbb{C}$ does have the cdh uniformisation property, as exhibited by its normalisation $X'_\mathbb{C}\to X_\mathbb{C}$.
\end{ex}

The main point about schemes with the cdh uniformisation property is that we can dominate their local rings with valuation rings with the same residue field.

\begin{lem} \label{lem:valuation-on-cdh-unif}
    Let $\xi\leadsto x$ be a specialisation on a scheme $X$. Suppose that the scheme $Z = \overline{\{\xi\}}$ (with the reduced subscheme structure) has the cdh uniformisation property. Then there exists a valuation subring $k(\xi)^+\subseteq k(\xi)$ with residue field $k(x)$. 
\end{lem}

\begin{proof}
Let $Z'\to Z$ and $x'\in Z'$ be as in \cref{def:cdh-unif}. Let $\xi'\in Z'$ be the generic point, which maps to $\xi$. By \cref{lem:valuation-on-regular} there exists a valuation subring $k(\xi')^+\subseteq k(\xi')$ with residue field $k(x')=k(x)$. Then $k(\xi')^+\cap k(\xi)\subseteq k(\xi)$ is a valuation subring with the same residue field. 
\end{proof}

For our purposes, the cdh uniformisation property is important since it ensures that tameness can be checked at valuations with generic support.

\begin{lem}  \label{tameness-support-generic-point}
    Let $f\colon Y\to X$ be a finite \'etale  morphism of quasi-excellent schemes over a base $S$. Suppose that $X$ has the cdh uniformisation property. If 
    \[
        \Spa(f) \colon \Spa(Y,S) \la \Spa(X,S)
    \]
    is tame at all points of $\Spa(Y,S)$ whose support is a generic point of~$Y$, then $\Spa(f)$ is tame everywhere.
\end{lem}

\begin{proof}
Suppose that $\Spa(f)$ is tame at all points with generic support.
We want to show that $\Spa(f)$ is tame at $(y,k(y)^+)\in \Spa(Y, S)$ (we omit the centre of $k(y)^+$ on $S$ from the notation). Let $(x, k(x)^+)$ be its image in $\Spa(X, S)$ and let $x'$ be a generic point of $X$ specialising to $x$. By \cref{lem:valuation-on-cdh-unif}, there exists a valuation subring $k(x')^+\subseteq k(x')$ with residue field $k(x)$. We thus have a map $V=\Spec(k(x')^+)\to X$ sending the generic point to $x'$, the special point to $x$, both without changing the residue field. In particular, the fibre $Y_x$ containing $y$ can be regarded as a closed subscheme of $W = Y\times_X V$. The ring $k(y')^+ = \cO_{W,y}$ is a valuation ring with residue field $k(y)$ and fraction field $k(y')$ for a unique $y'\in Y_{x'}$. Since $y\mapsto x$, the map $k(x')^+\to k(y')^+$ thus constructed is local, so that $(k(y'), k(y')^+)/(k(x'), k(x')^+)$ is an extension of valuation rings. Since $W\to V$ is finite \'etale, this extension is moreover unramified. Composing each of these valuation rings with the valuation rings in the extension 
\[
    (k(y')^\succ=k(y), k(y)^+)  /  (k(x')^\succ = k(x), k(x)^+)
\]
we obtain a point $(y', k(y')^+\times_{k(y)} k(y)^+)\in \Spa(Y, S)$ lying over $(x', k(x')^+\times_{k(x)} k(x)^+)\in \Spa(X, S)$. Since $Y\to X$ is flat, $y'$ is a generic point of $Y$, and thus by our assumption, the corresponding extension of valuation rings is tame. We are now in position to apply \cref{composition-tame}, which implies that $(k(y), k(y)^+)/(k(x),k(x)^+)$ is tame, so that $\Spa(f)$ is tame at $(y, k(y)^+)$.
\end{proof}

In accordance with \cref{def:tame-rig}, we call a finite \'etale  map of separated finite type formal schemes $\fY\to \fX$ {\bf tame relative to $K^\circ$} if the induced finite \'etale map of universal compactifications
\[
    \Spa(\fY^{\rm ad}, K^\circ) \la \Spa(\fX^{\rm ad}, K^\circ)
\]
is tame.
Note that~$\fX^{\rm ad}$ is indeed an adic space and satisfies the hypotheses of \cref{thm:Solodov-comp} (see \cref{rmk:hypotheses-compactification}), so the universal vertical compactification exists and similarly for~$\fY$.
The above definition is clearly local on~$\fX$.
If~$\fX$ is not necessarily separated and quasi-compact, we can define tameness of a finite étale morphism $\fY \to \fX$ by requiring it to be tame when restricted to all affine opens (for which the universal vertical compactification is known to exist).

\begin{prop} \label{prop:tameness-comparison-etale}
    Let $\fX$ be a formal scheme locally of finite type over $K^\circ$ and let $f \colon \fY\to \fX$ be a finite \'etale map.
    Consider the following conditions:
    \begin{enumerate}[(a)]
        \item \label{propitem:tameness-comparison-etale-a} The finite \'etale map of $k$-schemes $f_k \colon \fY_k\to\fX_k$ is tame relative to $k$;
        \item \label{propitem:tameness-comparison-etale-b}The map of formal schemes $f \colon \fY\to\fX$ is tame relative to $K^\circ$;
        \item \label{propitem:tameness-comparison-etale-c} The finite \'etale map of rigid spaces $f_\rig \colon \fY_{\rm rig}\to\fX_{\rm rig}$ is tame relative to $K$.
    \end{enumerate}
    Then \labelcref{propitem:tameness-comparison-etale-a} $\Leftrightarrow$ \labelcref{propitem:tameness-comparison-etale-b} $\Rightarrow$ \labelcref{propitem:tameness-comparison-etale-c}, and if $\fX$ is $\eta$-normal and $\fX_k$ has the cdh uniformisation property, then all three conditions are equivalent.
\end{prop}

\begin{proof}
The assertion is local on~$\fX$, so we may assume that~$\fX$ is separated and quasi-compact.
By \cref{cor:base-change-comp} the generic fibre of the induced map of universal compactifications $\Spa(\fY^{\rm ad}, K^\circ) \to \Spa(\fX^{\rm ad}, K^\circ)$ identifies with $\Spa(\fY_\rig, K) \to \Spa(\fX_\rig, K)$ and its special fibre is $\Spa(\fY_k,k) \to \Spa(\fX_k,k)$.
Since tameness is stable under base change, it is now clear that~\labelcref{propitem:tameness-comparison-etale-b} implies~\labelcref{propitem:tameness-comparison-etale-a} and~\labelcref{propitem:tameness-comparison-etale-c}.

Let us now show that~\labelcref{propitem:tameness-comparison-etale-a} implies~\labelcref{propitem:tameness-comparison-etale-b}.
The morphism $\Spa(\fY^{\rm ad},K^\circ) \to \Spa(\fX^{\rm ad},K^\circ)$ is finite étale by \cref{prop:tame-S-compactification}.
By assumption, it is tame at all points of the special fibre.
Every other point of $\Spa(\fY^{\rm ad},K^\circ)$ has a specialisation lying in the special fibre.
By the openness of the tame locus \cite[Corollary~4.4]{Hubner2021:AdicTameSite} it follows that $\Spa(\fY^{\rm ad}, K^\circ) \to \Spa(\fX^{\rm ad}, K^\circ)$ is tame.

Finally, let us deduce~\labelcref{propitem:tameness-comparison-etale-a} from~\labelcref{propitem:tameness-comparison-etale-c} in case~$\fX$ is $\eta$-normal and~$\fX_k$ has the cdh uniformisation property.
Suppose that $\Spa(\fY_\rig, K) \to \Spa(\fX_\rig, K)$ is tame.
Being étale over~$\fX_k$, the scheme $\fY_k$ inherits the cdh uniformisation property from~$\fX_k$. 
Therefore, by \cref{tameness-support-generic-point}, to show that $\Spa(\fY_k, k) \to \Spa(\fX_k, k)$ is tame, it suffices to test tameness at valuations whose support is a generic point of an irreducible component.
So let us take a valuation ring~$V_y$ on $k(\eta)$ for the generic point~$\eta$ of an irreducible component of $\fY_k$.
Since $\fX$ is $\eta$-normal, so is~$\fY$, and by \cref{rmk:sp-facts}(3) there is a (unique) point $y \in \fY_\rig$ with $\spe(y) = \eta$ for which
\[
k(y)^\succ = k(\eta).
\]
Let $\bar{y}$ be the composition of~$y$ with~$V_y$.
This is a point of~$\Spa(\fY_\rig,K)$.
We now map the whole constellation to~$\Spa(\fX_\rig,K)$:
\[
x = f(y),    \quad   \bar{x} = \Spa(f)(\bar{y}),   \quad   \eta_X = f(\eta),   \quad V_x = V_y \cap k(\eta_X).
\]
Then $k(x)^\succ = k(\eta_X)$ and $(k(\bar{x}),k(\bar{x})^+) = (k(x),k(x)^+) \succ (k(\eta_X),V_x)$.
We know by assumption that the extension of valued fields
\[
(k(\bar{y}),k(\bar{y})^+)/(k(\bar{x}),k(\bar{x})^+)
\]
is tame.
Therefore, \cref{composition-tame} implies that $(k(\eta),V_y)/(k(\eta_X),V_x)$
is tame, i.e.~$\Spa(f_k)$ is tame at~$(\eta,V_x)$.
\end{proof}

\subsection{Tame covers and Kummer covers (II): tameness at infinity}
\label{ss:tameness-comparison}

We consider a  strictly semistable formal scheme $\fX$ over~$K^\circ$ with rigid generic fibre~$\fX_\rig$ and special fibre~$\fX_k$.
When  $K$ is either algebraically closed or discretely valued, by \cref{cor:generic fibre and eta normalisation} 
we have equivalences
\[
 \FEt^t_{\fX_\rig} \simeq \FEt_\fX \simeq \FEt_{\fX_k},
\]
where $\FEt_\fX$ and $\FEt_{\fX_k}$ are the respective categories of finite Kummer \'etale covers.
In this subsection, we add tameness conditions at the boundary to the picture.
Let $\fY \to \fX$ be a finite (Kummer) étale morphism of log formal schemes, let $\fY_\rig \to \fX_\rig$ be its rigid generic fibre (a tame finite étale map of rigid spaces) and let $\fY_k \to \fX_k$ be its special fibre (a finite \'etale map of log schemes over the residue field~$k$). In \cref{prop:tameness-comparison} below, we shall prove that $\fY_\rig \to \fX_\rig$ is tame relative to $K$ (in the sense of \cref{def:tame-rig}) if and only if $\fY_k\to \fX_k$ is tame relative to $k$ (in the sense of \cref{def:tame-log}).
As a consequence we obtain equivalences of categories
\[
 \FEt^t_{\fX_\rig/K} \simeq \FEt_{\fX/K^\circ} \simeq \FEt_{\fX_k/k}.
\]
Ultimately, we shall deduce this result from \cref{prop:tameness-comparison-etale}. In order to apply it, we first need to be able to ensure that $\fX_k$ has the cdh uniformisation property.
We start our investigation with two preparatory lemmas.

\begin{lem} \label{lem:extend-valuative}
    Let $\alpha \colon P \to k$ be a morphism from an fs monoid to the multiplicative monoid of a field.
    Then there is a valuative monoid $V$ with $P \subseteq V \subseteq P^\gp$ such that~$f$ extends to a morphism $f_V \colon V \to k$.
\end{lem}

\begin{proof}
    Let~$F$ be the face $f^{-1}(k^\times) = P \setminus f^{-1}(0)$ of~$P$.
    By the universal property of localisation, we can extend~$f$ to the localisation $P-F$.
    Therefore, we may assume in the following that $f$ is local, i.e.\ that it maps all non-units to~$0$ (and all units to $k^\times$).

    By \cite[Proposition~I.2.4.1]{Ogus} there is a valuative monoid~$V$ dominating~$P$, i.e. $P \subseteq V \subseteq P^\gp$ such that $P \to V$ is local.
    We claim that~$f$ extends to any such monoid (in fact we do not need it to be valuative, it works with any monoid dominating~$P$).
    Restricting~$f$ to the units of~$P$ we obtain a group homomorphism
    \[
        f^\times \colon P^\times \longrightarrow k^\times
    \]
    and we want to extend it to~$V^\times$.
    Note that~$V^\times$ and~$P^\times$ are finitely generated abelian groups as they are both subgroups of the finitely generated abelian group~$P^\gp$.
    Moreover, the quotient $V^\times/P^\times$ is torsion-free by saturatedness of~$P$.
    We conclude that $V^\times/P^\times$ is free and $V^\times \simeq P^\times \oplus (V^\times/P^\times)$.
    Hence, we can easily extend~$f^\times$ to $f_V^\times \colon V^\times \to k^\times$.

    We now define a homomorphism of monoids $f_V \colon V \to k$  by setting $f_V|_{V^\times} = f_V^\times$ and sending all non-units to zero.
    This is indeed an extension of~$f$ because $P \to V$ is local and~$f$ maps all non-units to zero.
\end{proof}

\begin{lem} \label{lem:log-blowup-fibres}
    Let $X$ be an fs log scheme admitting a chart Zariski locally and let $f\colon X'\to X$ be a log blow-up.
    For every point $x\in X$ there exists an $x'\in f^{-1}(x)$ with $k(x')=k(x)$.
\end{lem}

\begin{proof}
Since the formation of log blowups is compatible with strict base change, we easily reduce to the case $X = \Spec(P\xrightarrow{\alpha} k)$ for a field $k$ and an fs monoid $P$. In this case, $X'$ is the blow-up corresponding to a finitely generated ideal subsheaf  $\overline{\mathcal{J}}\subseteq \overline{\cM}_X$. Since the latter is the constant sheaf with value $P/P^\times$ over $\Spec(k)$, the sheaf $\overline{\mathcal{J}}$ is also constant, and hence $X'\to X$ is the log blowup corresponding to an ideal $J\subseteq P$ (the preimage of $\overline{\mathcal{J}}(X)\subseteq P/P^\times$ in $P$). Using \cref{lem:extend-valuative} we can extend $\alpha \colon P \to k$ to a monoid homomorphism $V \to k$ for a valuative monoid~$V$ with $P \subseteq V \subseteq P^\gp$.
In~$V$ the ideal~$J$ becomes principal as all finitely generated ideals of valuative monoids are principal.
By the universal property of a log blow-up (see \cite[Definition~III.2.6.2]{Ogus}, note that this is applicable even though~$V$ is not fine by passing to a suitable fine submonoid of~$V$), the map
\[ 
    \Spec(V\to k) \longrightarrow \Spec(P\xrightarrow{\alpha}k) = X
\]
factors through the blow-up $X'$. This shows that $X'(k)\neq \emptyset$, as desired.
\end{proof}

\begin{prop} \label{lem:log-sm-cdh-unif}
    Let $S = \Spec(M\xrightarrow{\alpha} k)$ where $k$ is a field and $M$ a saturated monoid that is finitely generated over a sharp divisible valuative monoid~$V$, and where $\alpha^{-1}(k^\times)=\{0\}$.
    Let $f\colon X\to S$ be a log smooth morphism which admits a chart Zariski locally; that is, Zariski locally on $X$ there exists a smooth map of saturated monoids $M\to P$ and a strict \'etale morphism
    \[
        X \la \Spec(P\to k[P] \otimes_{k[M]} k).
    \]
    Then the underlying scheme $\underline{X}$ has the cdh uniformisation property.
\end{prop}

\begin{proof}
We are in the situation of \cref{setup:surgery}.
The normalisation map $\nu \colon X^{(0)} \to X$ induces isomorphisms on residue fields by \cref{lem:surgery-properties}.
Moreover, the log scheme~$X_\star^{(0)}$ (which has the same underlying scheme as~$X^{(0)}$ but with the compactifying log structure defined by the locally constant locus in~$X^{(0)}$) is fs and smooth over~$k$ (with the trivial log structure).
By toric resolution of singularities (see \cite[(10.4)]{KatoToricSingularities}) there exists a log blowup $Z \to X_\star^{(0)}$ such that~$\underline{Z}$ is regular.
Now $\underline{Z} \to \underline{X}_\star^{(0)} \to \underline{X}$ is surjective and every point $x \in \underline{X}$ can be lifted to a point $x' \in \underline{X}_\star^{(0)}$ with trivial residue field extension.
Moreover, from \cref{lem:log-blowup-fibres} we can further lift~$x'$ to $z\in \underline{Z}$  with $k(z)=k(x')$.
This shows the cdh uniformisation property.
\end{proof}

We now come back to the setting we are interested in:
we want to study tameness at the boundary of finite Kummer étale morphisms $\fY \to \fX$ of log formal schemes over~$K^\circ$.

\begin{prop} \label{prop:tameness-comparison}
    Let $\fX$ be a strictly semistable log formal scheme over $K^\circ$ and let $\fY\to\fX$ be a finite Kummer \'etale map.  Assume that $K$ is either algebraically closed or discretely valued. 
 Then the following are equivalent:
    \begin{enumerate}[(a)]
        \item the finite Kummer \'etale map of special fibres $\fY_k\to \fX_k$ is tame relative to $k$,
    
        \item the finite \'etale map of generic fibres $\fY_\rig\to\fX_\rig$ is tame relative to $K$.
    \end{enumerate}
\end{prop}

\begin{proof}
We set $M = \Gamma_K^+$. 
Since the assertion is Zariski local on $\fX$, we may assume by \cref{lem:strictly-semistable-formal-scheme-reduces-to-setup364} that there is a chart
\[
    \fX \la \Spf(P \to K^\circ \otimes_{k[M]}k[P])
\]
with $M \to P$ a smooth morphism. 
For a Kummer homomorphism of monoids $P \to Q$ we consider the cartesian squares
\begin{equation} \label{diagram:cover-P-Q}
 \begin{tikzcd}
     \fY_Q    \ar[r]  \ar[d]  
     & \fX_Q   \ar[r]  \ar[d]  &
     \Spf(Q \to K^\circ \otimes_{k[M]}k[Q])  \ar[d]  \\
     \fY      \ar[r]          
     & \fX     \ar[r]       &    \Spf(P \to K^\circ \otimes_{k[M]}k[P]).
 \end{tikzcd}
\end{equation}
Passing to special fibres we obtain an analogous diagram of log schemes over~$k$.
By \cref{cor:fKet-local-structure} we can choose $P \to Q$ in such a way that $\fY_{Q,k} \to \fX_{Q,k}$ is strict étale.
Since strict étaleness can be checked on the special fibre, also $\fY_Q \to \fX_Q$ is strict étale.
The morphism $(\fX_Q)_k \to \fX_k$ is automatically tame relative to~$k$ as it is a torsor under a finite abelian group scheme of degree prime to the residue characteristic. 
So in the base change of the diagram~(\ref{diagram:cover-P-Q}) to $k$, the left and middle vertical arrows are tame relative to~$k$ and therefore the upper left horizontal arrow $(\fY_Q)_k \to (\fX_Q)_k$ is tame relative to~$k$ if and only if the lower left horizontal arrow $\fY_k \to \fX_k$ is tame relative to~$k$ (see \cref{lem:tame-log-properties}\ref{lemitem:tame-log-properties-XQ}). The same reasoning applies to the generic fibres.

This reduces us to proving the proposition for $\fY_Q \to \fX_Q$ instead of $\fY \to \fX$.
In particular, we may assume that our morphism is strict étale.
Note that by construction $\fX_Q\to \Spf(K^\circ)$ admits a global chart such that the special fibre $\fX_{Q,k}^{\log}$ satisfies the \cref{setup:surgery}, in particular because the composition $M \to P \to Q$ is smooth. Therefore \cref{lem:log-sm-cdh-unif} applies, and hence its special fibre $\fX_{Q,k}$ has the cdh uniformisation property. 
By \cref{lem:tame-log-properties}\ref{lemitem:tame-log-properties-strict} we can thus forget about the log structure on~$\fX$ and just study finite étale morphisms $\fY \to \fX$ of $\eta$-normal formal schemes such that~$\fX_k$ has the cdh uniformisation property.
Then the result follows from \cref{prop:tameness-comparison-etale}.
\end{proof}

Combining with \cref{cor:generic fibre and eta normalisation}, we obtain the main result of this section:

\begin{cor} \label{cor:Abhy-infty-equiv}
    Let $\fX$ be a strictly semistable log formal scheme over $K^\circ$. Assume that $K$ is either algebraically closed or discretely valued. 
    The functor ``generic fibre'' 
    $\FEt_{\fX_k}\simeq \FEt_\fX\to \FEt_{\fX_\rig}$ induces equivalences
    \[
        \FEtt_{\fX_k/k} \simeq \FEtt_{\fX/K^\circ} \simeq \FEtt_{\fX_\rig/K}.
    \]
\end{cor}

\cref{lem:log-sm-cdh-unif} is a crucial ingredient in the proof of \cref{cor:Abhy-infty-equiv}. In \cref{lem:log-sm-cdh-unif} it is important that the chart exists Zariski locally and not only étale locally because the cdh uniformisation property is not étale local.
But this property was key to constructing sufficiently nice points where to test tameness.
This is the reason why we work with \emph{strictly} semistable formal models in this subsection. Note that the nodal cubic in \cref{ex:nodal-curve-noncdh} is the special fibre of a semistable but not strictly semistable formal scheme over $\mathbb{R}(\!(\pi)\!)$.

\subsection{Finite generation, finite presentation, and examples}
\label{ss:tamepi1-rig-fgfp}

We are now ready to prove the main results of this section regarding finite generation and finite presentation of $\pitame(X/K)$. 

\begin{thm} \label{thm:fg-pi1-rig}
    Let $K$ be a non-archimedean field whose tame Galois group $\pitame(K)$ is finitely generated. Let $X$ be a qcqs connected rigid analytic space over $K$ with geometric point $\bar x\to X$. 
    Then $\pitame(X/K,\bar x)$ is finitely generated.
\end{thm}

\begin{proof} 
Let $C$ be a completed algebraic closure of $K$. There exists a finite extension $K'\subseteq C$ of $K$ such that the connected components of $X_{K'}$ are geometrically connected (see \cite[Corollary 3.2.3]{Conrad:IrreducibleComponents}). By descent (\cref{cor:descent-of-fg-fett}), we may replace $X$ with one of those components and $K$ with $K'$ and hence assume that $X$ is geometrically connected. By \cref{cor:fg-alg-closed}, we may then replace $K$ with $C$ and $X$ with $X_C$. We have thus reduced to the case where $K$ is algebraically closed. 
   
Let $Y\to X$ be a surjection from a regular (and hence smooth, since $K$ is algebraically closed) qcqs rigid space $Y$, which exists by \cref{prop:h-locally-regular}. By Temkin's uniformisation result recalled as \cref{thm:Temkin-unif}, there exists an \'etale surjection $Z\to Y$ such that $Z$ is qcqs and admits a strictly semistable formal model~$\fZ$.

Let $\fW^{\rm log}$ be a connected component of $\fZ$ endowed with the standard log structure and $W\subseteq Z$ the corresponding connected component of $Z$. By \cref{cor:Abhy-infty-equiv} we have an isomorphism of profinite groups
\[ 
    \pitame(W/K) \simeq \pitame(\fW^{\rm log}_k/k).
\]
Combining this with \cref{thm:pi1tame log is fg}, we conclude that the tame fundamental group of every connected component of~$Z$ is finitely generated. Applying descent (\cref{cor:descent-of-fg-fett}) to the surjective map $Z\to X$, we deduce the finite generation of $\pitame(X/K)$.
\end{proof}

\begin{cor}
    In the setting of \cref{thm:fg-pi1-rig}, the prime-to-$p$ quotient of the étale fundamental group $\pi_1(X)$ is finitely generated (where~$p$ is the residue characteristic exponent of~$K$).
\end{cor}

\begin{proof}
The category of all étale covers of~$X$ whose connected components have Galois closures of degree prime to~$p$ is a full subcategory of $\FEtt_{X/K}$ satisfying the conditions of \cref{lem:Galois-subcategory-criterion}.
The corresponding fundamental group is thus a quotient of $\pitame(X/K)$ and inherits the property of being finitely generated.
\end{proof}

Let us take a look at Berkovich's tame fundamental group~$\pi_1^B(X/K)$ studied in \cite[\S6.3]{Berkovich:EtaleCohomology}.
In \cref{rmk:Berkovich-covers} we noted that $\pi_1^B(X/K)$ is a quotient of $\pitame(X/K)$.
Consequently, we can deduce from \cref{thm:fg-pi1-rig} that it is finitely generated:

\begin{cor}
    In the setting of \cref{thm:fg-pi1-rig}, Berkovich's tame fundamental group $\pi_1^B(X/K)$ is finitely generated.
\end{cor}

The method of proof of \cref{thm:fg-pi1-rig} is inexplicit, in that it does not provide a formula for $\pitame(X/K)$. Below, we show how $\pitame(X/K)$ can be computed in practice.  We use \cref{cor:Abhy-infty-equiv} and the results of \S\ref{ss:surgery} to give a partial result regarding the finite presentation of $\pitame(X/K)$ of a rigid space $X$ over an algebraically closed $K$. Afterwards, we use the same approach to compute some examples.

\begin{thm} \label{thm:fp-sometimes}
    Let $K$ be an algebraically closed non-archimedean field and let $X/K$ be the rigid generic fibre of a connected quasi-compact strictly semistable formal scheme $\fX/K^\circ$. Let $U\subseteq \fX_k$ be the smooth locus of the special fibre (with no log structure). Suppose that $U$ admits a projective snc compactification. Then $\pitame(X/K)$ is finitely presented. 
\end{thm}

\begin{proof}
By \cref{cor:Abhy-infty-equiv} we have $\pitame(X/K)\simeq \pitame(\fX_k/k)$.
The special fibre $\fX_k$ endowed with the restriction of the standard log structure of $\fX$ satisfies the assumptions in \cref{setup:surgery} by \cref{lem:strictly-semistable-formal-scheme-reduces-to-setup364}. We conclude that $\pitame(\fX_k/k)$ is finitely presented thanks to \cref{cor:surgery-fp}.
\end{proof}

\begin{rmk}
    Our methods show that in order to check tameness, it suffices to check a finite number of rank one and two points of the universal compactification. More precisely, in the situation of \cref{thm:fp-sometimes}, let $U\hookrightarrow \overline{U}$ be an snc compactification, let $\beta_1, \dots, \beta_n$ be the generic points of $\fX_k$, and let $\gamma_1, \dots, \gamma_r$ be the generic points of $\overline{U}\setminus U$. As noted in \cref{rmk:sp-facts}, for each $i=1,\dots, n$ there is a rank one point $\zeta_i$ of $X$, the unique point specialising to $\beta_i$. Moreover, for each $i=1,\dots, r$, there is a rank two point $\xi_i$ of the universal compactification $\overline{X}$, whose valuation is described as the composition of the valuation corresponding to the irreducible component $Z$ of $\fX_k$ such that $\gamma_i$ lies in the closure of $Z \cap U$ with the discrete valuation on $k(Z) = k(Z \cap U)$ with valuation ring $\cO_{\overline{U}, \gamma_i}$. Then a finite \'etale cover $Y \to X$ is tame relative to $K$ if and only if it is tame above each of $\zeta_1,\dots,\zeta_n$ and $\xi_1, \dots, \xi_r$. 
\end{rmk}

\begin{lem}[Good reduction]
    Let $X/K$ be the generic fibre of a connected quasi-compact smooth formal scheme $\fX/K^\circ$.  Assume that $K$ is algebraically closed. 
 Let $\fX_k$ be the special fibre of $\fX$ (not equipped with a log structure!). Then $\pitame(X/K) \simeq \pitame(\fX_k/k)$.
\end{lem}

\begin{proof}
Recall that $\fX_k^{\log}$ denotes the special fibre endowed with log structure. Since smooth implies strictly semistable, we have $\pitame(X/K)\simeq \pitame(\fX_k^{\log}/k)$ by \cref{cor:Abhy-infty-equiv}. Since the log structure of $\fX_k^{\log}$ is constant, by \cref{cor:torus-fibration-tame} and the fact that the stalks of $\overline{\cM}_{\fX^{\log}_k}$ are divisible we get $\pitame(\fX^{\log}_k/k) \simeq \pitame(\fX_k/k)$. 
\end{proof} 

\begin{ex}[Product of polydiscs and circles] \label{ex:polydisc}
    Let $X = \Spa(K\langle x_1, \dots, x_n, y_1^{\pm 1},\dots, y_m^{\pm 1}\rangle)$.   
Then $X$ admits a smooth formal model 
    \[
        \fX = \Spf(K^\circ\langle x_1, \dots, x_n, y_1^{\pm 1},\dots, y_m^{\pm 1}\rangle)
    \]
    with special fibre $Y = (\bA^1)^n\times(\mathbb{G}_m)^m$. Suppose that $K$ is algebraically closed. By the K\"unneth formula for $\pitame(-/k)$ of smooth varieties with snc compactification \cite[Thm. 5.1]{Orgogozo2003:AlterationsGroupeFondamental} and the fact that $\pitame(\bA^1/k)=1$ and $\pitame(\mathbb{G}_m/k)\simeq \widehat{\bZ}'(1)$, we have 
    \[ 
        \pitame(X/K) \simeq \pitame(Y/k)\simeq  \widehat{\bZ}'(1)^m. 
    \]
\end{ex}

\begin{ex}[Annulus]
    Let $X=\Spa(K\langle x, y\rangle/(xy-\pi))$ be the affinoid annulus with outer radius $1$ and inner radius $|\pi|$ for some pseudouniformiser $\pi$.  Assume that $K$ is algebraically closed. 
The formal model $\fX = \Spf(K^\circ\langle x,y\rangle/(xy-\pi))$ is strictly semistable and its special fibre $\fX_k$ is the log scheme considered in \cref{ex:semistable-log-curve,ex:semistable-log-curve-2}. Consequently, we obtain
    \[
        \pitame(X/K)\simeq \pitame(\fX_k/k)\simeq \widehat{\bZ}'(1).
    \]
\end{ex}

\begin{ex}[Punctured disc]
   Let $X = \Spa(K\langle x\rangle)\setminus\{0\}$ be the punctured affinoid disc.  Assume that $K$ is algebraically closed. 
   We can write $X$ as the increasing union of affinoid annuli $X_n = \{|\pi|^n\leq |x|\leq 1\}$. 
   Analytification and restriction yield homomorphisms
   \[
   \pitame(X_n/K) \la \pitame(X_{n+1}/K) \la \ldots \la \pitame(X/K) \la \pitame(\mathbb{G}_{m,K}/K).
   \]
   By the previous example, we have $\pitame(X_n/K) \simeq \widehat{\mathbb{Z}}'(1)$ for every $n$ and the transition maps $\pitame(X_n/K)\to \pitame(X_{n+1}/K)$ as well as the composition of maps to $\pitame(\mathbb{G}_{m,K}/K) \simeq \widehat{\mathbb{Z}}'(1)$ are isomorphisms. In order to show that $\pitame(X/K)$ is also isomorphic to $\widehat{\mathbb{Z}}'(1)$, it remains to show that a connected $Y \in \FEtt(X/K)$ remains connected after restriction to $X_n$ for large enough $n$. As any two points in $Y$ can be joined by a connected quasi-compact subset $Z$ and $Y = \bigcup_n Y_{|X_n}$  is an open covering, this subset $Z$ eventually belongs to some $Y_{|X_n}$ and so these points are in the same connected component of $Y_{|X_n}$ for large $n$. Thus the constant colimit $\varinjlim_n \pi_0(Y_{|X_n})$ injects into $\pi_0(Y) = *$. This concludes the proof. 
\end{ex}

The case $n=3$ of the following computation answers a question of Paul Alexander Helminck.

\begin{ex}[Standard semistable formal scheme] \label{ex:helminck}
    Assume that $K$ is algebraically closed. 
    Let $X = \Spa(K\langle T_1,\ldots,T_n\rangle/(\pi - T_1 \cdot \ldots \cdot T_n))$, for some integer $n \geq 1$ and some pseudouniformiser $\pi \in K$. 
    Then $\fX = \Spf(K^\circ\langle T_1, \ldots, T_n\rangle/(\pi - T_1 \cdot \ldots \cdot T_n))$ is a strictly semistable formal model of~$X$ and by applying 
    \cref{lem:strictly-semistable-formal-scheme-reduces-to-setup364}
    we get the following description of its special fibre 
    \[
        \fX_k = \Spec\left(P_{n}(\pi) \to k[T_1, \ldots,T_n]/(T_1 \cdot \ldots \cdot T_n)\right)
    \]   
    where $P_{n}(\pi) = \Gamma_K^+[e_1, \ldots, e_n]/(e_1 + \ldots + e_n = \nu(\pi))$ as in \cref{lem:strictly-semistable-formal-scheme-reduces-to-setup364}, and $e_i \mapsto T_i$.
    By \cref{cor:Paul} we obtain
    \[
        \pitame(X/K)  = \pitame(\fX_k/k) = \pi_1(P_{n}(\pi))
        = \Hom\big(P_{n}(\pi)^\gp, \widehat{\bZ}'(1)\big) 
        \simeq \widehat{\bZ}'(1)^{n-1}.
    \]    
    Of course \cref{cor:Paul} applies also to more general ``polyhedral affinoids''.
\end{ex}

\printbibliography

\end{document}
